\documentclass[10pt,amssymb,amsfonts,psfig]{amsart}
\usepackage{mathrsfs,epsfig,amsfonts,amssymb,color,amsmath,amscd,url,bm}
\usepackage[small,nohug,heads=vee]{diagrams} 
\diagramstyle[labelstyle=\scriptstyle]

\newtheorem{thm}{Theorem}[section]
\newtheorem{cor}[thm]{Corollary}
\newtheorem{lem}[thm]{Lemma}
\newtheorem{prop}[thm]{Proposition}
\newtheorem{obs}[thm]{Observation}
\theoremstyle{remark}
\newtheorem{rem}{Remark}[section]

\newtheorem{problem}[thm]{Problem}

\theoremstyle{definition}

\theoremstyle{plain}
\newtheorem*{thmA}{Main Theorem (physical version)}
\newtheorem*{thmB}{Main Theorem (dynamical version)}
\newtheorem*{thmC}{Theorem (Global Lee-Yang-Fisher Current)}

\newtheorem*{LY Theorem}{Lee-Yang Theorem}
\newtheorem*{LY Theorem2}{General Lee-Yang Theorem}
\newtheorem*{LY TheoremBC}{Lee-Yang Theorem with Boundary Conditions}
\newtheorem*{LY Theorem2BC}{General Lee-Yang Theorem with Boundary Conditions}

\numberwithin{equation}{section}
\numberwithin{figure}{section}

\newcommand{\Rmig}{R} %the Migdal-Kadinoff renormalization in U,V,W coords.
\newcommand{\Cmig}{C} %the invariant ``cylinder'' for \Rmig
\newcommand{\Cmigbl}{C_0}
\newcommand{\TOPmig}{\mathrm{T}} %the invariant ``cylinder'' for \Rmig
\newcommand{\BOTTOMmig}{\mathrm{B}} %the invariant ``cylinder'' for \Rmig
\newcommand{\FIXmig}{b} %the fixed points on the invariant ``cylinder'' for \Rmig
\newcommand{\CFIXmig}{e} %the complex attacting fixed points for \Rmig
\newcommand{\INDmig}{a} %the indeterminant points for \Rmig
\newcommand{\Smig}{S} %the principal locus for \Rmig, i.e. the line U+2V+W=0

\newcommand{\Rphys}{\mathcal R} %the Physical renormalization in Z,T,S coords.
\newcommand{\Cphys}{\mathcal C} %the invariant cylinder for \Rphys
\newcommand{\Solid}{\mathcal{SC}}

\newcommand{\Secphys}{\Pi}
\newcommand{\Secmig}{P}
\newcommand{\TOPphys}{\mathcal T} %Top the invariant cylinder for \Rphys
\newcommand{\Cphystl}{{\mathcal C}_1}
\newcommand{\Cphysbl}{{\mathcal C}_0}
\newcommand{\Cphyslow}{{\mathcal C}_*}
\newcommand{\BOTTOMphys}{\mathcal B} %Top the invariant cylinder for \Rphys
\newcommand{\FIXphys}{\beta} %the fixed points on the invariant cylinder for \Rphys
\newcommand{\CFIXphys}{\eta} %the complex attacting fixed points for \Rphys
\newcommand{\INDphys}{\alpha} %the indeterminant points for \Rmig
\newcommand{\Sphys}{\mathcal S} %the principal locus for \Rphys, i.e. the curve line z^2+2zt+1=0
\newcommand{\Icurve}{\mathcal G} %image of the indeterminacy pts
\newcommand{\Imig}{G} 
\newcommand{\PI}{{\mathcal A}}
\newcommand{\crittemps}{{\mathscr C}} %The set of all critical tempuratures
\newcommand{\epoints}{{\mathscr E}} %The set of all endpoints of high-tempurature hairs

\newcommand{\hexp}{{{\la}^h_{\rm min}}}
\newcommand{\hexpmin}{{{\la}^h_{\rm min}}}

\newcommand{\horexpmin}{{{\la}^\hor_{\rm min}}}
\newcommand{\horexpmax}{{{\la}^\hor_{\rm max}}}

\newcommand{\rh}{{r^h}}
\newcommand{\rhmin}{{r^h_{\rm min}}}
\newcommand{\rhmax}{{r^h_{\rm max}}}

\newcommand{\xis}{{\xi_0}}

\newcommand{\Line}{L}
\newcommand{\Lzero}{L_0}
\newcommand{\Lone}{L_1}
\newcommand{\Ltwo}{L_2}
\newcommand{\Lthree}{L_3}
\newcommand{\Lfour}{L_4}

\newcommand{\LLzero}{\LL_0}
\newcommand{\LLone}{\LL_1}

\newcommand{\Tongue}{\Upsilon}

\newcommand{\Par}{{{\mathcal P}}}
\newcommand{\inv}{{\iota}}

\newcommand{\Mconv}{{\rm M}}
\newcommand{\Hconv}{{\rm H}}
\newcommand{\tconv}{{\rm t}}

\newcommand{\Weight}{W}

\newcommand{\KSQRT}{\widehat \KK}

\newcommand{\ex}{{\mathrm{exc}}}

\newcommand{\vertbondpp}{\, \underset{\oplus}{\overset{\oplus}{|}} \, }
\newcommand{\vertbondpm}{\, \underset{\ominus}{\overset{\oplus}{|}} \,}
\newcommand{\vertbondmp}{\, \underset{\oplus}{\overset{\ominus}{|}} \, }
\newcommand{\vertbondmm}{\, \underset{\ominus}{\overset{\ominus}{|}} \, }

\def\note#1
{
}

\newcommand{\bignote}[1]{}

\newcommand{\QED}{\rlap{$\sqcup$}$\sqcap$\smallskip}

\def\sss{\subsubsection}

\newcommand{\correspond}{\Psi}

\newcommand{\rank}{\rm rank}
\newcommand{\di}{\partial}
\newcommand{\dibar}{\bar\partial}

\newcommand{\ra}{\rightarrow}

\newcommand{\imply}{\Rightarrow}

\newcommand{\hor}{{ah}}
\newcommand{\ver}{{av}}

\def\ssk{\smallskip}
\def\msk{\medskip}
\def\bsk{\bigskip}

\def\nin{\noindent}

\def\sm{\smallsetminus}

\def\tr{{\text{tr}}}

\newcommand{\ctg}{\operatorname{ctg}}
\newcommand{\diam}{\operatorname{diam}}
\newcommand{\dist}{\operatorname{dist}}

\newcommand{\cl}{\operatorname{cl}}
\newcommand{\inter}{\operatorname{int}}
\renewcommand{\mod}{\operatorname{mod}}

\newcommand{\tl}{\tilde}

\renewcommand{\Re}{\operatorname{Re}}
\renewcommand{\Im}{\operatorname{Im}}

\newcommand{\orb}{\operatorname{orb}}

\newcommand{\supp}{\operatorname{supp}}
\newcommand{\id}{\operatorname{id}}

\newcommand{\area}{\operatorname{area}}

\renewcommand{\Im}{\operatorname{Im}}

\newcommand{\Ker}{\operatorname{Ker}}
\newcommand{\tg}{\operatorname{tg}}

\newcommand{\isom}{\approx}

\newcommand{\Conf}{{\mathrm{Conf}}}

\newcommand{\Bottom}{{\mathrm{B}}}

\newcommand{\const}{\mathrm{const}}
\def\loc{{\mathrm{loc}}}

\newcommand{\eps}{{\epsilon}}

\newcommand{\De}{{\Delta}}
\newcommand{\de}{{\delta}}
\newcommand{\la}{{\lambda}}
\newcommand{\La}{{\Lambda}}
\newcommand{\si}{{\sigma}}

\newcommand{\Om}{{\Omega}}
\newcommand{\om}{{\omega}}

\newcommand{\al}{{\alpha}}

\newcommand{\AAA}{{\mathcal A}}
\newcommand{\BB}{{\mathcal B}}
\newcommand{\CC}{{\mathcal C}}
\newcommand{\DD}{{\mathcal D}}
\newcommand{\EE}{{\mathcal E}}

\newcommand{\II}{{\mathcal I}}
\newcommand{\FF}{{\mathcal F}}

\newcommand{\HH}{{\mathcal H}}
\newcommand{\KK}{{\mathcal K}}
\newcommand{\LL}{{\mathcal L}}

\newcommand{\NN}{{\mathcal N}}

\newcommand{\PP}{{\mathcal P}}
\newcommand{\QQ}{{\mathcal Q}}

\newcommand{\RR}{{\mathcal R}}
\newcommand{\SSS}{{\mathcal S}}
\newcommand{\TT}{{\mathcal T}}

\newcommand{\UU}{{\mathcal U}}
\newcommand{\VV}{{\mathcal V}}
\newcommand{\WW}{{\mathcal W}}

\newcommand{\YY}{{\mathcal Y}}

\newcommand{\C}{{\Bbb C}}

\newcommand{\D}{{\Bbb D}}

\newcommand{\N}{{\Bbb N}}

\newcommand{\R}{{\Bbb R}}
\newcommand{\T}{{\Bbb T}}

\newcommand{\Z}{{\Bbb Z}}

\newcommand{\LLINV}{{\mathcal L}_{\rm inv}}

\def\Bh{{\mathbf{h}}}

\def\BJ{{\mathbf{J}}}

\def\BJ{{\mathbf{J}}}

\def\B0{{\mathbf{0}}}

\newcommand{\Jac}{\operatorname{Jac}}

\newcommand{\CP}{ {\Bbb{CP}}   }

\catcode`\@=12

\def\Empty{}
\newcommand\oplabel[1]{
  \def\OpArg{#1} \ifx \OpArg\Empty {} \else
  	\label{#1}
  \fi}
		
\newcommand{\comm}[1]{}
\newcommand{\comment}[1]{}

\begin{document}

\bigskip\bigskip

\title[Lee-Yang zeros ]{Lee-Yang zeros for DHL \\  % for the Diamond Hierarchical lattice\\
         and 2D rational dynamics, \\
        {\tiny I. Foliation of the physical cylinder.}}
\author {Pavel Bleher, Mikhail Lyubich and Roland Roeder}
%\thanks{}
%   This work was supported in part by Sloan Research Fellowship
% and NSF grants DMS-8920768 and DMS-9022140.}
\date{\today}

\def\IMSmarkvadjust{0 pt}
\def\IMSmarkhadjust{0 pt}
\def\IMSmarkhpadding{0 pt}
\def\IMSpubltext{Published in modified form:}
\def\SBIMSMark#1#2#3{
 \font\SBF=cmss10 at 10 true pt
 \font\SBI=cmssi10 at 10 true pt
 \setbox0=\hbox{\SBF \hbox to \IMSmarkhpadding{\relax}
                Stony Brook IMS Preprint \##1}
 \setbox2=\hbox to \wd0{\hfil \SBI #2}
 \setbox4=\hbox to \wd0{\hfil \SBI #3}
 \setbox6=\hbox to \wd0{\hss
             \vbox{\hsize=\wd0 \parskip=0pt \baselineskip=10 true pt
                   \copy0 \break%
                   \copy2 \break% 
                   \copy4 \break}}
 \dimen0=\ht6   \advance\dimen0 by \vsize \advance\dimen0 by 8 true pt
                \advance\dimen0 by -\pagetotal
	        \advance\dimen0 by \IMSmarkvadjust
 \dimen2=\hsize \advance\dimen2 by .25 true in
	        \advance\dimen2 by \IMSmarkhadjust

%
%   Check for publication info
%
%  \newread\jref
  \openin2=publishd.tex
  \ifeof2\setbox0=\hbox to 0pt{}
  \else 
     \setbox0=\hbox to 3.1 true in{
                \vbox to \ht6{\hsize=3 true in \parskip=0pt  \noindent  
                {\SBI \IMSpubltext}\hfil\break
                \input publishd.tex 
                \vfill}}
  \fi
  \closein2
  \ht0=0pt \dp0=0pt
 \ht6=0pt \dp6=0pt
 \setbox8=\vbox to \dimen0{\vfill \hbox to \dimen2{\copy0 \hss \copy6}}
 \ht8=0pt \dp8=0pt \wd8=0pt
 \copy8
 \message{*** Stony Brook IMS Preprint #1, #2. #3 ***}
}

\SBIMSMark{2010/4}{September 2010}{}

 \begin{abstract} 
  In a classical work of the 1950's, Lee and Yang proved that the zeros of the partition functions of a
ferromagnetic Ising models always lie on the unit circle. Distribution of these zeros is physically important 
as it controls phase transitions in the model. We study this distribution for 
the  Migdal-Kadanoff  Diamond Hierarchical Lattice (DHL). In this case, it can be described in terms
of the dynamics of an explicit rational function $\RR$ in two variables (the renormalization transformation).
We prove that $\RR$ is partially hyperbolic on an invariant cylinder $\CC$. 
The Lee-Yang zeros are organized in a transverse measure for the central-stable foliation of $\RR|\, \CC$. 
Their distribution is absolutely continuous. 
Its density is $C^\infty$ (and non-vanishing) below the critical temperature.
Above the critical temperature, it is  $C^\infty$ on a open dense subset,
but it vanishes on the complementary Cantor set of positive measure. 
% From the global complex point of view, these distributions get interpreted as slices of a dynamical $(1,1)$-current
% in the projective space.
This seems to be the first occasion of a complete rigorous description of the Lee-Yang distributions
beyond 1D models. 
% seems to be fair, after all
\end{abstract}

\setcounter{tocdepth}{1}
 
\maketitle
\tableofcontents

\comm{ \section{Plan}

-- Introduction;

-- Description of the model;

-- Structure of the RG transformation:
   - invariant cylinder,
   - in $\CP^2$ (indeterminacy and critical loci in various coordinates),
   - preimages of lines and degrees,
   - complex fixed point; Fatou and Julia sets,
   - solid cylinder theorem;

-- Pluri-potential theory;

-- Low temperature dynamics; 

-- High temperature dynamics;

-- Invariant cone field and dominated splitting;

-- Partial hyperbolicity (horizontal expansion);

-- Vertical foliation and distribution of zeros;

-- Intertwined basins;

-- Thermodynamical consequences: critical exponents etc.}

\section{Introduction}

\subsection{Phenomenology of Lee-Yang zeros}
The Ising model is designed to describe magnetic matter and, in particular, to
explain the appearance of spontaneous magnetization in ferromagnets
% (first-order phase transitions) 
and transitions between ferromagnetic and paramagnetic phases as the temperature $T$ varies. 
%(second-order phase transitions).

The matter in a certain scale is represented by a graph $\Gamma$.
Let $\VV$ and $\EE$ stand respectively for the set of its vertices 
(representing atoms) and edges (representing magnetic bonds between the atoms). 
%Two vertices, $v$ and $w$, connected by an edge are called neighbors.
%we respectively write $(v,w)\in \EE$ or  $\{v,w\}\in \EE$ for the corresponding 
%oriented or disoriented :) edge. 

 A magnetic state of the matter is represented by 
 {\it spin configuration} $\si: \VV\ra \{\pm 1\}$  on $\Gamma$. 
The spin $\si(v)$ represents a magnetic momentum of an atom  $v\in \VV$. 
The total magnetic momentum of the configuration is equal to
\begin{equation}\label{M conventional}
    {\Mconv}(\si)=  \sum_{v\in \VV} \si(v).
\end{equation}
  Each configuration $\si$ has energy $H(\si)$ depending 
on the interactions $J(v,w)$ between the atoms and the external magnetic field $h(v)$.
In the simplest isotropic case, $J$ and $h$ are constants, 
and the Hamiltonian assumes the form:   
\begin{equation}\label{Hamiltonian}
       \Hconv(\si) = -J \sum_{\{v,w\} \in \EE}  \si(v)\si(w) -  h \Mconv (\si),
\end{equation}
where the first sum accounts for the energy of interactions between the atoms
% (where only the neighbors interact), 
while the second one accounts to the 
energy of interactions of the matter with the external field.

 By the
Gibbs Principle, the spin configurations are distributed according to the
Gibbs measure that assigns to configuration $\si$ a probability proportional to
its {\it Gibbs weight} $\Weight(\si)= \exp (-H(\si)/T)$, where $T$ is the
temperature.  Various observable magnetic quantities (e.g., magnetization $M$)   % or free energy $F$)
are calculated by averaging of the corresponding functionals (e.g., $\Mconv(\si)$) % or $F(\si)$)
over the Gibbs distribution.    

The total Gibbs weight $Z= \sum \Weight(\si)$ is called the {\it partition function}.
%  or the {\it statistical sum}.  
It is a Laurent polynomial in two
variables $(z,\tconv)$, where $z=e^{-h/T}$ is a ``field-like'' variable and
$\tconv =e^{-J/T}$ is ``temperature-like''.%
\footnote{We will often refer to them as just ``field'' and ``temperature''.} 
For a fixed $\tconv$, the complex zeros of $Z(z,\tconv)$ in $z$ are called the
{\it Lee-Yang zeros}.  Their role comes from the fact that some important observable
quantities can be calculated as electrostatic-like potentials of the equally
charged particles located at the Lee-Yang zeros.  (For instance, the free
energy is equal to the logarithmic potential of such a family of particles.)

A celebrated  Lee-Yang Theorem \cite{YL,LY} asserts that for the
ferromagnetic\footnote{i.e., with $J>0$, which favors the same orientation of neighboring spins} 
Ising model on any graph, for any real temperature $T>0$, 
{\it the Lee-Yang zeros lie on the unit circle $\T$} in the complex plane (corresponding
to  purely imaginary  magnetic field $h=-iT\phi$).  \footnote{ We will take a
liberty to use either $z$-coordinate or the angular coordinate $\phi=\arg z
\in\R/2\pi Z$ on $\T$ without a comment.  }

Magnetic matter in various scales can be modeled by a hierarchy of graphs
$\Gamma_n$ of increasing size (corresponding to finer and finer scales of matter).  For
suitable models,  the Lee-Yang zeros of the partition functions $Z_n$ will have an asymptotic
distribution $d\mu_\tconv = \rho_\tconv d\phi/2\pi$ on the unit circle.
This distribution  supports singularities of the magnetic observables (or
rather, their  thermodynamical limits), and hence it captures phase transitions
in the model.  For instance, Lee and Yang showed that the spontaneous magnetization of the matter (as
the external field vanishes) is equal to  $\rho_\tconv(0)$.  So, the matter is
ferromagnetic (meaning that it exhibits non-zero spontaneous magnetization) at
temperature $\tconv$  if and only if $\rho_\tconv(0)>0$.

The Lee-Yang zeros for the 1D Ising model with periodic boundary conditions
(corresponding to the hierarchy of cyclic lattices $\Gamma_n=\Z/n\Z$) can be
explicitly calculated using the transfer matrix technique (see e.g., \cite{Baxter}):
\begin{equation}\label{in11a}
z_k^\pm=e^{i\phi_k^\pm},\quad \phi_k^\pm=
\pm \arccos \left[\sqrt{1-\tconv^4}\,\cos\left(\frac{\pi(k+1/2)}{n}\right)\right]\,;\qquad
k=0,1,\ldots,n-1;
\end{equation}
\noindent
see Appendix \ref{APP:STATMECH}.
Their asymptotic distribution is supported on two symmetric intervals, 
$I^+=[\phi^*,\pi-\phi^*]$ and $I^-=-I^+$, where $\cos \phi^* = \sqrt{1-\tconv^4}$,  
and its density is equal to
\begin{equation}\label{in11}
\rho_\tconv (\phi)=\frac{|\sin\phi|} {2\pi\sqrt{1-\tconv^4-\cos^2 \phi}} .
\end{equation}
We see that for positive temperature, the support $I^+\cup I^-$ does not
contain point $\phi=0$, and so the matter is paramagnetic and there are no
phase transitions.  As $T\to 0$ the gap between $I^+$ and $I^-$   closes up and
the Lee-Yang zeros get equidistributed on the unit circle (so, in this model,
the matter becomes ferromagnetic only at the zero-temperature limit).  Note
that $\rho_\tconv$ is real-analytic on $I^\pm$ and has power-like singularities
with exponent $(-1/2)$ at the end-points.

%\bignote{Is it worth a footnote explaining the distinction between
%``ferromagnetic phase'' (as used above) and ``ferromagnetic model'', as used
%below?}

For the ferromagnetic Ising model on lattices $\Z^d$ with $d\ge 2$, a similar
picture is believed to be true for high temperatures (above some critical
temperature  $T_c>0$), while below $T_c$ the Lee-Yang distributions are
conjectured to have full support with positive density. 
% \bignote{Is there a good (theoretical or computer) evidence for this conjecture?}  
This scenario would lead to a second-order phase transition: a
ferromagnet for  $T<T_c$ turns into a paramagnet for $T> T_c$.  However, these
conjectures are hard to prove rigorously as no exact formulas for the Lee-Yang
zeros are available.  

%\begin{rem}
For the two-dimensional lattice, the phase transitions can be rigorously
justified by means of the Onsager exact solution, see \cite{Baxter}.  In all
dimensions $d>1$, it was proven that for high temperatures, the Lee-Yang zeros
do not accumulate on the point $\phi=0$ (no spontaneous magnetization)
\cite{GMR,Rue1}, while for low temperatures, they have positive density at
$\phi=0$
(the spontaneous magnetization is observed) \cite{Peierls,GRIFFITHS}.  %See also \cite[Ch 5]{Ruelle_book}.
However, unlike the one-dimensional Ising model, 
for sufficiently low temperatures, $\rho_\tconv(\phi)$ is not 
real-analytic at $\phi = 0$  \cite{ISAKOV}, see Remark \ref{REM:NON_ANALYTICITY}.
%\end{rem}

In a recent breakthrough, it was proven in \cite{BBCKK} that the Lee-Yang zeros
$\phi_k^n(\tconv)\in \T$ for the $\Z^d$ Ising model with periodic boundary
conditions, $\Gamma_n = \Z^d/(n\Z)^d$, can be
calculated at sufficiently low temperature $\tconv$ as
\begin{equation}\label{in12}
       %\phi_k^n(\tconv)= g_t\left(\pm \frac{\pi(k+1/2)}{n^d}  + \gamma_n \right)+O(\la^{-n}) \quad k=0,1,\dots, n^d-1
       \phi_k^n(\tconv)= g_\tconv\left( \frac{\pi k}{n^d}  + \frac{\pi}{2n^d} \right)+O(\la^{-n}) \quad k=0,1,\dots, 2n^d-1
\end{equation}
\noindent
where  $\la>1$ and  $g_t(\phi)$ is a $C^2$-diffeomorphism of the circle
% that is a $C^2$ function of $\phi$ and  $t$.  
smoothly depending on $\tconv$. 
In particular, for
sufficiently low temperatures, the limiting density $\rho_\tconv (\phi)$ is $C^2$.

At high temperatures, a quantum field theory interpretation  gives a prediction
of the power exponents of the densities $\rho_\tconv$ near the end-points of
$I^\pm$, see Fisher \cite{Fis1} and Cardy \cite{Car}. For instance, for $d=2$
the exponent is predicted to be  $(-1/6)$, while for $d>6$ it is predicted to be
$1/2$.

Study of the Lee-Yang zeros is an active direction of research in contemporary
statistical mechanics, see \cite{MSh}, \cite{BB}, \cite{Ruelle_ANN} and references therein for recent developments.

\subsection{Diamond hierarchical model}\label{DHL intro}

The Ising model on hierarchical lattices was introduced by Berker and Ostlund \cite{BO} 
and further studied by  Bleher \& \v Zalys \cite{BZ1,BZ2,BZ3}  
and Kaufman \& Griffiths \cite {KG1}).                          

Let $\Gamma$ be an oriented graph with two  vertices marked and ordered.  
The corresponding {\it hierarchical lattice} is a sequence of graphs $\Gamma_n$ with two marked
and ordered vertices such that $\Gamma_0$ is an interval, $\Gamma_1=\Gamma$, and
$\Gamma_{n+1}$ is obtained from $\Gamma_n$ by replacing each edge of $\Gamma_n$
with $\Gamma$ so that the marked vertices of $\Gamma$  match with the  vertices
of $\Gamma_n$ and their order matches with the orientation of the corresponding edges of $\Gamma_n$. 
We then mark two vertices in $\Gamma_{n+1}$ so that they match with
the two  marked vertices of $\Gamma_n$. 

 For instance,  the {\it diamond
hierarchical lattice} (DHL) illustrated on Figure~\ref{FIG:DIAMOND GRAPHS}
corresponds to the diamond graph $\Gamma$.%
\footnote{ In fact, in this case, we do not need to orient $\Gamma$ and order the marked vertices
since the diamond is symmetric with respect to a reflection interchanging the marked vertices.} 
Our paper is fully  %completely 
devoted to this lattice.  

\begin{rem} The definition of the total magnetic momentum that we will use for the DHL 
will be slightly different from (\ref{M conventional}) 
(see (\ref{M}) and Appendix \ref{reg lattice}
for a motivation). Also, we will use  $t: =\tconv^2= e^{-2J/T}$ 
for the temperature-like variable as it  makes formulas nicer. 
\end{rem}

\begin{figure}
\begin{center}
\input{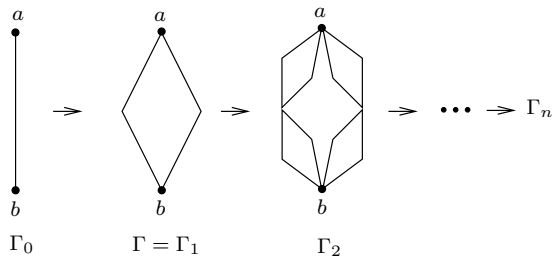}
\end{center}
\caption{\label{FIG:DIAMOND GRAPHS} {Diamond hierarchical lattice (DHL).}}
\end{figure}

 It was  shown in \cite{BZ3} that for any temperature $t \in [0,1]$, the
Lee-Yang zeros for the Ising model on the DHL are dense on the unit circle.
In this paper, we will describe the asymptotic distributions $d\mu_t=\rho_t\,d\phi/2\pi$
  of the Lee-Yang zeros for various temperatures $t\in I=[0,1]$. 
These distributions are illustrated on Figure~\ref{FIG:CYLINDER_BASINS}.  
It shows the cylinder $\CC=\T\times I$  in the
angular coordinate  $ \phi \in [0, 2\pi]$ on the circle $\T$.  In the blue  % dark
region the density $\rho_t$ of the  Lee-Yang distributions is a positive $C^\infty$
function, while in the orange % grey 
region it vanishes. We can see blue % dark  
``tongues" going from the bottom to the top of the cylinder and  orange  % light grey
``hairs'' sticking  from the top.  The tongues fill the cylinder densely.
However, the hairs fill a set of positive area -- in fact,  % this area  gets
of almost full area near the  top.  This creates a false impression that
everything is orange % light 
near the top of Figure \ref{FIG:CYLINDER_BASINS}.   One can also see that
the lowest temperature is reached by hairs for zero field $h$  ($\phi=0$): this
is the critical temperature $t_c$  that separates the ferromagnetic and
paramagnetic phases.

\begin{figure}
\begin{center}
\input{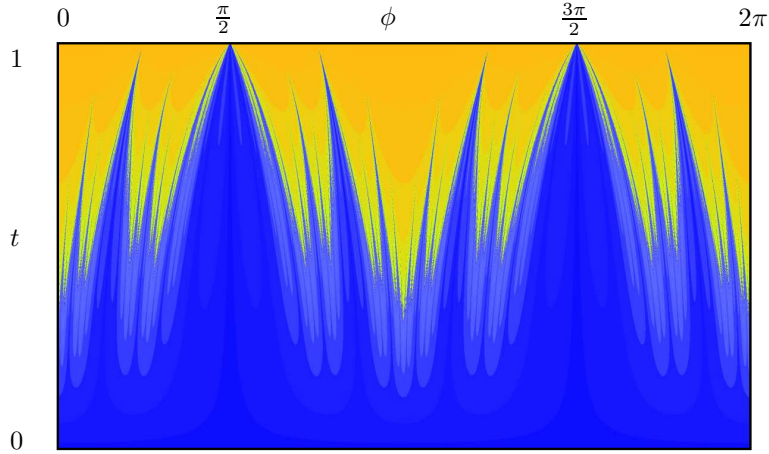}
\end{center}
\caption{\label{FIG:CYLINDER_BASINS}
{Distribution of Lee-Yang zeros and RG dynamics. 
% Dark grey 
Blue is the region where the Lee-Yang distributions have positive $C^\infty$ density.
Dynamically, it is the basin of attraction of the bottom.}} %  $\BOTTOMphys$.}} 
\end{figure}

\msk

Here is a precise statement:

 \begin{thmA} \label{global}
\it{
For any temperature $t\in [0,1)$ the limiting distribution $\mu_t(\phi)$ of the Lee-Yang zeros 
exists and it is absolutely  continuous with respect to the Lebesgue measure on $\T_t \equiv  \T\times \{t\}$:
$d\mu_t=\rho_t(\phi)\, d\phi$. % with continuous density $\rho_t(\phi)$. \note{Is it continuous?} 
It has the following properties: 
\begin{enumerate}
  \item  For $0\leq t<t_c$, the density $\rho_t(\phi)$ is a positive $C^\infty$ function on the circle 
    $\T_t$. Moreover, $\mu_0$ is the Lebesgue measure on $\T_0$ (i.e., $\rho_0(\phi)\equiv 1$).

  \item For $t=t_c$, the density $\rho_t(\phi)$ is a positive $C^\infty$ function on 
  $\T_{t_c} \setminus\{0, \pi\}$ with a power singularity at $\phi=0, \pi$:
$$
      \rho_t(\phi)\asymp |\phi|^\sigma \ \mbox{near $0$}, 
       \quad \rho_t(\phi)\asymp |\phi-\pi|^\sigma \ \mbox{near $\pi$},
$$ 
with some exponent $\sigma=0.064...\in (0,1)$.  

\item For $1>t>t_c$,  
the density $\rho_t$ vanishes on  a Cantor set $K_t \supset \{0,\pi\}$ of positive Lebesgue measure. 
Moreover, its measure tends to $2\pi$ as $t\to 1$.
On each component of the complementary set $O_t = \T_t \sm K_t$,
the density $\rho_t$ is $C^\infty$.
% Moreover, $\rho_t(\phi)$ has a power-like behavior near  $\T\sm O_t$. \note{Is this power-like behavior true?}
% \bignote{It is not generally true that the exponent is equal to 1, and even that there is a particular exponent.}
% In particular, $\rho_t(\phi)\asymp |\phi|$ near $\phi=0$. 
\item For $t=1$, the distribution $d\mu_t$ becomes purely atomic:
  it is supported on a countable dense subset  of $\T_1$. 
% $(2k\pm 1/2)\pi/2^n$, $n=0,1,\dots; \ k=0,\dots, n-1.  
\end{enumerate}
Moreover, there is a family of homeomorphisms $g_t: \T_0\ra \T_t$,
$t\in [0,1)$,  such that $g_t(\phi)$
is smooth in $t\in [0,1)$ for any $\phi\in \T$,  $h_0=\id$, and  $\rho_t= (g_t^{-1})'(\phi)$ a.e. on $\T$. 
For $t<t_c$, the family $g_t(\phi)$ is $C^\infty$ in two variables. 

%There are $2\cdot 4^n$ distinct Lee-Yang zeros $\phi_{k,\pm}^n (t)\in \T$ on
%level $n$ at temperature $t$ given by
%$$
%    \phi_{k,\pm}^n(t)= g^n_t\left(\frac {(2 k\pm 1/2)\, \pi} {4^n}\right) \quad k=0,1,\dots, 4^n-1,
%$$
%where the $g^n_t: \T \ra \T$ are real-analytic functions depending real-analytically on $t$
%that converge uniformly to $g_t$.
}
\end{thmA}

\comment{***************
``Smooth/analytic dependence on $t$'' above means that for any $\phi \in
[0,2\pi]$, the curve $t\mapsto g_t(\phi)$, $t\in [0,1)$,  is smooth/analytic.
Such a family of homeomorphisms $g_t: \T\ra \T$ is called  a {\it
smooth/analytic motion} on the circle.   
%We also prove that $\mu_t= (g_t)_*(\mu_0)$ for some  smooth motion. 
%Moreover, $g_t$ is a real analytic diffeomorphism  for $t<t_c$. 
******************}

\ssk

We see, in particular,  that $\rho_t(0)>0$ below $t_c$ and it vanishes above
$t_c$, so we observe at $t_c$ the ferromagnetic-paramagnetic phase transition.
However, unlike the scenario described above for the standard lattices, the
Lee-Yang zeros do accumulate on $\phi= 0$ even in the paramagnetic phase
$t>t_c$.  So, though  for $t>t_c$,  the magnetization $M(\phi,t)$  vanishes at
$\phi=0$, it is not analytic nearby.   

\ssk

The distributions $\mu_t(\phi)$ described above for the $\Z^d$ and the DHL are
examples of {\em global distributions}.  One can obtain 
{\em tangent distributions} 
as follows.  We fix an arbitrary point $\tl \phi$ inside the
support of $\mu_t(\phi)$ as a reference point and 
% \footnote{In this discussion, we suppose that $\mu_t$ is absolutely continuous.} 
rescale the zeros near $\tl \phi$, by the affine map 
\begin{eqnarray}
    \phi\mapsto \frac{L_n}{2\pi} \ \rho_\tconv (\tl \phi) \cdot (\phi - \tl \phi^n),
% s^n_k = \frac{L_n}{2\pi} \ \rho_\tconv(\tl \phi)(\phi^n_k - \phi^n_{k_0}),
\end{eqnarray}
\noindent
where $L_n$ is the total number of Lee-Yang zeros at level $n$
and $\tl\phi^n$ is the one that is closest to $\tl \phi$.  
We say that the Lee-Yang zeros are {\em locally rigid at $\tl\phi$} 
if  the  % points $s^n_k$ 
rescaled zeros converge locally uniformly to the $\Z$ lattice, as $n \ra \infty$.  

It follows directly from (\ref{in11a}) that the Lee-Yang zeros for the $\Z^1$
lattice are locally rigid everywhere.  For the $\Z^d$ lattice (with periodic boundary
conditions), infinitesimal rigidity follows at sufficiently low temperatures from
(\ref{in12}), since $L_n = 2n^d$.

Because there are $2\cdot4^n$ Lee-Yang zeros at level $n$ for the DHL, an
expression of the form (\ref{in12}) is not sufficient to show their local
rigidity.  Instead, we will show that the LY zeros $\phi_k^n(t) \in O_t$
can be expressed as $g^n_t(\phi_k^n(0))$,
where the $g^n_t$  are diffeomorphisms locally $C^1$  converging to
the  maps $g_t$.  
This is sufficient for the local rigidity, 
 see  Proposition \ref{PROP:LOCAL_RIGIDITY}.
%the Lee-Yang zeros for the DHL are locally rigid everywhere
%in the region $\{|t|< t_c\}$ (and in fact, in the whole
%low temperature basin $\{\rho(\phi,t)>0\}$).

\comment{%%%%%%%%%%%%%%%%%%
However, for the diamond lattice we can express the $2\cdot 4^n$ distinct Lee-Yang zeros
$\phi_k^n (t)\in \T_t$ at level $n$ % and temperature $t$ 
as
\begin{eqnarray}\label{EQN:GN}
    %\phi_{k,\pm}^n(t)= g^n_t\left(\pm \frac{\pi(k+1/2)}{4^n}\right) \quad k=0,1,\dots, 4^n-1,
    \phi_k^n(t)= g^n_t\left(\frac{\pi k}{4^n}\right) \quad k=0,1,\dots, 2 \cdot 4^n-1,
\end{eqnarray}
where the $g^n_t: \T \ra \T$ are diffeomorphisms depending smoothly on $t \in [0,t_c)$, see \S \ref{SEC:BOTTOM BASIN}.  
% In Proposition \ref{PROP:C2_CONVERGENCE} we prove
% that on any compact subset $0 \leq t \leq t_c-\eps$ the $g^n_t(\phi)$ converge to
% $g_t(\phi)$ within the $C^1$ topology.  
It follows that the Lee-Yang zeros for
the DHL are infinitesimally rigid everywhere % below the critical temperature,  
in the region $\{|t|< t_c\}$
(and in fact, in the whole low temperature basin $\{\rho(t)>0\}$).
}%%%%%%%%%%

\msk
Below we will re-interpret the above results in terms of the renorm-group.
% of the dynamics of renormalization transformation. 

\subsection{Migdal-Kadanoff RG equations}\label{MK RG eq-s}

There is a general physical principle that the values of physical quantities
depend on the scale where the measurement is taken.
The corresponding quantities are called {\it renormalized}, and the
(semi-group) of transformations relating them at various scales is called {\it
renorm-group (RG)}. However, it is usually hard to justify rigorously existence
of RG, let alone to find exact formulas for RG transformations.  The beauty of
hierarchical models is that all this can actually be accomplished.
 
In  \cite{Mig1}, \cite{Mig2}, Migdal suggested approximations to  RG for the
classical Ising model on $\Z^d$.  They were further developed by Kadanoff
\cite{Kad}, and became known as the {\it Migdal-Kadanoff approximate RG
equations}.  It was then  noticed by Berker and Ostlund \cite{BO} that these
equations become exact for suitable hierarchical Ising models 
(see also  \cite{BZ1}  and  \cite {KG1}).  In particular,
the DHL corresponds to the 2D lattice $\Z^2$.  The Migdal-Kadanoff RG equations in
this case assume the form:
\begin{equation}\label{R-intro}
    (z_{n+1},t_{n+1}) = \left( \frac{z_n^2+t_n^2}{z_n^{-2}+t_n^2}, \ \frac{z_n^2+z_n^{-2}+2}{z_n^2+z_n^{-2}+t_n^2+t_n^{-2}}\right):=
                       \Rphys(z_n,t_n).
\end{equation}
where $z_n$ and $t_n$ are the renormalized field-like and temperature-like
variables on $\Gamma_n$.  The map $\Rphys$ that relates these quantities is
also called the {\it renormalization transformation}.

The Lee-Yang zeros for $\Gamma_n$ are solutions of the algebraic equation
$Z_n(z, t)=0$, so they form a real algebraic curve $\SSS_n$ on the cylinder
$\CC$ ({\it the Lee-Yang locus of level $n$}), see Figure \ref{FIG:LEE_YANG_ZEROS}.  Equation
(\ref{R-intro}) shows that  $\SSS_n$ is the pullback of $\SSS_0$ under the
$n$-fold iterate of $\Rphys$, i.e.,  $\SSS_n= (\Rphys^n)^*\SSS_0$.  In this way, the
problem of asymptotical distribution of the Lee-Yang zeros is turned into a
dynamical one.

\begin{figure}
\begin{center}
\input{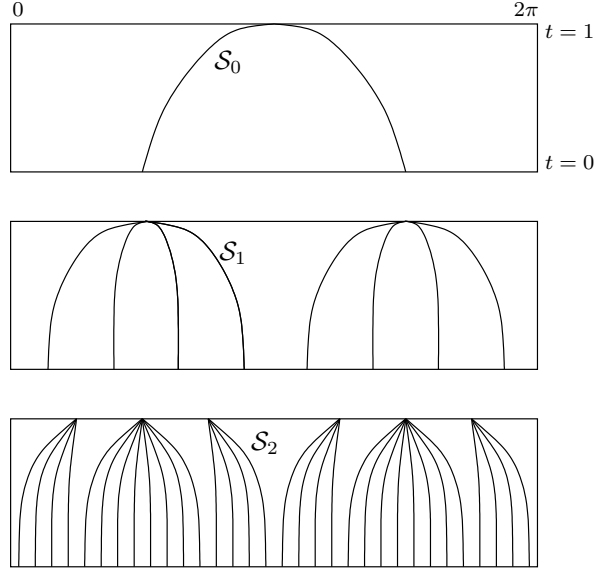}
\end{center}
\caption{\label{FIG:LEE_YANG_ZEROS} {The level $n$ Lee-Yang zeros $\SSS_n$ for $n=0,\,1,$ and $2$.}}
\end{figure}

%The picture behind the above theorem is that these loci converge to a foliation $\FF$ on $\CC\sm \TOPphys$ 
%by smooth vertical curves $t\mapsto g_t(\phi)$, $\phi\in \BOTTOMphys$ (connecting the bottom to the top). 
%Moreover, the $\SSS_n$ get equidistributed (exponentially fast) with respect to  
%the holonomy invariant transverse measure $\mu$ for this foliation
%coinciding with the Lebesgue measure on $\BOTTOMphys$. 

\subsection{Renormalization dynamics on the cylinder}

The first observation is  that the cylinder $\CC$ is $\Rphys$-invariant. 
Next, its bottom $\BOTTOMphys$ is $\RR$-invariant as well, and $\RR$ restricts to $z\mapsto z^4$ on $\BB$.   
Moreover, $\BB$ is superattracting, so there is an open  basin $\WW^s(\BB)$ where
the orbits converge to $\BOTTOMphys$: this is exactly the blue region on Figure~\ref{FIG:CYLINDER_BASINS}. 

The top $\TOPphys$ of $\CC$ is also invariant except for two indeterminacy
points $\alpha_\pm=(\pm \pi/2, 1)$ that ``blow up'' to a curve $\Icurve$ going
across the cylinder (see Figure~\ref{FIG:R_ON_CYLINDER} below).  Because of
this phenomenon, the degree of $\RR$ on the top drops to 2, (namely, $\RR:
z\mapsto z^2$ on $\TT$), and its basin $\WW^s(\TT)$  (roughly, the orange region
on Figure~\ref{FIG:CYLINDER_BASINS}) is not open, but rather  a  ``Cantor
bouquet'' of hairs sticking from $\TOPphys$.  

Despite this,  $\Rphys$ acts in a surprisingly nice way on the proper curves
(i.e., curves connecting the bottom to the top) --  namely,
a proper curve in
$\Cphys$ crossing $\Icurve$ only once lifts to four proper curves,
compare Figure \ref{FIG:LEE_YANG_ZEROS}.  
In this sense, the action of $\Rphys$ on proper curves has degree four.

% Moreover,  $\SSS_{n+1}= \RR^*(\SSS_n)$. 
Our main dynamical result asserts that $\RR$ is {\it partially hyperbolic} on
the cylinder $\Cphystl:= \Cphys\sm \TOPphys$.  This means that $\RR$ admits an
invariant horizontal tangent cone field $\KK^h(x)\subset T_x\CC$ such that the
horizontal tangent vectors $v\in \KK^h(x)$ get exponentially stretched under
iterates of $\RR$. 

Let us also consider the complementary vertical cone field $\KK^v(x)=T_x \CC\sm \KK^h(x)$.
A smooth curve $\gamma(t)$ in $\Cphystl$ going though this cone field 
 is called {\it vertical}.
% A {\it proper} vertical curve is a vertical curve connecting $\BOTTOMphys$ to $\TOPphys$.  
A {\it vertical foliation} on $\Cphystl$ is a foliation
whose leaves are proper vertical curves.

Given a vertical foliation $\FF$, the {\it holonomy} transformations
$g_t: \BOTTOMphys\ra \T\times \{t\}$, $t\in [0,1)$, are defined by the property
that $x$ and $g_t(x)$ belong to the same leaf of $\FF$. 
% They form a smooth motion on $\T$. We say that $\FF$ is {\it absolutely
% continuous/analytic} if all the homeomorphisms $g_t$ and $g_t^{-1}$ are
% absolutely continuous/analytic.

A  {\it central foliation} for $\Rphys$ is an invariant vertical
foliation.

Recall that a measurable map $g: \T\ra \T$ is called {\it absolutely continuous}
if preimages $g^{-1}(X)$ of null-sets $X\subset \T$ are null-sets 
(where a {\it null-set} means a set of zero Lebesgue  measure). 
Note that this is not a symmetric notion:
it may happen that a homeomorphism $g$ is absolutely continuous while the inverse one, $g^{-1}$, is not
(and this is what actually happens below).   

\begin{thmB}\label{Theorem B}
 \it{ The renormalization transformation $\Rphys$ is partially hyperbolic on $\Cphystl$,
and it has a unique  central foliation $\FF^s$.  
This foliation is $C^\infty$ on  $\WW^s(\BOTTOMphys)$
but is not  absolutely continuous on $\WW^s(\TOPphys)$.%
\footnote{Meaning that restrictions of the  homeomorphism $g_t^{-1}$ to $\WW^s(\BOTTOMphys)$ are $C^\infty$
but their restrictions to $\WW^s(\TOPphys)$  are not absolutely continuous.}

Given any proper vertical curve $\gamma$ on $\CC$, the pullback $(R^n)^*\gamma$ comprises $4^n$ 
proper vertical curves, and $(R^n)^*\gamma\to \FF$ exponentially fast (away from the top).  

The basin $\WW^s(\BOTTOMphys)$ is open and dense in $\CC$. The basin $\WW(\TOPphys)$ has positive area,
with density 1 at the top. The union of the two basins has full area in $\CC$. 
}
\end{thmB}
 
Our proof of this geometric result is based upon  Counting Arguments (making use of Bezout's Theorem).
We call this method the ``Enumerative Dynamics''.
 
 A {\it transverse invariant measure} $\mu$ for $\FF$ is a family of measures
$\mu_t$, $t\in [0,1)$,  such that $\mu_t= (g_t)_*(\mu_0)$.  It is uniquely
determined by $\mu_0$.  The Lee-Yang distributions $\mu_t$ form a transverse
invariant measure for $\FF$ equal to the Lebesgue measure on $\BOTTOMphys$.
% Moreover, $\Rphys^* \mu=4\mu$.  

This fact  makes a connection between the physical and dynamical versions of
the Main Theorem that will allow us to derive (easily) the former from the
latter.  

%The slices of our foliation $\FF$ by the dark and light regions are the strong stable foliations of the bottom
%and the top respectively. However, as this foliation is transversally $C^\infty$ in the dark regions,
%it is singular in the light region. This is the dynamical reason why the Lee-Yang distributions 
%become peculiar above the critical temperature.  

Let us also mention that the dynamical picture described in the last part of
the Main Theorem gives one more illustration  of the ``intertwined basins''
phenomenon studied by Kan, Yorke et al  \cite{Kan,AYYK}, and more recently by Bonifant and Milnor \cite{BM}. 

In the upcoming Part II of this work, 
we will study  the global structure of the
renormalization transformation and  zeros of the partition function  ({\it Lee-Yang-Fisher zeros})
in the complex projective space $\CP^2$.
The distribution of the zeros will be interpreted as the dynamical {\it (1,1)-current} of $\RR$,
while the partition function itself will become the {\it  potential} of this current.
In this way the classical Lee-Yang-Fisher Theory gets tightly linked to the contemporary
Dynamical Pluripotential Theory.

\subsection{Structure of  the paper}
Let us now outline the structure of the paper indicating ideas of the proofs.
Since this paper is naturally placed on the borderline of three fields
(statistical mechanics, dynamics, and complex geometry) we have attempted to
make exposition reader-friendly for a non-expert in any one of these fields, by
motivating the problems and supplying needed background and basic references.

We begin in \S \ref{SEC:MODEL} with  relevant background material in
statistical mechanics: description of the Ising model on graphs, formulation of
the Lee-Yang Theorem, and comments on physical significance of the Lee-Yang
zeros.  In particular, we supply explicit formulas for the free energy and
spontaneous magnetization in terms of the asymptotic  distributions of these
zeroes.  Then we pass to the diamond hierarchical model and derive the
Migdal-Kadanoff  Renorm-Group (RG) Equations.  They lead to the renormalization
transformation $\Rphys$.    

In \S \ref{SEC:STRUCTURE} we describe the structure of $\Rphys$ on the invariant
cylinder $\Cphys$.  It is strongly influenced by the presence of two indeterminacy
points $\INDphys\pm=(\pm i, 1)$ on the top $\TT$ that blow up to the curve
$\Icurve$.  (On Figure \ref{FIG:CYLINDER_BASINS}, these points are clearly seen
as the tips of the two main tongues of the blue region.) Because of them, $\RR$
does not evenly cover the cylinder: the region below $\Icurve$ is covered four
times while its complement is covered only twice. 
However, we show that $\Rphys$ acts properly with degree
four on the space of proper vertical curves.
We will derive from here by a Counting Argument  
the Lee-Yang Theorem for the DHL (\S \ref{LY for DHL sec}).

In \S \ref{SEC:GLOBAL STRUCTURE} we follow up  this discussion  with a
description of the global features of $\RR$ on the complex projective space
$\CP^2$: its critical and indeterminacy loci,  superattracting fixed points 
and their separatrices. (In fact, here it is more convenient to deal with the map $R$ that
comes directly from the Migdal-Kadanoff RG Equations, without passing to the
``physical'' $(z,t)$-coordinates.  This map is semi-conjugate to $\RR$ by a
degree two rational change of variable $\CP^2\ra \CP^2$.)
%We note that $R^*$ acts with degree
%four on the first homology.   \note{be careful} Combined with the previous
%discussion of the action on vertical curves in $\CC$, this yields, by a
%counting argument, the Lee-Yang Theorem in our setting. \note{to be added}

In \S \ref{sec: alg cone field} we prove that $\RR$ admits a horizontal
invariant cone field $\KK^\hor(x)$ on $\CC$.  We construct it explicitly by taking
the principal Lee-Yang locus $S\equiv S_0$ (which comprises two vertical
segments) and translating it around the cylinder. It gives us two transverse
vertical foliations on $\CC$.  Then we define  $\KK^\hor(x)$  as the horizontal
cone tangent to these foliations at $x$.  Using Bezout's Theorem, we check
invariance of this cone field. 
Unfortunately,  this cone field degenerates at the top. % $\TOPphys$.
We partially fix this problem 
by  modifying $\KK^\hor(x)$ near the top in such a way that the new field
 $\KK^h(x)$  degenerates only at the indeterminacy points $\INDphys_\pm$.

In \S \ref{SUBSEC:DOMINATED} we prove that $\RR|\, \CC$ admits a {\it dominated splitting}.
This means that there exists a ``vertical'' tangent line field $\LL^c(x)$
and constants $C>0$, $\la>1$ such that
\begin{equation}\label{domination def intro}
      \| D\RR^n(x)\,  w \| \leq C \la^{-n} \| DR^n(x)\, v\| \quad
      \mbox{for any}\ \ x\in \CC\sm \UU, \ w\in \LL^c(x), \ v\in \KK^h(x)
\end{equation}
(where $\UU$ is a neighborhood of the indeterminacy points),
so the ``horizontal'' vectors get stretched exponentially faster than the
``vertical'' ones.% 
\footnote{This property is also referred to as {\it
projective hyperbolicity} of $\RR$.} Integrating the line field $\LL^c(x)$, we
obtain an invariant family of smooth vertical curves filling in the whole
cylinder. However, at this stage of the discussion we do not know yet that the
integration is unique, so the integral curves may not form a foliation.

In \S \ref{SEC:EXPANSION} we prove our main dynamical result  that the map $\RR|\,
\CC$ is {\it horizontally expanding}  (and thus partially hyperbolic).
This means that under iterates the horizontal vectors get stretched exponentially fast:
\begin{equation}\label{hor exp def intro}
    \| D\RR^n(x)\, v \|\geq c \la^n\|v\|, \quad \quad x\in \CC, \ v\in \KK^h(x),
\end{equation}
where $c>0$, $\la>1$. 
To establish this property, we consider a central projection $\pi$ in $\CP^2$ onto the line at infinity.
By a Counting Argument, we show that $\pi\circ R^n$ restricted to the horizontal sections of
the solid cylinder is a Blyschke product $B_n$ (in  appropriate natural coordinates) vanishing at the origin
to order $2^{n+2}$. Such a $B_n$  expands the circle metric at least by $2^{n+2}$,
which gives us (\ref{hor exp def intro}) with $\la = 2$.  
We then provide a second proof of this expanding property that exploits the combinatorics 
of the DHL partition functions and a variant of the Lee-Yang Theorem that we call the {\it LY Theorem with Boundary Conditions}.

\comm{*******
   The proof is based on the Schwarz Lemma applied to complex horizontal
curves.  To make it work, we need to know that these curves enjoy some
topological expansion.  It comes from two properties: 

\nin
a) The action of $\RR$ on real closed horizontal curves on $\CC$ has  degree 4; 

\nin
b) Nearby imaginary points escape to the fixed points 
    (by the result on the dynamics in the solid cylinder of \S \ref{SEC:GLOBAL STRUCTURE}).

\noindent An extra difficulty arises near the indeterminate points $\INDphys_\pm$.
***********}

In \S \ref{SEC:BOTTOM BASIN} we discuss the basin $\WW^s(\BB)$ of the bottom
$\BB$ (the blue region of Figure~\ref{FIG:CYLINDER_BASINS}).  We explain where
the tongues observed on this picture come from and prove that
$\WW^s(\BB)$ supports a $C^\infty$ foliation, the stable foliation of $\BB$. 

% In the next section, \S \ref{SUBSEC:JULIA_AND_FATOU}, 
% we define the Fatou and Julia sets for $R$ and show that the Julia set
% coincides with the closure of preimages of the invariant complex line
% $\{z=1\}$ (corresponding to the vanishing external field).  It is based on
% Green's criteria for Kobayashi hyperbolicity of the complements of several
% algebraic curves in  $\CP^2$ \cite{GREEN:PAMS,GREEN:AJM} that generalize the classical
% Montel Theorem.  We then use this result to prove that points in the interior
% of the solid cylinder $\D \times I$ are attracted to a superattracting fixed
% point $\eta=(0,1)$ of $\RR$. 

In \S \ref{SEC:HIGH_TEMP}, we turn our attention to the top
$\TT$ of $\CC$.  We prove that its basin $\WW^s(\TT)$ contains a comb of curves
of positive measure.  Moreover, the density of its slices by  horizontal
circles  $\T\times \{t\}$ goes to 1 as $t\to 1$. 

We then derive in \S \ref{SEC:TWO BASINS}, applying standard distortion
techniques  to horizontal curves, that almost any orbit on the cylinder
converges either to the bottom $\BB$ or to the top $\TT$ (see
\cite{BLOKH_LYUBICH} for the one-dimensional prototype of this method).

In the next section, \S \ref{SEC:VERTICAL_FOLIATION}, we use horizontal expansion to prove {\it
unique} integrability of the invariant vertical line field $\LL^c$ yielding the
desired invariant central foliation $\FF^c$.  We then collect in \S \ref{SEC:THERM} consequences
about regularity of the the Lee-Yang distributions (as formulated above)  and
calculate various critical exponents.

In the last section, \S \ref{SEC:PERIODIC_LEAVES},
we analyze smoothness  of  periodic leaves that terminate at periodic
points on the top.
Such leaves are real analytic near the bottom and the top,
but we show that they must loose analyticity somewhere in the middle
 (in fact, generically they can have only finite smoothness).
This is another manifestation of the phase transitions in this model.

%In the last section, \S \ref{SEC:POTENTIAL}, we discuss the link between the
%Lee-Yang distributions and dynamical pluripotential theory in $\CP^2$, and
%prove the existence of the Global Lee-Yang-Fisher Current.
%%%%% Namely, we prove that there is a global $(1,1)$-current $\mu^c$ in $\CP^2$
%%%%% such that $R^* \mu^c = 4\mu^c$,  whose slices by the complex lines
%%%%% $t=\const\in [0,1)$ coincide with the Lee-Yang distributions.  

% Since $R$ is algebraicly stable (as shown in \S \ref {}),

We finish with several Appendices.  
% A
 In Appendix \ref{APP:COMPLEX GEOM} we    
 collect needed background in complex geometry: 
rational maps, indeterminacy points and their blow-ups, degrees and divisors. 
% (normality, Kobayashi
% hyperbolicity, currents and their pluri-potentials, line bundles over $\CP^2$, etc.) 

% B
In Appendix \ref{APP:R_NEAR_ALPHA} we  supply some calculations on the cylinder $\Cphys$, 
particularly, near its top $\TOPphys$ and the indeterminacy points $\alpha_\pm$  
(performing their ``blow-ups'' in various coordinates).

% C
In Appendix \ref{SUBSEC:COMPLEXIFICATION_CONES} we construct an extension of $\KK^h(x)$ that is invariant on an
appropriate complex neighborhood of $\Cphys \sm \{\INDphys_\pm\}$.  
It gives us a supply of complex horizontal curves that are used in  \S \ref{SEC:TWO BASINS}
to obtain the Koebe distortion estimates. 

% D
In Appendix \ref{App: crit locus} we describe the global critical locus of the map $\Rmig$
in $\CP^2$.  

% E
In Appendix \ref{SEC:PARTIAL_HYPERBOLICITY} we give a computational verification that
$\Rphys|\,\Cphys$ is  horizontally expanding using the explicit formula for $\Rphys$.
% This
% means that under iterates the horizontal vectors get stretched exponentially
% fast:
% $$
% 5     \| D\RR^n(x)\, v \|\geq c \la^n\|v\|, \quad \quad x\in \CC, \ v\in \KK^h(x),
% $$ 
% where $c>0$, $\la>1$.
Formally speaking, this can substitute the conceptual proof of \S \ref{SEC:EXPANSION}.
However, this would not provide any insight into the nature of $\Rphys$,
neither could it be useful in another related situation. 

% F 
In Appendix \ref{APP:STATMECH} we re-prove the classical Lee-Yang Theorem and
extend it to the LY Theorem with Boundary Conditions
% use the proof to derive the variant needed in the combinatorial proof of expansion from 
used in \S \ref{SEC:EXPANSION}.  We then describe the Lee-Yang zeros in
the one-dimensional model, and explain in what sense the hierarchical lattices
give an approximation to the standard lattices $\Z^d$.  

% G
 In Appendix
\ref{APP:PROBLEMS} we collect several open problems.  
% H
Finally, in Appendix
\ref{APP:NOTATION} we provide a list of notation that are frequently  used
throughout the paper.

% \bignote{Mention Cantat's and Sabot's work on cocycles}

\subsection{Basic notation and terminology}
  $\II=[0,1]$,  
$\C^*=\C\sm \{0\}$, $\T=\{|z|=1\}$, $\D_r=\{|z|<r\}$, $\D\equiv \D_1$, $\D^*=\D \sm \{0\}$, 
 $\N=\{0,1,2\dots\}$. 

Given two variables $x$ and $y$, 
$x\asymp y$ means that $c \leq |x/y|\leq C$ for some constants $C>c>0$. 

A {\it path} (in some topological space) is an embedded interval.

\vbox{
\msk {\bf Acknowledgment.}
This project was designed by the first two authors in May 1991 during Pavel
Bleher's visit to the IMS at  Stony Brook.  It was resumed  in the fall 2005 in
Toronto during the Fields Institute Program on Renormalization in Dynamics and
Mathematical Physics.  The work of the first author is supported in part by the
NSF grants DMS-0652005 and DMS-0969254.  The work of the second author has
been partially supported by NSF, NSERC and CRC funds.

 We thank Robert Shrock for interesting discussions and comments.
(In particular, formula (\ref{vert exp}) answers one of Shrock's
questions.) 
}

\section{Description of the model}
\label{SEC:MODEL}

\subsection{Background: Ising models on graphs}\label{background}
   Let $\Gamma$ be a graph   
representing a magnetic matter in  a certain scale. 
Let $\VV$ and $\EE$ stand respectively for the set of its vertices 
(representing atoms) and edges (representing magnetic bonds between the atoms). 
Two vertices, $v$ and $w$, connected by an edge are called neighbors:
we respectively write $(v,w)\in \EE$ or  $\{v,w\}\in \EE$ for the corresponding 
oriented or unoriented edge, respectively.

A {\it spin configuration} on $\Gamma$ is a function $\si: \VV\ra \{\pm 1\}$. 
The spin $\si(v)$ represents a magnetic momentum of an atom  $v\in \VV$. 
The total magnetic momentum of the configuration is equal to%
\footnote{As we have mentioned in the introduction, 
this definition is different from (\ref{M conventional})
as the summation here is taken over the bonds rather than the atoms:
see Appendix \ref{reg lattice} for a motivation for this unconventional definition.}   
\begin{equation}\label{M}
    M(\si)= \frac 1{2} \sum_{(v,w)\in \EE} (\si(v)+\si(w))=n_+(\si)-n_-(\si),
\end{equation}
where $n_+(\si)$ and $n_-(\si)$ stand respectively for the number of $\{++\}$ and $\{--\}$ bonds. 

The Ising model depends on three physical parameters:

\begin{itemize}

\item
$J$ -- the coupling constant (strength of the magnetic bonds between the atoms);
%$J=(J_{vw}))_{(v,w)\in \EE}$ -- a symmetric matrix of coupling constants (strength of the magnetic bonds between the atoms);
% % $J_{vw}=J_{wv}$;

\item
$h$ -- strength of the external magnetic field;

\item 
$T$ -- temperature.
\end{itemize}

The total energy of the configuration $\si$ is given by the Hamiltonian
$$
 H(\si) = - J I(\si) -  h M(\si),
$$
%
%%\begin{equation}\label{Hamiltonian}
%%   H(\si) = -\frac J{2} \sum_{(v,w)\in \EE} \si(v)\si(w) -\frac h{2}  \sum_{(v,w)\in \EE} (\si(v)+\si(w)),
%%\end{equation}
%where the first sum accounts for the energy of interactions between the atoms
%% (where only the neighbors interact), 
%while the second one accounts to the 
%energy of interactions of the matter with the external field.
where 
\begin{eqnarray}\label{I}
  I(\si):= \sum_{\{v,w\}\in \EE} \si(v)\si(w)= n_+(\si)+ n_-(\si)-n_0(\si),
\end{eqnarray}
with   $n_0(\si)$ being the number of $\{+-\}$ bonds in $\si$.  

Let $\Conf = \Conf(\Gamma)$ be the {\it configuration space}, i.e., the space of all
spin configurations. The {\it Gibbs weight} of a configuration $\si$ is equal
to%
\footnote{We let the Boltzmann constant $k=1$.} 
\begin{equation}\label{Gibbs weights}
         \Weight (\si)\equiv \Weight (\si;  J/T,  h/T) =  e^{\frac {-H(\si)}{T}}= t^{-I(\si)/2} z^{-M(\si)},
\end{equation}
where   $z= e^{-h/T}$ and  $t = e^{-2J/T}$ are field-like and
temperature-like  variables.
In the physical ferromagnetic  region % $J>0$, $h\in \R$, $T>0$, so
 we have $0<z<\infty$, $0<t<1$. 
However,  it is insightful to extend magnetic observables beyond this region. 
Let $\Par = \{(z,t)\}\subset \C^2$ stand for a relevant parameter space.

The Gibbs weights are invariant under simultaneous
change of sign of the external field and the spins:
$$
    \Weight(-\si;  J/T, - h/T)= \Weight(\si;  J/T,  h/T ).     
$$
This is the {\it basic symmetry} of the Ising model.  
It can be also formulated as follows. 
Consider the ``total configuration space'' $\widehat \Conf= \Conf\times \Par$   
fibered over $\Par$.
The  Gibbs weights $W(\si;z,t)$ endow it with the fibered Gibbs measure.
The basic symmetry translates into invariance of this measure under the involution
\begin{equation}\label{total involution}
    \hat\inv: \widehat\Conf\ra \widehat \Conf;\quad    \hat\inv: (\si; z,t)\mapsto (-\si; z^{-1}, t). 
\end{equation}

The {\it partition function} (or the {\it statistical sum})
is the total Gibbs weight of the space:
$$
  Z_\Gamma =Z_\Gamma(z, t)  = \sum_{\si\in \Conf} \Weight(\si).
$$
 It is a Laurent polynomial in $z$ and $t$.
Moreover, as a consequence of the basic symmetry, $Z$ is invariant under the 
involution  $\inv: (z,t)\mapsto (z^{-1},t)$,
so it has a form 
\begin{equation}\label{symmetric Laurent}
   Z_\Gamma = \sum_{n=0}^d a_n(t) (z^n+z^{-n}), \quad \mbox{where}\ d=|\EE|. 
\end{equation}
%Its  degree in $z$ is equal to $|\EE|$, so
Moreover, $\ a_d=t^{-d/2}$. 
Thus, for any given $t\in \C^*$, $Z_\Gamma(t,z)$ has $2|\EE|$ roots $z_i(t)\in \C$. 
They are called {\it Lee-Yang zeros}. 

The {\it Gibbs distribution} is the probability measure on $\Conf$
with probabilities of the configurations proportional to the
Gibbs weights:
$$
         P(\si)= \frac{\Weight (\si)}{Z}. 
$$ 
Note that it gives a bigger weight to less energetic configurations.  

The entropy  of a configuration $\si$ is defined as $S(\si) = -\log P(\si)=\log Z + H(\si)/T$. 
The {\it  free energy} is defined as  
\begin{equation}\label{def F_Gamma}
      F_\Gamma =   H(\si) - T S(\si)  =  - T\log Z_\Gamma. 
\end{equation} 
It is independent of the configuration $\si$ (in the Gibbs state)
and hence coincides with its average over $\Conf$.  
% Averaging it over the configuration space, we obtain the total 
% {\it free energy}: %  of  the Ising model is the logarithm of the partition function:
% $$
%         F=  F_\Gamma = \frac 1{|\EE|} \sum_{\si\in \Conf} F(\si) \Weight (\si)= \frac 1{|\EE|} \log Z.   
% $$
\begin{rem} One can define in the same way the entropy and the free energy for an arbitrary
probability distribution on $\Conf$. Then the Gibbs distribution is singled out by one of two
equivalent properties: 
(i)  {\it it minimizes the free energy}; (ii) {\it the free energy is evenly distributed over configurations}.
\end{rem}

%Recalling that $Z(\cdot,t)$ is a symmetric Laurent polynomial in $z$ with leading terms
%$$ 
%       t^{-|\EE|/2} z^{\pm |\EE|}=\exp\left\{\frac{|\EE|}{T} (J/2\mp h)\right\},
%$$
  Equations (\ref{def F_Gamma})  and (\ref{symmetric Laurent}) imply:
\begin{equation}\label{F_Gamma}
     F_\Gamma  =   - T \, \sum \log |z-z_i(t)|+ |\EE|\,T\, (\,\log|z|+\frac 1{2} \log |t|\,) %or - |\EE| (h +  J/2 ),
\end{equation}
where the summation is taken over the $2|\EE|$ Lee-Yang zeros $z_i(t)$ of   $Z(\cdot, t)$ 
% for a given $t\in (0,1)$.
(here $\log|z|$- and  $\log |t|$-terms account respectively for the denominator 
 and the leading coefficient of $Z(\cdot, t)$). 
 
\begin{rem}
  In the physical region, the variables $z$, $t$ and $Z$ are positive, so no absolute values are needed in (\ref{F_Gamma}).
They are introduced in order to extend $F_\Gamma$ to the complex plane 
in such a way  that $F_\Gamma-|\EE|\,T\, \log |z|$ is superharmonic.    
We will still refer to this extension as the ``free energy''. 
\end{rem}

The {\it magnetization} of the matter is the average of the the magnetic momentum over the Gibbs distribution:
\begin{equation}\label{M_Gamma}
   M_\Gamma = \sum M(\si) P(\si) =   -\frac {\di  F_\Gamma} {\di h}   =  % \frac z{T} \frac {\di  F_\Gamma} {\di z} =
                 - z \sum \frac{1}{z-z_i}+|\EE| .  
\end{equation}    

\begin{rem}
  Note that this expression gives a meromorphic continuation of $M_\Gamma$ to the complex plane.
It will still be referred as the ``magnetization''.   
\end{rem}

  The Ising model is called {\it ferromagnetic} if $J>0$, and {\it anti-ferromagnetic} otherwise.
The Gibbs distribution of a ferromagnetic model favors neighboring spins 
with the same orientation.  

Notice that for a ferromagnetic model,  $t= e^{-J/T} \in [0,1]$,
where $t=0$ and $t=1$ correspond respectively to zero and infinite temperature. 

\begin{LY Theorem}[\cite{YL,LY}] \label{THM:YANG_LEE} For a ferromagnetic Ising model,
 for any temperature  $t\in (0,1)$,  the Lee-Yang zeros $z_i(t)$ lie on the unit circle $\T$.
\end{LY Theorem}

This is a fundamental theorem of statistical mechanics. 
In Appendix \ref{APP:STATMECH} we will provide a proof of it in this general form
(in fact,  even in a slightly more general one).
In \S \ref{LY for DHL sec} we will prove it for  DHL using 
dynamics of  the Migdal-Kadanoff renormalization. 

Given a subsystem of atoms,  
$\UU \subset \VV$, and a partial configuration
$\si_\UU : \UU \ra \{ \pm 1\} $, we can define  {\it conditional configurations}
as all configurations $\si: \VV \ra \{ \pm\} $ that agree with $\si_\UU$ on $\UU$.
Let $\Conf(\Gamma|\, \si_\UU)$ stand for the space of all such  configurations.
The {\it conditional partition function} is defined as the total weight of this space:
$$
     Z_{\Gamma|\, \si_\UU} =  \sum_{\si\in \Conf (\Gamma|\si_\UU) } \Weight(\si).
$$

\begin{LY TheoremBC}\label{THM:LY_BC}Consider a ferromagnetic Ising model on a
connected graph $\Gamma$ and let $\sigma_\UU \equiv 1$ on a nonempty
$\UU \subsetneq \VV$.  Then, for any temperature $t \in (0,1)$ the Lee-Yang zeros $z_i^+(t)$ of the conditional
partition function $Z_{\Gamma|\, \si_\UU}$ lie outside the closed disc $\overline \D$.
\end{LY TheoremBC}

\noindent 
This interpretation follows directly from
the proof of the classical Lee-Yang Theorem; see Appendix \ref{APP:STATMECH}.
From the Basic Symmetry of the Ising model we get that for
$\sigma_\UU \equiv -1$ and $t \in (0,1)$ the Lee-Yang zeros $z_i(t)$ lie in
the open disc $\D$.

\subsection{Multiplicativity of the partition function}
For a subgraph $\Gamma'\subset \Gamma$, let $\bar \Gamma'$ stand for its {\it
closure} obtained by adding to $\Gamma'$ %all its neighbors. 
all of the vertices adjacent to $\Gamma'$ and all of the edges connecting them to $\Gamma'$.
Let $\di \Gamma'=
\bar \Gamma'\sm \Gamma'$.

\begin{lem}\label{multiplicativity}
  Let $\si_\UU$ be a conditional configuration 
and let $\Gamma_i$ be the connected components %of $\VV\sm \UU$.
obtained after removing the vertices $\UU$ and all of the edges
ending at them from $\Gamma$.
Then 
$$
     Z_{\Gamma|\, \si_\UU} = \prod_i Z_{\bar{\Gamma}_i|\, \si_{\di \Gamma_i }}. 
$$
\end{lem}

\begin{proof}
  Clearly, 
\begin{equation}\label{partial confs}
  \Conf (\Gamma|\, \si_\UU)\isom \prod \Conf(\bar \Gamma_i|\, \si_{\di \Gamma_i}).  
\end{equation}
Since there are no interactions between the partial configurations $\si|\, \Gamma_i$, 
we have  the additivity property for the energy:
$$
    H(\si|\, \si_\UU) = \sum H(\si_{\bar \Gamma_i} | \si_{\di \Gamma_i}).   
$$
This implies multiplicativity for the corresponding  Gibbs weights and 
(together with (\ref{partial confs})) for the conditional partition functions.
\end{proof}

\subsection{Thermodynamic limit}

For finite graphs, the partition function $Z$ is a Laurent polynomial with non-negative coefficients,
so the free energy $F=-T\log Z$ is real analytic in the physical region -- 
% (for positive $(z,t)$), 
 there are no phase transitions.  To observe phase transitions,
one should pass to a thermodynamic limit. Already in the original paper by Lee \& Yang \cite{LY},
the phase transitions were explicitly related to the asymptotic distribution of the zeros of the partition functions.  
In this section we will give a more rigorous account of these classical results.

Assume that we have a ``lattice'' given by a ``hierarchy'' of graphs $\Gamma_n$ of increasing size
(corresponding to finer and finer scales of the matter)%
\footnote{At this moment,  the terms ``lattice'' and ``hierarchy'' are used in a purely heuristic sense.
Formally speaking, we just have a sequence of graphs with $|\Gamma_n|\to \infty$.} 
with partition functions $Z_n$,
free energies $F_n$ and magnetizations $M_n$. To pass to the thermodynamic limit
we normalize these quantities {\it per bond}.%
\footnote{Viewing $|\EE|$ as the ``volume'' of the system, the normalized quantities get interpreted as
``specific''  free energy and magnetization.}
Let us say that our hierarchy of graphs has a {\it thermodynamic limit} if
\begin{equation}\label{L1 assumption}
     \displaystyle{\frac 1{|\EE_n| }F_n(z ,t)\to F(z,t)} \quad \mbox {for any} \ z\in \R_+, \ t\in (0,1).
\end{equation}
In this case, the function $F$ is called  the {\it free energy} of the lattice.

\begin{prop}\label{G and M}
  Assume that a hierarchy of graphs has a thermodynamic limit. Then for any $t\in (0,1)$, the limit (\ref{L1 assumption})  
in $z$ exists in $L^1_\loc(\C)$ and  the zeros of the partition functions $Z_n$ are asymptotically
equidistributed with respect to some measure $\mu_t$ on the unit circle $\T$.
Moreover, the limiting free energy  $F(z, t)$ admits the following electrostatic representation:
\begin{equation}\label{electrostat rep}
    F(z,t)= - 2 T \int_\T \log |z-\zeta|\,  d\mu_t (\zeta)  +T(\, \log |z| + \frac 1{2} {\log|t|}\, )
             \quad \mbox{for a.e. $z\in \C$}, 
\end{equation}
so $F(z,t)-T\log |z|$ is superharmonic in $z$ on the whole plane $\C$, and  is harmonic
on $\C\sm \supp\mu_t$. %  (modulo an adjustment on a null-set).

Furthermore, the magnetizations $\displaystyle{\frac 1{|\EE_n|} M_n}$ converge locally uniformly on $\C\sm \supp \mu_t$,
and the limiting magnetization $M$ admits the following Cauchy integral representation:
\begin{equation}\label{Cauchy rep}
     M(z,t) = -\frac{\di F}{\di h} =   -2 z \int_\T  \frac {d\mu_t (\zeta) }{z-\zeta}+1 \quad \mbox{for $z\in \C\sm \supp \mu_t$}, 
\end{equation}
so $M$(z,t) is holomorphic in $z$ % (modulo an adjustment on a null-set) 
on $\C\sm \supp \mu_t$. 
\end{prop}

% Under the above circumstances, the functions $F$ and $M$ are called 
% the {\it free energy} and  {\it magnetization}  of the  lattice.

\begin{proof}
We will fix some $t\in (0,1)$ and will consider all the functions in $z$-variable only.
Let $z_i^n$ stand for the zeros of the $Z_n$.  
Let us clear up the denominators of the Laurent polynomials $Z_n$ to obtain ordinary polynomials
$\tl Z_n = z^{d_n} Z_n$ (where $d_n=|\EE_n|$).  They have the same zeros as the $Z_n$, 
so by the Lee-Yang Theorem, they do not vanish on $\D$. Hence they admit  well defined roots
$\phi_n:= \tl Z_n^{1/d_n}$ on $\D$ that  are positive on the real line. 

Since the polynomials $\tl Z_n$  have positive coefficients, we have: 
\begin{equation}\label{C-bound}
     |\phi_n(z)| \leq \phi_n(1)\leq \exp \left\{ - \frac 1{T d_n} F_n(1) \right \}\leq C \quad \mbox{for any}\ z\in \bar \D,  
\end{equation}
where the last bound follows from existence of the thermodynamic limit. 

By Montel's Theorem, 
the sequence of functions $\phi_n$ is normal on $\D$.
Since it converges on $(0,1)$, it converges locally  uniformly on $\D$ to a holomorphic function $\phi$. 
Hence the free energies $d_m^{-1} F_n= -T(\log |\phi_n|-\log |z|)$ converge locally uniformly on $\D^*$ 
to the harmonic function $F: = -T(\log | \phi | - \log |z|)$. 

By the basic symmetry $z\mapsto 1/z$, we have $d_m^{-1} F_n\to F$ locally uniformly on $\C\sm \bar\D$.
Moreover, by the same symmetry and (\ref{C-bound}),  
the functions $d_m^{-1} F_n$  are uniformly bounded from above globally on $\C$.
By the Compactness Theorem for superharmonic functions (see \cite[Thm 4.1.9]{Ho}),
the function $F$ admits a superharmonic extension to the whole plane $\C$ and
\begin{equation}\label{L1 convergence}
            \frac 1{d_n}  F_n \to F \  \mbox{in}\  L^1_\loc(\C).
\end{equation}  

Let $\de_z$ stand for the unit mass located at $z$, and let
$\mu^n\equiv \displaystyle{\mu^n_t= \frac 1{2|\EE|}\sum \de_{z_i^n}}$.
  Since $\displaystyle{ \frac 1{2\pi} \log|\zeta|}$ is the fundamental solution of the Laplace 
equation, (\ref{F_Gamma}) implies that
\begin{equation}\label{De Fn}
    -\frac 1{4\pi T} \frac {\De F_n}{d_n} =  \frac 1{2 d_n}\sum \de_{z_i^n}-\frac 1{2}\de_0, 
\end{equation}
where  the Laplacian $\De$ is understood in the sense of distributions.

Since the distributional Laplacian is a continuous operator,
 (\ref{L1 convergence}) implies that 
% $F(\cdot, t)$ is a subharmonic function and 
$
  d_n^{-1} \De F_n (\cdot, t)\to \De F(\cdot, t)
$
in the weak topology on the space of measures. Together with  (\ref{De Fn}),
this implies that the Lee-Yang zeros are
equidistributed with respect to measure
$$
   \mu\equiv \mu_t:= -\frac 1{4\pi T}\De F(\cdot, t) +\frac 1{2} \de_0. 
$$ 

Let $u^n(z)\equiv u_t^n(z)$ and $u(z)\equiv u_t(z)$ 
stand for the electrostatic potentials of $\mu^n$ and $\mu$ respectively. 
% We would like to show that $u_t^n\to u_t$ in $L^1(\C)$.
For $z\in \C\sm \T$, the kernel 
$\zeta \mapsto \log  |z-\zeta|$ is a continuous function on $\T$ depending continuously (in the uniform topology)
on $z$. It follows  that $u^n(z)\to u(z)$ locally uniformly on $\C\sm \T$ . But by (\ref{F_Gamma}),
$$
 \frac 1{d_n} F_n= -2 T\, u^n(z) +  T(\log |z|+\frac 1{2} \log |t|).
$$  
Since (\ref{L1 convergence}) implies that  for some subsequence $n_k$, 
$\displaystyle{\frac 1{|d_{n_k}|} F_{n_k}(z)\to F(z)}$ a.e., representation (\ref{electrostat rep}) follows. 

Taking its $\di/\di h$-derivative (where $h=-T\log z$), we obtain representation (\ref{Cauchy rep}).

\comment{
Let us check that $L^2$-norms of the $u_^n$ are on $\D_2$ uniformly bounded. Indeed, by the Cauchy-Schwarz Inequality,
$$
    |u_n(z)|^2 \leq \int (\log |z-\zeta|)^2\, d\mu(\zeta),
$$
 so 
$$
  \|u_n\|^2\leq \int d\mu(\zeta) \int_{\D_2}  (\log |z-\zeta|)^2 d\area(z)\leq C.
$$
By Lemma \ref{}, $u_n\to u$ in $L^2$. 
}

\end{proof}

The basic symmetry of the Ising model implies that  
the free energy  is an even function of the field $h$,
while the magnetization is odd. 
In terms of the $(z,t)$-variables, $F$ \& $M$ are respectively even \& odd   under the involution $\iota$
(which is also clear from explicit representations (\ref{electrostat rep}) and (\ref{Cauchy rep})). 
If the magnetization has different limits at $z=1$ from above and below,
 $M^+(1)>0 $ and $M^-(1)=-M^+(1)<0$,  
then one says that the {\it first order phase transition} occurs (at $h=0$) , 
and call $M^+(1)$ the {\it spontaneous magnetization} of the model. 
The following statement makes physical relevance of the Lee-Yang zeros
particularly clear: 

\begin{cor}
  Assume the distribution $\mu_t$ is absolutely continuous on the unit circle
% (with respect to the Lebesgue measure on $\T$)
and its  density $\rho_t(\phi)= 2\pi d\mu_t(\phi)/d\phi$ is H\"older continuous at $\phi=0$. %   at  $\phi=0$.  
Then the first order phase transition  at $h=0$ occurs if and only if  $\rho_t(1)\not=0$,
and the corresponding spontaneous magnetization $M^+(1)$ is equal to $\rho_t(0)$.  
\end{cor}

\begin{proof}
 Formula (\ref{Cauchy rep}) with $d \mu_t = \rho_t d\phi/2\pi$ can be also written as follows:  
$$
   M(z,t)=  -2\left(\frac 1{2\pi i} \int_\T \frac {\rho_t(\zeta) d\zeta}{\zeta-z}- 
     \frac 1{2\pi i} \int_\T \frac{\rho_t(\zeta)d\zeta}{\zeta}\right) - 1.
$$
By the Sokhotsky Theorem  (see \cite[Theorem 7.6]{KRESS}),  the jump  at $z=1$  
(from inside to outside of $\T$) of the Cauchy integral 
% $\displaystyle{\frac 1{2\pi i}}\int \frac {\rho_t(\zeta)d\zeta}{\zeta-z}$ 
in parentheses is equal to $\rho(1)$. 
Hence the jump of $M$ (from inside to outside) is equal to $-2\rho_t(1)$.
On the other hand, it is equal to $-2 M^+(0)$. 
\end{proof}

\begin{rem}\label{REM:NON_ANALYTICITY}
If the limiting distribution $\mu_t$ has a density $\rho_t(\phi)$ that is
real-analytic in a neighborhood of $\phi = 0$, then (\ref{electrostat rep})
allows for a analytic continuation of $F(z,t)$ in a neighborhood of $z=1$ by a
deformation of contours.

For the $\Z^d$ Ising model, $d > 1$, Isakov \cite{ISAKOV} has shown
that no such analytic continuation exists at sufficiently low temperatures.
Thus, for these models $\rho_t(\phi)$ cannot be real-analytic $\phi=0$ for low temperatures.
\end{rem}

\subsection{Diamond Hierarchical Lattice (DHL) and Migdal-Kadanoff renormalization}  
%   We will now describe a hierarchy of graphs corresponding to finer scales of the matter. 
%The physical quantities (external field and temperature) in different scales are not the same
%but are rather related by a renormalization transformation. In the model in question,
%this transformation assumes a fairly simple explicit form.  

   Let us start with the simplest possible graph $\Gamma_0$: just two vertices, $a$ and $b$,
connected with one edge.
The space $\Conf_{\Gamma_0}$ consists of four configurations with the following energies:
$$
   H\left(\vertbondpp \right)= -J - h, \quad  H\left( \vertbondpm \right)=H\left(\vertbondmp \right)= J, \quad H\left(\vertbondmm \right)= -J + h, 
$$ 
and the following Gibbs weights:
\begin{eqnarray}\label{UVW-zt}
  U &=& \Weight\left( \vertbondpp  \right)=  z^{-1} t^{-1/2}, \nonumber \\
  V &=& \Weight\left(\vertbondpm \right)=\Weight\left(\vertbondmp \right)= t^{1/2},   \\ 
  W &=& \Weight\left(\vertbondmm \right)=  z t^{-1/2}.  \nonumber
\end{eqnarray}                                        
They sum up to  the following partition function:
\begin{equation}\label{Z-0}
  Z\equiv Z_{\Gamma_0}= U+2V+W =  \frac {z^2+2tz +1}{z\sqrt{t}}.
\end{equation}

Let us now replace the interval $\Gamma_0$ with a diamond $\Gamma_1$
with vertices $a,b,c,d$
(so that it  shares with $\Gamma_0$ the vertices $a$ and $b$), see Figure \ref{FIG:DIAMOND GRAPHS}. 
Restricting that the spins at $a$ and $b$ are both $+$ and summing over the
four spin configurations $(+,+)$  $(+,-)$ \& $(-,+)$, and $(-,-)$ at the
vertices $c$ and $d$ yields a sum of four conditional partition functions (two
of which are equal):
$$
  U_1 := Z_{\Gamma_1 | \, {++}} = U^4 +  2U^2 V^2 + U^4 = (U^2 + V^2)^2
$$
\noindent
Similarly
$$
   V_1 := Z_{\Gamma_1 | \, +-}=Z_{\Gamma_1 | \, -+}=  U^2V^2 + 2U V^2 W  + V^2 W^2 = V^2 ( U + W)^2, 
$$
$$
   W_1 :=  Z_{\Gamma_1 | \, --} =  V^4 +  2V^2 W^2 + W^4 = (W^2+V^2)^2. 
$$
The full partition function of $\Gamma_1$ is equal $Z_{\Gamma_1}=U_1+2V_1+W_1$.

Replacing each edge of the diamond with $\Gamma_1$, 
we obtain a lattice $\Gamma_2$ with 16 edges.
Inductively, replacing each edge of the diamond with the lattice $\Gamma_{n-1}$,
we obtain the lattice $\Gamma_n$ with $4^n$ edges,%
\footnote{This description of the DHL is ``dual'' to the one given in the Introduction, \S \ref{DHL intro}.}
 see Figure \ref{FIG:DIAMOND}. 

%\begin{figure}
%\begin{center}
%\input{figures/lattice.pstex_t}
%\end{center}
%\caption{\label{FIG:LATTICE} Inductive construction of $\Gamma_n$.}
%\end{figure}

\begin{figure}
\begin{center}
\input{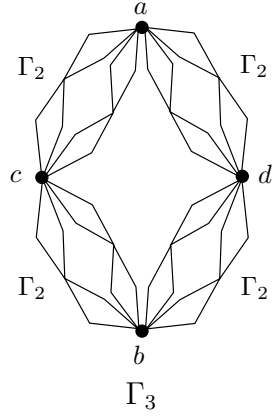}
\end{center}
\caption{\label{FIG:DIAMOND} Graph $\Gamma_3$ built from four copies of $\Gamma_2$.}
\end{figure}

All lattices $\Gamma_n$  share four vertices, 
$a$, $b$, $c$ and $d$, with the original diamond.
Restricting the spins at $\{a, b\}$ %to be $\{+,+\}$, $\{+,-\}$ \& $\{-,+\}$, and $\{-,-\}$, as shown in Figure \ref{FIG:DEFN_UVW}, 
we obtain three conditional partition functions, $U_n$, $V_n$ and $W_n$ as follows:\\
%\begin{figure}
\begin{center}
\input{figures/definition_UVW.pstex_t}
\end{center}
%\caption{\label{FIG:DEFN_UVW} Definition of $U_n, V_n$, and $W_n$.}
%\end{figure}
\noindent
The total partition function is equal to 
$$
   Z_n= Z_{\Gamma_n}= U_n+2V_n+W_n. 
$$

Similarly to the above formulas for $U_1$, $V_1$ and $W_1$, we have:

\msk
\begin{center}
{\bf Migdal-Kadanoff RG Equations: }
\end{center}
\ssk
$$
  U_{n+1}= (U_n^2+V_n^2)^2, \quad V_{n+1}= V_n^2 (U_n+W_n)^2, \quad W_{n+1}= (V_n^2+W_n^2)^2 .   
$$

\begin{proof}
    Let us check the first equation (the others are similar).   
There are four spin configurations at the vertices $(c, d)$: 
$(+,+)$, $(+,-)$ \& $(-,+)$ and $(-,-)$, as shown in Figure \ref{FIG:MK_derivation}.   
By the multiplicativity of the partition function (Lemma \ref{multiplicativity}),
the corresponding conditional partition functions are equal respectively  to
$ U_n^4$, $U_n^2 V_n^2$ (twice) and $V_n^4$. 
Summing these up, we obtain the desired equations.

\begin{figure}[h]
\begin{center}
\input{figures/MK_derivation_small.pstex_t}
\end{center}
\caption{\label{FIG:MK_derivation} Derivation of the Migdal-Kadanoff Equations.}
\end{figure}

\end{proof}

% Let us consider $(U:V:V)$ as homogeneous coordinates in the projective space $\CP^2$,
% and let $R: \CP^2\ra \CP^2$ be a degree 4 rational map given by the formulas:

Let us consider the following homogeneous degree 4 polynomial map $\hat \Rmig: \C^3\ra \C^3$ 
\begin{equation}\label{hat R}
  \hat \Rmig : (U, V, W)\mapsto ( (U^2+ V^2)^2, \  V^2(U + W)^2, \ (V^2+W^2)^2)
\end{equation}
called the {\it Migdal-Kadanoff Renormalization}. 
By the Migdal-Kadanoff RG Equations, the conditional partition functions of $\Gamma_n$ 
 are given by the orbits of this 
map,  $(U_n, V_n, W_n) = \hat \Rmig^n (U,V,W)$. 
The full partition function is obtained by the $\hat R^n$-pullback of the linear form $Z = U+2V+W$:  
\begin{equation}\label{Z_n}
   Z_n  =  Z \circ \hat \Rmig^n. 
\end{equation}
  The renormalization operator $\hat \Rmig$ descends to a rational transformation $\Rmig: \CP^2\ra \CP^2$.
In the affine coordinates $u= U/V$, $w= W/V$, it assumes the form:
\begin{equation}\label{uv coord}
   \Rmig: (u,w)\mapsto \left( \frac{u^2+1}{ u+w},\ \frac{w^2+1}{u+w} \right)^2, 
\end{equation}                                                              
where the external squaring stands for the squaring of both coordinates,
$ (u,w)^2 = (u^2, w^2)$. 
According to  (\ref{UVW-zt}), these coordinates are related to 
 the  ``physical'' $(z,t)$-coordinates as follows:
\begin{eqnarray}\label{zt-uv}
(u,w) = \correspond(z,t) = \left(\frac{1}{zt},\frac{z}{t}\right).
\end{eqnarray}
\comment{***********
\begin{equation}\label{zt-uv}
          z^2= \frac W{U}=\frac w{u},\quad t^2= \frac {V^2}{UW}= \frac 1{uw}.
\end{equation}
**************}
In  $(z,t)$-coordinates, the renormalization transformation assumes the form:
\begin{equation}\label{R}
    \Rphys: (z,t)\mapsto  \left( \frac{z^2+t^2}{z^{-2}+t^2}, \ \frac{z^2+z^{-2}+2}{z^2+z^{-2}+t^2+t^{-2}}\right).
\end{equation} 
     The iterates $(z_n, t_n)= \Rphys^n (z,t)$ are related to $(U_n,V_n,W_n)$ by means of  (\ref{UVW-zt}):
$U_n= z_n^{-1} t_n^{-1/2}$, etc.
Physically, they are interpreted as  the {\it renormalized}  field-like and temperature-like variables.
%If we freeze the interaction strength $J$  \note{OK?}
%we can recover the corresponding renormalized field $h_n$ and temperature $T_n$. 
% The renormalization describes explicitly  how these physical quantities depend on the scale where
% their measurement is taken. \note{makes sense?}

\subsection{Basic Symmetries}\label{sec: basic symmetries}
%  The basic symmetry of the Gibbs weights (\ref{Gibbs weights}) is invariance under simultaneous
%change of sign of the external field and the spins:
%$$
%    \Weight(-\si;  J/T, - h/T)= \Weight(\si;  J/T,  h/T ).     
%$$
By the basic symmetry of the Ising model, the change of sign of $h$ interchanges
the conditional partition functions $U_n$ and $W_n$ keeping $V_n$ and
the total sum $Z_n$ invariant. Consequently, the RG transformation $\hat \Rmig$
commutes with the involution $(U,V,W)\mapsto (W,V,U)$, 
which is also obvious from the explicit formula (\ref{hat R}).   
Accordingly, the transformation $\Rmig$ commutes with the permutation $(u,w)\mapsto (w,u)$,
while $\Rphys$ commutes with $(z,t)\mapsto (z^{-1}, t)$. 

All these transformations have real coefficients, 
so all of them commute with the corresponding complex conjugacies:
$(U,V,W)\mapsto (\bar U, \bar V, \bar W)$, etc. 
%Since conjugacies and permutations commute, 
%we obtain the Klein group $(\Z/2\Z)\times (\Z/2\Z)$ of symmetries.   \note{comment out?}

Finally, there is an extra ``accidental'' symmetry of the DHL:
the generating diamond $\Gamma_1$ is symmetric under reflection across
the vertical axis.  
%Therefore, $\Gamma_n$ will have this same symmetry for all $n > 1$.  
It results in the squared terms in the Migdal-Kadanoff RG Equations that makes the LY distributions $\mu_t$
symmetric under the half-period  translation $\phi\mapsto\phi +\pi$. 
It will play an important role in \S \ref{SUBSEC:COMBINATORIAL_EXPANSION}.

% \section{Structure of the RG transformation}

\section{Structure of the RG transformation I: Invariant cylinder}  %  and zero magnetic field leaf $\II_0$}
\label{SEC:STRUCTURE}

We will now begin to explore systematically the RG dynamics. 
Its generator was represented above in several coordinate systems, in particular,  
as a transformation $\Rmig$ (\ref{uv coord}) in the affine coordinates $(u,w)$
and as a transformation $\Rphys$ (\ref{R}) in the physical coordinates $(z,t)$.
From a physical point of view we are primarily interested in the latter.  
However, $\Rmig$ possess  better global dynamical properties. 
% see \S\ref{SUBSEC:INDETERMINANT_PTS} and \S \ref{SUBSEC:ALG_STABILITY}, in particular.  
For these reasons we will treat both mappings in parallel.

Since $\Rmig$ and $\Rphys$ represent the same map in different coordinates,
the corresponding change of variables (\ref{zt-uv}) should be equivariant, % semi-conjugate them,
i.e., the following diagram must commute:

\begin{eqnarray}\label{comm diagram}
\begin{CD}
\CP^2 @>\Rphys>> \CP^2 \\
@VV\correspond V 	@VV\correspond V\\
\CP^2 @>\Rmig>> \CP^2
\end{CD}
\end{eqnarray}
wherever the maps are well defined.
And this is indeed true and  can be verified directly using the explicit formulas for the maps.

\comm{***
\subsection{Relationship between $\Rmig$ and $\Rphys$:}
\label{SUBSEC:SEMICONJUGACY}
The mapping $\correspond:\C^2 \ra \C^2$ given by (\ref{zt-uv}) extends to a
rational map $\correspond:\CP^2 \ra \CP^2$.  

Writing homogeneous coordinates $[Z:T:Y]$ in the domain (with $(z,t)$ corresponding to
$[z:t:1]$) and $[U:V:W]$ in the codomain (with $(u,w)$ corresponding to
$[u:1:w]$), the extension is given by
\begin{eqnarray*}
[U:V:W] = \correspond([Z:T:Y]) = [Y^2:ZT:Z^2].
\end{eqnarray*}
\note{$W$-notation}
Generic points $[U:V:W]$ have two preimages under $\correspond$ and the critical locus of $\correspond$
is the union of the two lines $Z=0$ and $Y=0$, with the former collapsing under $\correspond$ to $[1:0:0]$.
On this collapsing line lies $\gamma=[0:1:0]$ which is the only indeterminacy point for $\correspond$.

\begin{prop}\label{PROP:SEMI_CONJUGACY}The mapping $\correspond$ induces a semi-conjugacy
between $\Rphys$ and $\Rmig$ so that the following diagram commutes:
\begin{eqnarray}
\begin{CD}
\CP^2 @>\Rphys>> \CP^2 \\
@VV\correspond V 	@VV\correspond V\\
\CP^2 @>\Rmig>> \CP^2
\end{CD}
\end{eqnarray}
\end{prop}

\noindent
The proof is a direct calculation and we omit it.
***}

In this section we will describe basic features of $\Rphys$ (viewed statically) on the invariant cylinder
(that supports the Lee-Yang zeros) and of $\Rmig$ on the corresponding invariant M\"obius band. 

\subsection{Invariant cylinder and M\"obius band.}
\label{SUBSEC:MAP_ON_CYL}
%The {\it Migdal-Kadanoff renormalization (RG) transformation} 
% $R$ is a rational function in two variables given by the following explicit formula:
% $$
%   R(z,t) = \left( \frac{z^2+t^2}{z^{-2}+t^2}, \ \frac{z^2+z^{-2}+2}{z^2+z^{-2}+t^2+t^{-2}}\right).
% $$ 
% We are  interested in the dynamics of $R$ on the invariant cylinder $\Cphys = \T\times \II$. 
Let us consider the round cylinder $\Cphys = \T\times\II$  naturally sitting in $\C^2$.
It is obvious from (\ref{R}) that $\Cphys$ is invariant under $\Rphys$.
% In this section we will describe basic features of $\Rphys: \Cphys\ra \Cphys$ viewed statically. 

\begin{rem} Note that the whole bi-infinite cylinder $\hat \CC= \T\times \R$ is also invariant under $\RR$,
and in fact, $\RR(\hat \CC)\subset \CC$. (Here the upper cylinder ($t>1$) corresponds 
to the {\it anti-ferromagnetic} region.) 
\end{rem}

%\bignote{Does it make sense that renormalization brings everything to the ferromagnetic region?} 

We keep identifying the unit circle $\T$ with $\R/2\pi\Z$, 
so that points $(z,t)$ ( $z\in \T$)  of the cylinder  $\Cphys$ will often be written
in the angular coordinate as $(\phi, t)$, where  $z=e^{i\phi}$, $\phi\in \R/2\pi\Z$. 
Let $\T_t= \T\times \{\phi\}$ be the horizontal sections of the cylinder.
We will use special notation for the bottom and the top sections: 
$$
  \BOTTOMphys=\T_0\quad \TOPphys=\T_1. 
$$
We will also use  special notation $\Cphystl:= \Cphys\sm \TOPphys$ and  $\Cphysbl:= \Cphys\sm \BOTTOMphys$
respectively for the topless and the bottomless cylinder.                            

The cylinder is foliated by the vertical intervals
$$
    \II_\phi= \{ (\phi, t):\ 0\leq t \leq 1\}.
$$
The interval $\II_0 = \{\phi=0\}$ plays a distinguished role, both physically and dynamically.
Physically, it corresponds to the {\it vanishing} magnetic field.
Dynamically, it is singled out by the property of being {\it invariant} under $\Rphys$.
Its endpoints $\FIXphys_0=(0,0)\in \BOTTOMphys $ and $\FIXphys_1 = (0,1)\in \TOPphys$ 
are {\it superattracting fixed points} for $\Rphys|\, \II_0$,
and there is a unique {\it repelling fixed point} $\FIXphys_c=(0,t_c)\in \II_0$.
This is exactly the {\it critical point of the Ising model}%
\footnote{We hope it will be clear from the context whether the term ``critical'' is used in 
                            the physical or dynamical sense.}
 mentioned in the introduction
and marked on Figure \ref{FIG:CYLINDER_BASINS}.  
%\footnote{This point is traditionally labeled by $c$ since it the {\it critical} point for the Ising model.
% However, to avoid confusion with the dynamical meaning of critical points, 
% we will refer to it as ``the real repelling point''.}
All points $\FIXphys \in \II_0$ below $\FIXphys_c$ converge to $\FIXphys_0$,
while all points above it converge to $\FIXphys_1$.
  
\msk
Let us now switch to the affine coordinates $(u,w)=\Psi(z,t)$ from (\ref{zt-uv}).
Consider a topological  annulus
\begin{equation}\label{C-D}
     \Cmigbl= \{(u,w)\in \C^2:\ w=\bar u, \ |u|\geq 1\}.
\end{equation}
in $\C^2$, and let $\Cmig$ stand for its closure in $\CP^2$.
Let $\TOPmig= \{(u,\bar u): \ |u|=1\} $ be the  ``top'' circle of $\Cmig$,
while  $\BOTTOMmig$ be the slice of $\Cmig$ at infinity. 

Though the change of variables $\Psi$ is not globally invertible,
it is nearly such on the cylinder $\Cphys$:

\begin{prop}\label{cylinder-band conjugacy}
\begin{itemize}
\item [(i)] $\correspond$ restricts to a diffeomorphism $\Cphysbl\ra \Cmigbl$;
\item [(ii)]   $\Cmig=\correspond(\Cphys)$ is an $\Rmig$-invariant M\"obius band,
         and $\BOTTOMmig$ is an $\Rmig$-invariant circle,
         a ``median'' of the band
         (given as the unit circle  in the coordinate $\zeta= w/u$). 
\item[(iii)]     $R$ acts on $\Cmigbl$ as
$\displaystyle{  u\ra \left( \frac{u^2+1}{2\Re u}\right)^2}$,
and it acts on $\BOTTOMmig$  as $\zeta\mapsto \zeta^4$; 
\item [(iv)]  The map $\Psi: \BOTTOMphys\ra \BOTTOMmig$ is 2-to-1, and
 $\correspond : \Cphys \ra \Cmig$ continuously semiconjugates $\Rphys$ to $\Rmig$.
\end{itemize}
\end{prop}

\begin{proof}
It is obvious from  (\ref{zt-uv}) that  $\correspond$ 
maps $\Cphysbl$ smoothly to $\Cmigbl$.
Moreover,  $(\phi,t)$   
can be recovered from $(u,v)=\correspond (\phi,t)$ as the polar coordinates of $u^{-1}= te^{i\phi}$.
This yields (i). 

Let us use coordinates $(\xi =1/u, \zeta=w/u)$ near the line at infinity $\{\xi=0\}$. 
In these coordinates, the map $\correspond$ assumes the form $\xi=tz,\ \zeta= z^2$,
which makes obvious its continuity. Hence it maps $\Cphys$ onto $\Cmig$. 

Moreover, the circle $\BOTTOMphys$ is mapped to the circle $\BOTTOMmig=\{\xi=0,\, |\zeta|=1\}$
by $\zeta= z^2$. So, topologically $\Cmig$ is obtained from the cylinder $\Cphys$ by 
identifying the antipodal points $z$ and $-z$ on  $\BOTTOMphys$. This makes the M\"obius band.

Invariance of $\Cmig$ and $\BOTTOMmig$ follow from the semi-conjugacy or directly from (\ref{uv coord}).
Expressions (iii)  are also straightforward.
\end{proof}

%The renormalization assumes a particularly simple form on $\Cmig$:
%\begin{equation}\label{R on C}
%    \Rmig: u\ra \left( \frac{u^2+1}{2\Re u}\right)^2.
%\end{equation}

In the coordinate $u$ on $\Cmig$, the interval $\II_0$ becomes the real ray $I_0= \{u\in \R, u\geq 1\}$
and the fixed points $\FIXphys_1, \FIXphys_c$, $\FIXphys_0$ on $\II_0$
become fixed points $\FIXmig_1=\{u=1\} , \FIXmig_c$,  $\FIXmig_0=\{\zeta=1\}$ for $\Rmig$.  
% Points $\pm i\in \TOPmig$ are the indeterminacy points for $R$. 

\subsection{Decomposition $\Rphys=f\circ Q$ and structure of $f$ on the cylinder}\label{f-sec}  

To understand further the geometric structure (of the first iterate) of  the renormalization $\Rphys$, 
it is convenient to decompose it as $f\circ Q$, where $Q(z,t)= (z^2, t^2)$ and
\begin{equation}\label{f}
   f(z, t) = \left( \frac{z+t}{z^{-1}+t}, \  \frac{\cos\phi + 1 } {\cos\phi+ s }\right),\quad s=\frac{1}{2}(t+t^{-1}).
\end{equation} 
As a reflection of the basic symmetry of the Ising model, 
 $f$ commutes with the involution $\si: (\phi, t)\mapsto (-\phi,t)$.%
% it is enough to consider it on the half-cylinder $\LL=[0,\pi]\times [0,1]$.

The cylinder $\Cphys$ is invariant under both $Q$ and $f$.  However, we should be
careful: $f$ is not well defined on the whole cylinder; it has a {\it point of
indeterminacy} $\INDphys = (\pi ,1)\in \TOPphys \cap \II_\pi $ which decisively
influences the dynamics.  When we approach $\INDphys$ from inside the cylinder at
angle $\om\in [-\pi/2, \pi/2]$ with the leaf $\II_\pi$, the map $f$ converges
to the point 
\begin{equation}\label{Icurve}
    (\phi, t) = ( 2\om, \sin^2\om):=  \Icurve(\om) \in \Cphys
\end{equation} 
(see \cite[p. 419]{BZ3} and also calculations in Appendix \ref{APP:BLOW_UPS}).  Thus,
the point $\INDphys$ {\it blows up} to the  {\it blow-up locus} 
\begin{equation}\label{Icurve-2}
     \Icurve= \{t=\sin^2 \phi/2 \equiv \frac {1- \cos \phi}{2}\}.
\end{equation}
Note that $\Icurve$ touches the top $\TOPphys$ at $\INDphys$
itself, and touches the bottom $\BOTTOMphys$ at $\FIXphys_0$ 
(see Figures \ref{FIG:CYLINDER_MAP} and \ref{FIG:F}).
% The curve $\Gamma$ 
It divides the cylinder into two pieces: $\Cphys_-$ 
(below $\Gamma$) and $\Cphys_+$ (above it).

The interval $\II_\pi$ (that ends at the point of indeterminacy $\INDphys$) is 
{\it critical}: it collapses under $f$ to the fixed point $\FIXphys_0$. 

In the angular coordinate, we have:
\begin{equation}\label{Df}
   Df = \frac{2}{\eta^2} \left
                         (\begin{array}{cc}  
                                                   \eta & 0 \\
                                                    0 & 1- t 
                         \end{array}   \right) 
        \cdot
                         \left
                          (\begin{array}{cc} 
                                                  1+t \cos\phi & - \sin\phi \\
                                                 -t(1-t) \sin\phi & (1+t)(1+\cos\phi))     
                          \end{array} \right),     
\end{equation}
where $\eta= 1+2t\cos\phi+ t^2\in [0,4] $. 
It follows that
\begin{equation}\label{det}
   \Jac f = \frac{4 (1-t)(1+\cos\phi)}{\eta^2} \geq 0,
\end{equation}
and $\Jac f = 0$ only on the critical interval $\II_\pi$. 
Thus, $f$ is an {\it orientation preserving local diffeomorphism} on $\Cphys\sm \II_\pi$.

\comment{*****
\bignote{Should we replace the above formula with the one in $(u,v,)$-coordinates?}
******}

Note  that $f|\BOTTOMphys \, : z \mapsto z^2$ while $f|\TOPphys = \id$. 
This drop in the degree is caused by the point of indeterminacy
in the following way.
% Let us cut the cylinder along the critical leaf $\II_\pi$
% to unroll it into  the rectangle $[0, 2\pi]\times I$.  
Let us consider the zero level Lee-Yang locus $\SSS=\{Z=0\}$,
where $Z$ is the partition function (\ref{Z-0}):
\begin{equation}\label{S}
     \SSS = \{ z^2+2tz+1=0\} = \{ t=-\cos\phi:\ \phi \in [\pi/2, 3\pi/2] \, \} .
\end{equation}
It has two branches over $I$ (symmetric with respect to $\II_\pi$)
each of which is mapped diffeomorphically  onto $\II_\pi$   (see Figure \ref{FIG:CYLINDER_MAP}).

\begin{figure}
\begin{center}
\input{figures/cylinder_map.pstex_t}
\end{center}
\caption{\label{FIG:CYLINDER_MAP} The map $f:\Cphys \rightarrow \Cphys$.}
\end{figure}

The curve $S$ divides $\Cphys$ into two domains: $\La^s$ (containing $\II_\pi\sm \{\INDphys\}$)
and $\La^r$ (containing $\II_0$). 
The domain $\La^s\sm \II_\pi$ is composed of two topological triangles mapped diffeomorphically 
onto the corresponding triangles of  $\Cphys_-\sm \II_\pi$,
namely, the right-hand side triangle of $\La_s\sm \II_\pi$ is mapped onto the left-hand side triangle of $\CC_-\sm \II_\pi$,%
\footnote{The map is quite peculiar on the boundary of $\La^s\sm \II_\pi$ 
 as it blows up $\INDphys$ to the $\Icurve$-boundary
of $(\Cphys_-\sm \II_\pi)$ while collapses $\II_\pi$ to $\FIXphys_0$.} see Figure \ref{FIG:F}.
On the other hand,  $\La^r$ is mapped diffeomorphically onto the whole $\Cphys\sm \II_\pi$. 
Accordingly, we have two diffeomorphic branches of the inverse map,
the ``singular'' branch $f^{-1}_s :  \Cphys_- \sm \II_\pi \ra \La^s$ and the 
``regular'' one, $f^{-1}_r: \Cphys\sm \II_\pi \ra \La^r$. % (see Figure \ref{FIG:F}).  

In particular, we conclude that
{\it the map $f$ has degree 2 over $\inter \Cphys_-$ and degree 1 over $\inter \Cphys_+$}.
So, $f$ {\it is not a proper map}. 

 A path $\gamma: [0,1]\ra \Cphys$ is called {\it proper} 
if it connects the bottom of the cylinder to its  top without passing through $\di \Cphys$ in between.%
\footnote{An open path  $\gamma_1: (0,1): \ra \inter\Cphystl$ or a half-open path   $\gamma_1: [0,1): \ra \Cphystl$
 that extends to a proper path  $\gamma: [0,1]\ra \Cphys$ will also be called ``proper''. } 
%Let us orient it so  that $\gamma(0)\in \BOTTOMphys$.
%A path $\de: [0, \la_*)\ra \C$ (or $\de: [0, \la_*]\ra \C$)  is called a {\it lift of $\gamma$} under $f$
%if $\de(0)\in \BB$ and $f(\de(\la))= \gamma(\la)$. 
A crucial property of the cylinder dynamics is that  $f^{-1}$ acts  properly on proper paths:
% (though $f$ itself is not proper): 

\begin{lem}\label{vertical lifts}
  If $\gamma:[0,1]\ra \Cphys$ is a proper  path then the full preimage 
$f^{-1}\gamma$ contains two proper  paths, $\de_1$ and $\de_2$. 
These two paths can meet only at $\INDphys$. Moreover, if $\gamma$ crosses $\Icurve$ only once,
then $f^{-1}\gamma= \de_1\cup \de_2:=\de_r\cup \de_s$, where $\de_r=f_r^{-1} \gamma$ is the
``regular'' lift of of $\gamma$ while $\de_s= f_s^{-1}\gamma$  is the ``singular'' lift
ending at $\INDphys$.%
\footnote{In case $\gamma(1)=\INDphys$,  both lifts end at $\INDphys$. }
\end{lem}

%\begin{lem}\label{vertical lifts}
%  Any  topologically vertical path $\gamma:[0,1]\ra \Cphys$ has two topologically vertical lifts
%$\de_1$ and $\de_2$ under $f$.
%\end{lem}

\begin{figure}
\begin{center}
\input{figures/F.pstex_t}
\end{center}
\caption{\label{FIG:F} Cylinder $\Cphys$ shown in $(\phi,t)$ coordinates.
The left-hand side triangle of $\La_s\sm \II_\pi$ is mapped onto the right-hand side triangle of $\CC_-\sm \II_\pi$.   
The proper path $\gamma$ is lifted by $f$ to two paths, 
the regular lift $\de_r$ and the singular lift $\de_s$.  
The singular lift $\de_s$ reaches $\TOPphys$ at $\INDphys$.}
\end{figure}

%path $\de_s$ is on the wrong side of $\II_\pi$. 
%Also, let us depict the interval $\II_\pi$, and mark $\beta_0$ on the bottom rectangle
%(in both corners $\phi=0,2\pi$), and $\alpha_0$ on the top  rectangle.
%Also, symbol $\gamma$ should be moved closer to the corresponding path. 
%I would also remove marks $0$ and $2\pi$ from the bottom rectangle. 

\begin{proof}
Since the endpoints of $\gamma$ belong to different components of $\di \Cphys$, 
we can orient it so that $\gamma(0)\in \BOTTOMphys$. 
This initial point has two preimages on $\BOTTOMphys$; let $p$ be either of them.
We will show that there is a proper  path $\de \subset f^{-1}\gamma$ that begins at $p$. 
%We will show that there is a lift $\de(\la)$ of 
%an initial piece of $\gamma(\la)$,  $0\leq \la\leq \la_*\in (0,1]$, that begins at $p$.
 
Let $\INDphys_0= (\pi,0)\in \BOTTOMphys$ stand for the bottom point of the critical interval $\II_\pi$
(which collapses to $\FIXphys_0$).
If $p=\INDphys_0$ then $\gamma(0)=\FIXphys_0$
and the interval $\II_\pi$ is a desired path $\de$.%
\footnote{If an initial piece of $\gamma$ lies in $\CC_-$ then there is a lift $\de'$ 
of $\gamma$ that begins at $\alpha_0$. This possible extra lift is disregarded in our discussion.} 

So, assume $p\not=\INDphys_0$. Then $f: \Cphys\ra \Cphys$ is a local diffeomorphism near $a$, 
so there is a local lift $\de$ of $\gamma$ that begins at $a$. 
Continuing lifting it as far as possible, we obtain a lift
$\de(\la)$, $0\leq \la<\la_*\in (0,1]$,  that cannot be extended further.

%Since $\gamma(t)\in \Cphys\sm \BOTTOMphys$ for $t>0$, it never hits $\FIXphys_0$,
%so the lift $\de(t)$ never hits the critical interval $\II_\pi$. 

What can go wrong at $\la_*$? If $x := \gamma(\la_*)\not \in \TOPphys\cup \Icurve$ then 
$x$ would have a disk neighborhood $U\subset \inter \Cphys$ such that $f: f^{-1}(U)\ra U$ were a covering map (of degree 1 or 2),
and the lift would admit a further extension. So, $x\in \TOPphys\cup \Icurve$.

If $x\in \TOPphys\sm \{\INDphys\}$ then $\la_*=1$ and $x$ has a relative half-disk neighborhood $U\subset \Cphys$ such that
$f^{-1}(U)$ is a half-disk neighborhood of $f^{-1} x= x$ mapped homeomorphically onto $U$. 
In this case we let $\de(1)=x$ and obtain the desired lift.

If $x\in \Icurve\sm \{\INDphys\}$ then $x$ has a  neighborhood $U\subset \inter \Cphys$
such that $f^{-1}U=U^r\sqcup U^s$   where $U^r=  f_r^{-1}(U)$
% (recall that $f_r^{-1}$ and $f_s^{-1}$ are the regular and singular branches of the inverse map).
is a disk homeomorphically mapped onto $U$,
while  $U^s= f_s^{-1}(U\cap \Cphys_-)$ is a ``wedge centered at $\INDphys$'' homeomorphically mapped onto $U\cap \Cphys_-$.  
Then for all $\la$ near $\la_*$,
\begin{equation}\label{reg and sing branches}
\mbox{ either $\de(\la)= f_r^{-1}(\gamma(\la)) \subset U^r$  
or $\de(\la)= f_s^{-1}(\gamma(\la)) \subset U^s$.}%
\footnote{Note that the curve $\gamma(\la)$ can cross $\Icurve$ infinitely many times at $\la\to \la_*$.
  If this happens then $\de(\la)=f_r^{-1}(\gamma(\la))$ for $\la$ near $\la_*$.}
\end{equation} %
But in the former case, $\de(\la)$ can  be extended beyond $\la_*$,
contrary to our assumption.
In the latter case,  $\de(\la)$ is forced to converge (as $\la\to \la_*$) to the center  $\INDphys$ of the wedge $U^s$. 
This gives us the desired proper path terminating at $\alpha$. 

Finally, if $x=\INDphys$ then $\la_*=1$ and $\INDphys$ can be the only accumulation point for $\de(\la)$ as $\la\to 1$ 
(for, if $y\in \Cphys\sm\{\INDphys\}$ is another accumulation point then $\gamma(\la)$ would accumulate on $f(y)\not=\INDphys$
as $\la \to 1$). Thus, $\de(\la)\to \INDphys$ as $\la\to 1$, and we obtain a proper path again.     

Remark that if $\de\not=\II_\pi$  then the path $\de(\la)$ cannot meet $\II_\pi\sm \{\INDphys\}$. 
Indeed, under this assumption, $\FIXphys_0\not\in \gamma$,   
while $\II_\pi$ collapses to $\FIXphys_0$. 

\ssk So, we have constructed two proper paths, $\de_1$ and $\de_2$.
They cannot meet at any point of $\Cphys\sm \II_\pi$ since $f$ is a local homeomorphism over there. 
By the above remark,  they cannot meet at any point of $\II_\pi\sm \{\INDphys\}$ either.
Hence  $\INDphys$ is their only possible meeting point.

\ssk    
Assume now that  $\gamma$ crosses $\Icurve$ only once, and let $\gamma(\la_*)\in \Icurve$ be this intersection point.
Assume first that $\gamma(0)\not=\FIXphys_0$.
Then by the previous argument, the arc $\gamma(\la)$, $0\leq \la < \la_*$,
has two lifts $\de_r(\la)$ and $\de_s(\la)$ as in (\ref{reg and sing branches}).
%  such for $t$ near $t_*$ we have: 
% $$ 
%    \de_r(t)=f_r^{-1}(\gamma(t)), \quad \de_s(t)=f_s^{-1}(\gamma(t)).
% $$ 
Then $\de_s(\la)$, $0\leq \la\leq \la_*$, is a proper  path 
terminating at $\INDphys$ (the singular lift of $\gamma$),
while $\de_r(\la)$ extends further to a proper  path parameterized by the full interval $[0, 1]$
(the regular lift of $\gamma$). 

Thus, for $\la<\la_*$, both preimages of  $\gamma(\la)$  are captured by the above lifts  $\de_r(\la)$ and  $\de_s(\la)$,   
while for $\la>\la_*$, the only preimage of $\gamma(\la)$ is $\de_r(\la)$.  We conclude that $f^{-1}(\gamma)=\de_r\cup\de_s$. 
% \footnote{This argument ``degenerates'' when $\gamma(0)=\INDphys_0$, but can be treated easily.} 

Finally, if $\gamma(0)=\FIXphys_0$ then $\gamma(\la)\in \CC^+$ for $\la>0$,
so such $\gamma(\la)$ has only one preimage.
This preimage is captured by the lift $\de_r$ that begins at $\FIXphys_0$. 
It follows that
$$
   f^{-1}(\gamma)= \de_r \cup  f^{-1} (\FIXphys_0)=\de_r\cup \II_\pi:=\de_r\cup \de_s.
$$
\end{proof}

Figure \ref{FIG:F} shows $\Cphys$ in $(\phi,t)$ coordinates, a
proper path  $\gamma$ that crosses $\Icurve$ in only one point, and the
regular and singular lifts $\de_r$ and $\de_s$ of $\gamma$ under $f$.

If $\eta$ is not a full proper path, but merely a proper path in $\Cphys_-$
(connecting  $\BOTTOMphys$ to $\Icurve$ without passing through $\di \Cphys_-$ in
between) we will also call the lift of $\eta$ under $f_s^{-1}$ the 
``singular lift'' of $\eta$.

%Thus we {\it always have at least two vertical lifts of $\gamma$ by $f$}.
%Moreover, if $\gamma$ intersects $\Gamma$ at a single point then we have exactly two lifts, 
%the regular one and the singular one, 
%and their union is the full  $f^{-1}(\gamma)$. 

\subsection{Structure of $\Rphys$ on the cylinder}\label{str on CC} 
  The above properties of $f$ immediately translate into the following properties 
% (of the one iterate) 
of the renormalization operator $\Rphys$: 

\begin{itemize}
\item[(P1)] {\it  Symmetries}:
As we have already mentioned in \S \ref{sec: basic symmetries},
the Basic Symmetry of the Ising model implies that $R$ commutes with the involution
$\iota: (z,t)\mapsto (z^{-1}, t)$. % or $\iota: (\phi,t)\mapsto (-\phi,t)$ in the angular coordinate. 
On the other hand,  since $R=f\circ Q$, we have: $R\circ \rho=R$, where $\rho: (z,t)\mapsto (-z,t)$.
It follows that the basins of the top and the bottom of $\Cphys$ are invariant under the Klein
group $(\Z/2\Z)\times (\Z/2\Z)$ comprising $\id$ and three involutions, $\iota$, $\rho$ and $\iota\circ \rho$. 
These symmetries are clearly visible on Figure \ref{FIG:CYLINDER_BASINS} as it is%
\footnote{or rather, its $2\pi$-periodic unfolding} 
$\pi$-periodic and is invariant under reflections in the axes $\phi=0,\pi$, $\phi=\pm \pi/2$.    

\item[(P2)] $\Rphys$ has {\it two points of indeterminacy}, $\INDphys_{\pm} =(\pm\pi/2, 1)\in \TOPphys$.
    Each of them blows up onto the singular curve $\Icurve$.  (See Appendix \ref{APP:BLOW_UPS} for detailed formulae.) 

\item[(P3)] 
Formula (\ref{det}) implies that the {\it critical locus} of $\Rphys|\Cphys$ comprises the bottom $\BOTTOMphys$,
the top $\TOPphys$, and  two vertical intervals, $\II_{\pm \pi/2}$, 
terminating at the points of indeterminacy.  These intervals collapse
under $\Rphys$ to the fixed point $\FIXphys_0\in \BOTTOMphys$.   $\Rphys$ is an orientation
preserving local diffeomorphism on the complement of the critical set, $\Cphys
\setminus \left(\II_{\pm \pi/2} \cup \TOPphys \cup \BOTTOMphys\right).$
   
\item[(P4)] 
$\Rphys$ is {\em postcritically finite} in the following sense: 
$\Cphys \setminus \left(\II_{\pm \pi/2} \cup \TOPphys \cup \BOTTOMphys\right)$ is backward invariant, and $\Rphys$ is a local
diffeomorphism on this set.  (Note: while postcritically finite maps typically have rather simple dynamics,  
$\Rphys$ is not {\em postsingularly finite}, since images of the curve $\Icurve$
are not eventually periodic,  leading to dynamical complexity of $\Rphys$.) 

\item[(P5)]

  $\Rphys|\, \BOTTOMphys: z\mapsto z^4$, while $\Rphys|\, \TOPphys: z\mapsto z^2$. Moreover, the
bottom circle is {\it uniformly superattracting}, namely, letting
$\Rphys(z,t)=(z',t')$, we have: $t' = O(t^2)$ for $t$ near $0$. The top circle is
{\it non-uniformly superattracting}, namely, near the top we have
$$
     \tau' = O\left(\frac{\tau^2}{\cos^2\phi}\right) = O\left( \frac{\tau^2}{\eps^2}\right) , 
    %       \quad {\mathrm {where}} 
     \quad \tau=1-t,\ 
   % \tau'= 1-t', \ 
        \eps=\pi/2-\phi \ \mod \pi Z, 
$$ 
    so that, the superattraction rate explodes near the points of indeterminacy.  
    See Figure \ref{FIG:CYLINDER_BASINS} for a computer image of the basins of attraction for $\BOTTOMphys$ and $\TOPphys$.
% Thus, $B$ has an open basin of attraction $W^s(\BOTTOMphys)$ within $\Cphys$ and $T$ 
% has a measure-theoretic basin of attraction $W^s(T)$ given by Pesin theory \cite{PESIN1,PESIN2,PS_PESIN}.  
%Figure\ref{FIG:CYLINDER_BASINS} shows a computer plot of these two basins of attraction.

%\bignote{The discussion of the basins belongs to the corresponding sections (``Low temp'' and ``High temp'' dynamics).
%   In fact, the high temp dynamics does not follow from the Pesin theory since the latter does not encounter pts of indeterminacy.
%   It is closer to Sinai-Bunimovich billiards or Katok-Strelcyn dynamics with singularities, but I doubt it can be reduced to 
%   either. Anyway, we do have a direct argument that treats it.}

\item[(P6)] The preimage $Q^{-1}(\Sphys)$ comprises two curves $S_{\pm}$ 
that are tangent to $\TOPphys$ at the indeterminacy points $\INDphys_{\pm}$
    and are symmetric with respect to $\II_{\pm}$ respectively. 
%The  domains $\La_{\pm}$ below them are called  ``principal tongues''. 
%    Each principal tongue is mapped by $R$ onto $\Cphys_-$, and this map is diffeomorphic on $\inter \La_{\pm}\sm I_\pm$.     
The domains below them are called the {\it (primary) central tongues} $\La_{\pm}$, see Figure \ref{FIG:R_ON_CYLINDER}.
The vertical intervals $\II_{\pm \pi/2}$ cut the corresponding tongues $\La_\pm\sm $ into  two topological triangles. 
Each of these (open) triangles  is mapped diffeomorphically by $\Rphys$ onto the appropriate triangle of $\Cphys_-\sm \II_\pi$. 
The inverse diffeomorphisms are called  {\it singular branches} of $\Rphys^{-1}$.

\begin{figure}
\begin{center}
\input{figures/R.pstex_t}
\end{center}
\caption{\label{FIG:R_ON_CYLINDER}
The mapping $\Rphys:\Cphys \rightarrow \Cphys$, the regions $\La_{\pm}$ in grey,
and the region $\Pi = \Cphys \sm \La_{\pm}$ in white (above).}
\end{figure}
 
\item[(P7)]
The  complement $\Pi := \Cphys \sm \La_+\cup \La_-$
 consists of two domains each of which is mapped diffeomorphically onto the cut rectangle
    $\Cphys\sm \II_\pi$.  The inverse diffeomorphisms are called {\it regular branches} of $\Rphys^{-1}$. 

\item[(P8)] $\Rphys$ has degree 4 over $\Cphys_-$ and it has degree 2 over $\Cphys_+$.

\item[(P9)] By Lemma \ref{vertical lifts}, 
every proper path $\gamma$ in $\Cphys$ has at least 4 proper lifts $\de_i$.
These lifts can meet only at the indeterminacy points $\INDphys_\pm$. 
If $\gamma$ crosses $\Icurve$ at a single point, then $\Rphys^{-1}\gamma=\cup\, \de_i$.
Two of these lifts (contained in $\La_\pm$) are ``singular'':
they terminate at the points $\INDphys_{\pm}$;  the other two are ``regular''. 
 
% This implies:

%\begin{lem}\label{vert lifts}
%  Every vertical path $\gamma$ in $\Cphys$ that does not cross $\II_\pi$
%  has at least $4^n$ vertical lifts.
%\end{lem}

\item[(P10)] 
$\Rphys$ acts with {\it degree 4} on  closed curves:
If $\gamma$ is any closed curve on $\Cphys$ wrapping once around $\Cphys$, then
$\Rphys(\gamma)$ is a  closed  curve wrapping four times around $\Cphys$.

\end{itemize}

\comm{  **********************************
For  further reference, let us write down an explicit expression for the differential 
$D\Rphys= (Df\circ Q) DQ$:
$$
   \frac{4}{\zeta^2} 
                        \left(
                         \begin{array}{cc}  
                                                   \zeta & 0 \\
                                                    0 & 1- t^2 
                         \end{array}   \right) 
                         \left(
                          \begin{array}{cc} 
                                                  1+t^2 \cos 2\phi & - \sin 2\phi \\
                                                 -t^2(1-t^2) \sin 2\phi & (1+t^2)(1+\cos 2\phi))     
                          \end{array} \right)
                                      \left(  
                        \begin{array}{cc}  
                                                    1 & 0 \\
                                                    0 & t
                         \end{array}   \right) 
$$
   where $\zeta(\phi, t) = \eta(2\phi, t^2)$. 

\bignote{The formula for $DR$ should either go to the Appendix or be dropped altogether --
   if we do calculations in $(u,v)$-coordinates.}

**************************}

\subsection{Structure of $R$ on the M\"obius band}\label{str on Mob band}
Because of the conjugacy $\correspond: \Cphysbl\ra  \Cmigbl$ from Proposition \ref{cylinder-band conjugacy}, 
the structural properties (P1-P10) for $\Rphys:\Cphys \rightarrow \Cphys$
discussed in \S \ref{str on CC} have immediate analogs for $\Rmig: \Cmig \rightarrow \Cmig$.
Particularly important is
\begin{itemize}
\item[(P9$'$)] Every proper path $\gamma$ in $\Cmigbl$ lifts under $\Rmig$ 
               to at least $4$ proper paths in $\Cmigbl$.  
               If $\gamma$ crosses $\Imig$ at a single point, then $\Rmig^{-1}\gamma=\cup\, \de_i$.
\end{itemize}
(Here a path in $\Cmigbl$ is called {\it proper} if it goes from $\T$ to $\infty$).
%\noindent
% which we will use in \S \ref{SUBSEC:SOLID_CYLINDERS} and \S \ref{SUBSEC:ALG_CONES}.

Also, we have:
\begin{itemize}
  
\item
  The principal LY locus $\SSS$ in $\Cphys$ (see (\ref{S})) is turned into the 
the vertical  line $S=\{\Re u =-1\}$ in $\Cmig$. (This is seen directly from the
formula (\ref{Z-0}) for the partition function). We will refer to $S$ as the 
{\it principal LY locus} in the affine coordinates.

\item
The indeterminacy points $\alpha_\pm\in \TOPphys$ for $\Rphys$
are turned into indeterminacy point $\pm i\in \T$ for $\Rmig$.

\item
  The blow-up locus $\Icurve$  (\ref{Icurve-2}) is turned into the parabola
\begin{equation}\label{Imig}
   \Imig =\{u:\  |u|=\Re u + 2\} = \{x+iy:\ x= \frac 1{4} y^2 -1\}
\end{equation}
(use $u^{-1} = te^{i\phi}$). See Figure \ref{FIG:ALG_CONE_FIELD_INVARIANCE}.

% \item The formula for the Jacobian becomes assumes a very simple form:

\end{itemize}

%\bignote{Mention the Jacobian ?}

Let us rotate the LY locus $S$ around the circle $\T$.
We obtain a family of lines $\Smig_\phi= e^{i\phi} S$ tangent to $\T$ at $e^{i\phi}$.
Let $\Smig^c_\phi = \{e^{-i\phi}U + 2V + e^{i\phi}W = 0\}$ be the corresponding complex line in $\CP^2$.  
By Corollary \ref{COR:DEGREE_GROWTH_MK}, 
the pullback $\Rmig^*(\Smig^c_\phi)$ is a complex algebraic curve of  degree $4$. 
By Bezout's Theorem, it intersects the conic $\Lone=\{uw=1\}$ 
(which is the complexification of the circle $\T$) at 8 points counted with multiplicity.  

\begin{lem}\label{pulback of S-phi}
{\rm (i)} If  $\phi\not=\pi$ (i.e., $\Smig_\phi\not= \Smig$), then the conic $\Lone$ intersects the pullback 
$\Rmig^*(\Smig^c_\phi)$ transversally  at four transverse double points 
($\pm e^{\psi/2}$ and the indeterminacy points $\pm i$). 
So, each intersection has multiplicity 2.

{\rm (ii)}
  If  $\phi=\pi$ (i.e., $\Smig_\phi = \Smig$), then $\Lone$ intersects $\Rmig^*(\Smig^c_\phi)$ tangentially
at two first order tangential double points (the indeterminacy points $\pm i$). 
So, each intersection has multiplicity 4. 
\end{lem}

\begin{proof}
(i)  By Lemma \ref{description of folds},  the map $R$ is a Whitney fold at $\pm e^{i\psi/2}$.
By Lemma \ref{pullback of parabola}, the germ of $\Rmig^*(\Smig^c_\phi)$ at  $\pm e^{i\psi/2}$ 
is a transverse double point transversely intersected by the critical locus $\Lone$. 
  
To understand the germ of $\Rmig^*(\Smig^c_\phi)$ at an indeterminacy point $a\in \{\pm i\}$,
let us blow it up  
and  lift $R$ to a map $\tl R: \tl \CP^2\ra \CP^2$, see Appendix \ref{APP:BLOW_UPS}. 
The blow-up locus $\Imig=\tl R(\EE_\ex)$ intersects $\Smig_\psi$ transversely at two points
which are regular values for $\tl R$ (Lemma \ref{five lines and conic}). 
% see Figure \ref{FIG:ALG_CONE_FIELD_INVARIANCE}. 
Hence  the curve $\tl R^* (\Smig_\psi^c)$ 
intersects the exceptional divisor $\EE_\ex$ transversely at two points. 
Projecting the corresponding germs to $\CP^2$, we obtain  two branches of  $R^* (\Smig_\psi^c)$  at $a$. 

\ssk (ii)
By Lemma \ref{description of folds}, 
the blow-up map $\tl R: \tl \Lone\ra \Lone$
is a Whitney fold at the intersection  point $\tl a = \tl \Lone\cap \EE_\ex$.
Hence the germ of ${\tl R}^*(\Smig_\psi)$  
has a transverse double point at $\tl a$ and intersects $\EE_\ex$ generically.
Its projection to $\CP^2$ is a pair of regular curves tangent to $\Lone$ at $a$
(see Lemma \ref{pencil}).  
\end{proof}

\begin{rem}
  The above lemma is reflected in the geometry of the initial LY loci
 illustrated on Figure \ref{FIG:LEE_YANG_ZEROS}.
The locus $\Sphys_1=\Rphys^{-1}\Sphys_0$ looks like two tangent parabolas near the indeterminacy points $\alpha_\pm$  
(part (ii) of the lemma). The next locus,
$\Sphys_2= \Rphys^{-1}\Sphys_1$, comprises $32$ branches meeting transversely at the top, 
as part (i) asserts.   
\end{rem}

\section{Structure of the RG transformation II: Global properties in $\CP^2$}\label{SEC:GLOBAL STRUCTURE}

The Lee-Yang Theorem  % (see \S \ref{SEC:MODEL}) 
places special emphasis of the dynamics of $\Rphys$ on the cylinder $\Cphys$. 
However, it is instructive to understand the global dynamics of $\Rphys$
on the projective space $\CP^2$, 
which has important consequences for the dynamics of $\Rphys| \Cphys$.
In this section we will describe basic global properties of $\Rphys$, 
along with those of $\Rmig:\CP^2 \ra \CP^2$.    %  given by (\ref{uv coord})  

%\subsection{Relationship between $\Rmig$ and $\Rphys$:}
\subsection{Semiconjugacy $\correspond$}\label{SUBSEC:SEMICONJUGACY}
The mapping $\correspond$ given by (\ref{zt-uv}) is a degree two
rational map $\CP^2 \ra \CP^2$.  
In homogeneous coordinates $[Z:T:Y]$ in the domain (with $z=Z/Y,\ t=T/Y$) 
and $[U:V:W]$ in the image (with $u= U/V,\ w=W/V$), it assumes the form
\begin{eqnarray*}
% [U:V:W] = \correspond([Z:T:Y]) = [Y^2:ZT:Z^2].
U=Y^2,\quad V=ZT, \quad W=Z^2. 
\end{eqnarray*}
A generic point $[U:V:W]$ has two preimages under $\correspond$.
The critical locus of $\correspond$
is the union of the vertical axis $\{Z=0\}$ and the line at infinity $\{Y=0\}$. 
Under $\correspond$, the former collapses to an $R$-fixed point  $e=[1:0:0]$,
while the latter maps onto the vertical axis $\{U=0\}$.
Since $e$ does not lie on this axis, the intersection  $\gamma=[0:1:0]$ 
of  the two critical lines must be an indeterminacy point for $\correspond$
(and in fact, this is the only one). 

This collapsing line $\{Z=0\}$ and the associated indeterminacy point $\gamma$
 created by the change of variable $\correspond$ is what makes 
the physical coordinates $(z,t)$ less suitable for describing the global structure of the
renormalization. 
%
%Under the map $\Rphys$, the line $\{Z=0\}$ collapses to its fixed point $\CFIXmig=[0:1:1]$,
%while $\gamma$ becomes an extra indeterminacy point.

\comm{***
\begin{prop}\label{PROP:SEMI_CONJUGACY}The mapping $\correspond$ induces a semi-conjugacy
between $\Rphys$ and $\Rmig$ so that the following diagram commutes:
\begin{eqnarray}
\begin{CD}
\CP^2 @>\Rphys>> \CP^2 \\
@VV\correspond V 	@VV\correspond V\\
\CP^2 @>\Rmig>> \CP^2
\end{CD}
\end{eqnarray}
\end{prop}

\noindent
The proof is a direct calculation and we omit it.
***}

\subsection{Indeterminacy points for $\Rphys$ and $\Rmig$}
\label{SUBSEC:INDETERMINANT_PTS}

In homogeneous coordinates on $\CP^2$, the map $\Rmig$ has the form:
\begin{equation} \label{EQN:MK_HOMOG}
  \Rmig : [U: V: W] \mapsto [(U^2+ V^2)^2: \  V^2(U + W)^2: \ (V^2+W^2)^2)]
\end{equation}
\noindent
which is just (\ref{hat R}) with $(U,V,W)$ interpreted as the homogeneous coordinates.  
% The points of indeterminacy are precisely
% where all three coordinates in (\ref{EQN:MK_HOMOG}) vanish.  
We find two points of indeterminacy: $\INDmig_+ := [i:1:-i]$ and $\INDmig_-:=[-i:1:i]$.
They lie on the M\"obius band $\Cmig$ 
and correspond under $\correspond | \Cphys$ to the indeterminate points $\INDphys_\pm \in \Cphys$ that we
discussed in \S \ref{SUBSEC:MAP_ON_CYL}.

If we now write $\Rphys$ in homogeneous coordinates we obtain
\begin{eqnarray}\label{EQN:R_HOMOG}
\begin{split}
 & \Rphys:[Z:T:Y] \mapsto \\ & [Z^2(Z^2+T^2)^2:  T^2(Z^2+Y^2)^2: (Z^2+T^2)(T^2 Z^2+Y^4)].
\end{split}
\end{eqnarray}
\noindent
We find the indeterminate points $\INDphys_\pm = (\pm i,1)  \in \Cphys$,
two symmetric points $(\pm i , -1)$, 
and two additional points of indeterminacy,
$\B0=(0,0)$ and $\gamma=[0:1:0]$
(here all the points except the last one are written in the physical coordinates $z=Z/Y,\, t=T/Y$).
In this way, 
when we turn $\Rmig$ into $\Rphys$ by the change of variable $\correspond$ 
we create two accidental points of indeterminacy, $\B0$ and $\gamma$,  
which makes the global properties of the map more awkward. 
% In particular, we will see an essential
% consequence in \S \ref{SUBSEC:ALG_STABILITY} when we
% consider the growth in degrees for preimages of curves under iterates of $\Rphys$
% and $\Rmig$.  

One can resolve all the indeterminacies of $\Rphys$ using suitable blow-ups.  
We will only need resolutions of $\INDphys_\pm$, 
which are described in Appendix \ref{APP:BLOW_UPS}.

\subsection{Superattracting fixed points and their separatrices}
\label{SUBSEC:FIXED POINTS} 

\sss{Description in terms of $R$ (\ref{EQN:MK_HOMOG})}

We will often refer to $\Lzero:= \{V=0\}\subset \CP^2$ as the  {\it line at infinity}. 
It contains two symmetric fixed  points, $e=(1:0:0)$ and $e'=(0:0:1)$.
In local coordinates $(\xi= W/U, \, \eta= V/U)$ near $e$, the map $R$ assumes form
\begin{equation}\label{dynamics near e}
   \xi'= \left( \frac{\xi^2+\eta^2}{1+\eta^2}\right)^2\sim (\xi^2+\eta^2)^2,\quad 
    \eta'=\eta^2\left(\frac{1+\xi}{1+\eta^2}\right)^2\sim \eta^2,
\end{equation}
so $|Rx|\leq 2 |x|^2$ for small $x=(\xi, \eta)$.
This shows that $e$ is superattracting: 
$$
      |R^n x|\leq |2x|^{2^n}.
$$ 
By symmetry, $e'$ is superattracting as well.
Let $\WW^w(e)$ and $\WW^s(e')$ stand for the attracting basins of these points. 
 
Moreover, the line at infinity  $\Lzero=\{\eta=0\}$ is $R$-invariant, and the restriction $R|\Lzero$ is the 
power map $\xi\mapsto \xi^4$. Thus,  points in the disk $\{ |\xi|<1\}$ in $\Lzero$
are attracted to $e$, points in the disk $\{ |\xi|> 1\} $ are attracted to  $e'$,
and these two basins are separated by the unit  circle $\BOTTOMmig$.  
We will also call $\Lzero$ the {\it fast separatrix} of $e$ and $e'$. 

Let us also consider the conic 
\begin{equation}\label{Lone}
   \Lone = \{ \xi=\eta^2\}= \{V^2= UW\} 
\end{equation}
passing through points $e$ and $e'$.
It is an embedded $\CP^1$ that can be uniformized by coordinate $w=W/V= \xi/\eta$. 
Formulas (\ref{dynamics near e}) show that $\Lone$ is $R$-invariant,
and the restriction $R|\, \Lone$ is the quadratic map $w\mapsto w^2$.
Thus,  points in the disk $\{ |w|<1\}$ in $\Lone$
are attracted to $e$, points in the disk $\{ |w|> 1\} $ are attracted to  $e'$,
and these two basins are separated by the unit  circle $\TOPmig$. 
We will call $\Lzero$ the {\it slow separatrix} of $e$ and $e'$.

If a point $x$ near $e$ (resp. $e'$) does not belong to the fast separatrix $\Lzero$,
then its orbit is ``pulled'' towards the slow separatrix $\Lone$ at rate $\rho^{4^n}$,
with some $\rho<1$, 
and converges to $e$ (resp. $e'$) along $\Lone$ at rate $r^{2^n}$, with some $r<1$.  

The second formula of (\ref{dynamics near e}) also shows that
the strong separatrix $\Lzero$ is transversally superattracting:
all nearby points are pulled towards $\Lzero$ uniformly at rate $r^{2^n}$
(see also the proof of Lemma \ref{description of folds}).  
It follows that these points either converge  to one of the fixed points, $e$ or $e'$,
or converge to the circle $\BOTTOMmig$. 

Given a neighborhood $\Om$ of $\BOTTOMmig$, let 
\begin{equation}\label{complex basin of B}
  \WW^s_{\C,\loc} (\BOTTOMmig)= \{x\in \CP^2:\ R^n x\in \Om\ (n\in \N) \ {\mathrm {and}}\ 
        \R^n x\to \BOTTOMmig\ \mathrm{as}\ n\to \infty \}
\end{equation}                                                                 
(where $\Om$ is implicit in the notation, and an assertion involving $\WW^s_{\C,\loc}$ means
that it holds for arbitrary small suitable neighborhoods of $\BOTTOMmig$). 

We conclude:
\begin{lem}\label{nbd of B}
 $\WW^s(e)\cup \WW^s(e')\cup \WW^s_{\C,\loc}(\BOTTOMmig)$ fills in some neighborhood of $\Lzero$.
\end{lem}
        
As the weak separatrix $\Lone$ is concerned, formula (\ref{tau squared})
from the proof of Lemma~\ref{description of folds} shows that it is transversally  superattracting
away from the indeterminacy points $a_\pm$. On the other hand, the latter act as strong repellers.
We will see in \S \ref{SEC:HIGH_TEMP}  
that this competition makes $\TOPmig$ a  non-uniformly hyperbolic attractor.    

\sss{Description in terms of $\Rphys$}
In the physical coordinates, the superattracting fixed points become  $\CFIXphys =(0,1)$ 
and $\CFIXphys'=[1:0:0]$.
The pullback of the line at infinity $\Lzero$ under the semi-conjugacy $\correspond$ comprises
two lines, $\LL_0=\{t=0\}$ and $\{z=0\}$, where the latter is the blow-up
of the fixed point $e$ under $\correspond^{-1}$.
These two lines form the fast separatrix of the fixed points
(recall that the latter collapses to $\CFIXphys =(0,1)$ under $\Rphys$),
which is an annoying artifact of the physical coordinates. 
A related nuisance is that $\LL_0$, unlike $\Lzero$, is not transversally superattracting
any more. Namely, it is superattracting away from the origin $\B0$,
but the latter blows up to the whole line $\{z=0\}$.  
Still, we will sometimes refer to $\LL_0$ itself as the ``fast separatrix'',
as long as it does not lead to confusion.    

The slow separatrix of the fixed points is the line $\{t=1\}$. 

The restrictions of $\RR$ to the separatrices $\LL_0$ and $\LL_1$ become the power maps
$z\mapsto z^4$ and $z\mapsto z^2$ respectively. The invariant circles on these lines
(separating the basins of the fixed points) 
become  $\BOTTOMphys=\T\times \{0\}$  and $\TOPphys=\T\times \{1\}$,
which are the bottom and the top of the physical cylinder   $\Cphys$
that we discussed in \S \ref{SEC:STRUCTURE}.    

Lemma \ref{nbd of B} implies:

\begin{lem}\label{nbd of Bphys}
 $\WW^s(\CFIXphys)\cup \WW^s(\CFIXphys')\cup \WW^s_{\C,\loc}(\BOTTOMphys)$ fills in some neighborhood of 
$\LLzero\sm \{\B0\}$.
\end{lem}

\subsection{Critical locus} 
  The critical locus of $\Rmig$ is described in Appendix, \S \ref{App: crit locus}.
Besides the separatrices $\Lzero$ and $\Lone$,
it comprises the line $\Ltwo$ that collapses to the low temperature fixed point $\FIXmig_0$,
and two symmetric pairs of lines, $\Lthree$ and $\Lfour$. 
The latter wander under the dynamics.  

The critical locus is schematically depicted on Figure \ref{FIG:CRITICAL_CURVES}, 
while its image, the critical value locus,  is depicted on Figure \ref{FIG:MK_CONFIGURATION}.  

\begin{figure}
\begin{center}
\input{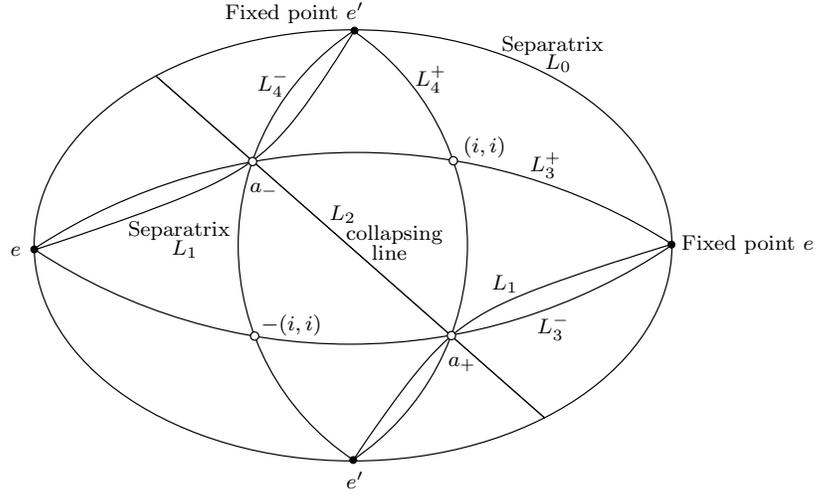}
\end{center}
\caption{\label{FIG:CRITICAL_CURVES}Critical locus for $\Rmig$ shown with the separatrix $\Lzero$ at infinity.}
\end{figure}

\begin{figure}
\begin{center}
\input{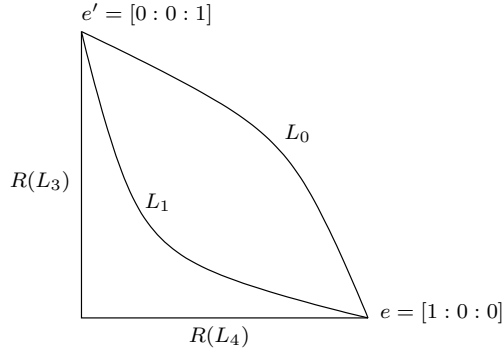}
\end{center}
\caption{Critical values locus of $R$.\label{FIG:MK_CONFIGURATION}}
\end{figure}

\bignote{}

In terms of the physical coordinates, the critical locus comprises:
\begin{itemize}
\item  $\correspond^{-1}\Lzero$: the fast separatrix $\LL_0\cup \{z=0\}$;
\item $\correspond^{-1}\Lone$:  the slow separatrix $\LL_1$ and its companion $\{t=-1\}$ 
      \\ (mapped to $\LL_1$ under $\Rphys$); 
\item $\correspond^{-1} \Ltwo$: two collapsing lines,  $z=\pm i$;
\item $\correspond^{-1} \Lthree$: two conics $zt=\pm i$;
\item $\correspond^{-1} \Lfour$: two lines $z=\pm i t$ (symmetric to the above conics). 
\end{itemize}

\begin{figure}
\begin{center}
\input{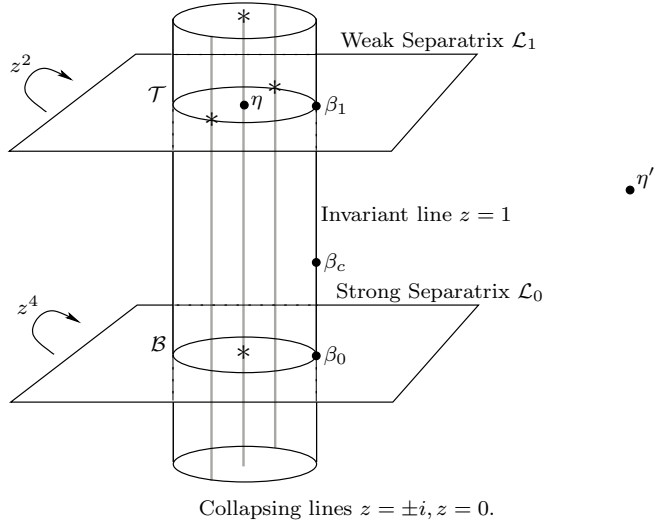}
\end{center}
\caption{\label{FIG:GLOBAL_CYLINDER}The LY cylinder $\Cphys$ situated between the strong separatrix $\LL_0$ and the weak separatrix $\LL_1$.  The collapsing lines at $z=0,\pm i$ are shown in grey and 
the indeterminate points $\INDphys_\pm, \B0, \gamma$ are depicted by stars.  The superattracting
fixed point $\eta' = (0,\infty)$ is ``symbolically'' shown at a finite location.}
\end{figure}

\comm{****
\subsection{Global structure of $\Rmig$ on complex projective space}
\label{SUBSEC:GLOBAL_MK}
We will now summarize global properties of $\Rmig$, as given in (\ref{EQN:MK_HOMOG}).
Most apparent is the symmetry $\rho: [U:V:W] \mapsto [W:V:U]$
 which is equivalent to the symmetry $\rho$ for $\Rphys$ expressed in Property (P1).

\msk
At infinity with respect to the usual affine coordinates $(u,w)$ is the line $V = 0$, which is forward invariant and
transversally superattracting like $v_{n+1} = O(v_n^2)$ (in either of the possible affine coordinates $v= \frac{V}{U}$ or $v = \frac{V}{W}$.)    Within $V=0$ we have $\Rmig(U/W) =
(U/W)^4$ leaving invariant the ``bottom circle'' $\BOTTOMmig = \{|U| = |W|, V=0\}$ from $\Cmig$.  Also on this line are the
two superattracting fixed points
$\CFIXmig := [1:0:0]$ and $\CFIXmig' := \rho(\CFIXmig) = [0:0:1]$ which attract the orbit of any point not on $\BOTTOMmig$.
These points are superattracting within $\C^2$ so that
any point in a sufficiently small neighborhood of $V=0$ must be in one of the
three stable sets $W^s(\CFIXmig), W^s(\CFIXmig')$, or $W^s(\BOTTOMmig)$.

There are three additional fixed points for $\Rmig$ on $V =0$ within $\BOTTOMmig$.  One of
these is $\FIXmig_0:=[1:0:1] \in \Cmig$ that was mentioned in \S \ref{SUBSEC:MAP_ON_CYL}, corresponding to the
``low-temperature fixed point'' $\FIXphys_0 \in \Cphys$.

In the affine coordinates $u = U/V, w=W/V$ we find three more fixed points.
Two are the fixed points $\FIXmig_1,\FIXmig_c \in \Cmig$ that correspond to
$\FIXphys_1, \FIXphys_c \in \Cphys$.  The third is not in $\Cmig$ and is
repelling.  

Thus, $\CFIXmig$ and $\CFIXmig'$ are the only attracting fixed points for $\Rmig$. It is unknown whether there are any
attracting periodic points for $\Rmig$ with period greater than $1$, however computer experiments suggest that
they do not exist.
See \S \ref{SUBSEC:SOLID_CYLINDERS}.

\msk
The critical points for $\Rmig$ are the union of five curves:
\begin{eqnarray*}
C_1 &:=& \{V=0\},\\
C_2 &:=& \{V^2 = UW\},\\
C_3 &:=& \{U = -W\},\\
C_4 &:=& \{U^2 = -V^2\}, \mbox{ and }\\
C_5 &:=& \{W^2 = -V^2\}.
\end{eqnarray*}

Using (\ref{EQN:MK_HOMOG}) one can check that $C_1$ and $C_2$ are forward invariant, that $C_3$ collapses to $[1:0:1] \in C_1$, and that $C_4$ and $C_5$
are mapped to the lines $U=0$ and $W=0$, respectively.  Thus, the critical value locus is
\begin{eqnarray*}
D_1 &:=& \{V=0\},\\
D_2 &:=& \{V^2 = UW\},\\
D_3 &:=& \{U=0\}, \mbox{ and } \\
D_4 &:=& \{W=0\}.
\end{eqnarray*}

Since $D_1$ and $D_2$ are forward invariant and $D_4 = \rho(D_3)$, there is one free critical orbit $\Rmig^n(D_3)$, up to the symmetry $\rho$.

\msk

We return our attention to $\CFIXmig$ and $\CFIXmig'$, which are the only attracting fixed points of $\Rmig$.
The two forward invariant critical curves $D_1 = \{V=0\}$ and $D_2 = \{V^2 = UW\}$ intersect at
$\CFIXmig$ and at $\CFIXmig'$.  As mentioned earlier, the dynamics of $D_1$ is given by taking the fourth power of either local coordinate $(U/W)$ or $(W/U)$.
The dynamics on $D_2$ is equally simple: if we use $z=W/U$ as uniformizing coordinate, then, except at the two points of indeterminacy $\INDmig_\pm \in D_2$, we
have $\Rmig(z) = z^2$.

Therefore both $\CFIXmig$ and $\CFIXmig'$ are superattracting fixed points with
$D_1$ and $D_2$ as invariant separatrices.  Points in a neighborhood of
$\CFIXmig$ (or $\CFIXmig'$, respectively) not on either of these separatrices will converge to
$\CFIXmig$ ($\CFIXmig'$, respectively) asymptotically to $D_2$ at rate $|z|^{2^n}$ (or
$|1/z|^{2^n}$, respectively).

Since $\CFIXmig$ and $\CFIXmig'$ are attracting they have open basin of attraction $W^s(\CFIXmig)$ and $W^s(\CFIXmig')$ consisting of all
points whose orbits converge to $\CFIXmig$ or $\CFIXmig'$. We will study $W^s(\CFIXmig)$ and $W^s(\CFIXmig')$ in \S \ref{SUBSEC:SOLID_CYLINDERS}.

The configuration of $D_1,D_2,D_3,$ and $D_4$ relative to $\CFIXmig$ and $\CFIXmig'$ is shown schematically in Figure
\ref{FIG:MK_CONFIGURATION}.

\begin{figure}
\begin{center}
\input{figures/MK_configuration.pstex_t}
\end{center}
\caption{Configuration of critical (value) curves $D_1,D_2,D_3,$ and $D_4$ for $\Rmig$ and the superattracting fixed points
$\CFIXmig, \CFIXmig'$.\label{FIG:MK_CONFIGURATION}}
\end{figure}
*****}

\subsection{Topological degree}
The {\em topological degree} $\deg_{top}(f)$ of a rational mapping $f: \mathbb{P}^k \rightarrow \mathbb{P}^k$ is the number of 
preimages under $f$ of a generic point $\zeta \in \mathbb{P}^k$.

\begin{prop}
We have $\deg_{\rm top}(\Rmig) = \deg_{\rm top}(\Rphys) = 8$.
\end{prop}

\begin{proof}
This can be seen for $\Rmig$ by taking, e.g.,  a point $\zeta = (u,u)\in \C^2$  far away,
and hence close to the separatrix $\Lzero$.  
Such a point  has $8$ preimages under $\Rmig$,
since the transverse degree of $\Rmig$ at $\Lzero$ is equal to 2,
while $\deg(\Rmig|\, \Lzero)=4$.

Similarly for $\Rphys$, consider a generic point $\zeta$ sufficiently close to, but not on, the separatrix $\LLzero$.
\end{proof}

%\subsection{Degree for preimages of curves under $\Rmig^n$ and $\Rphys^n$}
\subsection{Algebraic degrees and pullbacks of curves}
\label{SUBSEC:ALG_STABILITY}

The reader can consult  Appendix \ref{APP:DEGREE} for needed
background in elementary algebraic geometry.  
%In particular, the definition of {\em divisor} is presented
%there, which will will need to properly account for multiplicities when pulling
%back algebraic curves.  Throughout this paper, whenever taking the pull-back of
%an algebraic curve, we will do so in the sense of divisors.

\sss{Case of $\Rmig$}
Since $\Rmig$ is given in homogeneous coordinates by relatively prime equations of degree $4$, we have
$\deg \Rmig  = 4$. 

The notion of algebraic stability
 is essential to understanding pullbacks of curves (considered as divisors)
 under iterates of rational mappings, see Appendix \ref{APP:DEGREE}.

\begin{prop}\label{alg stab}
The mapping $\Rmig:\mathbb{CP}^2\rightarrow \mathbb{CP}^2$ is algebraically stable.
\end{prop}

\begin{proof}
The only collapsing curve is $L_2$, whose orbit lands on the low-temperature fixed point $\FIXmig_0$.
\end{proof} 

It follows that $\deg \Rmig^n  = \left( \deg \Rmig\right)^n=4^n$, and hence we have: 
% and the following Corollary to Lemma \ref{LEM:PULL_BACKS_AS}:

\begin{cor}\label{COR:DEGREE_GROWTH_MK}
If $D$ is an algebraic curve of degree $d$,
then the pullback $(\Rmig^n)^* D $ is a divisor of degree $d\cdot 4^n $.
\end{cor}

\sss{Case of $\Rphys$}

Since $\Rphys$ is given in homogeneous coordinates by relatively prime equations of degree $6$,
we have $\deg \Rphys = 6$.  In particular, for any algebraic curve $X$ we have $\Rphys^* X$ is a
divisor of degree $6 \cdot \deg X$.
%
%\begin{lem}\label{deg pullback X}
% Let $X$ be an algebraic curve in $\CP^2$ that does not pass through
%the  points $\B0$, $\CFIXphys'$, and is transverse to the line $\LLzero$. 
%Then $\deg \Rphys^*X= 4\, \deg X$.
%\end{lem}    \note{too many conditions?}
%
The degrees of pullbacks under iterates of $\Rphys$ are less organized:

\begin{obs}
The mapping $\Rphys:\CP^2 \ra \CP^2$ is not algebraically stable.
\end{obs}

Indeed,  $\Rphys$ maps the lines $Z=\pm i T$ to the point of indeterminacy
$\gamma=(0:1:0)$ since the first and third coordinates of (\ref{EQN:R_HOMOG}) contain the factor
$(Z^2+T^2)$.

\begin{rem}
In this case, algebraic instability results in a drop of degree for the second
iterate of $\Rphys$.  We have ${\rm deg}(\Rphys^2) =
28 < 36 = \left({\rm deg}\Rphys\right)^2$, since the common factor of $(Z^2+T^2)^4$ appears in the
expression for $\Rphys^2$, which must be canceled 
(compare Remark \ref{geom deg deficit}).  A cohomological calculation (which we omit) yields the {\em dynamical degree} $\delta(\Rphys) := \lim_{n\rightarrow \infty} \left(\deg \Rphys^n \right)^{1/n} = 4$.
\end{rem}

%However, we can easily compute the dynamical degree
%$\delta(\Rphys) := \lim_{n\rightarrow \infty} \left(\deg \Rphys^n \right)^{1/n}$:
%
%\begin{prop}
%We have $\delta(\Rphys) = 4$.
%\end{prop}
%
%\begin{proof}
%The semi-conjugacy $\correspond$ guarantees that $ \deg \Rmig^n$ and       
%$\deg \Rphys^n $ can only differ by a bounded amount.                 \note{does it ?}
%%% Therefore, $\delta_\Rphys = \delta_\Rmig = {\rm deg}(\Rmig) = 4.$
%\end{proof}

% Therefore, $\deg (\Rphys^n)^*X$ will grow asymptotically like
% $(\delta_\Rphys)^n\cdot \deg X = 4^n \cdot \deg X$, even though $\Rphys^* X$ has
% degree $6 \cdot \deg X$.  (See Appendix \ref{APP:DEGREE}.)

\begin{rem}\label{junk zeros}
Note that the preimage $\Psi^{-1} (R^{-n} S)$ is exactly the Lee-Yang locus $\SSS_n$ of degree $2\cdot 4^n$.
On the other hand, the preimage $\RR^{-n} (\SSS)=R^{-n} (\Psi^{-1} S)$ contains,
besides  $\SSS_n$, 
some ``junk'' components that collapse to the indeterminacy point $\gamma$ under some
iterate $\RR^k$, $k=0,1,\dots, n-1$.
So, the commutative diagram (\ref{comm diagram}) should be applied with caution. 
\end{rem}

\section{Proof of the Lee-Yang Theorem for DHL}\label{LY for DHL sec}

In this section we will give an easy proof of the Lee-Yang Theorem for the DHL
by means of ``enumerative dynamics''.

\begin{thm}\label{dynamical LY thm}
  The Lee-Yang locus $\SSS_n$ intersects any complex line $\Secphys_t:= \C\times \{t\}$, $t\in [0,1)$,
in $2\cdot 4^m$ distinct points on the unit circle $\T$.
\end{thm}

\begin{proof}
By (\ref{symmetric Laurent}), the partition function $Z_n$ is a symmetric Laurent polynomial in $z$ of degree $4^n$,
so it has $2\cdot 4^n$ zeros on every complex line in question.
But by (\ref{Z_n}), $Z_n= Z\circ \RR^n$ 
(where the $Z_n$ should be written in the physical coordinates), 
so every point $(z,t)$ of $\RR^{-n} Z$ 
which is a regular point for all  $\RR^k$, $k=0,1,\dots, n-1$ (compare Remark \ref{junk zeros}  below)  
is a Lee-Yang zero. But Property P9 supplies us with $2\cdot 4^n$ such zeros on the unit circle
of $\Secphys_t$. Hence it accounts for all of the zeros.
\end{proof}

Let us formulate the corresponding statement in the Migdal coordinates. In these coordinates,
the horizontal complex lines $\Pi_t$ %  =\{(z,t): z\in \C\}$
turn into the conics 
$$
 \Secmig_t:=\{ |uw|= t^{-2} \}.
$$

A complex line $L=\{au+bw+c=0 \}$ is called {\it Hermitian} if it is invariant under the
antiholomorphic involution $(u,w) \mapsto (\bar w , \bar u)$.%
\footnote{Equivalently, $b= \bar a$, $c\in \R\sm \{0\}$ or $|b|=|a|$, $c=0$. } 
The slice of such a line by the real plane $\{w=\bar u\}$ is a real line
(otherwise it would be a single point).

\begin{thm}\label{LY in Mig coord-s}
  Let $t\in [0,1)$ and let $L$ by any Hermitian complex line crossing the top $\TOPphys=\T$ 
of the cylinder $\Cmig$.
Then the pullback $(R^n)^*L$ intersects the horizontal complex line $\Secmig_t$ in 
$2\cdot 4^n$ simple points, all on the cylinder $\Cmig$. 
\end{thm}

\begin{proof}
By Corollary \ref{COR:DEGREE_GROWTH_MK},  $(R^n)^* (L)$ has degree $4^n$,
so by Bezout Theorem, it has $2\cdot 4^n$ intersection points with the conic $\Secmig_t$.
On the other hand, $L \cap \Cmig$ comprises two vertical intervals on the cylinder $\Cmig$.
By property (P9) from \S \ref{str on CC},  $(R^n)^*(L) \cap \Cmig$  comprises  at least 
$2\times 4^n$ vertical curves on $\Cmig$  (connecting the top  $\TOPmig= \T$ to the bottom $\BOTTOMmig$ at infinity). 
They have at least $2\cdot 4^n$ different intersections with the circle $\{ |u|= t^{-1}\}=\Secmig_t\cap \Cmig$.
Hence all the intersection points of $(R^n)^* (L)$  with $\Secmig_t$ are captured on the cylinder $\Cmig$,
and all of them are simple. 
\end{proof}

\section{Algebraic cone field}\label{sec: alg cone field}

In this section we will construct a horizontal invariant cone field on the cylinder. 
It appears as the tangent cone
field  to a  pair of transverse algebraic foliations obtained by translating
the principal Lee-Yang locus around the cylinder or the M\"obius band. 
In the affine coordinates on the M\"obius band,  
 these foliations assume a particular simple linear form.

\subsection{Algebraic cone fields}
\label{SUBSEC:ALG_CONES}
Let us consider the % $R$-invariant  
M\"obius band $\Cmig$ introduced in \S \ref{SUBSEC:MAP_ON_CYL},
\begin{equation*} 
\inter \Cmig = \{(u,w) \in \C^2 : u = \overline{w} \,\,\mbox{and} \,\, |u| > 1\}\isom \C\sm \bar \D. 
\end{equation*}
\noindent
Recall that $\Smig_\psi$ stands for the line tangent to $\T$ at $e^{i\psi}$.
We define an algebraic horizontal cone field $K^\hor(u)$ on $\inter \Cmig$ as follows. 
% Since $u$ is outside of the unit disc, 
For any  $u\in \C\sm \bar \D$, 
there are two tangents lines $\Smig_{\psi_1}$ and $\Smig_{\psi_2}$ passing through $u$.  
Then, $K^\hor(u)$ is the open  cone%
\footnote{Here a ``cone'' comprises two symmetric wedges.
Also, more precisely  one should think of $K^\hor(u)$ as the {\it tangent} cone at $u$.}
bounded by $\Smig_{\psi_1}$ and $\Smig_{\psi_2}$ that does not contain $\T$ (see Figure \ref{FIG:DEFINING_CONE_FIELD}).

\begin{figure}
\input{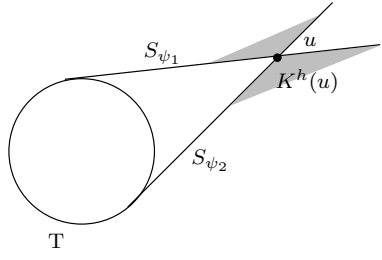}
\caption{\label{FIG:DEFINING_CONE_FIELD} An algebraic horizontal cone $K^\hor(u)$.}
\end{figure}

\begin{figure}
\begin{center}
\input{figures/alg_cone_field.pstex_t}
\end{center}
\caption{\label{FIG:CONEFIELD1} A horizontal cone $\KK^\hor(\omega)$.}
\end{figure}

This construction can be described in terms of the principal LY locus $\Smig=\{ \Re u = -1\}$
on $\Cmig$ (see \S \ref{str on Mob band}). 
The line $\Smig$ is a concatenation of two rays,  $S^+=\{u\in \Smig: \ \Im u>0\}$ and $S^-=\{u\in \Smig: \ \Im u<0\}$,
 meeting at $-1\in \T$.
Rotating these rays around the origin, we obtain two linear foliations $\Phi^\pm$  of $\inter \Cmig$
(comprised of the leaves  $\Smig_\theta^\pm = e^{i\theta} S^\pm$).
The cone $K^\hor(u)$ is bounded by (the tangent lines to) the leaves
 of these foliations passing through $u$.  
% $:= \{{\rm Re}(e^{-i\theta} u) = 1\}$ of the principal locus we obtain each of the
%possible tangent lines to $\T$. Thus, at each $u \in \inter \Cmig$ the cone
%$K(u)$ is bounded by two appropriate rotations $\Smig_\theta$ and $\Smig_\phi$
%of the principal locus $\Smig$.

Described in this way,
the construction can be immediately transferred to the cylinder $\Cphys$.
Rotating the principal LY locus $\SSS$ around the cylinder, we obtain two transverse foliations of $\inter\Cphys$.
Then the horizontal cone field  is formed by the tangent cones $\KK^\hor(\om)$
bounded by the tangent lines to the two leaves meeting at $\om$ (see Figure \ref{FIG:CONEFIELD1}).
Clearly  $\KK^\hor(\om)=D\Psi^{-1}(K^\hor(u))$, where $u=\Psi(\om)$. 

\begin{rem}
In fact, this cone field extends to the bottom $\BOTTOMphys$ of $\Cphys$, so it is well defined on the 
topless cylinder $\Cphystl$. However, it degenerates to a line field at the top.
\end{rem}

A smooth path $\gamma(t)$ in $\Cmig$ is called {\it horizontal} 
if  it goes through the cones $K^\hor(x)$, i.e. $\gamma'(t)\in K^\hor(\gamma(t))$ 
whenever  $\gamma(t)\in \inter \Cmig$. (The same definition applies to the cylinder $\Cphys$.)

\begin{lem}\label{S transv to Icurve}
The blow-up locus $\Imig$ (respectively $\Icurve$) is horizontal.    
\end{lem}   

\begin{proof}
  See Figure \ref{FIG:hl Imig}. 
\end{proof}

\begin{figure}
\begin{center}
\input{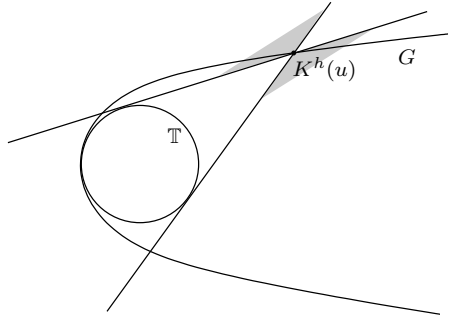}
\end{center}
\caption{\label{FIG:hl Imig} The blow-up locus $\Imig$ is horizontal.}
\end{figure}

Let us now define the {\it vertical cones } as the complements to the horizontal ones,  $K^\ver(u)=T_u\Cmig \sm K^\hor(u)$.
% as follows: (and similarly for $\KK^h(\om)=D\Psi^{-1}(K^v(u))$). 
% {\it Vertical-like} paths are those that go through the vertical cones.
A smooth path $\gamma(s)$ is called {\it vertical} if it goes through the vertical cones,
i.e., $\gamma'(s)\in K^\ver(x)$.
(The same definitions apply to the cylinder $\Cphys$.)
We call a path {\it strictly vertical} if it goes through the $\inter K^\ver(u)$.

\begin{lem}\label{exactly 4}
  If $\gamma$ is a proper vertical path in $\Cmig$, then $\Rmig^{-1}\gamma$ comprises four 
proper paths (and similarly, for proper paths in $\Cphys$). 
\end{lem}

\begin{proof}
Obviously, vertical paths are transverse to horizontal ones -- 
so, by Lemma \ref{S transv to Icurve}, they are transverse to the blow-up locus $\Imig$. 
Combined with Property (P9$'$), this yields the assertion.
\end{proof}

Given a cone $K$ in a linear space $E$, let $PK\subset PE$ stand for its projectivization.  
% For instance, $PK(u)\subset PT_u\Cmig\isom \PR^1$. 
Let us say that a cone field $K(u)$ is {\it strictly forward invariant} if
       $$ D\Rmig(PK(u)) \Subset \inter PK(\Rmig u). $$

\begin{prop}
\label{PROP:MK_CONEFIELD}
The horizontal  cone fields  $K^\hor(u)$ and $\KK^\hor(u)$ are strictly forward invariant under
the corresponding dynamics,  $\Rmig$ and $\Rphys$.
\end{prop}

\begin{proof}
% It is sufficient to deal with the field $K^h(u)$.
Equivalently,  the vertical  cone field $K^\ver(u)$ is
strictly backward invariant. Since the cones are tangent to the pair of foliations $\Phi^\pm$,
this is equivalent to the property that the pullbacks  $\Rmig^{-1}(\Smig_\psi^\pm) $ 
of the $\Phi^\pm$-leaves are strictly vertical. 

By  Lemma \ref{exactly 4}, each  pullback $\Rmig^{-1}(\Smig_\psi^\pm) $  
comprises  four disjoint proper paths in ${\Cmig}$. 
As the line $\Smig_\psi$ is the concatenation of two rays  $\Smig_\psi^\pm$,
the pullback $R^{-1}(\Smig_\psi)$ comprises eight disjoint proper paths $\gamma_i=\gamma_{i,\psi}$ in $\inter \Cmig$.

To prove the desired,  it suffices
to show that for any angles $\psi$ and $\theta$, 
each path $\gamma_{i,\psi}$ has at most one intersection point 
with any  line $\Smig_\theta$,
and the intersection is transverse.  In so, $\gamma_i$ could not
cross $\Smig_\theta$ through the horizontal cone $K^\hor(u)$, 
for it would be disjoint from the whole closed vertical cone $\cl K^\ver(u)$ 
(viewed as a subset of $\C$).   
But the latter contains $\T$, so $\gamma_i$ would fail to land on $\T$
(see  Figure \ref{FIG:ALG_CONE_FIELD_INVARIANCE}).

\begin{figure}
\input{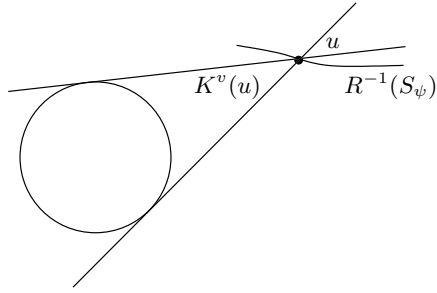}
\caption{\label{FIG:ALG_CONE_FIELD_INVARIANCE} Illustration to the proof  of invariance of the cone field.}
\end{figure}

Let $\Smig^c_\psi$ be the complexification of $\Smig_\psi$.  Since $\Rmig^*(\Smig^c_\psi)$ is a complex
algebraic curve of  degree $4$ (by Corollary \ref{COR:DEGREE_GROWTH_MK}),
its slice by the complex line  $\Smig^c_\theta$ consists of $4$ points counted
with multiplicity.  We will show that the intersection points that lie in $\inter \Cmig$
are transverse and belong to distinct radial components of $\Rmig^{-1}(S_\psi)$.
Let us consider several cases.

%To this end, let us consider the circle at infinity  $\BOTTOMmig$, the bottom of the M\"obius band 
%(see \S \ref{SUBSEC:MAP_ON_CYL}).
%The map $\Rmig$ acts on $\BOTTOMmig$ as $\zeta\mapsto \zeta^4$ (where $\zeta=w/u=e^{-2i\phi}$).
%Hence $\Rmig^{-1}(\Smig_\psi)$ intersects  $\BOTTOMmig$ at $4$ equidistributed  points
%$\zeta= e^{-i(\psi/2+\pi k/2)}$ ($k\in \Z/4\Z$).
%They prescribe the asymptotical slopes at infinity  
%of the $8$ radial paths $\gamma_i$ of $\Rmig^{-1}(\Smig_\psi)$
%(namely, $\psi/4+ \pi k/4$, $k\in \Z/8Z$). 

\ssk {\it Case 1 (generic)}. 
Let  $\theta\not=\psi/2,  \pi/2\ \mod \pi$.  
Then    $\Smig_\theta$ does not meet $R^{-1}(\Smig_\psi)$ on $\T$.
Since $R$ is even, both pullbacks $R^{-1}(\Smig_\psi^\pm)$ are symmetric with respect to the origin.
Hence the rays comprising  $R^{-1}(\Smig_\psi^\pm)$ come in symmetric pairs $\gamma_i$, $\gamma_i'$, 
$i=1,\dots, 4$. 
Then one ray in each pair (say $\gamma_i$)  near infinity is separated by  $\Smig_\theta$
from $\T$.
Since $\gamma_i$ lands on $\T$, it must intersect $\Smig_\theta$.  

Since the $\gamma_i$ are pairwise disjoint, 
this gives us 4 distinct intersection points  of $R^{-1}(\Smig_\psi)$ with $\Smig_\theta$,
Since the total number of intersection points counted with multiplicity is at most 4, 
we have accounted for all of them.
Thus,  each $\gamma_i$ intersects $\Smig_\theta$ exactly once and the intersection is transverse
(while the $\gamma_i'$ are disjoint from $\Smig_\theta$).

\ssk {\it Case 2.}
Let $\theta= \psi/2$ or $\pi/2\ \mod \pi$, but $\psi\not=\pi\ \mod 2\pi$.
By Lemma \ref{pulback of S-phi} (i), the algebraic curve $R^*(\Smig_\psi^c)$ has a double point 
at $e^{i\theta}$,   and $\Smig_\theta$ intersects it non-tangentially with multiplicity 2. 
It must intersect two other branches $\gamma_i$ in $\inter \Cmig$, 
and thus we have accounted for all four intersection points.  
The conclusion follows. 

%\ssk{\it Case 3.} Let $\theta= \pi/2\ \mod \pi$  
%(so that $a=e^{i\theta}$ is an indeterminacy point of $R$)
%but $\psi\not=\pi$. 
%Then two branches $\gamma_i$ meet transversely at $a$, 
%so $\Smig_\theta$ has a double intersection with $R^{-1}(\Smig_\psi)$ at $a$. 
%It must intersect two other branches $\gamma_i$ in $\inter \Cmig$, 
%and  we have accounted for all four intersection points again.

\ssk {\it Case 3.} 
Finally, let $\theta= \psi/2 =  \pi/2\ \mod \pi$
(this is the most degenerate case, but it occurs exactly when $\Smig_\psi=\Smig$ is the principal LY locus,
see Figure \ref{FIG:LEE_YANG_ZEROS}). 
In this case,  all four branches $\gamma_i$ meet at  the indeterminacy points $e^{i\theta}$, 
and $\Smig_\theta^c$ intersects $R^*(\Smig_\psi^c)$ at this point
with multiplicity 4 (as described Lemma \ref{pulback of S-phi} (ii)). 
It accounts for all intersection points, so no intersections occur in $\inter\Cmig$.  
\comm{***
If there is an intersection of $\Rmig^*(\Smig_\phi)$ with $\Smig_\theta$ at the
point $e^{i\theta} \in \T$, then this point is a tangency between
$\Rmig^*(\Smig_\phi)$ with $\Smig_\theta$ because both are tangent to $\T$.
The intersection must therefore be counted with multiplicity $2$ (or $4$
in the rare case that $4$ of the radial curves from $\Rmig^*(\Smig_\phi)$ meet
$\Smig_\theta$ at $e^{i\theta}$.)

Similarly, if there are not four endpoints of radial curves from
$\Rmig^*(\Smig_\phi)$ on one side of $\Smig_\theta$ and four on the other
side, there is necessarily a tangency between $\Rmig^*(\Smig_\phi)$ and
$\Smig_\theta$ at infinity (with respect to $u$) and the resulting intersection
must be counted with multiplicity $2$.  

In each of these degenerate cases, one can check that the multiplicity
appropriately decreases the number of possible points of intersection (counted
without multiplicity) between $\Rmig^*(\Smig^c_\phi)$ and $\Smig^c_\theta$ in
$\CP^2$ so that for topological reasons each radial component from
$\Rmig^*(\Smig_\phi)$ can intersect $\Smig_\theta$ at most once.***}
\end{proof}

%When $t=1$ the
%two curves $\Sphys_{\psi_1}$ and $\Sphys_{\psi_2}$ coincide, and the horizontal cone
%degenerates to a single tangent direction.  
%This degeneracy of $\KK$ near $\TOPphys$ will require some care later.
%$cause some challenges later.
%% (The same problem occurs for $K(u)$ as points approach $\TOPmig$.)

%\begin{rem}
%Note that invariance of the cone field $\KK(u)$  can also be proven directly in a
%manner similar to the proof of Proposition \ref{PROP:MK_CONEFIELD}, but
%it is more involved due to the ``artificial'' points of indeterminacy for $\Rphys$
%at $\gamma=[0:1:0]$ and their affect on the degrees and intersection numbers of
%pullbacks of curves.
%\end{rem}

\begin{rem}
 Eric Bedford has informed us that a similar algebraic method for constructing an
invariant cone field (for certain birational maps) had been earlier used in \cite[\S 5]{BD_GOLDEN}.  
\end{rem}

%\msk
%A similar construction of invariant cone fields on an invariant real slice of
%a map is given in \cite[\S 5]{BD_GOLDEN}.  The authors construct cones bounded by
%intersections of certain algebraic curves with the real slice and use bounds on
%the maximal number of complex intersections in combination with knowledge of
%how each of these intersections must occur in the real slice to show that the
%cone-field is invariant.  A further similarity between the cone fields constructed in \cite{BD_GOLDEN}
%the ones constructed above is that the both lack a
%uniformity of how the image cones lie within the cone field.

Corollary \ref{exactly 4} and Proposition \ref{PROP:MK_CONEFIELD} imply: 

\begin{cor}\label{strictly vert pullbacks}
  If $\gamma$ is a proper vertical path in $\Cmig$, then $\Rmig^{-1}\gamma$ comprises four 
proper strictly vertical paths (and similarly, in $\Cphys$). 
\end{cor}

\subsection{Modified algebraic cone field}\label{SUBSEC:MODIFIED_ALG_CONES}
The algebraic cone field we have just constructed has a disadvantage that it
degenerates near the top.  We will now modify it near the top so that it will
become non-degenerate everywhere away from the indeterminacy points
$\alpha_\pm$.

 Given a small threshold  $\bar \tau>0$, let us consider the following annular neighborhood of the top:
\begin{equation}\label{EQN:VV}
  \VV \equiv \VV_{\bar \tau} =\{ x\in \Cphys :\ \tau\leq \bar \tau \}. % \quad \tl\VV\equiv \VV_{2\bar \tau}.
\end{equation}
% Recall that $\SSS_\psi$ stands for the translated principal LY locus based at the point $\psi \in \TOPphys$
% (see Figure \ref{FIG:CONEFIELD1}).
Given $\eta >0$,
let us consider two parabolas
$
    \YY^\pm_\eta = \{\tau= \eta \eps^2 \}
$
centered at the indeterminacy points $\alpha_\pm$.

Consider two parabolic regions $\PP^\pm_\eta$ below the curves $\YY^\pm_\eta$
(see Figure \ref{FIG:REGION_PARABOLIC}), and  let
\begin{equation}\label{EQN:VV'}
  \VV' \equiv \VV'_{\eta,\bar \tau}  = \VV_{\bar \tau} \sm (\PP^+_\eta \cup \PP^-_\eta). % \quad  \tl\VV'  = \tl \VV \sm (\PP_+ \cup \PP_-).
\end{equation}

For a small threshold $\bar \eps>0$, let us consider the following regions in $\VV$:
$$
      \UU\equiv \UU_{\bar\eps}=\{x\in \VV :\ |\eps(x)| <  \bar\eps\} \,\,\, \mbox{and} \,\,\,  \UU' =\{x\in \VV' :\ |\eps(x)| < \bar\eps\}.
$$

For the remainder of the construction, we let $\bar \tau = \eta\, \bar \eps^2 $ so
that the regions $\UU$ and $\UU'$ meet  the parabolas $\YY^\pm$ at their bottom, 
see Figure \ref{FIG:REGION_PARABOLIC}.
To be definitive, we fix $\eta = 1/18$,
% and drop the dependence on $\eta$ from the notation.  
so the only free parameter left in our disposal is $\bar \eps$.

% They have the same image under $\Rphys$.
Note that for $x\in \YY^\pm$,
the slope of the lines that bound the algebraic cone $\KK^a(x)$
is equal (in absolute value) to
\begin{equation}\label{slope on YY}
   s^a(x) = \sqrt{\tau(2-\tau)} \sim \sqrt{2\tau} = |\eps|/3, \quad \text{and}\quad s^a(x)\geq |\eps|/3.  
\end{equation}

\begin{figure}
\begin{center}
\input{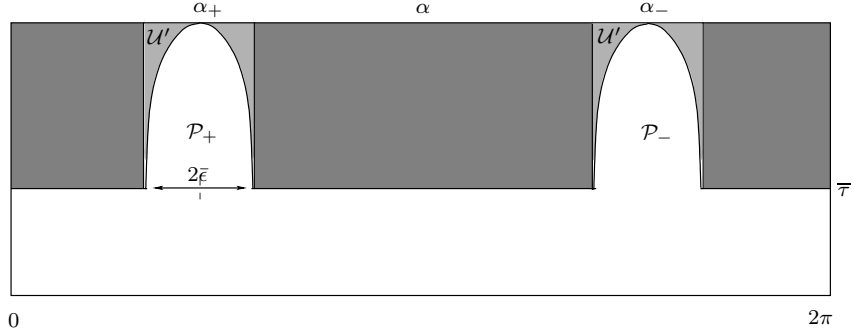}
\end{center}
\caption{\label{FIG:REGION_PARABOLIC}
$\VV'$ is the union of all shaded regions.
Note that the figure is not to scale.}
\end{figure}

\begin{lem}\label{inclusions}
  For $\bar \eps > 0$ sufficiently small, we have:
\begin{itemize}
\item [(a)] $\Rphys(\UU) \cap \UU = \emptyset$;
\item[(b)] $\Rphys^{-1}(\UU)$ is contained in a small neighborhood of 
%\item[(b)] $\Rphys^{-1}(\UU)\cap \VV$ is contained in a small neighborhood of 
the points\\  $(\phi = \pm \pi k /4,\ \tau=0 )\in \TOPphys$ with  $k=\pm 1, \pm 3$.
\end{itemize}
\end{lem}

\begin{proof}
a) By blow-up formula (\ref{EQN_PHI_PRIME_BLOWUP}),
the image $\Rphys(\UU)$ is contained in a small neighborhood of the singular curve $\Icurve$,
which is disjoint from $\UU$.

\msk 
b) 
By (a), there are no points of $\UU$ in $\Rphys^{-1} \UU$.  But in $\Cphys \sm \UU$,
the map $\Rphys$ is continuous, so $\Rphys^{-1}(\UU)$ is localized
near $\Rphys^{-1}(\alpha_\pm)$.

%Fix a small $\bar\eps_0>0$.
%By a), $ \Rphys^{-1}(\UU_{\bar\eps})\cap \VV\subset \VV\sm \UU_{\eps_0}$ for $\bar\eps\leq \bar \eps_0$. 
%But in $\VV\sm \UU_{\eps_0}$, the map $\Rphys$ is continuous, so $\Rphys^{-1}(\UU_{\bar\eps})$ is localized
%near $\Rphys^{-1}(\alpha_\pm)$.   
\end{proof}

For $x\in \Cphys$, let us define a continuous horizontal cone field $\KK^h(x) \equiv \KK^h_{\bar \eps}(x)$  whose boundary lines have slopes
with the absolute value  $s(x)$ such that:

\begin{itemize}

\item [(o)]  $\KK^h(x)\supset \KK^\hor(x)$ {\it everywhere};

\item [(i)] $\KK^h(x)=\KK^\hor(x)$ in the {\it white} region $\Cphys\sm \VV'$ (see Figure \ref{FIG:REGION_PARABOLIC});

\item [(ii)]   $s(x)\sim |\eps(x)| / 3$ in  the {\it grey} region $\UU'$;

\item [(iii)]   $s(x)= \bar w := \sqrt{\bar \tau(2-\bar \tau)} \sim \bar \eps/3$ in the {\it black} region $\VV \sm \UU=\VV'\sm \UU'$.

\end{itemize}

Conditions (o)--(ii) are compatible due to (\ref{slope on YY}).  
For a cone field satisfying these three conditions,
we have $s(x)=s_0$ on the horizontal boundary of the black region
and $s(x)\sim \bar\eps/3$ on the vertical one.
Hence it can be extended to the black region satisfying (iv).  

Note that in $\Cphys \sm \VV$, we have $s^a(x)=\sqrt{\tau(2-\tau)}\geq \bar w$. Hence

\begin{equation}\label{bar w}
    \mbox{ $s(x) \geq \bar w \sim \bar\eps/3 \quad $ in $\quad\Cphys \sm \UU$.}
\end{equation}

\begin{lem}\label{LEM:MOD_CF_INVARIANT}
For sufficiently small $\bar \eps > 0$, 
the cone field $\KK^h$ is strictly forward invariant:    %  on the region $\VV'_{\eta,\bar\tau}$,
\begin{equation}\label{invariance of KK}
  \Rphys (\KK^h(x))\Subset \KK^h(\Rphys x).
\end{equation}   %  provided $x, \Rphys x \in \VV'_{\eta_,\bar\tau}$.
\end{lem}

\begin{proof}
For $x\in \Cphys \sm  \VV'$, we make use of the invariance of the algebraic cone field and conditions (o) and (i):
 $$
  \Rphys(\KK^h(x)) = \Rphys (\KK^\hor(x)) \Subset \KK^\hor(\Rphys x)\subset \KK^h(\Rphys x).
$$

So, assume $x\in \VV'$.
The cones $\KK^h(x)$ are bounded by two lines spanned by the tangent vectors $v_\pm=(1, \pm  s(x))$.
Let us estimate the absolute value $s'$  of the slope of $D\Rphys(v_\pm)$.

Select $\bar\eps$ so small that (\ref{DR near alpha}) applies in $\UU$.
For $x\in \UU'$, it yields:
$$
 s' \leq \frac {|\eps| \tau( \tau +  \eps^2/3) }{\si^2(\tau+ 2\eps^2/3) } =\frac {|\eps| \tau}{\si^2}  % =  \frac{|\kappa|}{1+\kappa^2} 
            < |\kappa| < \bar \eps / 18 < \bar w.
$$
By Lemma \ref{inclusions}, $\Rphys x \not \in \UU$, so property (\ref{bar w}) ensures (\ref{invariance of KK}).

Let $x=(\phi, \tau) \in \VV'\sm \UU'$.
Then for $\bar \eps$ sufficiently  small, we have:  
$$
    |\cos \phi| \geq \bar \eps/2 , \quad |\tg \phi| < 4 / (3 \bar \eps) , \quad \text{and}\quad  \tau\leq \bar\tau  \leq \bar\eps^2/18.
$$ 
Letting $(a,b) = D\Rphys (v_\pm)$, we obtain from  (\ref{EQN:A_EXPANSION}): 
\begin{eqnarray}\label{EQN:A_EST}
  a \geq 2-    2\tg\phi\cdot \bar\eps /3  \geq 10/9 > 1
\end{eqnarray}
\noindent
and
\begin{equation}\label{EQN:B_EST}
|b| \leq  \frac{2\tau^2 |\tg \phi| }{\cos^2 \phi}  + 
     \frac{\tau \bar \eps}{3 \cos^2 \phi}
 \leq 16\, \frac{\bar \tau^2}{\bar \eps^3} +
  2\, \frac{\bar \tau}{\bar \eps} 
 \leq \left( \frac{16}{ 18^2}   + \frac{2 }{18} \right)\,  \bar\eps 
<    \frac{\bar \eps}{6}.
\end{equation}
\noindent
So, the slope $s'=|b/a| < \bar w$ as well.
Due to property (\ref{bar w}), this implies (\ref{invariance of KK}) in case $\Rphys x\not\in \UU$.

Finally, let $x \in \VV'\sm \UU'$ and $\Rphys x=(\eps', \tau') \in \UU$. 
Then Lemma \ref{inclusions} b) gives $|\cos \phi| > C^{-1} > 0$ and $|\tg \phi| < C$, independent of $\bar \eps$.
Estimate (\ref{EQN:B_EST}) simplifies to
$$
    |s'| \leq |b| = O( \tau^2   +  \tau\bar \eps) = o(\tau) = o(\sqrt{\tau'}) \,\,\, \mbox{as} \,\, \bar\eps \ra 0.
$$
But  $s(\Rphys x) \geq \sqrt{\tau'}$ as follows from condition (o).   
This concludes the proof.
\end{proof}

The horizontal cone field $\KK^h$ extends continuously to the indeterminacy points
as degenerate cones $\KK^h(\alpha_\pm)= \{ d\tau=0\}$.
Obviously, this continuous extension is invariant.
% (which should be properly understood after the blow-up of $\alpha_\pm$.

We will also call an smooth path $\gamma(t)$ in $\Cphys$ {\em horizontal} if at
each point $\gamma(t) \in {\rm int} \ \Cphys$ we have $\gamma'(t) \in
\KK^h(\gamma(t))$.  In the remainder of the paper, all horizontal paths
will considered with respect to respect to $\KK^h$ (and not $\KK^\hor$), unless
otherwise specified.  
%Note that any horizontal path is naturally oriented by its $\phi$ coordinate and
%that this orientation is preserved since $D\Rphys$ preserves the orientation of horizontal cones.

\begin{rem}
One obtains a pushed forward conefield $K^h := D\correspond \ \KK^h$ on $\Cmig$ that is invariant
under $\Rmig$ and non-degenerate away from $\INDmig_\pm$.
\end{rem}

\comm{******
\sss{Expansion within $\VV'$}
\begin{lem}\label{LEM:EXPANSION_IN_VPRIME}For any $\eta, \bar \tau > 0$ sufficiently small, there exists $\lambda > 1$ so that for all $x \in \VV_{\eta,\bar \tau}$ and any $v \in \KK^h(x)$ we have
\begin{eqnarray}\label{EQN:EXPANSION_IN_VPRIME}
d(\phi\circ R)(v) > \lambda \ d \phi(v).
\end{eqnarray}
In particular, (\ref{EQN:EXPANSION_IN_VPRIME}) holds in the region $\VV'$ used to define $\KK^h$.
\end{lem}

\begin{proof}
It suffices to check for the region $\VV'$ considered in the proof of Lemma
\ref{LEM:MOD_CF_INVARIANT}.  For $x \in \VV' \sm \UU'$,
(\ref{EQN:EXPANSION_IN_VPRIME}) follows from (\ref{EQN:A_EST}).

For $x \in \UU'$ the minimum of $d(\phi \circ \Rphys)(v)$ 
occurs at $v_\pm = (1,\pm s(x))$, where $s(x) \sim |\eps(x)|/3.$ 
Equation (\ref{DR near alpha}) gives
\begin{eqnarray*}
d(\phi \circ \Rphys)(v_\pm) = \frac{2}{\eps^2 + \tau^2} \left(\eps^2 + \tau \pm \frac{\eps^2}{3} \right)  \geq \frac{4\eps^2/3 + 2 \tau}{\eps^2 + \tau^2}  > \frac{4}{3}.
\end{eqnarray*}
\end{proof}
************************}

\subsection{Vertical cone fields and laminations}\label{sec: strictly vert lam}

In this section we develop  a language adopted, for definiteness, to the modified cone field $\KK^h$
on the cylinder $\Cphys$,
but a similar language, with obvious adjustments, can be applied to the algebraic
cone field $\KK^\hor$, as well as the corresponding cone fields on the M\"obius band $\Cmig$. 

We let  $\KK^v(x)= T_x\Cphys\sm \KK^h(x)$ be the complementary vertical cone field.  
(In particular, $\KK^v(\alpha_\pm)= \{d\phi\not=0\}$ is the complement to the horizontal line.) 
%It is backward invariant under $\Rphys$ in the following sense.

Let $\KK(x)\subset \KK^v(x)$ be a continuous cone field on $\Cphys$. 
Let us define the pullback $D\Rphys^*(\KK)$ as follows. 
Let $y=\Rphys x$. 
When $D\Rphys_x$ is well defined and invertible  (i.e.,  $x\not\in \BOTTOMphys\cup \TOPphys\cup \II_{\pm \pi/2}$,
we let  $D\Rphys^*(\KK)(x) = D\Rphys_x^{-1}(\KK(y))$. 
When $D\Rphys_x$ is well defined but is not invertible, we let $D\Rphys^*(\KK)(x)  = \Ker D\Rphys_x$.
Finally, for $x\in \{\alpha_\pm\}$ we let   $D\Rphys^*(\KK)(x)  = \KK^v(\alpha_\pm)$. 

It is easy to see that the pullback is continuous (and is contained in $\KK^v$).  

\begin{cor}\label{COR:R_CONEFIELD}
The vertical cone field $\KK^v$ is {\em backwards invariant}: $D\Rphys^*(\KK^v)\subset \KK^v$.
\end{cor}

\msk
In this setting,   ``vertical paths'' are understood in the sense of the vertical cone field $\KK^v$
rather than $\KK^{av}$. So, by definition,  a {\it vertical path}\footnote{In the remainder of the paper, all vertical paths will be considered with respect to $\KK^v$ (and not $\KK^{av}$)
unless otherwise specified.} is a 
a smooth path $\gamma(s)$ in $\Cphys$ such that $\gamma'(s)\in \KK^h(\gamma(s))$. 
Being a graph over a temperature interval, it can be parameterized accordingly:
$\phi= \gamma(t)$.  % $0\leq t < 1$. 
Moreover, $\gamma'(s)$ is finite  except possibly at the indeterminacy points  $\alpha_\pm = \gamma(1)$
(if $\gamma$ terminates at one of them). 
%In this parameterization, 
%$$
%      |\gamma'(t)|=  O((1-t)^{-1/2}).
%$$ 
%%%%whenever $0\leq t\leq 1-\tau$, but it may happen that $|\gamma'(t)|\to \infty$ as $t\to 1$. 

\msk
{\it A vertical lamination} $\FF$ in $\Cphys$ is a family of 
 vertical paths (the {\it leaves} of the lamination) which are disjoint in the topless cylinder $\Cphystl$
(but are are allowed to meet on $\TOPphys$)
that has a {\it local product structure}. 
The latter means that for any path $\gamma_0\in \FF$
and any point $x_0=(\gamma_0(t_0), t_0)\in \Cphystl$ there exists an $\eps>0$ such that any leaf
$\gamma\in \FF$ passing near $x_0$ can be locally represented as a graph % $\phi=\gamma(t)$
over $(t_0-\eps, t_0+\eps)$, and this graph depends $C^1$-continuously on the 
{\it transverse parameter} $\phi = \gamma(t_0)$.
The $\supp(\FF)$ is the union of the leaves of the lamination. 

A vertical lamination is called {\it proper} if all its leaves are proper. 

%filling in some closed subset of $\Cphystl$, the {\it support} of $\FF$, and such that the
%restricted  paths  $\gamma:  [0,1-\eps]\ra  \R/ 2\pi \Z$ depend $C^1$-continuously on its bottom
%point $\phi=\gamma(0)\in \BOTTOMphys$.   

For instance, a finite family of disjoint proper vertical paths
form a proper vertical lamination. 

In case  when   $\supp \FF$ is open in $\Cphystl$,  % $\supp \FF=\Cphystl$,
the lamination $\FF$ is called {\it a strictly vertical foliation} (of its support). 
For instance, the ``genuinely vertical'' foliation on the cylinder $\Cphystl$ 
is formed by the intervals $\II_\phi$, $\phi\in \R/ 2\pi \Z$.
%Other examples are given by the foliations $\Phi^\pm$ 
%that we have just used for defining the cone fields. 
%%%formed by the left and right branches 
%%%of the translated LY loci $\SSS_\psi$ (\ref{translated LY loci}) 
%%% (by definition, they go though the boundary of the vertical cones). 

% Let $\FF_n=(\Rphys^n)^* (\FF_0)$ be its pullback by $\Rphys^n$.
Lemma  \ref{LEM:MOD_CF_INVARIANT} implies that 
the pullbacks $(\Rphys^n)^*(\FF)$ of a (proper) vertical lamination are (proper) vertical. 

We will mostly be dealing with laminations whose leaves begin on the bottom of the cylinder
(in fact, mostly with  proper laminations), 
and will use the bottom angle $\phi\in \Z/2\pi \Z$ as the  transverse
parameter, $\gamma= \gamma_\phi$.

\section{Central line field and dominated splitting}
\label{SUBSEC:DOMINATED}
 
In this section % \S \ref{SUBSEC:DOMINATED}
we will use the cone field constructed in \S \ref{sec: alg cone field}
to prove that $\Rphys: \Cphys \rightarrow \Cphys$ is {\it projectively hyperbolic},
or admits a {\em dominated splitting}. % (degenerating near the indeterminacy points).  
We will start with constructing an invariant vertical line field.

\subsection{Central line field}
  A  {\it central line field} $\LL$ on $\Cphys$ is an $\Rphys$-invariant continuous tangent line field $\LL(x)\subset \KK^v(x)$,  $x\in \Cphys\sm \{\alpha_\pm\}$.
Here ``invariance'' means that $D\Rphys_x (\LL(x))\subset \LL(\Rphys x)$ whenever $x\not\in \{\alpha_\pm\}\cup \Rphys^{-1}\{\alpha_\pm\}$. 

\begin{prop}\label{Central line field prop}
  There exists a unique central line field on $\Cphys$.
Moreover, if $\LL (x)\subset\KK^v(x) $ is any vertical line field then
$$
   (D\Rphys^n)^*\LL \to \LL^c
$$
uniformly and exponentially on compact subsets of $\Cphys \sm \{\alpha_\pm\}$. 
\end{prop}

\sss{Hyperbolic metric}
To prove this proposition, we will make use of the ``projective hyperbolic metric'' on the vertical cones.
Any interval $I=(a,b)\subset \R$ can be viewed as the hyperbolic line endowed with the hyperbolic metric
$$
  \dist_I (x,y) = \log\frac{y-a}{x-a}+\log\frac {b-x}{b-y}, \quad  a<x<y<b.
$$
Since cross-ratios are projective invariants, the hyperbolic metric is invariant under M\"obius isomorphisms
$\phi: I \ra J$. Moreover, it gets contracted under inclusions: if $I\Subset J$ then 
$$
   \dist_J(x,y)\leq \la \dist_I(x,y),\quad  x,y\in I,
$$
where $\la< 1$  depends only on the hyperbolic diameter of $I$ in $J$.  
Putting these two properties together, we obtain the following ``Schwarz Lemma'' for projective maps:

\begin{lem}\label{Schwarz Lemma for proj maps}
  Let $\phi: I\ra J$ be a M\"obius transformation with $\phi(I) \Subset J$. Then 
$$
   \dist_J(\phi(x),\phi(y))\leq \la \dist_I(x,y),  \quad x,y\in I,
$$
where $\la<  1$ depends only on the hyperbolic diameter of $\phi(I)$ in $J$.  
\end{lem}

Due to projective invariance, the above discussion can be carried to any 
intervals $I, J$ in the projective line $P\R^1$.

\sss{Contraction of the projective cone field.}

 As the cones $\KK^v(x)$ represent intervals in the projective tangent lines,
they can be endowed with the hyperbolic metrics $d_x$. 

\begin{lem}\label{exp shrinking of cones}
 For any neighborhood $\UU$ of the indeterminacy points $\{\alpha_\pm\}$,
 there exist $C>0$ and $\la > 1$ such that for any $x\in \Cphys\sm \UU$
$$
   \diam (D\Rphys^n)^* (P\KK^v)(x) \leq C\la^{-n},
$$
where the $\diam$ stands for the angular size of the cones. 
% and $ D\Rphys^{-3} :  P\KK^v(\Rphys^3x)\ra P\KK^vx)$ is uniformly contracting in the hyperbolic metric.   
\end{lem}

\begin{proof}
By Lemma  \ref{Schwarz Lemma for proj maps}, the differential $D\Rphys^{-1}: P\KK^v(Rx)\ra P\KK^v(x)$ contracts the hyperbolic metric
by a factor $\mu(x)<1$ depending only on the hyperbolic diameter of $D\Rphys^* (P\KK^v)(x)$ in $P\KK^v(x)$.
(For a critical point $x$, the projective cone $D\Rphys^* (P\KK^v)(x)$ collapses to a point, and we let $\mu(x)=0$). 
By continuity of the cone fields, this factor is uniform away from $\UU$.
Since the $\alpha_\pm$ blow up to the curve $\Icurve$ that does not contain $\alpha_\pm$, 
the orbit of $x$ can visit $\UU$ with frequency bounded by $1/2$
(provided $\UU$ is sufficiently small). 
Hence the hyperbolic diameter of the cones $(D\Rphys^n)^* (P\KK^v)(x)$ decay exponentially with rate $O(\mu^{n/2})$.
Since the projective intervals $P\KK^v(x)$ have angular size bounded away from $\pi$,
the  $\diam (D\Rphys^n)^* (P\KK^v)(x)$ are $O$ of their hyperbolic size,
and the conclusion follows.  
\end{proof}

\sss{Proof of Proposition \ref{Central line field prop}.}
 Let us take a vertical line field $\LL$ on $\Cphys$
and pull it back by the dynamics: $\LL_n= (D\Rphys^n)^*\LL$. 
By Lemma \ref{exp shrinking of cones}, for any $m\leq n$ we have:
$$
  \dist (\LL_n(x), \LL_m(x))\leq C \la^{-m}, \quad x\in \Cphys\sm \UU.
$$
Hence the $\LL_n$ uniformly and exponentially converge to a limit,
which is the desired central line field $\LL^c$.
\QED

\subsection{Dominated splitting}

We say that the horizontal cone field $\KK^h$ and central line field $\LL^c$
give  a {\it dominated splitting} of the map $\Rphys: \Cphys\ra \Cphys$ if
for any neighborhood $\UU$ of the indeterminacy points $\alpha_\pm$, 
there exist constants $c>0$ and $\la>1$ such that for any two tangent vectors
$v^h\in \KK^h(x)$ and $v^c\in \LL^c(x)$ of unit length we have:
\begin{equation}\label{dom split def}
    \|D\Rphys^n_x v^h\|\geq c\la^n \| D\Rphys^n_x v^c \|, \quad x \in \Cphys\sm \UU.
\end{equation}
(In other words, horizontal vectors grow exponentially faster that the central ones.)

%Note that since the central lines on $\Cphys\sm \UU$
%are uniformly transverse to the horizontal cones, condition (\ref{dom split def}) 
%is in fact independent of the particular choice of $v^h$.   

\begin{rem}
  For diffeomorphism, the splitting is usually given by two sub-bundles of the tangent bundle (see \cite{Pujals}).
However, such a definition is not suitable for the non-invertible case when the unstable sub-bundle 
may not exist. That is why we give a definition in terms of cone fields.
Of course, in the invertible case, both definitions are equivalent. 
\end{rem}

\begin{lem}\label{LEM:COMPARABLE_HORIZONTAL_VECTORS}
For any $x \in \Cphys \sm \{\INDphys_\pm\}$ and $i \geq 3$,
if $v_1, v_2 \in D\Rphys^i (\KK^h(x))$ satisfy $v_1 -  v_2 \in \LL^c(\Rphys^i x)$, then
$\|v_1\| \asymp \|v_2\|$.
\end{lem}

\begin{proof}
Let $\UU$ be a neighborhood of $\{\INDphys_\pm\}$ chosen sufficiently small so
that if $x \in \UU$ then $\Rphys^{n} x \not \in \UU$ for $n=1,2$.  
Lemma \ref{Schwarz Lemma for proj maps} implies that $D \Rphys^i
(P\KK^h(x))$ has uniformly bounded hyperbolic diameter for $i \geq 3$.  
Let $L_1$ and $L_2$ be any two lines through $D \Rphys^i (P\KK^h(x))$ and let $\theta(L_1,L_2)$, $\theta(L_1,\LL^c)$, and $\theta(L_2,\LL^c)$ be the angles
between them and between each line and $\LL^c \equiv \LL^c(x)$.  The uniform bound on the hyperbolic diameter implies that
that there is some constant $C > 0$ so that 
\begin{eqnarray*}
 \theta(L_1, L_2) \leq C \cdot \theta(L_j,\LL^c), \qquad j=1,2.
\end{eqnarray*}
The result then follows from basic trigonometry.
\end{proof}

\begin{cor}\label{COR:HORIZONTAL_VECS_ITERATE_COMPARABLY}
For any neighborhood $\UU$ of $\{\INDphys_\pm\}$ and any $x \in \Cphys \sm \UU$, if  $v_1^h, v_2^h \in \KK^h(x)$ are unit tangent vectors 
then $\|D\Rphys^n v_2^h\| \asymp \|D\Rphys^n  v_2^h\|$.
\end{cor}

\begin{proof}
Since $x \not \in \UU$, the projection of $v_1^h$ onto $v_2^h$ along $\LL^c$ will have length comparable to the length of $v_2^h$.
Thus, the result follows for $n \geq 3$ from Lemma \ref{LEM:COMPARABLE_HORIZONTAL_VECTORS}.
If $n \leq 2$, the result follows from the fact that $D\Rphys$ can only contract
a horizontal vector by a bounded amount (Lemma
\ref{LEM:BOUNDED_CONTRACT}).
\end{proof}

\noindent Note that Corollary \ref{COR:HORIZONTAL_VECS_ITERATE_COMPARABLY} implies that condition (\ref{dom split def}) is in fact independent of the particular choice of $v^h$.

\begin{prop}
   The horizontal cone field $\KK^h$ and central line field $\LL$
give  a {\it dominated splitting} of the map $\Rphys: \Cphys\ra \Cphys$.
\end{prop}

\begin{proof}
Since a single iterate of $\Rphys$ can only contract horizontal vectors by a bounded amount  (Lemma
\ref{LEM:BOUNDED_CONTRACT}), it suffices to consider $n \geq 3$. 
Let $x_n := \Rphys^n x$ be the orbit of any $x \in \Cphys \sm \UU$.
Let $v^h(x_n)$ be the unit  vector on the boundary of $D\Rphys^3 (\KK^h(x_{n-3}))$ pointing ``northeast'' and
let $v^c(x_n)$ be the unit central vector pointing ``north''.  By definition, the vector
\begin{equation}\label{vector w_n}
w_n =  v^h(x_n) + v^c(x_n)
\end{equation}
will satisfy $(D\Rphys^3)^* w_n \in \KK^v(\Rphys^{n-3} x)$.

We pull back $w_n$ under the dynamics and decompose it as 
$$(\Rphys^n)^* w_n := w = w^h + w^c$$
 with
$w^h$ parallel to the unit vector on the boundary of $\KK^h(x)$ pointing
``northeast'' and $w^c \in \LL^c(x)$.  
By Proposition \ref{Central line field prop}, 
\begin{equation}\label{exp growth of ratio}
\| w^c\| \geq c\la^n \|w^h\| .
\end{equation}

But 
$$
   w_n = \Rphys^n(w^h) + \Rphys^n(w^c) \ \mathrm{with}\ \Rphys^n(w^h)\in D\Rphys^3 (\KK^h(x_{n-3})),\  \Rphys^n(w^c)\in \LL^c.  
$$
Since $\Rphys^n(w^h)$ and $v^h(x_n)$ differ by an element of $\LL^c(x_n)$,
Lemma \ref{LEM:COMPARABLE_HORIZONTAL_VECTORS} gives 
$$
 \|\Rphys^n(w^h)\| \asymp \|v^h(x_n)| = 1,
$$
and hence % $\| \RR^n(w^h) - v^h(x_n) \| =O(1)$, and 
% This implies a uniform upper bound for 
$ \|\Rphys^n(w^c)\| = \| w_n -  \RR^n(w^h)  \| = O(1) $.
%and hence a uniform lower bound for
We see that $\|\Rphys^n(w^h)\|/\|\Rphys^n(w^c)\|$ is bounded from below, 
and hence  (\ref{exp growth of ratio}) can be written as
$$
    \frac{ \| \Rphys^n(w^h)\|}{\|\Rphys^n(w^c)\|} \geq c\la^n \frac{\|w^h\|}{\|w^c\|}
$$
But this is just the homogeneous form of the dominated splitting condition (\ref{dom split def}).
Since this condition is independent of the particular choice of vectors $w^h$ and $w^c$, we are done. 
\end{proof}

\sss{Central curves}

Let us say that a smooth curve is {\it central} it is tangent (on $\Cphys\sm \{\alpha_\pm\}$)
 to the central line field $\LL^c$. 

\begin{prop}\label{PROP:EXISTENCE_OF_CENTRAL_CURVES}
  Through any point $x\in \Cphys \sm \{\INDphys_\pm \}$ passes a vertical central curve.
\end{prop}

\begin{proof}
It follows from the Peano Existence Theorem (see \cite{Walter}) that continuous
line fields are integrable, so we can find a central curve through any point
$x\in \Cphys\sm \{\INDphys_\pm\}$.  Since the central line field is transverse
to the genuinely horizontal foliation, this curve is a graph over the $t$-axes,
and for standard reasons can be extended in both ways to the boundary of the
cylinder.  This is the desired  vertical central curve.

If $x\in \{\alpha_\pm\}$, then one can take $\II_\pm$ as the desired central curve.
(In fact, there are whole central tongues $\La_\pm$ filled with vertical central curves landing at $\alpha_\pm$ --
see \S \ref{central fol sec}).  
\end{proof}

\begin{rem}
  At this stage we do not know yet that there exists a unique central curve through a given $x$.
In fact, as we have just mentioned, this is not the case for the indeterminacy points $\alpha_\pm$ 
(and hence for their preimages). However,  we will prove in \S \ref{central fol sec}  that the uniqueness holds on $\Cphystl$.
\end{rem}

\section{Horizontal expansion}
\label{SEC:EXPANSION}

In this section we will prove that
the map  $\Rphys:\Cphys\ra \Cphys$ is {\it horizontally expanding},
in the following sense:

\ssk\nin $\bullet$ $\Rphys$ has an invariant horizontal cone field 
$\KK^h(x)$  on $\Cphys\sm \{\alpha_\pm\}$;

\ssk\nin $\bullet$ There exist constants $c >0$ and $\la>1$  such that 
$$
           | D\Rphys^n(x) (v)| \geq c \la^n \|v\|, \quad n=0,1,\dots
$$
  for any $x\in \Cphys\sm \{\alpha_\pm\}$ and $v\in \KK^h(x)$.
Moreover, $\la$ is called a {\it rate of expansion}.

% The goal of this section is to prove the following result:

\begin{thm}\label{hor expansion thm}
    The map $\Rphys:\Cphys\ra \Cphys$ is horizontally expanding on $\Cphys$ with the rate $\la=2$
with respect to the horizontal cone field $\KK^h$ from \S \ref{SUBSEC:MODIFIED_ALG_CONES}. 
\end{thm}

\subsection{Global Approach}\label{SUBSEC:GLOBAL_APPROACH}
Let us consider the solid cylinder
$$
  \Solid := \{ (z,t):\ |z|\leq  1,\    t\in [0, 1] \}.
$$
It is foliated by the horizontal leaves
$$
   \Secphys_t^* : = \{ (z,t):\ |z|\leq  1\}, \quad t\in [0,1] \}. 
$$
% The corresponding object in the affine $(u,w)$-coordinates is 
%$$
%    \Solidmig := \{ (u,w): \ |uw|\geq 1, \ |u|\geq 1 \}
%$$
In the $(u,w)$-coordinates, the horizontal complex lines $\Pi_t:=\{(z,t): z\in \C\}$
correspond to the conics $ \Secmig_t =\{ |uw|= t^{-2} \}$,
and the leaves of the solid cylinder become 
$$
  \Secmig_t^* := \{ (u,w)\in \Secmig_t : \ |u|\geq t^{-1}\},    \quad t \in [0,1]. 
$$
Here the bottom leaf $\Secmig_0^*$ is the complement  of the unit disk,  $ \CP^1 \sm \D$,
in the coordinate $\zeta=u/w$ on the line  at infinity $L_0 = \{V=0\}\isom \CP^1$.

Recall that the cylinder $\Cmig$  itself (or rather, the M\"obius band) is given by 
$$
     \Cmig = \{ w=\bar u, \ |u|\geq 1\},
$$
see \S \ref{SUBSEC:MAP_ON_CYL}. 
The leaf $\Secmig_t^*$ intersects it by the round circle $S_t = \{ |u| = t^{-1} \}$. 
All the leaves $\Secmig_t^*$ meet at the attracting fixed point $\CFIXmig=(1:0:0)$ at infinity.

Let us now consider the central projection from the origin to the line $L_0$ at infinity: 
$$
   \pi: \CP^2\ra L_0 \,\quad    (u,w) \mapsto \zeta = u/w.
$$

\begin{lem}\label{phi_n}
  For any $t\in [0,1)$,  the map 
\begin{equation}
  \psi_n= \pi\circ R^n : \Secmig_t^* \ra \Secmig_0^*
\end{equation}
is a branched covering of degree $2\cdot 4^n$. 
\end{lem}

\begin{proof}
We fix some $t\in [0,1)$, and skip it from the notation, so $\Secmig^* \equiv
\Secmig_t^*$, etc.  Since $\Rmig$ is algebraically stable, Lemma
\ref{LEM:PUSH_PULL_AS} gives that the push-forward $(\Rmig^n)_* \Secmig$ is a
divisor of degree $4^n \cdot 2$.

Bezout's Theorem gives $4^n \cdot 2$ intersections of $(\Rmig^n)_* \Secmig$
with with the complexification of any radial line $L :=\{ \arg u = \phi \
\mod \pi \}$.  However, the M\"obius band $\Cmig$ is in invariant and $\Rmig$
acts with degree $4$ on its first homology, so that $\Rmig^n(\Secmig \cap
\Cmig)$ is a horizontal curve of degree $4^n$ on $\Cmig$, forcing all $4^n \cdot 2$
intersections take place in $\Cmig = \{|u| \geq 1\}$.  In particular $(\Rmig^n)_* \Secmig$ does
not pass through the origin $u=w=0$ so that the central projection $\pi:
R^n(\Secmig) \ra \BOTTOMmig$ is well defined, and is a degree $2\cdot 4^n$
branched covering.

Furthermore, $\pi : R^n(\Secmig)\cap \Cmig \ra  \BOTTOMmig$  is also a degree
$2\cdot 4^n$ covering so that $\psi_n: \Secmig  \ra L_0$ is a rational map
between two Riemann spheres of degree $2\cdot 4^n$ (commuting with the natural
antiholomorphic involutions) that restricts to a covering map between the
circles $\Secmig\cap \Cmig\ra \BOTTOMmig$ (the fixed points loci for the
involutions) of the {\it same} degree.  By the Argument Principle, any point
$\zeta\in \Secmig^*_0$ has at least $2\cdot 4^n$ preimages in the disk
$\Secmig^*$.

\end{proof}

Let us now consider the fixed point $e$ at the center of the both  disks $\Secmig_t$ and  $\Secmig_0$. 
  Obviously,   $\psi_n (e)= e $. %  and  $\pi(e') = e'$.

\begin{lem}\label{order of vanishing}
The map $\psi_n$ has  branching of degree $ 2^{n+2}$ at $\CFIXmig$. 
\end{lem}

\begin{proof}
Use local coordinates $(v= V/U, \, w=W/U)$ near $e= (1:0:0)$.
In these coordinates, the map $R$ assumes the form
\begin{equation}\label{R in (v,w)}
   v'= v^2 \left( \frac{1+w}{1+v^2} \right)^2\sim v^2 ,\quad  w'= \left(\frac{w^2+v^2}{1+v^2}\right)^2\sim v^4 (1+ (w/v)^2)^2,
\end{equation}
while the curve $\Secmig$ becomes the parabola $w=v^2$. 
Let us use $s = v$ as a local parameter for $\Secmig$.
Then all the images $R^n(\Secmig)=(v_n(s), w_n(s))$ are naturally parameterized by $s$ as well,
and (\ref{R in (v,w)}) imply inductively:
$$
   v_n\sim s^{2^{n+1}},\quad w_n \sim s^{2^{n+2}} .
$$ 
Hence given a small $w_n$, we find $2^{n+2}$ parameter values $s$ near  the origin such that $\psi_n(s)=w_n$. 
%  For a point $x= (u,v)$  near $e$, let $R^n x= (u_n, v_n)$.
%Since the points on the strong separatrix converge to $e$ at rather $|u|^{4^n}$,
%while on the weak separatrix the rate is $|v|^{2^n}$, we have:
%$|u_n|\sim |v_n|^{2^n}$.   \note{should be done more carefully} 
%Hence the curve $R^n(\Secmig_t)$ has the order of tangency  $2^n$ with the weak separatrix,
and the conclusion follows.
\end{proof}

\begin{rem}
  Note that because of the symmetry $R(-u,-w) = R(u,w)$,
the images $\Rmig^n (\Secmig)$ are covered twice by $\Secmig$. 
If we consider $\Rmig^n (\Secmig)$ as a divisor of multiplicity two,
then its  order of tangency with the $u$-axis at $e=(\infty,0)$ is $2^{n+2}$.
But if we treat $\Rmig^n (\Secmig)$ as a simple curve then the order of tangency is twice smaller.  
\end{rem}

\begin{lem}\label{Blaschke}
   The Blaschke product $B$  vanishing at the origin to order $k$
expands the Euclidean metric on the circle $\T$ at least by $k$.
\end{lem}

\begin{proof}
  Under the circumstances,
$$
   B (z) =  z^k \tl B, \quad \text{where}\  \tl B \ \text{is another Blaschke product}.
$$
In the angular coordinate $z=e^{i\phi}$ it assumes a form
$$
    \phi\mapsto k \phi + h(\phi),
$$
where $h'(\phi)>0$ since $\tl B$ is orientation preserving on the circle.
The assertion follows.
\end{proof}

\begin{proof}[Proof of Theorem \ref{hor expansion thm}.]
Let us uniformize $\Secmig$ by the Riemann sphere so that the $\Secmig\cap \Cmig$ becomes the unit circle
$\T=\{ |\la|=1 \}$.
Lemma \ref{phi_n} implies that  in this coordinate,  the map $\zeta = \psi_n(\la)$ becomes a Blaschke product.
By Lemma \ref{order of vanishing}, it vanishes of order $2^{n+2}$ at the origin.
By Lemma \ref{Blaschke}, it expands the circle metric at least by   factor $2^{n+2}$.
Hence $R^n$ expands the cylinder metric along $\Secmig\cap \Cmig$  at least by  $c\, 2^{n+2}$ with some $c>0$.

But then the same is true for any $v \in K^h(x), \, x \in \Cmig$.  If $N$ is any
small neighborhood of the indeterminate points $\INDmig_\pm$, then all
horizontal vectors $v\in K^h(x)$ that are equivalent mod $\LL^c$ have a
comparable length (uniformly over $x\in \Cmig\sm N$).
Corollary~\ref{COR:HORIZONTAL_VECS_ITERATE_COMPARABLY} then gives that the
lengths of the iterated vectors remain comparable.

Finally, if $x \in N$, then  Lemma \ref{LEM:BOUNDED_CONTRACT} gives that one
iterate of $\Rmig$ can contract the horizontal length only by a bounded factor
and the result follows.
 
\comment{%%%%%%%%%%%%%%%%%%%%%%%
But then the same is true for any $v\in \KK^h(x)$, $x\in \Cmig$. 
Indeed, by compactness, the angle between the central line $\LL^c$ and the  cone $\KK^h(x)$ 
is bounded away from $0$ as long as $x\in \Cmig\sm N$, where $N$ is a neighborhood of  the indeterminacy points $\alpha_\pm$. 
But then all horizontal vectors $v\in \KK^h(x)$ that are equivalent mod $\LL^c$
 have a comparable length (uniformly over $x\in \Cmig\sm N$). 
Hence expansion by $DR^n(x)$   can be detected by taking any horizontal vector $v\in \KK^h(x)$,
as long as $x, R^n x\in \Cmig\sm N$.

Finally, the last condition can be eliminated
since, by Lemma \ref{LEM:BOUNDED_CONTRACT}, one iterate of $\Rmig$ can contract the horizontal length only by a bounded factor. 
&&&&&&&&&&&&&&&&&&&}

%Indeed, away from $\{\alpha_\pm\}$ this is true since the horizontal cones are transverse to the critical lines $\II_{\pm \pi/2}$;
%near  $\{\alpha_\pm\}$ this is true since the cones meet the exceptional divisor at the regular points \note{checked?}
%or by the direct estimate of Lemma  \ref{hexp}. 
\end{proof}

\comm{****
begin{proof}
We will work in the affine $(u,w)$-coordinates. 
Let us consider the central projection from the origin to the line at infinity: 
$$
   \pi: \CP^2\ra \BOTTOMmig,\quad    (u,w) \mapsto \zeta = u/w.
$$
Consider a horizontal section of the solid cylinder,   \note{introduce}
given in these coordinates by
$$
     S= S_R = \{ uw= R^2,\quad |u| \geq R \}.
$$
Then the map 
$$
    \psi:= \pi \circ R^n : S\ra \D
$$
is a Blaschke product of degree $2\times 4^n$
vanishing at the origin with order $\sim 2^n$.
So
$$
    \psi(z) = z^{2^k} \phi (z),
$$
(where $\phi$ is another Blaschke product). This map expands the circle
metric at least by $c 2^n$.
Since the curves $R^n S$ have  bounded slope, this gives a bound on the
horizontal expansion of $R$.
end{proof}
************}

\subsection{Combinatorial Approach}\label{SUBSEC:COMBINATORIAL_EXPANSION}

We will now present another proof of Theorem~\ref{hor expansion thm} that is based on a
combinatorial interpretation of the DHL and the Lee-Yang Theorem with Boundary
Conditions.  The notation is
from \S \ref{background}.

Recall that the partition function of the Ising model on $\Gamma_n$ is given as
\begin{equation}\label{in3}
Z_n=\sum_\sigma e^{- H_n(\sigma)/T}=\sum_\sigma t^{-I(\si)/2} z^{-M(\si)},
\end{equation}
\noindent
where $M(\si)$ and $I(\si)$ are the magnetic moment (\ref{M}) and interaction (\ref{I}) of the configuration $\si$.
Let $\si_+ \equiv +1$ and $\si_- \equiv -1$.
Clearing denominators we obtain the modified partition function
\begin{equation}\label{in6}
\check Z_n(z):= t^{I(\si_-)}z^{4^n}Z_n(z) =\sum_{j=0}^{N} a_j z^j, %\qquad a_j=\sum_{\si:\; M(\si)=j-4^n} t^{[I(\si^-)-I(\si)]/2},
\end{equation}
with $N=2\cdot 4^n$.
Recall the basic symmetry of the Ising model
\begin{equation*}
a_{N-j}=a_j,\quad j=0,1,\ldots N;\qquad a_0=a_N=1.
\end{equation*}
\noindent
which is obtained under the involution $\si \mapsto -\si$ and from the invariance $I(\si) = I(-\si)$.

The generating graph $\Gamma$ is symmetric under reflection across the vertical
line through the marked vertices $a$ and $b$.  This allows us to
factor\footnote{$V_n$ also factors, but will not use it.} the conditional
partition functions $U_n$ and $W_n$ as
\begin{eqnarray*}
U_n =  {\rm U}_n^2 \qquad W_n =  {\rm W}_n^2,
\end{eqnarray*}
\note{notation?}
\noindent
where ${\rm U}_n$ and ${\rm W}_n$ correspond to the conditional partition
functions of the right (or left) half of $\Gamma_n$, having the same boundary
conditions as $U_n$ and $W_n$.

Both halves of $\Gamma_n$ have valence $2^{n-1}$ at $a$ and $b$.  In particular,
if $\sigma(a) = \sigma(b) = +1$, there are at most $4^n/2-2^n$ edges both of
whose endpoints have spin $-1$. This gives that ${\rm U}_n$ has no terms in $z$ of
degree greater than $4^n/2-2^n$.  Similarly, ${\rm W}_n$ has no terms in $z$ of
degree lower than $-4^n/2 + 2^n$.

Clearing denominators, one obtains 
\begin{eqnarray*}
\check {\rm U}(z) &:=& t^{I(\si_+)/2}z^{4^n/2} \ {\rm U}(z) = \sum_{j=0}^{N_0} a_j^+ z^j \,\,\, \mbox{and} \\
\check {\rm W}(z) &:=& t^{I(\si_-)/2}z^{4^n/2-2^n} \ {\rm W}(z) = \sum_{j=0}^{N_0}a_j^- z^j,\\
\end{eqnarray*}
\noindent
where $N_0 = 4^n-2^n$.  
It follows from the Lee-Yang Theorem with Boundary Conditions that the zeros
$b_1,\ldots,b_{N_0}$ of $\check {\rm W}(z)$ all lie in $\D$.

The basic symmetry of the Ising model appears as the following symmetry between $\check {\rm U}$
and $\check {\rm W}$:
\begin{eqnarray*}
a^-_{N_0-j} = a_j^+, \qquad j=0,1,\ldots,N_0.
\end{eqnarray*}
%It follows from the bijection between
%the set of configurations $\sigma$ satisfying $\sigma(a) = \sigma(b) = +1$ and $M(\sigma) = j-
%4^n/2$ and the configurations $\sigma'$ satisfying $\sigma'(a) = \sigma'(b) = -1$ and
%$M(\sigma') = 4^n/2 - j$ obtained by $\sigma \mapsto -\sigma$ and the invariance $I(\sigma) = I(-\sigma)$.
\noindent
Consequently,
\begin{eqnarray*}
\check {\rm W}_n(z) = \prod_{j=1}^{N_0}(z-b_j) \,\,\, \mbox{and} \,\,\,  \check {\rm U}_n(z) = \prod_{j=1}^{N_0}(1-b_j z) = \prod_{j=1}^{N_0}(1-\bar{b}_j z),
  \end{eqnarray*}
\noindent
using also that $\check {\rm U}_n(z)$ has real coefficients.

Since $z_n^2 = W_n/U_n = {\rm W}_n^2/{\rm U}_n^2$, we obtain the Blaschke Product
\begin{eqnarray}\label{EQN:ZBP}
z_n = \frac{{\rm W}_n}{{\rm U}_n} = \frac{z^{2^n} \check{\rm W}_n}{\check {\rm U}_n} = z^{2^n}\prod_{j=1}^{N_0}\frac{z-b_j}{1-\bar{b}_j z}
\end{eqnarray}
having all of its zeros in $\D$ and a zero at $z=0$ of multiplicity $2^n$.

\begin{rem}The degree of (\ref{EQN:ZBP}) is half of the degree for the Blaschke product in terms of $u$ and $v$ that was obtained using the global approach because of the relationship $z^2 = W/U$.  This is why the multiplicity of $z=0$ in (\ref{EQN:ZBP}) is only $2^n$ instead of $2^{n+1}$.
\end{rem}

The remainder of the proof continues as in \S \ref{SUBSEC:GLOBAL_APPROACH}.

\section{Low temperature dynamics: basin of $\BOTTOMphys$ and its stable foliation}\label{SEC:BOTTOM BASIN}

The bottom circle $\BOTTOMphys$ is superattracting within $\Cphys$ by Property (P5)
from \S \ref{SUBSEC:MAP_ON_CYL}, so there is an open set 
$\WW^s(\BOTTOMphys) \subset \Cphystl$ 
consisting of points whose orbits converge to $\BOTTOMphys$, 
called the {\it basin of attraction} of $\BOTTOMphys$. 
Obviously, $\WW^s(\BOTTOMphys)$ is completely invariant under $\Rphys | \Cphystl$.

% \subsection{Basin of attraction $\WW^s(\BOTTOMphys)$.}
\subsection{Low temperature cylinder $\Cphyslow$}
\label{SUBSEC:BASIN_BOTTOM}   

Recall that $t_c$ is height of the real repelling fixed point $\FIXphys_c$ in the
invariant line $\{\phi=0\}$, and let $\FIXphys_c'= (\pi, t_c)$, 
$$ 
 \Cphyslow= \{ (\phi,t)\in \Cphys:\ t\leq t_c\}.
$$ 
 We will call $\Cphyslow \sm \{\FIXphys_c, \FIXphys_c' \}$ the {\it low temperature cylinder}.

Let us begin with a simple observation: % showing that the basin $\WW^s(\BOTTOMphys)$ is quite big:
% it contains the whole low temperature cylinder (except the point $\FIXphys_c$)
% and a family of tongues attached to a dense set of points on $T$.  

\begin{lem}\label{BB}
The basin $\WW^s(\BOTTOMphys)$ contains the low temperature cylinder
$\Cphyslow\sm \{\FIXphys_c, \FIXphys_c'\}$.
% a family of tongues $\Tongue_{\pm}^{k,n}$ 
%attached to the points $\alpha^{k,n}_\pm$. 
\end{lem}

\begin{proof}
We will check that points  $x\in \Cphyslow \sm \{\FIXphys_c\}$ converge to the bottom $\BOTTOMphys$ at least as fast
as they do on the low temperature interval $(\FIXphys_0, \FIXphys_c)\subset \II_0$.
 
For $x=(\phi,t)\in \Cphys$, we let $t(x)=t$.  Since the function \begin{displaymath}y \mapsto
\frac{y+1}{y+s}\end{displaymath} is non-increasing on $(-1,1]$ for $s \geq 1$, (\ref{f}) gives
that $t (\Rphys x)\leq  t(\Rphys(0,t(x)))$, with equality attained only if $x\in \II_0,
\II_\pi$. 

Let $t_n=t(\Rphys^n x)$ and let $q_n = t(\Rphys^n(0,t_0))$. 
By the above observation, if $x\in \Cphyslow\sm \{\FIXphys_c\}$ then $t_1 < t_c$,
and furthermore,  $t_{n+1} \leq  q_n  \to 0$ as $n\to \infty$. Thus,
$\Cphyslow\sm \{\FIXphys_c\} \subset \WW^s(\BOTTOMphys)$. 
\end{proof}

\subsection{Complex stable lamination of $\BOTTOMphys$ in $\C^2$}
In $\C^2$, the circle $\BOTTOMphys$ is {\it hyperbolic}, with a complex
one-dimensional transverse stable direction (the $t$-direction) and a real
one-dimensional transverse unstable one (the unstable direction within the
plane $t=0$).

Given an $\eps>0$, the $\eps$-{\it local basin} $\WW^s_{\C, \eps}(\BOTTOMphys)$
is the set of  points $\zeta \in \C^2$ that are $\eps$-close to $\BOTTOMphys$
and whose orbits converge to $\BOTTOMphys$ while remaining $\epsilon$-close to
$\BOTTOMphys$.  Alternatively, we can use any forward invariant open set containing
$\BOTTOMphys$ to define a local stable set $\WW^s_{\C, \loc}(\BOTTOMphys)$, with the specific
open set that is used implicit in the notation.

Similarly, for $x=(\phi, 0)\in \BOTTOMphys$, a local stable manifold
$\WW^s_{\C,\loc}(x)\equiv \WW^s_{\C,\loc}(\phi)$ is a $1$-dimensional holomorphic
curve containing $x$ consisting of all  points $\zeta$ near $x$ whose orbits
are forward asymptotic to the orbit of $x$, while remaining close to the orbit
of $x$.

In this subsection we will construct the local stable lamination
\footnote{Here, a lamination is a family of disjoint holomorphic curves that has
a local product structure, in a sense similar to that given in \S \ref{sec:
strictly vert lam}.} of the bottom circle $\BOTTOMphys$ in $\C^2$ and will show
that the leaves $\WW^s_{\C,\loc}(\phi)$, $\phi\in \T$, of this lamination are
holomorphic curves filling in a topological real $3$-manifold contained within
$\WW^s_{\C,\loc}(\BOTTOMphys)$.  

In the case of diffeomorphisms, the construction of the stable laminations for
hyperbolic sets is a standard background of the general theory, see
\cite{HPS,PDM,SHUB}. However, we have not been able to find an adequate
reference in the non-invertible case (notice, however,
the remark in \cite[p. 79]{PDM}), so we will give a direct argument in
our situation.

We will make use of the simple structure of the postcritical locus near
$\BOTTOMphys$ (comprising two lines $\{ t=\pm \pi/2\}$ collapsing to the fixed
point $\FIXphys_0$) and of the holomorphic $\la$-lemma.  This is similar to the
method used in \cite[\S 2.4]{HP_NEWTON} and \cite[\S 4]{ROE}.

\begin{prop}\label{PROP:STABLE_MANIFOLD}
For sufficiently small $\eps>0$, local stable manifolds $\WW^s_{\C,\loc}(x)$, $x\in \BOTTOMphys$,
are holomorphic curves whose union $\bigcup_{x \in \BOTTOMphys} \WW^s_{\C,\loc}(x)$ form a lamination 
supported on the local basin of $\WW^s_{\C, \loc}(\BOTTOMphys)$.
Moreover, $\bigcup_{x \in \BOTTOMphys} \WW^s_{\C,\loc}(x)$ is  a topological real $3$-manifold.
\end{prop}

\begin{proof}
Let $\QQ \subset \BOTTOMphys$ be the set of iterated preimages of $1$ under $z
\mapsto z^4$.  We will construct a family of analytic discs $\DD_z$ through the
the set $\QQ$ in such a way that each disc intersects $\BOTTOMphys$ in exactly
one point $z \in \QQ$ and so that $\DD_z$ is obviously the stable disc of $z$.
We will furthermore verify that each $\DD_z$ can be expressed as the graph of a
function $z = \eta(t)$ for all $t$ in an appropriate small disc $\D_\rho$.

This family of discs provides a {\em holomorphic motion} of
$\QQ$, parameterized by $\D_\rho$:
\begin{eqnarray*}
h:\QQ \times \D_\rho \rightarrow \C.
\end{eqnarray*}
\noindent
Then, the $\lambda$-Lemma \cite{LY_LAMBDA,MSS} for holomorphic motions immediately gives that $h$
extends to the closure providing a continuous mapping $\bar h:\BOTTOMphys \times
\D_\rho \rightarrow \C$.

Recall that the line $t=0$ is superattacting away from the origin.  Therefore,
choosing $0 < a < 1$ and $C > 0$ we can easily further restrict $\rho>0$ so
that $|t_n| \leq C a^n$ for any choice of $(z,t) \in \cup_{z\in \QQ} \DD_z$.
(Here $t_n = t(\Rphys^n(z,t))$.) Therefore, points in the closure, and in
particular, points in the image of $\BOTTOMphys \times \D_\rho$ under $(\phi,t)
\mapsto \left(\bar h(\phi,t),t\right)$ converge at least geometrically to
$\BOTTOMphys$.

Given $(\phi,0) \in \BOTTOMphys$, the stable leaf $\WW^s_{\C,\loc}(\phi)$ is
parameterized by $t \ra \left(\bar h(\phi,t),t\right)$ for $|t| < \rho$.  The
union of all stable leaves, which is given by the image of $\BOTTOMphys \times
\D_\rho$ under $(\phi,t) \mapsto \left(\bar h(\phi,t),t\right)$, is the desired
lamination.  It is contained within
$\WW^s_{\C,\loc}(\BOTTOMphys)$ by the discussion in the previous paragraph.  

\msk
We now construct the holomorphic motion $h:\QQ \times \D_\rho \rightarrow
\C$.

Consider a neighborhood $\Delta_{\delta,\rho}$ of $\BOTTOMphys$ of the form
\begin{eqnarray*}
\Delta_{\delta,\rho} = \{(z,t) \in \C^2 \,\, : \, 1-\delta < |z| < 1+\delta,\,\, |t| < \rho \}
\end{eqnarray*}

Let $\partial ^v \Delta_{\delta,\rho}$ be the vertical boundary $|z| = 1 \pm \delta$ and $\partial ^h \Delta_{\delta,\rho}$
the horizontal boundary $|t| = \rho$.

\begin{lem}\label{LEM:GOODNBHD}
We can choose $\delta, \rho > 0$ sufficiently small so that
\begin{enumerate}
\item The complex vertical cone-field $|dt(v)| > |dz(v)|$ is backward invariant under $\Rphys$.

\item $\Rphys(\partial ^v \Delta_{\delta,\rho})$ is entirely outside of $\Delta_{\delta,\rho}$ and $\Rphys(\partial ^h \Delta_{\delta,\rho})$ is entirely contained within $|t| < \rho$.
\end{enumerate}
\end{lem}

\begin{proof}
On $\BOTTOMphys$
we have
\begin{eqnarray*}
D\Rphys = \left[\begin{array}{cc} 4z^3 & 0 \\ 0 & 0 \end{array}\right]
\end{eqnarray*}
so that, by continuity the $(1,1)$ term of $D\Rphys$ dominates the remaining
terms in any sufficiently small neighborhood of $\BOTTOMphys$.  This is
sufficient for the desired invariance of the conefield in (1).

We now further restrict $\partial ^h \Delta_{\delta,\rho}$ so that (2)
holds.  However, this again follows easily by continuity because on $t=0$ we
have that $\Rphys$ is given by $z \mapsto z^4$ which maps the boundary of any
annulus of the form $1-\delta < |z| < 1+\delta$ well outside of the annulus and
because the line $t=0$ is superattracting.  \end{proof}

We will call an analytic disc $\DD$ in $\Delta_{\delta,\rho}$ {\em admissible} if:
\begin{enumerate}
\item $\DD$ intersects $\BOTTOMphys$ in a single point,
\item  $\partial \DD$ intersects $\partial \Delta_{\delta,\rho}$ only in the vertical boundary $\partial^v \Delta_{\delta,\rho}$, and
\item The tangents to $\DD$ lie within the vertical cone-field Lemma \ref{LEM:GOODNBHD} above.
\end{enumerate}
\noindent
These properties are chosen to ensure that
any admissible disc can be written as the the graph of some function $z=\eta(t)$ for some analytic $\eta:\D_\rho \ra \C$.

We now construct the family of discs $\DD_z$ for all $z \in \QQ$.  We do this ``by hand'' for the first preimages of $1$, letting:
\begin{eqnarray}\label{EQN:FIRST_DISCS}
\DD_{\pm 1} = {\pm 1} \times \D_\rho, \mbox{  and  }
\DD_{\pm i} = {\pm i} \times \D_\rho,
\end{eqnarray}
\noindent
because this first level of discs is ``delicate'' since $\DD_{\pm i}$ contain portions of the collapsing lines.
Note that each is clearly the stable disc of its basepoint, and clearly admissible.

We now inductively define admissible stable discs $\DD_z$ over any $z \in
\QQ$, assuming that $z^4 \neq 1$.  Suppose that  $z$ is an $n$-th preimage
of $1$.  By construction, there is an admissible stable disc  $\DD_{z'}$ over
$z' = z^4$, which transversally intersects the critical value locus of $\Rphys$.
(Within $\Delta_{\delta,\rho}$ the critical value locus of $\Rphys$  is the
union of the line $t=0$ and the point $(1,0)$.)
Therefore, $R^{-1}(\DD_{z'})$ is a finite union of
analytic discs.  Let $\DD_z$ be the component of $R^{-1}(\DD_{z'}) \cap
\Delta_{\delta,\rho}$ containing $z$.   By properties (1) and (2) from the lemma,
we see that since $\DD_{z'}$ was an admissible disc, so is
$\DD_z$.

Continuing in this way one defines a family of admissible stable discs over every $z \in \QQ$.
The result is the desired holomorphic motion $h:\QQ \times \D_\rho \rightarrow \C$.

\msk
The map $(\phi,t) \mapsto \left(\bar h (\phi,t),t\right)$ is
clearly an immersion of $\BOTTOMphys \times \D_\rho$ into $\C^2$ with image $\bigcup_{x
\in \BOTTOMphys} \WW^s_{\C,\loc}(x)$.
Shrinking $\rho$ slightly, if necessary, this immersion can be made into an embedding.
\end{proof}

The $\la$-Lemma  gives the following regularity for $\bigcup_{x
\in \BOTTOMphys} \WW^s_{\C,\loc}(x)$.  Globally, it is just a topological
manifold, however each slice with $t=t_0$ for $|t_0| < \rho$ is the image of
the unit circle $\BOTTOMphys$ under a quasiconformal homeomorphism with
dilatation \begin{eqnarray*} K \leq \frac{\rho + |t_0|}{\rho - |t_0|}.
\end{eqnarray*}

%Note that at this point we do not know that the stable lamination $\bigcup_{x
%\in \BOTTOMphys} \WW^s_{\C,\loc}(x)$ forms all of
$\WW^s_{\C,\loc}(\BOTTOMphys)$.  
%This will be shown in
%\ref{SUBSEC:REAL_ANALYTIC}.

\subsection{Stable foliation of $\BOTTOMphys$ in the cylinder}  \label{SUBSEC:BOTTOM_FOLIATION}
 Let us now consider the slices of the local stable manifolds by the cylinder,
$$
  \WW^s_\loc(x)\equiv \WW^s_\loc(\phi)= \WW^s_{\C,\loc}(\phi)\cap \Cphys, \quad x=(\phi,0), \ \phi\in \T.
$$ 
They are real analytic curves that form a foliation of a neighborhood of $\BOTTOMphys$ in $\Cphys$.
Note that the $\WW^s_\loc(\pm \pi/2)$ are arcs of the collapsing intervals $\II_{\pm \pi/2}$.

We will now globalize this foliation.
 Let $\WW^s_n(\phi)$ stand for the lift  % connected component of $\Rphys^{-n}(\WW^s(4\phi,0))$
of $\WW^s(4\phi,0)$ that begins at $(\phi,0)$.  
Since $\Rphys(\WW_\loc(\phi))\subset \WW_\loc(4\phi)$, we have: 
$$
  \WW^s_\loc(\phi)\equiv \WW^s_0(\phi) \subset \WW_1^s(\phi)\subset \WW^s_2(\phi)\subset\dots
$$
By Property (P4), for $\phi \neq \pm \pi/2$, each $\WW^s_n(\phi)$ is
a real analytic curve, while $\WW^s_n(\pm \pi/2)=\II_{\pm \pi/2}$ for all $n\geq 1$. 
Hence the sets
$$
   \WW^s(\phi)=\bigcup_{n=0}^\infty \WW^s_n(\phi).
$$  
are real analytic curves for all $\phi\in \T$. They are called the {\it global stable manifolds}
of the points  $x=(\phi,0)\in \BOTTOMphys$.%
\footnote{This terminology is not completely standard, as usually the global
stable manifold of $x$  is defined as the set of {\it all} points that are forward
asymptotic to the  $\orb x$.  In our situation, it would be 
$\cup \Rphys^{-n} \WW^s(4^n \phi)$, which is disconnected.} 
Note that $\WW^s(\pm\pi/2)=I_{\pm\pi/2}$, while $\WW^s(0)=[\FIXphys_0, \FIXphys_c)$.  

By construction, $\Rphys(\WW^s(\phi))\subset \WW^s(4\phi)$,
and, in fact,  $\WW^s(\phi)$ is the lift of $\WW^s(4\phi)$ by  $\Rphys$
that begins at $(\phi,0)\in \BOTTOMphys$  (compare Lemma \ref{vertical lifts}).

\begin{lem}\label{stble man str vert}
  The stable manifolds $\WW^s(\phi)$ are strictly vertical curves.
\end{lem}   

\begin{proof}
The stable manifolds $\WW^s(x)$ are tangent to the $\Ker D\Rphys(x)$.  
From representation $R=f\circ Q$ in \S \ref{f-sec}
we see that  the $\Ker D\RR(x)=\Ker DQ(x)$ are orthogonal to $\BOTTOMphys$. % at any $x\in \BOTTOMphys$,  
Hence the local stable manifolds $\WW_\eps^s(x)$ go through the vertical algebraic cones $\KK^v(x)$ 
(for $\eps>0$ sufficiently small), so they are strictly vertical. 
Since the cone field $\KK^v(x)$ is backward invariant (see Cor. \ref{COR:R_CONEFIELD}), 
the global stable manifolds $\WW^s(x)$ are strictly vertical as well. 
\end{proof}

\subsection{Stable tongues.}\label{stable tongues sec}

A {\it tongue} $\Tongue$   
attached to a ``tip'' $x \in \TOPphys$ is a domain in $\Cphys$ bounded by two proper vertical paths
meeting only at $x$. % that are disjoint (except at $x$) and 
Note that $\Tongue$ meets $\BOTTOMphys$ in an interval $\BOTTOMphys_\Tongue$, called its {\it bottom}.
A tongue is called {\it stable} if it is contained in $\WW^s(\BOTTOMphys)$ and is foliated by proper
stable manifolds $\WW^s(\phi)$, $(\phi,0)\in \BOTTOMphys_\Tongue$ (terminating at $x$). 

\comm{***
   Since the critical leaves $\II_{\pm}$ are collapsed to
the fixed point $\FIXphys_0\in \BOTTOMphys$, they are contained in $\WW^s(\BOTTOMphys)$.  Hence all their lifts
under the iterates of $\Rphys$, which are also properly embedded vertical curves by
Property (P9), are also
contained in $\WW^s(\BOTTOMphys)$. Since the basin $\WW^s(\BOTTOMphys)$ is open, each of these lifts is
contained in $\WW^s(\BOTTOMphys)$ together with some tongue. 
****}

\begin{prop}\label{principal tongues}
   There are two stable tongues $\Tongue(\alpha_\pm)$ attached to the indeterminacy points $\alpha_\pm$ respectively. 
They are symmetric with respect to the collapsing intervals $\II_{\pm\pi/2}$ and have positive angles at the tip.
\end{prop}

%More can be said about the geometry of the tongues in a neighborhood of $\TOPphys$.
%As we can clearly see on Figure \ref{FIG:CYLINDER_BASINS},
%each of the tongue $\Tongue_{\pm}$                                      
%attached to the indeterminacy points  $\INDphys_\pm$
%contains a symmetric wedge of some positive angle $2\omega_*$.        

\begin{proof}
By the blow-up formula (\ref{Icurve}) % of \S \ref{APP:BLOW_UPS} give that 
points $x$ approaching $\INDphys_{\pm\pi/2}$ at angle $\omega$ from $\II_{\pm\pi/2}$ are mapped to the point
$\Icurve(\om) = (2\omega,\sin^2 \omega)$.  
This curve intersects the critical level $\{t=t_c\}$            
at two points, $\Icurve(\pm \om_c)$.                        
Moreover, 
$$
 \Icurve[-\om_c, \om_c]\subset \Cphyslow\sm \{\beta_c, \beta_c'\}\subset \WW^s(\BOTTOMphys)
$$ 
% for $\om\in [-\om_c, \om_c]$.
Hence there is a region $U$ under the arc $\Icurve[-\om_c,\om_c]$ 
(comprised of two  symmetric topological triangles) 
foliated by stable manifolds $\WW^s(\phi)$ (see Figure \ref{FIG:FOLIATION_TONGUE}).
By Lemma \ref{vertical lifts}, this region lifts by $\Rphys$ 
to two tongues $\Tongue'(\alpha_\pm)$ attached to $\alpha_\pm$. 
Each $\Tongue'(\alpha_\pm)$ has angle $2\om_c>0$ at the tip.
The desired tongues $\Tongue(\alpha_\pm)$ are the maximal stable tongues containing $\Tongue'(\alpha_\pm)$. 
By $\pm\pi/2$-symmetry of the basin $\WW^s(\BOTTOMphys)$ (see Property (P1) in \S \ref{str on CC}),
they are symmetric with respect to the corresponding axes $\II_{\pm \pi/2}$.
\end{proof}

\begin{figure}
\begin{center}
\input{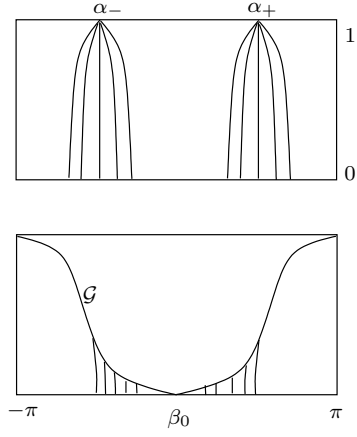}
\end{center}
\caption{\label{FIG:FOLIATION_TONGUE} 
Topological triangles foliated by stable manifolds (below)
and the corresponding stable tongues attached to the indeterminacy points $\INDphys_\pm$ (above).}
\end{figure}

The above tongues $\Tongue_\pm$ will be called the {\it primary} stable tongues of $\WW^s(\BOTTOMphys)$. 

%\begin{rem}
% In \S \ref{SUBSEC:LONG_TEETH} we will give an upper bound on the angle of the $\Tongue_\pm$ at the tip.
%\end{rem}    \note{make this discussion after Prop 6.7}

For $n\in \N$, let 
\begin{equation}\label{PI set}
   \PI^n=\Rphys^{-n} \{\alpha_\pm\} = \{(\phi_{n,k},1)\}_{ k=0}^{2^n-1}, \ \mbox{where}\ \phi_{n,k}=\frac{\pm\pi/2 + 2\pi k}{2^{n}}.
\end{equation}
and let $\PI=\cup \PI_n$ be the {\it pre-indeterminacy set}.
%of all preimages  of the points of indeterminacy $\INDphys_{\pm}=(\pm \pi/2, 1)\in \TOPphys$ 
%under the iterates of $\Rphys$.
%(These are the points with diadic angles  $(\pm\pi/2 + 2\pi k)2^{-n}$, $n\in \N$, $k=0,\dots, n-1$.)  

\begin{prop}\label{tongues are a.e.}
There is a family of disjoint stable tongues  $\Tongue_k(\alpha)$
attached to the points $\alpha\in \PI$ such that
\begin{itemize}  
\item [(i)] $\Tongue_0 (\alpha_\pm)\equiv \Tongue(\alpha_\pm)$;
\item [(ii)] If $\alpha\not=\alpha_\pm$ then  $\Tongue_k (\alpha)$ is a regular lift of some $\Tongue_j(\Rphys(\alpha))$;
\item [(iii)] If  $\alpha=\alpha_\pm$ but $k\not=0$ 
        then  $\Tongue^{n+1}_k(\alpha)$ is a singular lift of some $\Tongue_j(\beta)$ with $\beta\in \PI$;
%  A lift of any tongue $\Tongue_\pm^{n,j}$
\item [(iv)] The union $\cup \Tongue_k(\alpha)$ is backwards invariant;
\item [(v)] The union of the bottoms of all the tongues is an open set of full Lebesgue measure in $\BOTTOMphys$. 
\end{itemize}
\end{prop}

\begin{proof}
Lemmas \ref{stble man str vert}, \ref{S transv to Icurve}, \ref{vertical lifts}, 
and Corollary  \ref{COR:R_CONEFIELD} imply that the lifts of stable tongues by $\Rphys$
are stable tongues. So, taking all possible lifts of the principal tongues $\Tongue(\alpha_\pm)$
by the iterates of $\Rphys$, we obtain an infinite family of stable tongues attached to points of $\PI$.  

Since the stable manifolds are disjoint, the stable tongues with different tips are disjoint. 
If one of these tongues overlaps with a principal tongue $\Tongue(\alpha_\pm)$ then it must be contained in it,
by maximality of the latter. 
It follows that any two tongues in the family are either disjoint or nested.          
Keeping only the maximal tongues, we obtain the desired family of tongues.  

All the properties of this family are straightforward except the last one.
%  Each tongue is open, so the first assertion is obvious. 
This one follows from ergodicity of the map $z\mapsto z^4$  with respect to the Lebesgue measure on $\T$,
which implies that $\cup \Rphys^{-n} (\BOTTOMphys_{\Tongue_\pm})$ has full measure. 
\end{proof}

% All the tongues in this family except $\Tongue_\pm$ will be called {\it secondary}. 

\begin{prop}\label{s-tongues inside La}
  The family of stable tongues $\Tongue_k(\alpha_+)$ is contained in the central tongue  
$\La_+$, and their bottoms form an open set of full Lebesgue measure in the bottom of $\La_+$ 
($=(\pi/4, 3\pi/4)$). Moreover, each of these tongues has positive angle at its tip.  
The same property is valid for $\alpha_-$. 
\end{prop}

\begin{proof}  
Let us say that a topological rectangle $\Pi\subset \Cphys_-$ is a {\it singular stable rectangle}
if it is bounded by an interval on $\BOTTOMphys$, an arc of $\Icurve$, and two proper vertical paths 
in $\Cphys_-$, and is foliated by stable leaves.%
\footnote{We allow one of the vertical sides to degenerate to $\beta_0$, making $\Pi$ degenerate to a triangle.} 
 For instance, the intersection of any stable tongue
$\Tongue_k(\alpha_+)$ with $\Cphys_-$ is   singular stable rectangle 
(by Lemmas \ref{stble man str vert} and \ref{S transv to Icurve}).

Let $\Pi_k$ be the family of {\it maximal} singular stable rectangles. 
Proposition \ref{tongues are a.e.} (v) implies that their bottoms have full measure in $\BOTTOMphys$. 

Any singular stable  rectangle lifts to a stable tongue attached to the indeterminacy points $\alpha_+$
with positive angle at the tip 
(equal to $|\om_1-\om_2|$ where $\Icurve(\om_i)$  are the upper vertices of $\Pi$).
Lifting the rectangles $\Pi_k$, we obtain the desired family of tongues. 
\end{proof}

However, the above discussion does not imply that there are infinitely many stable tongues $\Tongue_k(\alpha_\pm)$:
this will be justified in the following piece.  

\subsection{Long hairs growing from $\TOPphys$}
\label{SUBSEC:LONG_TEETH}
Let us recall dynamics on the invariant interval $\II_0$ (see \S \ref{SUBSEC:MAP_ON_CYL}).
Let 
$$
    \II_0^+= \{(\phi,t)\in \II_0:\ t>t_c\}, \quad   \II_0^\de= \{(\phi,t)\in \II_0:\ t\geq t_c+\de\}.
$$ 
The interval $\II_0^+$ is the stable manifold of the high temperature fixed point $\FIXphys_1\in \TOPphys$. 
In this section it will be convenient to orient $\II_0^+$ so that it begins at $\FIXphys_1$.  

We say that  a sequence of curves $\gamma_k$ {\it stretch along} $\II^+_0$ if for any $\de>0$ there is $k_0$
such that the curves $\gamma_k$, $k\geq k_0$, contain arcs that can be represented as graphs $\phi=\gamma_k(t)$   
over $\II_0^\de$ such that  $\gamma_k \to 0$ in $C^1(\II^+_0)$. 

Let us consider a sequence of preimages $\beta_k\in \TOPphys$ of $\FIXphys_1$ converging to $\FIXphys_1$, say 
$\beta_k= (2\pi/2^{k-1}, 1)$. Let $\gamma_k$ stand for the lift of $\II^+_0$ by $\Rphys^k$ that begins at $\beta_k$.

\begin{lem}\label{PROP:LONG_TEETH}
The curves $\gamma_k$ are pairwise disjoint,
and the orbits of points $x\in \cup \gamma_k$ converge to $\FIXphys$.
Moreover, the curves $\gamma_k$ stretch along $\II^+_0$.   
\end{lem}

\begin{proof}
The first assertion is obvious. 
The last one follows from the Dynamical $\la$-Lemma (see  \cite[pp. 80-85]{PDM})
applied near the hyperbolic fixed point  $\FIXphys_1$. 
\end{proof}

\begin{prop}\label{inf many tongues}
  There are infinitely many tongues $\Tongue_k(\alpha_\pm)$.
Each of them sticks at a positive angle out of the top $\TOPphys$.
\end{prop}

\begin{proof}
Let  $\gamma_k'$ be the curves $\gamma_k$ translated  horizontally by $\pi$.
They stretch along the interval $\II_\pi^+=\{(\phi,t)\in \II_0:\ t>t_c\}$
and hence (for $k$ sufficiently big)
intersect the singular curve $\Icurve$ transversally near the top.

By the symmetry $\Rphys(\phi+\pi)=\Rphys(\phi)$ (see Property (P1)),
the orbits of points $x\in \cup \gamma_k$ converge to $\FIXphys$.
So, $\cup \gamma_k'$ is disjoint from the basin $\WW^s(\BOTTOMphys)$,
and hence the singular stable rectangles from the proof of 
Proposition \ref{s-tongues inside La} can meet $\Icurve$ only in between 
the curves.

Since the tips of the stable tongues $\Tongue(\alpha)$ are dense in $\TOPphys$,
there is a tongue  $\Tongue(\alpha_k)$   ``squeezed'' in between any pair of curves $\gamma_k'$ and $\gamma_{k+1}'$. 
Then  the corresponding rectangles $\Tongue(\alpha_k)\cap \Cphys_-$ are contained in disjoint
(for $k$ big enough) maximal rectangle $\Pi_k$. The latter lift to disjoint stable tongues $\Tongue_k(\alpha_\pm)$. 
This proves the first assertion.

The second one follows as well since any singular stable rectangle $\Pi$
is separated from $\II_\pi$ by some curve $\gamma_k'$,
hence the corresponding tongue (the singular lift of $\Pi$) meets
the top non-tangentially.
\end{proof}

%To summarize, we observe a sequence of pairwise disjoint singular stable rectangles $\Pi_k\subset \Cphys_-\cap \{0<\phi<\pi\}$ 
%with top arcs $\Icurve(\om_k, \tl \om_k)$, where $0=\om_0< \om_1< \dots,$ $\om_k\to \pi/2$,
%%% and a sequence of $\II_\pi$-symmetric rectangles $\Pi_{-k}$. 
%They lift to stable tongues $\Tongue_k(\alpha_\pm)\subset \La_\pm$ that meet $\alpha_\pm$ at wedge
%$(\om_k, \tl \om_k)$. By symmetry, we also observe $\pm\pi/2$-symmetric tongues $\Tongue_k'(\alpha_\pm)$.
%Moreover, $\Tongue_0(\alpha_\pm)\cup \Tongue_0'(\alpha_\pm)=\Tongue_\pm$ are the primary stable tongues.
%So, they have angle $2\tl \om_0\in (0, \pi)$ at their tips.  

\begin{cor}
  The stable tongues  $\Tongue_k(\alpha_\pm)$ have angles $0 < \om_k(\alpha_\pm) < \pi$ at their tips.
All other stable tongues $\Tongue_k(\alpha)$,   $\alpha\not=\alpha_\pm$,   have  cusps at their tips.%
\footnote{This is a reason why the tongues do not appear to reach $\TOPphys$ on Figure \ref{FIG:CYLINDER_BASINS}.}
\end{cor}

\begin{proof}
 The angles $\om_k(\alpha_\pm)$ are positive by Propositions \ref{principal tongues} and \ref{inf many tongues}. 
Since $\sum  \om_k(\alpha_+) = \sum \om_k(\alpha_-)\leq \pi$, 
 each of the angles is also strictly smaller than $\pi$.  

By Property (ii) of Proposition \ref{tongues are a.e.},
the other stable tongues $\Tongue_k(\alpha)$ are regular pullbacks of the  $\Tongue_k(\alpha_\pm)$.
Since $\Rphys$ is transversely super-attracting  at $\TOPphys$ (away from $\INDphys_\pm$)
and expanding along $\TOPphys$,  these pullbacks have cusps at the tips.  
\end{proof}

\begin{figure}
\begin{center}
\input{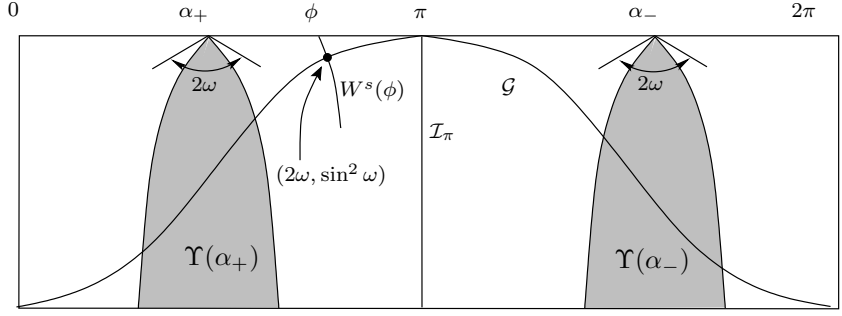}
\end{center}
\caption{\label{FIG:MAXIMAL_ANGLE} The primary stable tongues. }
\end{figure}

%In fact, by taking further preimages of
%$\WW^s(\pm\pi/2)$ we obtain leaves in $\FF(\BOTTOMphys)$ traversing the full height of the
%cylinder, connecting to each of the preimages of $\INDphys_\pm$.

\subsection{Regularity}
\label{SUBSEC:REGULARITY}

\begin{prop}\label{PROP:CINFTY_FOLIATION}
$\FF^s(\BOTTOMphys)$ is a $C^\infty$ foliation of $\WW^s(\BOTTOMphys)$.
\end{prop}

\noindent 
Since the leaves of $\FF^s(\BOTTOMphys)$ are integral curves of the
central line field $\LL^c$, it suffices to prove the following proposition.

\comment{*************
Let $
\begin{eqnarray*}
B_n(x) = \frac{1}{4^n} D\Rphys^n(x)
A(\Rphys^{n-1}x) A(\Rphys^{n-1} x) \cdots A(x) \,\, \text{where} \,\, A(x) = [1/4 \,\, 1] \,  D \Rphys(x)
\end{eqnarray*}
************}

\begin{prop}\label{PROP:LOW_TEMP_CONVERGENCE}
The sequence
$B_n(x):=\frac{1}{4^n} D\Rphys^n (x)$
converges uniformly on compact subsets of $\WW^s(\BOTTOMphys)$ at super-exponential rate to a
$C^\infty$ matrix-valued function $B(x)$.  Moreover, $\LL^c(x) = \ker B(x)$.
\end{prop}

\begin{proof}[Proof of Proposition \ref{PROP:LOW_TEMP_CONVERGENCE}.]
It suffices to prove the statement in any neighborhood of $\BOTTOMphys$,
since one can use the invariance
\begin{eqnarray}\label{EQN:B_INVARIANCE}
4 B_n(x) = B_{n-1}(\Rphys x) D\Rphys(x) \,\,  \mbox{and} \,\,  4 B(x) = B(\Rphys x) D\Rphys (x),
\end{eqnarray}
to extend the result to any compact subset of
$\WW^s(\BOTTOMphys)$.  For example, $\ker B(x) = \Rphys^* \ker B(\Rphys x)$ follows automatically from (\ref{EQN:B_INVARIANCE}) 
at the regular points of $\Rphys$ and is a simple check
near the critical points $\II_{\pm \pi/2}$.

For $x \in \Cphyslow$ we have
\begin{eqnarray}\label{EQN:SUPERCONVERGE_BOTTOM}
t_n \leq C q^{2^n}
\end{eqnarray}
\noindent
with $C > 1$ and $0 < q < 1$.  It is noteworthy that $C \equiv C(\eps)$ and $q \equiv q(\eps)$ can be chosen
uniformly on the region ${\Cphyslow^\eps} := \{x \in \Cphys \,: \, t(x) < t_c - \eps\}$ for any $\eps > 0$.

Note that $A(x):=\frac{1}{4} D\Rphys (x)$
is real-analytic and, by (\ref{EQN:DR}), it satisfies
\begin{eqnarray*}
A(\phi,0) = A_0 := \left[\begin{array}{cc} 1 & 0 \\ 0 & 0 \end{array} \right].
\end{eqnarray*} 
It follows from  (\ref{EQN:SUPERCONVERGE_BOTTOM}) that
\begin{eqnarray}\label{EQN:CONVERGENCE_OF_A}
|A(\Rphys^n x) - A_0| < C_0 q^{2^n}
\end{eqnarray}
\noindent for any $x \in \Cphyslow^\eps$.

By the chain rule
\begin{eqnarray*}
D B_n(x) = A(\Rphys^{n-1}x) A(\Rphys^{n-1} x) \cdots A(x).
\end{eqnarray*}
Moreover, Equation (\ref{EQN:CONVERGENCE_OF_A}) is sufficient for $B_n(x)$ to converge uniformly (and super-exponentially fast) to some
continuous $B(x)$ on $\Cphyslow^\eps$ for any $\eps$. 
It satisfies $B(\phi,0) = A_0$ and also $4 B(x) = B(\Rphys x) D\Rphys (x)$.

\msk

Since $\BOTTOMphys$ is superattracting, there is some forward invariant
neighborhood $\NN$ of $\BOTTOMphys$ so that for any  $x \in \NN$, any $v \in \LL^c(x)$ has
its length contracted under $D\Rphys$ by a definite factor, thus satisfying
$v \in \ker B(x)$.  Moreover, since $B(\phi,0) = A_0$ we can trim this
neighborhood, if necessary, so that ${\rank} B(x) = 1$ and thus $\LL^c(x) =
\ker B(x)$ for all $x \in \NN$.

\msk
We now show that $B(x)$ is $C^\infty$ in $\Cphyslow$.
%\note{Is the convergence actually in the $C^\infty$ topology?}
The proof depends on the superattacting nature of
$\BOTTOMphys$.  Let $\Rphys = (\Rphys_1,\Rphys_2)$.   Then, $\Rphys_2$ vanishes quadratically in $t$ when $t=0$,
giving that for any multi-index $\bm \beta$ there is some
$M_{\bm \beta}$ such that
\begin{eqnarray}\label{EQN:R2_DERIVARIVES}
\left|\partial_{\bm \beta} \Rphys_2\right(\phi,t)| \leq
\begin{cases}
\frac{t^2}{C^2} M_{\bm \beta} & \text{if $\beta_2 = 0$}\\
\frac{t}{C} M_{\bm \beta} & \text{if $\beta_2 = 1$}\\
M_{\bm \beta} & \text{if $\beta_2 \geq 2$}\\
\end{cases}
\end{eqnarray}
\noindent
for all $(\phi,t) \in \Cphyslow$.
Furthermore, since $A(x) - A_0$ vanishes at $t=0$, we have that
\begin{eqnarray}\label{EQN:DERIVE_OF_A}
\left\|\partial_{\bm \beta} A\right\| < C_{\bm \beta} \, t \,\,\,\, \mbox{if} \,\, \beta_2 = 0.
\end{eqnarray}

\msk
We'll first observe that $B(x)$ is $C^1$.  It is a consequence of the following estimates
\begin{eqnarray}
\left\|D\Rphys^n (x) \right\| &\leq& \lambda^n, \label{EQN:DXN}\\
\left|\partial_x t_n \right| &\leq&  \mu^n q^{2^n}, \mbox{and} \label{EQN:DTN}\\
\left\| \partial_x A(\Rphys^n x) \right\| &\leq& C_1 \nu^n q^{2^n}. \label{EQN:DAN}
\end{eqnarray}
\noindent
for appropriate $\lambda, \mu, \nu,$ and $C_1$.

\comment{******************
\begin{eqnarray*}
D\Rphys^n (x) = D\Rphys(x_{n-1}) D\Rphys(x_{n-2})\cdots D\Rphys(x).
\end{eqnarray*}
******************}

The first follows a bound $\|D\Rphys(x)\| \leq \lambda$ on $\Cphyslow$ and the chain rule.
Meanwhile (\ref{EQN:DTN}) follows from induction on $n$, since
the chain rule and (\ref{EQN:R2_DERIVARIVES}) give
\begin{eqnarray}\label{EQN:SIMPLE_CASE}
\partial_x t_n &=& \partial_\phi \Rphys_2(\phi_{n-1},t_{n-1})  \partial_x
\phi_{n-1} + \partial_t \Rphys_2(\phi_{n-1},t_{n-1})\partial_x t_{n-1} \\
&\leq& \hspace{0.5in}    \frac{t_{n-1}^2}{C^2}M_\phi  \cdot  \partial_x
\phi_{n-1} \hspace{0.25in} + \hspace{0.25in} \frac{t_{n-1}}{C}M_t \cdot \partial_x
t_{n-1}.\notag
\end{eqnarray}
\noindent
Equation (\ref{EQN:DAN}) follows from similar use of the chain rule, together
with (\ref{EQN:DERIVE_OF_A}), (\ref{EQN:DTN}), and (\ref{EQN:DXN}).

Then (\ref{EQN:DAN}) is sufficient for the series
\begin{align}
\partial_x B(x) = \partial_x A(x) A(\Rphys x) A(\Rphys^2 x)\cdots + A(x) \partial_x A(\Rphys x) A(\Rphys^2 x)\cdots + \cdots
\end{align}
\noindent
to converge uniformly on $\Cphyslow^\eps$ for any $\eps > 0$.

\msk
To prove the convergence of higher derivatives of $B(x)$ we will use
\begin{lem}
For $x \in \Cphyslow^\eps$ and any multi-index ${\bm \alpha} = (\alpha_1, \alpha_2) \neq 0$ we have
\begin{eqnarray}
\left\|\partial_{\bm \alpha} \Rphys^n(x) \right\| &<& \lambda_{\bm \alpha}^n, \label{EQN:PARTIALS1} \\
\left|\partial_{\bm \alpha} t_n\right| &<& \mu_{\bm \alpha}^n q^{2^n}, \,\, \mbox{and} \label{EQN:PARTIALS2} \\
\left\|\partial_{\bm \alpha} A(\Rphys^n x) \right\| &<& C_{\bm \alpha} \nu_{\bm \alpha}^n q^{2^n}, \label{EQN:PARTIALS3}
\end{eqnarray}
\noindent
for suitable $\lambda_{\bm \alpha}, \mu_{\bm \alpha}, \nu_{\bm \alpha} > 1$ and $C_{\bm \alpha} > 0$.
\end{lem}

\begin{proof}
The proof is similar to that for (\ref{EQN:DXN}-\ref{EQN:DAN}) except that in
place of the chain rule we will use the
the Fa\`a di Bruno formula \cite{FA} to estimate the higher partial derivatives of a composition.

If
\begin{eqnarray*}
h(x_1,\ldots,x_d) =
f(g^{(1)}(x_1,\ldots,x_d),\ldots,g^{(m)}(x_1,\ldots,x_d))
\end{eqnarray*}
\noindent and $|{\bm \alpha}| := \sum \alpha_i$, it gives:
\begin{eqnarray}\label{EQN:FDB}
\partial_{\bm \alpha} h = \sum_{1 \leq |\bm \beta| \leq |\bm \alpha|} \partial_{\bm \beta} f \sum_{s=1}^{|\bm \alpha|} \sum_{p_s({\bm \alpha},{\bm \beta})} {\bm \alpha}! \prod_{j=1}^s \frac{(\partial_{{\bm l}_j}{\bm g})^{{\bm k}_j}}{({\bm k}_j !)({\bm l}_j !)^{|{\bm k}_j|}}.
\end{eqnarray}
\noindent
Each ${\bm \beta}$ and ${\bm k}_j$ is a $n$-dimensional multi-index and
each ${\bm l}_j$ is a $d$-dimensional multi-index.  The final sum is taken
over the set
\begin{eqnarray*}
p_s({\bm \alpha},{\bm \beta}) = \{({\bm k}_1,\ldots,{\bm k}_s;{\bm l}_1,\ldots,{\bm l}_s) \, : |{\bm k}_i| > 0, \\
{\bm 0} \prec {\bm l}_1 \prec \cdots \prec {\bm l}_s, \, \sum_{i=1}^s {\bm k}_i
= \bm \beta \, \mbox{and} \, \sum_{i=1}^s |{\bm k}_i| {\bm l}_i = \bm \alpha.\}
\end{eqnarray*}
\noindent
Here, $\prec$ denotes a linear order on the multi-indices (its details will not be important for us), ${\bm \eta}! := \prod \eta_i !$, and for a vector ${\bm z}$ we have ${\bm z}^{\bm \eta} = \prod z_i^{\eta_i}$.

\msk
We will not need the precise combinatorial details of this formula.  For example,
(\ref{EQN:PARTIALS1}) follows directly from the existence of a polynomial
expression for $\partial_{\bm \alpha} h$ in the partial derivatives of $f$ and
$g$.

Suppose that both $f$ and $g$ are functions of two variables.
Then, all that we will need to prove (\ref{EQN:PARTIALS2}) and (\ref{EQN:PARTIALS3})
is that the Fa\`a di Bruno formula gives an expression of the form
\begin{eqnarray}\label{EQN:FDB2}
\partial_{\bm \alpha} h = \sum_{i} K_i \, \partial_{{\bm \beta}_i} f
 \prod_{j = 1}^{\beta_i^1} \partial_{{\bm \gamma}_{i,j}} g^{(1)} \, \prod_{j = 1}^{\beta_i^2} \partial_{{\bm \eta}_{i,j}} g^{(2)},
\end{eqnarray}
having two additional properties:
\begin{enumerate}
\item $1 \leq {\bm \beta}_i,\ {\bm \gamma}_{i,j},\ {\bm \eta}_{i,j} \leq {\bm \alpha}$ for every $i,j$ and,
if either ${\bm \gamma}_{i,j} = {\bm \alpha}$ or ${\bm \eta}_{i,j} = {\bm \alpha}$, then
$|{\bm \beta}_i| = 1$; and
\item $K_i \geq 1$ for every $i$.
\end{enumerate}
\noindent
Here, each ${\bm \beta_i} = (\beta_i^1,\beta_i^2)$.

The proof of (\ref{EQN:PARTIALS2}) is done by an inductive use (\ref{EQN:FDB2})
similar to the usage of the chain rule for the first derivatives
(\ref{EQN:SIMPLE_CASE}).  It is the key step, so we'll prove it here and omit a
proof of (\ref{EQN:PARTIALS3}), which is simpler.

\msk

We already have (\ref{EQN:PARTIALS2}) when $|\bm \alpha| = 1$.  Therefore, can
suppose that it holds for all ${\bm \beta}$ satisfying
$|{\bm \beta}| < |\bm \alpha|$ in order to prove it for $\bm \alpha$.
Equation (\ref{EQN:FDB2})
gives
\begin{eqnarray}\label{EQN:FIRST_APPLICATION}
\partial_{\bm \alpha} t_n = \sum_{i} K_i \, \partial_{{\bm \beta}_i} \Rphys_2(x_{n-1})
 \prod_{j = 1}^{\beta_i^1} \partial_{{\bm \gamma}_{i,j}} \phi_{n-1} \, \prod_{j = 1}^{\beta_i^2} \partial_{{\bm \eta}_{i,j}} t_{n-1}.
\end{eqnarray}

\noindent
Let $\mu_{\bm \alpha}:= \max(N_{\bm \alpha},P_{\bm \alpha})$, where
\begin{eqnarray*}
N_{\bm \alpha} :=
\frac{1}{q^2}\sum_{i} K_i M_{{\bm \beta}_i}
\end{eqnarray*}
\noindent
and $P_{\bm \alpha}$ is the maximum of
\begin{eqnarray*}
\prod_{j=1}^{\beta_i^1} \lambda_{{\bm \gamma}_{i,j}} \prod_{j=1}^{\beta_i^2} \mu_{{\bm \eta}_{i,j}}
\end{eqnarray*}
\noindent
taken over all $i$ in (\ref{EQN:FIRST_APPLICATION}).

Equation (\ref{EQN:PARTIALS2}) follows for $n=1$, since
$\mu_{\bm \alpha} \geq N_{\bm \alpha} \geq \frac{1}{q^2} M_{\bm \alpha} \geq \frac{1}{q^2}
\left|\partial_{\bm \alpha} t_1\right|$.

We now suppose that (\ref{EQN:PARTIALS2}) is true for the
$(n-1)$-st iterate in order to prove it for the $n$-th iterate.
By the definition of $\mu_{\bm \alpha}$, it suffices to show that each term in (\ref{EQN:FIRST_APPLICATION}) has absolute value bounded by
\begin{eqnarray}\label{EQN:BASIC_INEQUALITY}
 q^{2^n} K_i M_{{\bm \beta}_i}  \left(
\prod_{j=1}^{\beta_i^1} \lambda_{{\bm \gamma}_{i,j}} \prod_{j=1}^{\beta_i^2} \mu_{{\bm \eta}_{i,j}}\right)^{n-1}.
\end{eqnarray}
\noindent
If $\beta_i^2 \geq 2$, then the factor of $q^{2^n}$ results from
at least two factors of $|\partial_{{\bm \eta}_{i,j}} t_{n-1}|$ in the $i$-th term from
(\ref{EQN:FIRST_APPLICATION}) and the induction hypothesis.  Otherwise, sufficient extra factors of
$q^{2^{n-1}}$ come from from $t_{n-1} < C q^{2^{n-1}}$ and
(\ref{EQN:R2_DERIVARIVES}).
\end{proof}

Using a generalization of the product rule, we have
\begin{align}\label{EQN:PARTIALS}
\partial_{\bm \alpha} B(x) =
\sum_{{\bm \alpha} = {\bm \alpha}_0 + {\bm \alpha}_1 + \cdots}
\frac{{\bm \alpha}!}{{\bm \alpha}_0 !  {\bm \alpha}_1 ! {\bm \alpha}_2 ! \cdots}
\partial_{{\bm \alpha}_0} A(x)  \cdot
\partial_{{\bm \alpha}_1} A(\Rphys x) \cdot
\partial_{{\bm \alpha}_2} A(\Rphys^2 x) \cdots
\end{align}
\noindent
where the sum is taken over all partitions of $\bm \alpha$ into a sum ${\bm \alpha}_0 + {\bm \alpha}_1 + \cdots$.
We apply (\ref{EQN:PARTIALS3}) to each $|{\bm \alpha}_i| \leq |\bm \alpha|$ and take appropriate maxima to find
\begin{eqnarray*}
\left|\partial_{{\bm \alpha}_n} A(\Rphys^n x) \right| &<& D_1 \nu^n q^{2^
n} \leq D_2
\end{eqnarray*}
\noindent
for each $n$ and suitable $\nu > 1, D_1 > 0,$ and $D_2 > 0$.

There are no more than $n^{|\bm \alpha|}$ partitions ${\bm \alpha} = {\bm \alpha}_0
+ {\bm \alpha}_1 + \cdots$ for which $n$ is the maximal index with $|{\bm
\alpha}_i| > 0$.  Each such term in the sum (\ref{EQN:PARTIALS}) can be bounded
by $K \nu^n q^{2^n}$, since there are at most $|\bm \alpha|$ terms in the
product for which ${\bm \alpha}_i \neq {\bm 0}$, each of which is bounded by $D_2$,
and the last one is bounded by $D_1 \nu^n q^{2^ n}$.  Thus, we bound
the sum (\ref{EQN:PARTIALS}) by
\begin{eqnarray*}
\sum_{n}  n^{|\bm \alpha|}K \nu^n q^{2^n},
\end{eqnarray*}
\noindent
which is convergent.

Thus the series (\ref{EQN:PARTIALS}) for $\partial_{\bm \alpha}
B(x)$ converges uniformly on $\Cphyslow^\eps$ for
any $\eps > 0$.  Since the multi-index $\bm \alpha$ was arbitrary, we conclude
that $B(x)$ is $C^\infty$ on $\Cphyslow$.
\end{proof}

\subsection{``B\"ottcher coordinate'' on $\WW^s(\BB)$  and convergence
of foliations}\label{SUBSEC:BOTTCHER}

Since the map $\RR: \CC\ra \CC $ has degree 4 in the first homology of $\CC_1$,  
the normalized  pullback
$$
    \phi_1: \CC_1\ra \T, \quad \phi_1 = \frac 14 \phi\circ \RR: \CC_1\ra \T
$$
is a  well defined map of degree 1.  Iterating the pullback, we obtain a sequence
of degree 1 maps
$$
   \phi_n : \CC_1\ra \T, \quad \phi_n = \frac {1}{4^n} \phi\circ \RR^n : \CC_1\ra \T.
$$
By Proposition \ref{SUBSEC:REGULARITY}, 
$$
   d\phi_n(x) = d\phi(R^n x)\cdot  B_n(x)\to (1,0) \cdot B(x):= \om (x),   \quad x\in \BB, 
$$
where $\om$ is a  closed $C^\infty$-smooth 1-form on $\WW^s(\BB)$ with period 1.
Hence $\om = d\Phi$ where $\Phi: \WW(\BB) \ra \T$ is a $C^\infty$ map  of degree 1.
Moreover,
\begin{equation}\label{Bottcher equation}
  \Phi|\, \Bottom \equiv  \phi, \quad \mbox{and}\quad  \Phi(\RR x) = 4 \Phi(x).
\end{equation}
Since $d\Phi$ vanishes on the  stable leaves $\WW^s(x)$ ($x\in \BB$)
and does not vanish  transversally,
it is a defining function for the stable foliation. 

The function $\Phi$ plays the role of the {\it B\"ottcher coordinate} on the basin of the bottom.

\msk
Instead of the angular coordinate $\phi$, we can do the same
construction with a more general function:

\begin{lem}\label{convergence of foliations}
Let $\psi: \CC_1\ra \T$ be a  $C^\infty$-map 
of degree $l$ tangent to $l \phi$ at the bottom (i.e., $\psi(\phi, t) = l\phi + o(t)$
as $t\to 0$). Then 
$$
    \psi_n := \frac {1}{l\cdot4^{n}} \psi (\RR^n x) \to \Phi(x) \quad \mbox{super-exponentially fast as}\ n\to \infty,
$$ 
in the $C^1$-topology on compact subsets of $\WW^s(\BB)$.
\end{lem}

\begin{proof}
   The same reason as above shows that $d\psi_n\to d\Phi$  super-exponentially fast
in the $C^1$-topology on compact subsets of $\WW^s(\BB)$. 
The assertion follows, 
since the $\psi_n$  agree with $\phi$  on $\BB$. 
\end{proof}

If $\psi$ is a defining function for some foliation $\FF$ on $\WW^s(\BB)$, 
then the $\psi_n$ are defining functions for the pullbacks $(\RR^n)^*(\FF)$,
and Lemma \ref{convergence of foliations} gives a strong sense in
which  these pullbacks converge to the stable foliation  of the bottom  $\FF^s(\BOTTOMphys)$.

\msk
Let us finally mention that the B\"ottcher coordinate $\Phi$ can be extended
continuously to a degree 1 map  $\tl \Phi: \CC_1\ra \T$ satisfying
B\"ottcher  functional equation (\ref{Bottcher equation}), 
see  Remark \ref{REM:EXTENSION_OF_BOTTCHER} below.

\section{High temperature dynamics: basin of the top of the cylinder}
\label{SEC:HIGH_TEMP}

Property (P5) from \S \ref{str on CC}  states that
the top $\TOPphys$ of the cylinder $\Cphys$ is non-uniformly superattracting. 
In this section we will prove that there is set of positive measure attracted to $\TOPphys$,
$$
   W^s(\TOPphys)= \{x\in \Cphys: \ \Rphys^n x\to \TOPphys\},
$$
that supports a ``stable bouquet'' $\FF^s(\TOPphys)$ consisting of
curves emanating from almost all points of $\TOPphys$.

\ssk

Near the top, we will make use of the local coordinate $\tau$,
and  near the indeterminacy points -- of the local coordinates  $(\tau, \eps)$, see (P5).
We say that ``$x$ is $\tau$-{\it below} $y$'' if $\tau(x)< \tau(y)$
(so, $x$ is, in fact, above $y$ on the cylinder $\Cphys$).

Recall the neighborhoods $\VV' \equiv \VV'_{\bar \tau,\eta}$ of $\TOPphys \sm
\{\INDphys_\pm\}$ obtained by removing the parabolic regions $\PP_\eta^\pm$
from the $\VV_{\bar \tau}$ neighborhood of the top (see \S \ref{SUBSEC:MODIFIED_ALG_CONES}.)
Let $q\in (0,1)$.
By property (P5), if $\eta$ and $\bar \tau$ are sufficiently small then
\begin{equation}\label{la}
    \tau(\Rphys x) <  q \, \tau (x) \quad \forall\ x\in \VV'.
\end{equation}

Let $\WW^s_{\eta,\bar \tau}(\TOPphys)$ be the set of points whose orbits
converge to $\TOPphys$ while remaining in $\VV'_{\bar \tau,\eta}$ and let
$\WW^s_{\eta}(\TOPphys)$ be the set of points whose orbits eventually land in
$\VV'_{\bar\tau, \eta}$ for some $\bar \tau$ and stay there (note that this
property is independent of $\bar\tau$).   Then, points of
$\WW^s_{\eta}(\TOPphys)$ are attracted to $\TOPphys$  with exponential rate
$O(q^n)$.  We will show below that this set supports a ``stable bouquet'' $\FF^s_\eta(\TOPphys)$ consisting of
curves emanating from almost all points of $\TOPphys$ and that it has positive two-dimensional Lebesgue measure.

\subsection{Vertical bouquet $\FF^s(\TOPphys)$.}
\label{SUBSEC:FOLIATION_NEAR_T}

A {\it bouquet} of curves in $\Cphys$ is a family of curves that are disjoint on $\Cphystl$
(which may or may not be a lamination). \note{Ok?}
The curves comprising the bouquet are called its {\it leaves}.

In the following proposition,  horizontal and vertical curves $\gamma$ are
understood in the sense of the cone fields $\KK^h$ and  $\KK^v$.  Vertical
curves are oriented by the local coordinate $\tau$.  Note that for sufficiently small $\eta$, the boundary of
$\VV'$ is horizontal because the tangent lines to the parabolas $\YY^\pm_\eta$
have slope $2\eta \epsilon$ which can be made less than $\epsilon/3$.

\begin{prop}\label{PROP:ETA_BASINS_LAMINATED}
For any sufficiently small $\eta> 0$,
the basin  $W^s_{\eta}(\TOPphys)$
supports an invariant bouquet $\FF^s(\TOPphys) \equiv \FF^s_\eta(\TOPphys)$
by smooth vertical paths landing transversely at almost all points of $\TOPphys$.
%%Given any
%%$l>0$, sufficiently small, there is a positive measure set $\TOPphys_l \subset \TOPphys$ having an
%%invariant vertical stable foliation $\FF^s(\TOPphys_l) = \bigcup_{\phi \in \TOPphys_l}
%$W^s(\phi)$ with each of the leaves of vertical length greater than $l$.
\end{prop}

\begin{proof}
% If $x = (\phi,t) \in \Cphys$.
Let $q\in (0,1)$.
% Recall   $\eps= \min|\phi \pm \pi/2|$ % $\tau=1-t$.
% $\tau_n=\tau \circ \Rphys^n$ and
% Let $x\in \VV^\eta_{\bar\tau}$.
%
A vertical curve $\gamma$  is called {\it semi-proper} if it begins on $\TOPphys$.
If additionally, $\gamma$ lands on the $\tau$-upper boundary of $\VV'$, it is  called {\it proper}.

 Let
$$
   \VV'_n = \{x:\ \Rphys^k x\in \VV', \ k=0,\dots, n-1\}.
$$
Let $\gamma^0\equiv \gamma^0_x \subset \VV'$ be the proper genuinely vertical interval
containing  $x\in \VV'\equiv \VV'_{\bar\tau,\eta}$.
For each $n>0$ and $x\in \VV'_n$,
we will inductively construct a semi-proper vertical curve
$\gamma^n=\gamma^n_x \subset \VV'_n$ containing $x\in \VV'_n$.
Assume we have already constructed curves $\gamma^{n-1}_y$ for all $y\in \VV'_{n-1}$.
Then for $x\in \VV'_n$,  we define $\gamma^n_x$ as the regular lift of $\gamma^{n-1}_{\Rphys x}$
truncated (if needed) by the $\tau$-upper boundary of $\VV'$.
Since the cone field $\KK^v$ is backward invariant,
we obtain a semi-proper vertical curve. Since the boundary of $\VV'$ is horizontal
and $x\in \VV'$,
the whole  curve $\gamma_x^n$ is contained in $\VV'$.
Since $\gamma^{n-1}_{\Rphys x}\subset \VV'_{n-1}$, we conclude that $\gamma_x^n\subset \VV'_n$.

Let now $x\in \WW^s_{\eta,\bar\tau}=\bigcap \VV'_n$, so that the curves $\gamma^n_x$ exist for all $n\in \N$.
Since each of them contains $x$, they all have a definite height $\tau_0 \geq  \tau(x)$.
Since $\Rphys$ is horizontally expanding, 
the curves $\gamma^n_x$  exponentially converge in the uniform topology
to a curve $\gamma_x$ containing $x$ and of height $\geq \tau_0$.
Moreover, being vertical, the curves $\gamma^n_x$ are uniformly Lipschitz,
so $\gamma_x$ is Lipschitz as well.

The results of \S \ref{SUBSEC:DOMINATED} gives
that the intersections $\cap D\Rphys^{-n}(\KK^v(\Rphys^n x))$ converge
geometrically to the central line field $\LL^c(x)$.
Thus, $\gamma_x$ is tangent to $\LL^c(x)$ at $x$.
Similarly, at any $y \in \gamma_x$ we have that $\gamma_x$ is tangent to
$\LL^c(y)$.  Since the $\LL^c$ is a continuous line field, the entire curve
$\gamma_x$ is $C^1$.

It lands at $\TOPphys$ transversely since the vertical cone field $\KK^v$ is
non-degenerate on $\TOPphys\sm \{\alpha_\pm\}$.
(Note that the curves $\gamma_x$ do not land at $\alpha_\pm$ since they are vertical
while $\di \VV'$ is horizontal.)

So, the basin  $\WW^s_{\eta, \bar\tau}$ supports an invariant bouquet  by smooth
vertical paths landing transversely at $\TOPphys$.
Pulling it back by the dynamics, we obtain a similar bouquet supported on the whole basin
$\WW^s_\eta$.
To complete the proof we need to show that the leaves of this bouquet land at almost all points
of $\TOPphys$.

Let $\eps_n(\phi)= \eps(\Rphys^n(\phi))$ (where $\phi$ is considered to be a point on $\TOPphys$).
Since the doubling map  $\Rphys: \phi \mapsto 2\phi$ preserves the Lebesgue measure on $\TOPphys$,
the Borel-Cantelli Lemma implies that for a.e. $\phi\in \TOPphys$,
eventually we have: $\eps_n(\phi)> q^{n/4}$. Hence for a.e. $\phi\in \TOPphys$, there exists $c=c(x)>0$
such that
\begin{equation}\label{slow recurrence}
    \eps_n(\phi)> c\, q^{n/4}, \quad n\in \N.
\end{equation}

Let $h(\phi) = \min\{\bar \tau, \eta \eps^2/2\}$,
so that any vertical curve $\gamma$ based at $\phi\in \TOPphys$
of $\tau$-height   $\leq h(\phi)$  is necessarily contained in $\VV'$.
(To check it, note that a vertical curve that begins at $\eps_0$-distance from $\alpha_+$
goes $\tau$-below the parabola $\tau= (\eps_0^2 - \eps^2)/4$,
and hence reaches the boundary parabola $\YY_+= \{\tau=\eta\eps^2\}$ at least at $\tau$-height
$\eta\eps_0^2/(1+4\eta)$.)

Let us now slightly modify the above construction of semi-proper vertical curves.
Let $\gamma_\phi^0 \subset \VV'$ be the proper vertical (straight) interval  based at $\phi\in \TOPphys$.
For each $n>0$ we will inductively construct a family of semi-proper vertical curves
$\gamma^n_\phi \subset \VV'$ based at $\phi\in \TOPphys$.
Assume we have already constructed curves $\gamma^{n-1}_\phi$.
Then we define $\gamma^n_\phi$ as the regular lift of $\gamma^{n-1}_{2\phi}$
truncated (if needed) by the $\tau$-upper boundary of $\VV'$.
Since the cone field $\KK^v$ is backward invariant, we obtain a family of semi-proper vertical curves.

We will now show that if $\phi$ satisfies (\ref{slow recurrence})
then the curves $\gamma^n=\gamma^n_\phi$ have a definite height
(depending on $\phi$ but independent of $n$).
Indeed, by construction, one of the curves $\Rphys^k(\gamma^n)$, $k=0,1,\dots, n$, is proper.
But the height of $\Rphys^n(\gamma^n)$ is bounded by
$q^n \bar \tau$ which is eventually smaller than
$$
   \frac \eta {2}  c^2 q^{n/2} \leq \min\{\, \frac \eta{2} \eps_n^2, \, \bar\tau\} = h(\phi_n).
$$
Hence there is $k_0=k_0(\phi)$ (independent of $n$)
 such that all the curves $\Rphys^k \gamma^n$ are not proper
for $k>k_0$. It follows that one of the curves $\Rphys^k \gamma^n$, $k=0,1,\dots, k_0$,
is proper, and hence, it has a definite height.
Then the same is true for the curve $\gamma^n$.  This  completes the proof.
\end{proof}

\subsection{Positive measure of $\WW^s(\TOPphys)$.}\label{SUBSEC:BASIN_T}

Let 
\begin{eqnarray*}
\WW^{s,o}_\eta(\TOPphys) := \{x \in \WW^s_\eta(\TOPphys) \, : \, \mbox{the curve $\gamma_x \in \FF^s_\eta$ containing $x$ extends beyond $x$} \}.
\end{eqnarray*}
It is a completely invariant subset of $\WW^s_\eta(\TOPphys)$ consisting of points
$x \in \WW^s_\eta$ whose orbits $\Rphys^n x$ converge to $\TOPphys$ within the interiors of the corresponding leaves
$\gamma_{\Rphys^n x}$.

Let 
$$ 
     \pi: \Cphys \to \BOTTOMphys=\T,  \quad (\phi, t)\mapsto \phi 
$$ 
be the natural projection onto the bottom of the cylinder. 

Given a horizontal curve $\xi$, let  $dl^h\equiv dl_\xi^h=\pi^* (d\phi)$ stand for the
{\it horizontal length} on it, i.e., 
the {\it pullback} of the standard Lebesgue measure on $\T$ to $\xi$ under $\pi$.
So, if $\xi$ projects invectively onto the horizontal axis
(or equivalently, if $l^h(\xi)\leq 2\pi$)  then 
  $l^h(X)= |\pi(X)|$ for any measurable set $X\subset \xi$
(where $|Y|$ stands for the Lebesgue measure of $Y\subset \T$). 
In general, 
\begin{equation}\label{pullback meas}
  l^h(X) =\int |(\pi| X)^{-1}(\phi)| d\phi\leq \deg(\pi|\, X) \, |\pi(X)|, 
\end{equation}
where $\deg (\pi| X) =\max_{\phi\in \T}  |(\pi|X)^{-1}(\phi)| $.

% Since $\Rphys$ is horizontally expanding
% \footnote{A similar statement holds in all of $\Cphystl$, as a consequence of Proposition \ref{PROP:EXPANDING}, but here
% we prefer to rely on the much simpler Lemma \ref{LEM:EXPANSION_IN_VPRIME}.}
% on $\VV'$ (Lemma \ref{LEM:EXPANSION_IN_VPRIME}),
%there is a constant $\lambda > 1$

By Theorem \ref{hor expansion thm}, 
$\Rphys$ expands the horizontal length: there exists $\la>1$ and $c > 0$ such that  
for any  horizontal curve $\xi \subset \VV'$ and any measurable set $X\subset \xi$, we have:
\begin{equation}\label{EQN:STRETCHING_IN_UPSILON}
  l^h(\Rphys^n X) \geq c \lambda^n \,  l^h(X).
\end{equation}

\begin{rem}
  In fact,  $\la=2$ but we will keep notation ``$\la$'' to distinguish it from
the combinatorial appearance of ``2''.  Note also that in region $\VV'$ (which
we are concerned with in this section)  expanding property
(\ref{EQN:STRETCHING_IN_UPSILON}) can be easily derived from Lemma
\ref{LEM:STRETCHING_NEAR_TOP}. 
\end{rem}

\begin{lem}\label{LEM:POS_MEASURE_AH_CURVE}
For any $\eta \in (0, 1/2)$ and $\de>0$, there exists a threshold $\bar\tau>0$ with the following property.  
Let  $\xi$ be a horizontal curve in the strip  $\VV_{\bar\tau}$ with $l^h(\xi) < 2\pi$. %  with $l^h(\xi)>\underline{l}$.
Then all points of $\xi$, except for a set of horizontal length $<\de$,
belong to $\WW^{s,o}_\eta(\TOPphys)$. 
%These points converge  to $\TOPphys$ at exponential rate $O(q^n)$, where $q=q(\eta) \to 0$ as $\eta\to 0$.
%Moreover, through each such point $x$, the leaf $\gamma_x \in \FF^s_\eta(\TOPphys)$ extends beyond $x$ by some definite amount.
\end{lem}

\begin{proof}[Proof of Lemma \ref{LEM:POS_MEASURE_AH_CURVE}:]
Making $\eta$ and $\bar \tau$ sufficiently small, we can assume if $\gamma
\subset \VV'_{2\bar{\tau},\eta}$ is any vertical curve then the vertical
length\footnote{In general, $\Rphys \gamma$ need not be vertical.  In all cases, the vertical length of $\Rphys \gamma$ can be considered as the total length its projection onto the vertical interval $\II_0$.} of $\Rphys \gamma$ is at least a factor of $q < 1/4$ smaller than the
vertical length of $\gamma$.  This slightly stronger condition than (\ref{la})
can be obtained using (\ref{EQN:A_EXPANSION}) and (\ref{DR near alpha}).

Let
\begin{align}\label{EQN:TOP_REMOVALS}
   \xi_n = \left\{ x\in \xi: \ |\eps(\Rphys^k x)| \geq 2 \sqrt{\frac{\bar
\tau}{\eta}}  \cdot 2^{-k} \mbox{ for } 0 \leq k \leq  n-1 \mbox{ and } \ |\eps(\Rphys^n x)| < 2 
\sqrt{\frac{\bar \tau}{\eta}}  \cdot 2^{-n}\right\},
\end{align}
and let $X_n = \Rphys^n(\xi_n)$.  
Note that the sets $\xi_n$ are pairwise disjoint.

Using (\ref{la}), one can inductively show that $\xi \sm \cup \xi_n \subset \WW^s_\eta(\TOPphys)$.
Let us now estimate $l^h(\cup \xi_n)$.
By construction
$$
     | \pi(X_n)| \leq 8 \sqrt{\bar\tau/ \eta} \ 2^{-n}.
$$
Making use of (\ref{pullback meas}), we obtain:
$$
     l^h(X_n) \leq 8 \sqrt{\bar\tau/ \eta}\, \deg (\pi|\, X_n) \, 2^{-n}.
$$
Together with (\ref{EQN:STRETCHING_IN_UPSILON}), this implies 
% that each arc $\xi_n(x)$ has horizontal length bounded by 
$$
            l^h(\xi_n) \leq 8 \sqrt{\bar\tau/ \eta}  \,  \deg (\pi|\, X_n)\,  2^{-n}\,  c^{-1} \lambda^{-n}.
$$
We will show that
\begin{equation}\label{bound on deg}
     \deg (\pi|\, X_n)\leq 2^n.
\end{equation}
Indeed, in this case  $ l^h(\xi_n(x)) \leq 8  \sqrt{\bar\tau/ \eta}  \, c^{-1} \lambda^{-n}$,
hence
\begin{equation}\label{total length}
  l^h(\cup \xi_n) =  \sum_{n=0}^\infty l^h(\xi_n(x))  \leq 8 \sqrt{\frac{\bar\tau} {\eta}}\, \frac{\la}{c(\la-1)},
\end{equation}
% and the total horizontal measure of all these sets is bounded by
which can be made arbitrarily small if $\bar \tau$ is selected small enough.

\ssk

Let us prove (\ref{bound on deg}).
For $x\in \xi_n$, let $x_k = \Rphys^k x$, and let $\phi=\pi(x_n)$.
Let $\gamma_n \equiv \gamma_n(x)\subset \II_\phi$ 
be the genuine vertical interval connecting $x_n$ to $\TOPphys$,
and let $\gamma_k\equiv \gamma_k(x)$ be
its lifts that connect $x_k$ to $\TOPphys$. Since $x_k \in \VV'$ for $k<n$, 
and the region $\VV'$ lies above the tongues $\La_\pm$ 
(defined by Property P6 and Figure \ref{FIG:R_ON_CYLINDER}),
the lifts $\gamma_k$ are regular, 
i.e., they lie in the regular lifts $I_k^j$ of the interval $\II_\phi$
(where the lift $I_k^j$ terminates at the point $(\phi+ 2\pi j)/2^k$ of $\TOPphys$, $j=0,1,\dots, 2^k-1$).
But each lift $I_k^j$ is  vertical since the vertical cone field $\KK^v$ is backward invariant.
Hence each $I_k^j$ crosses the horizontal curve $\xi$ at most once. 
Hence,  given $\phi=\pi(x_n)$, there are at most  $2^n$ points $x\in \xi_n$ such that
$\pi(x_n)=\phi$, and (\ref{bound on deg}) follows.   

\ssk
It remains to prove that $\xi \sm \cup \xi_n$ is actually a subset of $\WW^{s,o}_\eta(\TOPphys)$.
We will show that 
the leaf $\gamma_x \in \FF^s_\eta(\TOPphys)$ through any $x \in \xi \sm \cup \xi_n$ extends beyond $x$ by a definite amount.
It suffices to verify that this holds for each of the curves $\gamma_x^n$ used in the proof of
Proposition \ref{PROP:ETA_BASINS_LAMINATED} to construct $\gamma_x$.

There is a constant $K > 0$ so that 
if $\tau(x) \leq \bar \tau$ and 
\begin{eqnarray*}
|\eps(x)| \geq 2 \sqrt{\frac{\tau(x)}{\eta}},
\end{eqnarray*}
then any proper vertical curve $\gamma$ in $\VV'_{2\bar \tau,\eta}$ containing $x$
extends beyond $x$ by at least $K|\eps(x)|^2$.

To see it, let $\tl x$ be the point where $\gamma$ reaches the $\tau$-upper
boundary of $\VV'_{2\bar \tau,\eta}$.  If $\tau(\tl x) = 2\bar \tau$, then we
are done.  So, we can suppose that $x = (\eps,\tau)$ and $\tl x = (\tl \eps,
\tl \tau)$ are near $\INDphys_+$ and $\eps,\tl \eps > 0$.  If $\tl \eps \geq
3/4 \eps$, then $\tl \tau \geq (9/16) \eta \eps^2$ so that $\tl \tau - \tau \geq
(5/16) \eta \eps^2$.  Otherwise, $\tl \eps \leq 3/4 \eps$.  Since the vertical
cones $\KK^v(x)$ have slope $d \tau / d\eps \geq \eps(x)/3$ this forces that
$\tl \tau - \tau \geq (1/16) \eps^2$.

For any $x \in \xi \sm \cup \xi_n$, consider
the orbit $x_i = \Rphys^i x$.  By (\ref{EQN:TOP_REMOVALS}), one finds that
\begin{eqnarray*}
|\eps(x_i)| \geq 2 \sqrt{\frac{\bar \tau}{\eta}}  \cdot 2^{-i} \geq 2 \sqrt{\frac{\tau(x_i)}{\eta}},
\end{eqnarray*}
so that any proper vertical curve in $\VV'_{2\bar \tau,\eta}$ through $x_i$ extends beyond $x_i$ by at least 
$K|\eps(x_i)|^2 \geq 4^{-i} K \bar \tau/\eta$.

We now inductively prove that for every $n$ the curve $\gamma_{x_i}^n$ extends
beyond $x_i$ by at least $4^{-i} K \bar \tau/\eta$.  This holds when $n=0$,
since $\gamma_{x_i}^0$ is the proper genuinely vertical interval in
$\VV'_{2\bar \tau,\eta}$ through $x_i$.
Now suppose that it is true at step $n$ in order to check for step $n+1$.  Suppose that
some $\gamma_{x_i}^{n+1}$ extends beyond $x_i$ by less than $4^{-i} K \bar
\tau/\eta$.  It implies that $\gamma_{x_i}^{n+1}$ is not proper, hence 
$\Rphys: \gamma_{x_i}^{n+1} \ra \gamma_{x_{i+1}}^n$ is a homeomorphism.   Since
$\Rphys$ contracts the vertical lengths of vertical curves in $\VV'_{2\bar
\tau,\eta}$ by at least $q < 1/4$, this would then imply that  $\gamma_{x_{i+1}}^n$
extends beyond $x_{i+1}$ by less than $4^{-(i+1)} K \bar \tau/\eta$, contradicting the induction hypothesis.

In particular, each of the curves $\gamma^n_x \equiv \gamma^n_{x_0}$ extends
beyond $x$ by at least $K \bar \tau/\eta$, implying that $\gamma_x$ does as
well.

\end{proof}

\begin{rem}
In the above proof we could use the more obvious degree bound by $4^n$ instead of
the more delicate (\ref{bound on deg}) by either selecting $\eta$ so small that $q<1/16$ or 
using that $\lambda = 2$.
\end{rem}

\begin{cor}\label{basin of T}
 For any $\eta \in (0, 1/2)$ and $\de>0$, there exist a threshold $\bar\tau>0$ such that
$$
   \area \WW^s_{\eta, \bar\tau}(\TOPphys) \geq (1-\de) \area \VV_{\bar\tau}.
$$ 
Moreover, all the points in $\WW^s_{\eta, \bar\tau}$
get attracted to the top of the cylinder exponentially at rate $O(q^n)$, where $q=q(\eta)\to 0$ as $\eta\to 0$.
\end{cor}

\begin{proof}
For each $\tau \in (0,\bar\tau)$ let $\xi_{\tau} =  \{x:\ \tau(x)=\tau\}$.   \note{good notation!}
Applying Lemma \ref{LEM:POS_MEASURE_AH_CURVE},
we find a subset of $\xi_{\tau}$ of horizontal measure $>(2\pi-\de)$
that converges to $\TOPphys$ exponentially at rate $O(q^n)$. 
An application of the Fubini Theorem completes the proof.
\end{proof}

%%  This is the statement about horizontal curves of length > C\sqrt{\tau}.

The proof of Lemma \ref{LEM:POS_MEASURE_AH_CURVE} actually gives a slightly better statement, 
which we record here for later use. 

For a curve $\xi$, let $\tau(\xi)= \sup_{x\in \xi} \tau (x)$. 

\begin{lem}\label{LEM:SQRT_LENGTH_CURVE}
For any $\eta\in (0, 1/2)$, there exist $\bar\tau > 0$ and $C>1$
such that $\WW^{s,o}_\eta (\TOPphys)$ forms at least $3/4$ of the horizontal length
of any horizontal curve $\xi$ 
with $\tau(\xi) \leq \bar\tau$ and $l^h(\xi) \geq C \sqrt{\tau(\xi)}$.
\end{lem}

\begin{proof}
By (\ref{total length}), 
$$
    l^h(\xi\sm \WW^s_\eta(\TOPphys)) \leq \frac{8}{\sqrt{\eta}}\, \frac{\lambda}{c(\lambda-1)}\, \sqrt{\tau(\xi)}\leq 
                  \frac 1{4}\, C\sqrt{\tau(\xi)},  
$$
where we let  $\displaystyle{C = \frac{32}{\sqrt{\eta}} \frac{\lambda}{c(\lambda-1)}}$.
\end{proof}

%%%%%%%%%%%%%%%%%%%%%%%%%%%%%%%%%%%%%%%%%%%%%%%%%%%%%%%%%%%%%
% Here are a bunch of results from previous versions that we didn't end up needing.
%%%%%%%%%%%%%%%%%%%%%%%%%%%%%%%%%%%%%%%%%%%%%%%%%%%%%%%%%%%%%

\comment{*******************
Let us construct one more invariant horizontal cone field,
 $\tl\KK^h$, on the region $\VV'^\eta_{\bar\tau}$ (see Figure \ref{FIG:REGION_UPSILON}).    \note{should go after \S C.3}
It will have a virtue of being non-degenerate away from the indeterminacy points.
Namely, fix a small threshold $\bar \eps>0$, and let
$$
 \UU=\{x\in \VV'^\eta_{\bar\tau}: |\eps(x)| < \bar\eps\},\quad \QQ= \{x\in \VV'^\eta_{\bar\tau}: |\phi-\pi|< \de\}\quad,
        \QQ'=\{x\in \VV'^\eta_{\bar\tau}: |\phi-\pi|< 2\de\}.
$$   \note{choose $\de$}
Let the cone $\tl\KK^h(x)$ have boundary lines with the absolute value  $w(x)$ of the slope
equal to $ |\eps(x)| / 2$ in $\UU$,
equal to $1/2$ in  $\QQ$, equal to $\bar \eps/2$ in $\Cphys \sm (\UU\cup \QQ')$,
and is squeezed  between $\bar\eps/2$ and $1/2$ in $\QQ'\sm \QQ$
(arranged so that altogether the cone field is continuous).

\begin{lem}
For sufficiently small $\bar\eps$, $\bar\tau$ and $\eta$.
the cone field $\tl \KK^h$ is forward invariant on the region $\VV'^\eta_{\bar\tau}$,
i.e., $\Rphys (\tl\KK^h(x))\subset \tl\KK^h(\Rphys x)$, provided $x, \Rphys x \in \VV'^\eta_{\bar\tau}$.
\end{lem}

\begin{proof}
The cones $\tl\KK^h(x)$ are bounded by two lines spanned by the tangent vectors $v_\pm=(1, \pm  w(x)$.
Let us estimate the absolute value $s$  of the slope of $D\Rphys(v_\pm)$.

Let us select $\bar\eps$ so that (\ref{DR near alpha}) applies in $\UU$.
For $x\in \UU$, it yields:
$$
 s \leq \frac {|\eps| \tau( \tau +  \eps^2/2) }{\si^2(\tau+ \eps^2/2) } <  \frac 1 {2},
$$
so
\begin{equation}\label{invariance of tl K}
 \Rphys(\tl \KK^h(x))\subset \tl\KK^h(\Rphys x).
\end{equation}
For $x\in \VV'^\eta_{\bar\tau}$, let us use (\ref{EQN:A_EXPANSION}).
It shows that the first component of $D\Rphys (v_\pm)$ is bounded away from $0$
(with the bound depending on $\bar\eps$), while the second one is $O(\bar\tau)$
(with an absolute constant).
So, selecting $\bar\tau$ sufficiently small, we make $s\leq  \bar\eps/2$.
This implies (\ref{invariance of tl K})  provided $\Rphys x\not\in \UU$.

Finally, if $\Rphys x=(\eps', \tau') \in \UU$ then
%$\phi$ is close to $\pi/4$ $\mod \pi/2$,
%and the first component of $D\Rphys (v)$ is bounded away from $0$ independently of $\bar\eps$.
$$
    s=O(\tau)= O\sqrt{\tau'}\leq \sqrt{\eta} \cdot O(\eps'),
$$
so $s$ can be made smaller that $\eps'/2$ by selecting $\eta$ small enough.
Again,  (\ref{invariance of tl K}) follows.
\end{proof}

As usual, we let $\tl \KK^v$ be the complementary vertical cone field on $\VV'^\eta_{\bar\tau}$.
It is backward invariant under $\Rphys$.

Notice that at $x \in \VV'^\eta_{\bar \tau}$ we have $\KK^h(x) \Subset \tilde \KK^h(x)$ and hence $\tilde \KK^v (x) \Subset \KK^v(x)$.
 \note{Added by Roland, ok?}
********************************************}

\comment{***************************
\begin{lem}\label{LEM:EXPANSION_NEAR_TOP}
There exists $D > 0$ which is independent of $\bar \tau$ with the following property:
For any horizontal
curve $\xi$ through $x_0$, there is an $n_0 > 0$ with $l^h(\xi_{n_0}) \geq D \bar
\tau^{1/4}$.
\end{lem}

\begin{rem}\label{REM:NON_INT_CONES}
Near $\TOPphys$, the algebraic cone-field $\KK$ bounds the slope of a
horizontal curve $\xi$ at $x=(\phi,\tau)$ by $\frac{d\tau}{d \phi} = \sqrt{2
\tau}$, which is not uniquely integrable at $\tau = 0$.  Therefore, even if
$\xi$ is a long curve containing a point $x$ very near to $\TOPphys$, we cannot
guarantee that a long piece of $\xi$ is within a given neighborhood $\VV_{\bar
\tau}$.  However, if $\xi$ contains $x$ with $\tau(x) = \bar \tau/2$ and
$\tau(\xi) \geq \bar \tau$, then $l^h(\xi) \geq l_0(\bar \tau) := (\sqrt{2}-1) \sqrt{\bar
\tau}$.
\end{rem}

\begin{proof}[Proof of Lemma \ref{LEM:EXPANSION_NEAR_TOP}:]
Let us set up three appropriate regions.

Choose $\phi_0 < \pi/2$.  Property (P5) gives that if $\sigma>0$ is sufficiently small then points $x \in \AAA := \{x
\in \Cphys\,: |\phi(x)| \leq \phi_0, \, \tau(x) \leq \sigma \}$ will satisfy
\begin{equation*}
\tau(\Rphys(x))\leq M (\tau(x))^2.
\end{equation*}
\noindent

Choose $1/2 < c < 2 -\sqrt{2}$ so that Lemma \ref{LEM:BOUNDED_CONTRACTION}
gives that if $\sigma$ is small enough, then $v \in \KK(x)$ based in
\begin{equation*}
\HH_\pm:= \{x \in \Cphys\,:\, |\phi(x) \mp \pi/2| \leq \sigma, \tau(x) \leq
\sigma\}
\end{equation*}
\noindent
will satisfy $d\phi(D\Rphys(v)) \geq c \cdot d\phi(v)$.

Using the blow-up formula \ref{EQN:R_BLOWUP_ANGULAR} we see that if $\sigma$ is
sufficiently small and $\xi$ is a horizontal curve with $l^h(\xi) =
l_0(\sigma)$ and $\xi \cap \HH_\pm \neq \emptyset$, then
$\Rphys(\xi) \cap \HH_\pm = \emptyset$.

Let $1/c < \lambda < 2$.
According to Proposition \ref{LEM:STRETCHING_NEAR_TOP} if $\eta$ and $\sigma
$ are sufficiently small, then any horizontal curve $\xi \subset
\VV_{\eta,\sigma}$ is horizontally expanded by $\lambda$ under $\Rphys$.

We let $\hat \tau = \sqrt{\bar \tau/M}$.  We can choose
$\bar \tau$ sufficiently small so that
\begin{eqnarray*}
\hat \tau \leq \min\left\{\sigma,\sqrt{\sigma/\eta}\right\}.
\end{eqnarray*}
\noindent
In particular, $\hat \tau$ is small enough that every point in $\VV_{\hat
\tau}$ is either in $\VV_{\eta,\hat \tau}$ or in $\HH_\pm$.

Let $\xi'_n$ be the component of $\Rphys^n(\xi) \cap \VV_{\hat \tau}$
containing $\Rphys^n(x_0)$.  We will show that there is an $N>0$ so that for
any $n \geq N$ we have $l^h(\xi'_n) \geq c \cdot l_0(\bar \tau)$.

Let $\mu = c \cdot \lambda > 1$.  If $v \in \KK(x)$ with $x,\Rphys(x) \in
\VV_{\hat \tau}$, then $d\phi(D\Rphys^2(v)) \geq \mu \, d\phi(v)$.
In particular, if a truncation never occurs, then the lengths of the $\xi_n$
grow exponentially.

We can suppose that $\tau(\Rphys^n(x_0)) \leq \bar \tau /2$ so that if a
truncation occurs when forming $\xi_n$, then Remark \ref{REM:NON_INT_CONES}
gives $l^h(\xi_n) \geq l_0(\bar \tau)$.  Between successive truncations we must
have $l^h(\xi_i) \geq c \cdot l_0(\bar \tau)$ because any moment when there is
a contraction by at most $c$ is either followed by an expansion by $\lambda >
1/c$ or another truncation.

\msk

Because $x_0 \in \WW^s(\TOPphys)$ with $\Rphys | \TOPphys$ given by angle
doubling there are arbitrarily high $n$ with $|\phi(\Rphys^n(x_0))|$
arbitrarily small.  Thus, we can choose a moment when $|\phi(\xi_n)| < \phi_0$.
At the next iterate we have
\begin{eqnarray*}
\tau(\Rphys(\xi_n) &\leq& M \hat \tau^2 = \bar \tau, \, \mbox{and}\\
l^h(\Rphys(\xi_n)) &\geq&  \lambda \, l_0(\hat \tau) = D \bar \tau^{1/4}.
\end{eqnarray*}
\noindent
\end{proof}

Let $\bar \tau > 0$ and $C > 1$ be given by Lemma \ref{LEM:SQRT_LENGTH_CURVE}.
Making $\bar \tau$ smaller, if necessarily, we can assume that $D \bar
\tau^{1/4} \geq C \sqrt{\bar \tau}$.  Then, Lemma \ref{LEM:EXPANSION_NEAR_TOP}
gives a horizontal curve $\xi_{n_0}$ with $l^h(\xi_{n_0}) \geq C
\sqrt{\tau(\xi_{n_0})}$ so that $\WW^s_\eta(\TOPphys)$ forms at least $3/4$ of
the horizontal measure of $\xi_{n_0}$.

Let $\tl \xi_{n_0} \subset \xi_{n_0}$ be the subdisc of twice smaller horizontal radius, of which $\WW^s_\eta(\TOPphys)$ forms a
positive horizontal proportion, and let $\tl \xi = \Rphys^{-n_0}\left(\tl \xi_{n_0}\right) \subset \xi$.

\begin{lem}\label{LEM:DISTORTION_WHILE_APPROACHING_T}
Let $\tl \xi$ and $\tl \xi_{N}$ be as above.  Then $\Rphys^m: \tilde \xi \ra
\tilde \xi_{M}$ has bounded horizontal distortion.
\end{lem}

\begin{proof}

\bignote{We will either recover the appropriate distortion estimates and put them here, or use complex discs and Koebe.}

\end{proof}

By Lemma \ref{LEM:DISTORTION_WHILE_APPROACHING_T} we find that
$\WW^s_\eta(\TOPphys)$ also occupies a definite proportion of $\tilde \xi$
contrary to the fact that $x_0$ is a density point for $X=\WW^2(\TOPphys) \sm
\WW^s_\eta(\TOPphys)$.
\end{proof}
**********}

\comment{***********
\begin{rem}\label{REM:NON_INT_CONES}
Near $\TOPphys$, the algebraic cone-field $\KK$ bounds the slopes of an almost
horizontal curve $\xi$ at $x=(\phi,\tau)$ by $\frac{d\tau}{d \phi} = \sqrt{2
\tau}$, which is not uniquely integrable at $\tau = 0$.  Therefore, if $\xi$ is
a long curve containing a point $x$ very near to $\TOPphys$, we cannot
guarantee that a long piece of $\xi$ is within a given $\bar \tau$
neighborhood of $\TOPphys$.  More precisely, only a segment of length $O(\sqrt{ \bar
\tau})$ guaranteed to be within this neighborhood.
\end{rem}

\begin{proof}[Proof of Lemma \ref{LEM:EXPANSION_NEAR_TOP}:]
According to Proposition \ref{LEM:STRETCHING_NEAR_TOP} we can choose $\eta,
\bar \tau > 0$ so that any almost horizontal curve $\xi \subset
\VV_{\eta,\bar \tau}$ is horizontally expanded under $\Rphys$ by a factor
of $\lambda > 1$.

We further suppose that $\bar \tau$ is sufficiently small so that the blow-up
formula \ref{EQN:R_BLOWUP_ANGULAR} gives a
constant $\bar \kappa > 0$ so that all points $x$ sufficiently $\bar \tau$
close to $\INDphys_\pm$ having $\kappa(x) > \bar \kappa$ will have
$\tau(\Rphys(x)) \geq \bar \tau$.  (Here $\bar \kappa \approx \sqrt{\bar
\tau}$.)

Let
\begin{equation*}
\Delta := \{x \in \Cphys \,:\, \tau(x) \leq \bar \tau, \kappa(x) \leq \bar \kappa\}.\end{equation*}
so that is $\tau(x) \leq \bar \tau$ and $\tau(\Rphys(x)) \leq \bar \tau$, then $x \in \Delta$.  We will call the set of points in $\Delta$ having $\tau
(x) = \bar \tau$
the ``horizontal boundary'' of $\Delta$ and the set of points have $\kappa(x) = \bar \kappa$ the ``wedge boundary'' of $\Delta$.

Consider the rectangular neighborhoods of $\INDphys_\pm$
\begin{equation*}
\Pi_{\pm} := \left\{x \in \Cphys \,:\, \tau(x) \leq \bar \tau\, , \, |\phi(x) \mp \pi/2| \leq \sqrt{\frac{\bar \tau}{\eta}} \right\}.
\end{equation*}
\noindentNotice that almost horizontal curves contained in $\Delta \sm \Pi_\pm$ are expanded under $\Rphys$ be $\lambda > 1$.  Meanwhile, according
to Lemma \ref{LEM:BOUNDED_CONTRACTION} there is a bound $C > 0$ on the possible contraction of almost horizontal curves contained in $\Pi_\pm \cap \Del
ta$.

Because $x_0 \in \WW^s(\TOPphys)$ there is some $M >0$ so that
$\tau(\Rphys^n(x_0)) \leq \frac{\bar \tau}{2}$ for all $n \geq M$.  For each $n
\geq M$ we let $\xi_n$ be the component of  $\Rphys^n(\xi) \cap \{x \in \Cphys \,:\,
\tau(x) \leq \bar \tau\}$ containing $\Rphys^n(x_0)$.  Notice that if a
truncation $\xi_{n} \neq \Rphys(\xi_{n-1})$ occurs at some $n \geq M$ then
according to Remark \ref{REM:NON_INT_CONES} we must have $l^h(\xi) \geq l_*:=
(2-\sqrt2) \sqrt{\bar \tau}.$

Let $l_0 := C l_*$.  We will show that there is an $N\geq 0$ so that
$l^h(\xi_n) \geq l_0$ for $n \geq N$.

Recall the high-temperature fixed point $\FIXphys_1 = (0,1)$.
\begin{lem}\label{LEM:GROWTH_AT_RETURN_TIMES}
There exists $c > 1$ with the following property:If $L \subset \Delta$ is almost horizontal,
$\dist^h(L,\FIXphys_1) \geq 1$, and $n$ is the first iterate for which $\dist^h(\Rphys^n(L),\FIXphys_1)\geq 1$, then $l^h(\Rphys^n(L)) \geq \min(c \cdo
t l^h(L),l_*)$.
\end{lem}

We first observe how to use Lemma \ref{LEM:GROWTH_AT_RETURN_TIMES} to
complete the proof of Lemma \ref{LEM:EXPANSION_NEAR_TOP}.  Note that we can
make $\bar \tau$ small independent of $\eta$ obtaining
$\dist^h(\Pi_\pm,\FIXphys_1) > 1+l_*$.  Thus, contraction can only occur when
$\dist^h(\xi_n,\FIXphys_1) \geq 1$.

Suppose that a truncation never occurs.  Then $\xi_n \subset \Delta$ for each
$n$.  Otherwise there is some $j \geq M$ for which $\xi_{j}$ intersects the wedge
boundary of $\Delta$, which is mapped below $\bar \tau$ by $\Rphys$ implying
that a truncation occurs at the next iterate.

Therefore we can apply Lemma \ref{LEM:GROWTH_AT_RETURN_TIMES} which
implies that at the return times $n_i$ for which $\dist^h(\xi_{n_i},\FIXphys_1)
\geq 1$ the lengths undergo an exponential growth until they reach $l_*$.  Let
$N$ be the first return time when $l^h(\xi_N) = l_*$.  Then, $l^h(\xi_{n_i})
\geq l_*$ for each return time $n_i \geq N$.

We must therefore show that $n_i \geq N$ is a return time then for each $n_i < n < n_{i+1}$  we
have $l^h(\xi_n) \geq C l_* = l_0$.  This follows because a contraction of at
most a factor of $C$ can occur during the first iterate $n=n_i+1$ and then for
each future iterate before $n_{i+1}$ expansion occurs.

Otherwise, we let $N \geq M$ be the first moment when a truncation occurs.  As observed previously $l^h(\xi_N) \geq l_*$.
If $n_j,n_{j+1} \geq N$ are two successive moments of truncation, we must only check that for each $n_j < n < n_{j+1}$ we have
$l^h(\xi_n) \geq l_0$.  (If no further truncations occur after $n_j$ then we just check that $l^h(\xi_n) \geq l_0$ for all $n > n_j$.)

However, one can take further iterates until we have $n_\dagger \geq n_j$ so that $\dist^h(\xi_{n_\dagger},\FIXphys_1) \geq 1)$ with
$l^h(\xi_{n_\dagger}) \geq l^h(\xi_{n_j}) \geq l_*$, because contraction occurs only near $\INDphys_\pm$.
Then, the previous discussion gives that $l^h(\xi_n) \geq l_0$ for each $n_\dagger < n < n_{j+1}$.
\end{proof}

\begin{proof}[Proof of Lemma \ref{LEM:GROWTH_AT_RETURN_TIMES}:]
We cut $L = L' \cup L''$ with $L' \subset \Pi_+ \cup \Pi_-$ and $L'' := L \sm L' \subset \VV$.   See Figure \ref{FIG:REGION_PI}.

\begin{figure}
\begin{center}
\input{figures/region_Pi.pstex_t}
\end{center}
\caption{\label{FIG:REGION_PI}}
\end{figure}

For each $1 \leq i < n$ we can assume that $l^h(\Rphys^i(L'')) \leq l_*$.
(Otherwise we truncate and use only a piece of length $l_*$.) Therefore,
$\dist^h(L,\FIXphys_1)<1$ and $l^h(\Rphys^i(L'')) \leq l_*$ give that
$\Rphys^i(L'') \subset \VV$.  Therefore, $l^h(\Rphys^n(L'')) \geq
\min(\lambda^n \cdot l^h(L''),l_*) \geq \min(\lambda \cdot l^h(L''),l_*).$
We must work harder with $L'$.  The reason is that the horizontal length of an
almost horizontal curve $\xi \subset \Delta \sm \VV$ may be contracted by
a bounded amount under the first iterate.  (See Lemma  \ref{LEM:BOUNDED_CONTRACTION}.)
We apply Lemma \ref{One iterate near indeterminacy points} to $L'$ and take one more iterate,
finding
\begin{equation*}
\frac{l^h(\Rphys^2(L'))}{\dist^h(\Rphys(L'), \FIXphys_1)} \geq \frac{1}{2} \frac{l^h(L')}{\dist^h(L', \INDphys_\pm)}.
\end{equation*}

Therefore, using hyperbolicity of the fixed point $\FIXphys_1$, we find that the first iterate $n$ that $\dist^h(\Rphys^n(L'),\FIXphys_1) \geq 1$,
\begin{equation*}
l^h(\Rphys^n(L')) \geq \frac{1}{2\dist^h(L', \INDphys_\pm)} l^h(L').
\end{equation*}

Since $L' \subset \Pi_\pm$, we have that $\frac{1}{2\dist^h(L', \INDphys_\pm)} \gg 1$.  Therefore, at the first iterate $n$ for which $\dist^h(\Rphys^n(L'),\FIXphys_1) \geq 1$,
we have $l^h(\Rphys^n(L')) \geq \lambda' l^h(L')$, where $\lambda' > 1$.

\msk

Therefore, letting $c = \min(\lambda,\lambda') > 1$ we have that $l^h(\Rphys^{n}(L)) \geq \min(c \cdot l^h(L),l_0)$.
This completes the proof of Lemma \ref{LEM:GROWTH_AT_RETURN_TIMES}.
\end{proof}

*********}

\comm{****
\subsection{Long teeth in $\FF(\TOPphys)$}
\label{SUBSEC:LONG_TEETH}
We will show that by taking further preimages of the stable curves constructed in \S \ref{SUBSEC:FOLIATION_NEAR_T}
we obtain long curves in $W^s(\TOPphys)$ near $\II_0$ and $\II_\pi$.

Let $\gamma_0 \subset \II_0$ be the full stable curve of $\FIXphys_1$ stretching from $\FIXphys_1$ down to $\FIXphys_c$ and let $l_c$ be
the vertical length of $\gamma_0$.

\begin{prop}\label{PROP:LONG_TEETH}
Given any $\delta >0$, based at sufficiently near to $\FIXphys_1$ and $\INDphys =
(\pi,1) \in \TOPphys$ are stable curves $\gamma$ extending below $\tau = l_c -
\delta$.
\end{prop}

\begin{proof}
We need only prove the statement near $\FIXphys_1$ since the result follows for
$\INDphys$ under the symmetry $\Rphys(\phi+\pi,t)=\Rphys(\phi,t)$ described in Property
(P1).
The point $\FIXphys_1$ is hyperbolic, so within a sufficiently small neighborhood
$U$ of $\FIXphys_1$ the $\lambda$-Lemma from real dynamics \cite[pp. 80-85]{PDM} applies.  We take a
stable curve $\gamma_\epsilon \in \FF^s(\TOPphys)$ meeting $\TOPphys$ at $(\epsilon,1)
\in U$.  Because $\gamma_\epsilon$ is the stable curve for $(\epsilon,1)$ it is disjoint from $\gamma_0$ and
because it is a stable curve, it meets $\TOPphys$ at a non-zero angle.
We can therefore extend $\gamma_\epsilon$ as a $C^1$ curve that crosses $\TOPphys$ transversally and then
truncate, so that the resulting curve $\gamma'_\epsilon$ is contained in $U$ and disjoint from $\gamma_0$.

We denote by $\Rphys^{-1}$ the inverse branch of $\Rphys$ fixing $\FIXphys_1$.  Then, since
$\TOPphys = W^u(\FIXphys_1)$, and $\gamma'_\epsilon$ crosses $\TOPphys$ transversally, the
$\lambda$-Lemma gives that the preimages $\Rphys^{-n}(\gamma'_\epsilon)$ accumulate
parallel to $\gamma_0 = W^s(\FIXphys_1)$.  In particular, taking sufficiently many
preimages, we find a curve $\gamma = \Rphys^{-n}(\gamma'_\epsilon)$ extending below $\TOPphys$
arbitrarily close to $\tau = l_c$.
\end{proof}

Recall \S \ref{SUBSEC:BASIN_BOTTOM} where we obtained a lower bound $2\omega_0$
for the limiting limiting angle of a wedge that can fit into the principal
tongues $\LL_{\INDphys_\pm} \subset W^s(\BOTTOMphys)$ that reach $\TOPphys$ at
$\INDphys_\pm$. One consequence of Proposition \ref{PROP:LONG_TEETH} is an upper
bound $2\omega_1$, as shown in Figure \ref{FIG:MAXIMAL_ANGLE}.

Proposition \ref{PROP:LONG_TEETH} gives that
there is some $0 < \phi < \pi$ sufficiently close to $\pi$ so that $W^s(\phi)$ crosses $\Icurve$.
Since $W^s(\phi)$ and $\II_\pi$ are necessarily disjoint, $W^s(\phi)$ crosses $\Icurve$ at some point $(2\omega_1,\sin^2(\omega_1))$
with $0 < 2 \omega_1 < \pi$.  This angle $2 \omega_1$ gives the desired upper bound
since the portion of $W^s(\phi)$ that lies below $\Icurve$ pulls-back under the singular branch of $\Rphys^{-1}$ (Property (P6)) to
a smooth curve within $W^s(\TOPphys)$ meeting $\INDphys_+$ at angle $\omega_1$ with respect to $\II_{\pi/2}$.

\begin{figure}
\begin{center}
\input{figures/maximal_angle.pstex_t}
\end{center}
\caption{\label{FIG:MAXIMAL_ANGLE}}
\end{figure}
*****************************}

\comment{DIRECT CONSTRCTION OF THE CONE FIELD NEAR THE TOP
Throughout this section we will rely upon estimates of the behavior of $\Rphys$ in
a neighborhood of $\TOPphys$ and, in particular, close to the indeterminate points $\INDphys_\pm$.
These calculations are used repeatedly later in the paper, so we have put
them in Appendix \ref{APP:R_NEAR_ALPHA}, for reference.

\subsection{Cone-field near $\TOPphys$}
\label{SUBSEC:SQRT_CONE_FIELD}

For $x = (\phi,t) \in \Cphys$ we let $\tau(x) = 1-t(x)$.
Consider the horizontal cone field
\begin{equation}\label{hor cones}
   \KSQRT(x)  = \{ v \in T_x \Cphys: |d\tau (v)| < c\sqrt{\tau(x)}|d \phi(v)|\}, \ x=(\phi,\tau)\in \Cphys,
\end{equation}
\noindent
where $1 < c < \frac{3}{2}$.

\begin{prop}\label{PROP:SQRT_CONEFIELD}
There exists $\tau_0 > 0$ so that for all $x \in \Cphys$ with $0 < \tau(x) < \tau_0$
we have $D\Rphys({\KSQRT}(x)) \Subset \KSQRT(\Rphys(x))$.
\end{prop}

\begin{proof}
We first consider the case that $x$ is in a small neighborhood of $\INDphys_\pm$.  By symmetry we can
work in a neighborhood of $\INDphys_+$.

Let $\epsilon(x) = \pi/2 - \phi(x)$.   By linearity of the action of $D\Rphys$ it suffices to show that
the boundaries $v=(1,\pm c\sqrt{\tau})$ of the cone $\KSQRT(x)$ are mapped within $\KSQRT(\Rphys(x))$.
Using (\ref{DR near alpha}) we see that the $D\Rphys(x) v$ is proportional to $(1,a')$ where
\begin{eqnarray*}
a' = -\frac{\eps\tau(\tau-c\sqrt{\tau} \eps)} {|\sigma|^2(\eps^2+\tau-c\sqrt{\tau}\eps)},
\end{eqnarray*}
\noindent where $\zeta = \sqrt{\tau^2 + \epsilon^2}$.

Since $\tau' =\tau^2/\eps^2$, we need to check that $|a'|< c \tau/|\eps|$.
Since $\eps^2<|\zeta^2|$, it is enough to verify that
\begin{equation}\label{eps-tau}
  \frac{|\tau-c\sqrt{\tau} \eps| } {\eps^2+\tau-c\sqrt{\tau}\eps} < c.
\end{equation}

Note that denominator of this expression is a quadratic form in $(\eps, \sqrt{\tau})$
which is positive definite for $c\in (0,2)$.
Hence (\ref{eps-tau}) amounts to the following system of inequalities
$$ c\eps^2 - c(c+1)\eps\sqrt{\tau} +(1+c)\tau > 0, $$
$$ c\eps^2 - c(c-1)\eps\sqrt{\tau} +(c-1)\tau > 0.$$
These are two quadratic forms in $(\eps, \sqrt{\tau})$ that are  positive definite
in our range of $c$.

\msk

We now consider points outside of these neighborhoods of $\INDphys_\pm$.  We can therefore assume that
$|\pm \pi/w-\phi(x)|$ is bounded below will further restrict $\tau_0$ (if necessary) in order
to obtain invariance of $\KSQRT(x)$ for these points $x$.

Using expansion (\ref{EQN:A_EXPANSION}) for
\begin{equation*}
A :=\frac{1+2t^2\cos(2\phi)+t^4}{4} D\Rphys
\end{equation*}
\noindent
for any tangent vector $v = (1,a)$ based near $\TOPphys$ we have that
$D\Rphys(v)$ is proportional to
\begin{eqnarray*}
Av = (2(1+\cos(2\phi))^2-a\sin(2\phi)(2+2\cos(2\phi)) + O(\tau), O(\tau^2)).
\end{eqnarray*}

Since we assume that $|\pm \pi/2-\phi(x)|$ is bounded from below we have a
lower bound for $|1+\cos(2\phi)|$ that is independent of $\bar \tau$.
Therefore, the fact that $v \in \KSQRT(x)$ gives $|a| < c \sqrt{\tau}$ allows us
to choose $\tau$ sufficiently small so that $|a|\sin(2\phi)$ is negligibly
small compared to $1+\cos(2\phi)$.

Thus, choosing $\tau_0$ sufficiently small guarantees that the first
coordinate of $D\Rphys v$ is non-zero, independent of $\tau$.  The second
coordinate scales like $O(\tau^2)$, so if we re-scale $D\Rphys
v$ to the form $(1,a')$ we find that $a' = O(\tau^2)$.

In order to verify that $D\Rphys(\KSQRT(x)) \subset  \KSQRT(\Rphys(x))$
we need to check that $|a'| < c \sqrt{\tau'}$, where $\Rphys(x) =
(\phi',1-\tau')$.  Recall that $\TOPphys$ is non-uniformly superattracting with
asymptotics $\tau' = O(\tau^2/\cos^2(\phi))$.  Thus, $\tau'$ vanishes at most
like $O(\tau^2)$ and hence $c \sqrt{\tau'}$ vanishes at most like $O(\tau)$.
Since $|a'|$ vanishes as $O(\tau^2)$
it is easy to restrict $|\tau| < \tau_0$ so that $|a'| <
 c \sqrt{\tau'}$.
\end{proof}

Let $\Delta$ be the neighborhood of $\TOPphys$ found in Proposition \ref{PROP:SQRT_CONEFIELD} on which $\KSQRT(x)$ is forward invariant.  Clearly
$R(\Delta)$ contains a neighborhood of $\TOPphys$, as well.  We further restrict $\tau_0$ (if necessary) so that $R(\Delta)$
contains all $x \in \Cphys$ with $0 < \tau(x) < \tau_0$.  \note{$x$ or $\om$?}

\begin{cor}\label{COR:SQRT_VERT_CONES}
The vertical cone field $\KSQRT^v(x)$ complementary to $\KSQRT(x)$ is backward invariant under the regular lifts of $\Rphys$
for every $x$ satisfying $\tau(x) < \tau_0$.
\end{cor}

***************************************************************}

\section{Intertwined basins of attraction}\label{SEC:TWO BASINS}

Recall the sets $\WW^{s,o}_\eta(\TOPphys)$ from \S \ref{SUBSEC:BASIN_T}.
 We let
\begin{equation}\label{W^s_0}
\WW^{s,o}_0(\TOPphys) := \bigcap_{\eta > 0} \WW^{s,o}_\eta(\TOPphys).
\end{equation} 
Note that $\WW^{s,o}_0(\TOPphys)$ is a completely invariant set whose orbits
get attracted to $\TOPphys$ superexponentially fast.  

In this section we will prove the following result:

%\begin{lem}\label{inv family curves}
%  There is $K>0$ such that the family $\HH^h$ of almost horizontal curves
%with curvature bounded by $K$ is invariant by the dynamics of $\Rphys$.
%\end{lem}

\begin{thm}\label{attractors}
The union of the basins  $\WW^s(\BOTTOMphys)$ and  $\WW^{s,o}_0(\TOPphys)$ is a set of full area in the cylinder
$\Cphys$.
\end{thm}

Together with  Lemma \ref{LEM:STRETCHING_NEAR_TOP} we find

\begin{cor}\label{COR:LYAP_LOG2}
  Almost every point in $\WW^s(\TOPphys)$ has characteristic exponent $\log 2$.
\end{cor}

\subsection{Distortion control}

 To  prove Theorem \ref{attractors}, we will construct a family of horizontal
curves on which $\Rphys$ is expanding with bounded distortion.  Without the
indeterminacy points, this would be straightforward from partial hyperbolicity.

We will remove the union of two wedges extending downward from $\INDphys_\pm$:
\begin{eqnarray*}
 \Delta_{\bar \kappa} = \{x \in \Cphys \,:\, \tau(x) \geq \bar \kappa |\eps(x)| \}.
\end{eqnarray*}

 \begin{lem}\label{dist lemma}
Given any $\bar \kappa > 0$, there is a family $\HH_x$ of
``admissible'' horizontal curves centered each $x \in \Cphys$  with the following property:

If the orbit of 
$x \in \Cphys \sm \{\alpha_\pm\}$ avoids the wedge regions $\Delta_{\bar \kappa}$,
then there a neighborhood $\UU$ of $\INDphys_\pm$ and  sequence of times $n_i \equiv n_i(x)\in \Z_+$
such that $\Rphys^{n_i} x$ remains outside of $\UU$ and for any curve  $\xi\in \HH_x$ we have:

\ssk\nin {\rm (i)} $\xi\in \HH_x$ projects onto the  horizontal interval of radius 
   $r(x)\asymp \dist^h (x, \{\alpha_\pm\})$ centered at $x$;
  
\ssk\nin {\rm (ii)} The image $R^{n_i} \xi$ overflows some curve  $\eta_i \in \HH_{R^{n_i} x}$;

\ssk\nin {\rm (iii)} If $\widetilde{\eta}_i$ is the restriction of  $\eta_i$ to one half of its radius,
then, the inverse
branch $R^{-{n_i}} : \widetilde{\eta}_i \ra \xi $ is uniformly exponentially
contracting  with bounded distortion (with the contracting factor going to $0$
as $n_i\to \infty$).

\noindent 
\ssk
Moreover, the genuinely horizontal intervals 
$$
  \{  J_x= (\phi,t): \ t=t(x),\ |\phi - \phi(x)| < r(x) \} 
$$
are admissible.
\end{lem}

\subsection{Proof of Theorem \ref{attractors}}
Let us first derive Theorem \ref{attractors} from  Lemma~\ref{dist lemma}.

Take a small $\eta\in (0,1/2)$ and
 let $X_\eta$ be the complement of $\WW^s(\BOTTOMphys)\cup \WW^{s,o}_\eta(\TOPphys)$.
Assume $\area(X_\eta)>0$. 
By the Lebesgue and Fubini Theorems,
there is a point  $x \in X_\eta$ which is a density point for the slice of $X_\eta$
by any genuinely horizontal interval  $J = J_x$ centered at $x$.

By Proposition \ref{principal tongues}, we can choose $\bar \kappa$
sufficiently large so that the wedge regions $\Delta_{\bar \kappa}$ are
entirely contained in $\WW^s(\BOTTOMphys)$.  Since $x \not \in
\WW^s(\BOTTOMphys)$, the orbit $x_n:=\Rphys^n x$ avoids $\Delta_{\bar \kappa}$,
allowing us to apply Lemma \ref{dist lemma}.

Let $S \subset \N$ be the subsequence of times $n_i \equiv n_i(x)$ given by
Lemma \ref{dist lemma}.  We can choose a further subsequence $S' \subset S$
along which $x_n$ converges to some $y \in \Cphys \sm \{\INDphys_\pm\}$. 
We will keep the same notation $n_i$ for this subsequence.

Now, let $J$ be the genuinely horizontal interval radius $r(x)$.  Since $J $ is
admissible, for any $n_i \in S'$ there is an admissible curve $\eta_i \in
\HH_{x_{n_i}}$ centered at $x_{n_i}$ such that the inverse branch $R^{-n_i}:
\tilde{\eta}_{i} \ra J$ is contracting (exponentially in $n_i$) with bounded
distortion. 

Suppose $y\not\in \TOPphys$.  
Then, the curves $\tl \eta_i$ are part of a 
compact family of curves having the property that each curve from the family intersects 
$\WW^s(\BOTTOMphys)$ 
in a dense open set.  This implies that
$\WW^s(\BOTTOMphys)$ occupies a definite
portion of each $\tl \eta_i$.
Since $R^{-n_i}: \tl \eta_i \ra J$ has a bounded
distortion, the basin $\WW^s(\BOTTOMphys)$ occupies a definite portion of
$R^{-n_i}(\tl \eta_i)$, which is can be made an arbitrarily small neighborhood of
$x \in J$ by taking $n_i$ sufficiently large.  This contradicts the choice of
$x$ as a density point of $X_\eta$ on $J$.

If $y\in \TOPphys$, then let us consider $\bar\tau>0$ and $C>1$ from Lemma
\ref{LEM:SQRT_LENGTH_CURVE}.  Since $\tl \eta_i$ is a horizontal curve (with respect
to the algebraic cone field $\KK^\hor$), the horizontal length of $\gamma:=
\tl \eta_i \cap \VV_{\bar\tau}$ is at least $\sqrt{\bar\tau}$ for $n_i \in S'$ sufficiently
big, since it lies above one of the parabolas $\Sphys_\psi$.  Since $\gamma$ is near $\TOPphys$ and bounded away from $\INDphys_\pm$, 
Lemma \ref{LEM:STRETCHING_NEAR_TOP} gives that the horizontal length of $\gamma':= \Rphys \gamma$ is at
least that big.  But  $\tau(\gamma')=O(\bar\tau^2)$, so $ l^h(\gamma')\geq
C\sqrt{\tau(\gamma')} $. By Lemma \ref{LEM:SQRT_LENGTH_CURVE},
$\WW^{s,o}_\eta(\TOPphys)$ occupies at least $3/4$ of $\gamma'$.  

Since $\gamma$ is bounded away from $\INDphys_\pm$, the single iterate $\Rphys:
\gamma \ra \gamma'$ has bounded distortion (as does its inverse).  Therefore, $\Rphys^{-(n_i+1)}: \gamma' \ra J$
is exponentially contracting with bounded distortion.
Hence, by taking
sufficiently large $i$, the basin $\WW^{s,o}_\eta (\TOPphys)$ occupies a definite
portion of the arbitrarily small neighborhoods $R^{-(n+1)}(\gamma')\subset J$ of
$x$, contradicting again  the choice of $x$.

The contradictions show that $\area (X_\eta)=0$ for any $\eta > 0$, and the conclusion follows. 
\QED
%\end{proof}

\subsection{Proof of Lemma \ref{dist lemma} }

Let us formulate  a stronger, complex version of Lemma~\ref{dist lemma}.  Here
``horizontal holomorphic curves'' are understood in the sense of the complex
extension of the horizontal cone field $\KK^{\hor}$ constructed in Appendix
\ref{SUBSEC:COMPLEXIFICATION_CONES}.

Let $\pi(z,w) = z$.

\comment{%%%%%%%%%%%%%%
For each $x \in \Cphys \sm \{\INDphys_\pm\}$ let $\HH_x$ be the family of all
horizontal holomorphic curves $\xi$ that project under $\pi(z,w) = z$ onto the round disc of radius
$r(x)$ (to be specified below) centered at $\pi(x)$. 
In particular, $\HH_x$ includes the genuinely horizontal disks
$$
  \{  D_x= (\phi,t): \ t=t(x),\ |\phi - \phi(x)| < r(x) \}.
$$
%%%%%%%%%
}

\begin{lem}\label{complex dist lemma}
Given any $\bar \kappa > 0$, there is a family $\HH_x$ of ``admissible'' horizontal holomorphic curves
entered at each $x \in \Cphys$ with the following property:

If the orbit of
$x \in \Cphys \sm \{\alpha_\pm\}$ avoids the wedge regions $\Delta_{\bar \kappa}$,
then there is a neighborhood $\UU$ of $\INDphys_\pm$ and a sequence of times $n_i \equiv n_i(x)\in \Z_+$
such that $\Rphys^{n_i} x$ remains outside of $\UU$ and for any curve  $\xi\in \HH_x$ we have:

\ssk\nin {\rm (i)} $\xi \in \HH_x$ projects under $\pi$ on a complex disc of
radius $r(x) \asymp \dist^h(x,\{\INDphys_\pm\})$ centered at $\pi(x)$.

\ssk\nin {\rm (ii)} The image $R^{n_i} \xi$ overflows some curve  $\eta_i \in \HH_{R^{n_i} x}$;

\ssk\nin {\rm (iii)} If $\widetilde{\eta}_i$ is the restriction of $\eta_i$ to one half of its radius, then, the inverse branch $R^{-{n_i}} : \widetilde{\eta}_i \ra \xi $ is uniformly exponentially contracting  with bounded distortion
(with the contracting factor going to $0$ as $n_i\to \infty$). 

Moreover, the genuinely horizontal discs 
$$
  \{  D_x= (\phi,t): \ t=t(x),\ |\phi - \phi(x)| < r(x) \}.
$$
\noindent
are admissible.
\end{lem}

\begin{proof}

If $x \in \Cphys \sm \UU$, then $\HH_x$ will consist of all horizontal
holomorphic curves $\xi$ that project under $\pi$ onto the round disc of
constant radius $r(x) = r_1$ (to be specified below) centered at $\pi(x)$.  If
$x \in \UU$, then $\HH_x$ will consist of the restrictions to half their radius of all horizontal
holomorphic curves that project under $\pi$ onto a round disc of radius
$a|\eps(x)|$.

Since $\Rphys$ is horizontally expanding (Theorem \ref{hor expansion thm}),
there is an $N$ such that $D\Rphys^N$ expands horizontal vectors $v\in
\KK^\hor(x)$, $x\in  \Cphys \sm \{\alpha_\pm\} $, by definite factor. 

We can choose a neighborhood $\UU \subset \Cphys$ of $\{\INDphys_\pm\}$
sufficiently small so that each of the preimages $\Rphys^{-i} \UU$ for $1 \leq
i \leq N$ is in the neighborhood $\VV'$ of $\TOPphys$ in which
Lemma~\ref{LEM:STRETCHING_NEAR_TOP} gives that each iterate of $\Rphys$ expands
horizontal vectors.  Therefore, if $x \in \Cphys \sm \UU$ there exists $n(x)
\leq N$ so that $D\Rphys^{n(x)}$ expands any $v \in \KK^\hor(x)$ and $\Rphys^i
x \not \in \UU$ for $0 \leq i < n(x)$.

Proposition \ref{PROP:ESCAPE_FROM_ALPHA} gives $C> 0$ so for any sufficiently
small $a > 0$ and any $x \in \UU$ with $|\kappa(x)| < \bar \kappa$ there is an
iterate $n(x)$ so that if $\xi$ is an horizontal holomorphic curve centered at
$x$ of radius $a|\eps(x)|$ and $\xis$ is the subdisc of radius $a|\eps(x)|/2$,
then $\Rphys^{n(x)} \xis$ is horizontal and projects under $\pi$ onto a disc
has a definite radius $\geq C a$.  We further restrict $\UU$ so that $|\eps(x)|
< C$ for any $x \in \UU$.  It ensures that $\Rphys^{n(x)} \xis$ will be larger
than any curve $\xi$ of radius $a|\eps(x)|$ that is based at any $x \in \UU$.

Since the cone-field  $\KK^\hor$ is defined in a definite complex neighborhood
of $\Cphys \sm \UU$, on which $D\Rphys$ has bounded expansion, we can choose
$r_0$ sufficiently small for any horizontal holomorphic curve $\xi$ centered at
$x \in \Cphys \sm \UU$ of radius $\leq r_0$ we have that $\Rphys^i \xi$ is in the
domain of definition of $\KK^\hor$ for $1 \leq i < n(x)$.  In particular,
$\Rphys^{n(x)} \xi$ will be horizontal.  By continuity, we can also require
that $r_0$ be sufficiently small so that $\Rphys^{n(x)}$ is uniformly expanding
on any such curve  $\xi$.

If we choose $a$ sufficiently small so that $C\cdot a < r_0$ and choose $r_1 = C\cdot a$, it will
guarantee the overflowing property (ii).
\ssk

The above gives a sequence of further times
$n_i(x)$ and curves $\eta_i \in \HH_{\Rphys^{n_i(x)} x}$ so that
$\Rphys^{n_{i+1}-n_i} \eta_i$ is horizontal and overflows $\eta_{i+1}$.
Consequently, the inverse $\Rphys^{-n_i(x)} \eta_i \ra \xi$ is well-defined
(and hence univalent) for each $i$.

The Koebe Distortion Theorem gives that the restriction $\Rphys^{-n_i(x)} \
\widetilde{\eta}_i \ra \xi$ of each inverse branch to the disc of half the
radius will have bounded distortion.   By construction, $D\Rphys^{n_i(x)}$ is
exponentially expanding at the center of $\xi$, therefore the inverse branch is
exponentially contracting.

It follows from the proof of Proposition \ref{PROP:ESCAPE_FROM_ALPHA} that if
$\Rphys^{n_i(x)} x \in \UU$, then $\Rphys^{n_{i+1}(x)} x \not \in \UU$.
Therefore, passing to a subsequence, we can suppose that $\Rphys^{n_i(x)} x
\not \in \UU$ for each~$i$.  \end{proof}

\begin{rem}
  The whole proof of Lemma \ref{complex dist lemma} goes through 
in the purely real way  except one problematic  point: the distortion control in (iv).
In fact, with some extra work it should be possible to do it as well,
using the property  that the horizontal non-linearity of $\Rphys$
behaves like $1/\eps$ near the indeterminacy points $\alpha_\pm$. 
%This allows one to bound the horizontal distortion of the iterates $\Rphys^n$
%by
%$$
%   \sum_{k=0}^\infty \frac {|J_k|}{\dist (x_k, \{\alpha_\pm\}) }
%$$
%at the moments of the closest approaches to $\alpha_\pm$,
%be means of the standard distortion estimates   
\end{rem}

\comment{%%%%%%%%%%%%%%

\subsection{Proof of Lemma \ref{dist lemma}: real approach (outline)}

Finally,  we will outline a purely real approach to Lemma \ref{dist lemma}.
Namely, the whole proof of Lemma \ref{complex dist lemma} goes through 
with real curves on the cylinder (and in fact, it becomes easier since we do not need 
to construct the complex extension of the horizontal cone field),
except one point: the distortion control in (iv).
To establish  it in a purely real way, we need 
to estimate on the non-linearity of $R$ near the indeterminacy points. 

Recall that the {\em non-linearity} of an interval map $h: I\ra I$ is
defined as  
$$
 \NN h (x) =  h''(x)/h'(x).
$$
We define the {\it horizontal non-linearity}  $\NN^h\Rphys(x)$  as
the sup of the non-linearity of $\Rphys$ (measured with respect to the angular parameter)
on all admissible curves $\xi \in \HH_x$. Then estimate (\ref{}) implies:
\begin{lem}\label{hor non-linearity}
   \NN^h\Rphys(x) = O(\frac 1{\dist^h(x, \{\alpha_\pm\}}),
\end{lem}
as long as the admissible curves have uniformly bounded curvature
(which should be secured using the dominated splitting). 

This allows one to bound the horizontal distortion of the iterates $\Rphys^n$
by
$$
   \sum_{k=0}^\infty \frac {|J_k|}{\dist (x_k, \{\alpha_\pm\}) }.
$$
Due to exponential expansion, this gives us a uniform bound 
at the moments of the closest approaches to $\alpha_\pm$.
And then we have to deal with the rest of the orbit!

**************** }

\section{Central foliation, its holonomy and transverse measure} %  and distribution of Yang-Lee zeros}
\label{SEC:VERTICAL_FOLIATION}

{\it In what follows all laminations in question will be assumed strictly vertical.}
Given a lamination $\FF$ and $\tau\in (0,1)$, we let $\FF_\tau$ be the slice of $\FF$ by the 
truncated cylinder $\Cphys_\tau= \T\times [0,1-\tau]$. 

\subsection{Central foliation}\label{central fol sec}

Recall that  {\it central foliation} is a strictly vertical foliation invariant under $\Rphys^*$. 

\begin{thm}\label{cs foliation}
  The map  $\Rphys$ has a unique central foliation.
\end{thm}

\begin{proof}
According to Proposition \ref{PROP:EXISTENCE_OF_CENTRAL_CURVES}, through each
$x \in \Cphys \sm \{\INDphys_\pm \}$ is a central curve extending in both
directions to the boundary of the cylinder.  Taking the union of all such
curves through every point on $\Cphystl$, we obtain an invariant family $\FF^c$
of strictly vertical  $C^1$-curves filling in the whole cylinder $\Cphystl$. 

However, for a continuous vector field, there may exist many  integral curves through a given point,
so $\FF^c$ may fail to be a foliation. What saves the day is that  $\Rphys$ is horizontally expanding
(Theorem \ref{hor expansion thm}). Namely, assume that there exist two integral curves,
$\gamma_1$ and $\gamma_2$,  through a point $x\in \Cphystl$. Let us take two points $y_i\in \gamma_i$ 
on the same height, and connect them with a (genuinely) horizontal interval $\de=[y_1, y_2]$. 
By Theorem \ref{hor expansion thm}, the curves $\de_n:= \Rphys^n(\de)$ are almost horizontal,
and $l^h(\de_n)\to \infty$.  Hence the horizontal projections of the $\de_n$ eventually cover the whole
circle $\T$, and in particular, they intersect the critical interval $\II_{\pi/2}$. 
Hence $\de$ intersects $(\Rphys^n)^*(\II_{\pi/2})$ at some point $y$. 

But the interval $\II_{\pi/2}\sm \{\alpha_+\}$ is contained in the basin of $\BOTTOMphys$,    
where $\FF^c$ coincides with the stable foliation  $\FF^s(\BOTTOMphys)$,                       
which is a $C^\infty$ foliation (Proposition \ref{PROP:CINFTY_FOLIATION}). 
Pulling it back, we conclude that there is a unique integral curve $\WW^c(y)$ through $y$,
and $\FF^c$ is a $C^\infty$ foliation near it. But $\WW^c(y)$ is squeezed in between
the curves $\gamma_1$ and $\gamma_2$, and hence must merge with them at $x$ -- contradiction.     

This proves that the line field  $\LL^c$ is uniquely integrable,
so the family  $\FF^c$  of all integral curves forms the central foliation.
Since any such foliation must be formed by integral curves to $\LL^c$, it is unique.  
\end{proof}

Recall from the Introduction that the {\it holonomy} transformations
$g_t: \BOTTOMphys\ra \T_t\equiv \T\times \{t\}$,   \note{good notation!}
$t\in [0,1)$, are defined by the property
that $x$ and $g_t(x)$ belong to the leaf of $\FF^c$.

\begin{rem}\label{REM:EXTENSION_OF_BOTTCHER}
By means of the holonomy along the  central foliation to the bottom of $\CC$, 
we  can now obtain a continuous extension $\tl \Phi: \Cphystl \ra \T$ of the B\"otcher
coordinate that was constructed in \S \ref{SUBSEC:BOTTCHER} 
(namely, let $\tl \Phi(\phi,t) = g_t^{-1}(\phi)$).   
Since $\FF^c$ is $\Rphys$-invariant, 
$\tl \Phi$ also satisfies the B\"ottcher functional equation
$\tl \Phi(\Rphys(\phi,t)) = 4 \tl \Phi(\phi,t)$.
However, $\tl \Phi$ has  weak regularity outside of  $\WW^s(\BOTTOMphys)$: 
for example, it is not even absolutely continuous (see Corollary \ref{COR:AC_HOLONOMY} below).
\end{rem}

\subsection{Central tongues}
Recall a notion of a tongue from \S \ref{stable tongues sec}. 
A tongue is called {\it central} if it is bounded by two central leaves
(and hence it is saturated by intermediate leaves of the central foliation $\FF^c$).
For instance, the primary central tongues $\La_\pm$ from Property (P6) in  \S \ref{str on CC}
are central tongues in this sense  (as they are bounded by $\Rphys$-lifts of $\II_\pi$, 
which are central leaves). 

Recall also the pre-indeterminacy set $\PI=\cup \PI^n$ (\ref{PI set}).

\begin{prop}\label{central tongues}
 There is one maximal central tongue $\La(\alpha)$  attached to each pre-indeterminacy point $\alpha\in \PI^n$,
which is the regular pullback of one if the tongues $\La_\pm\equiv \La(\alpha_\pm)$ by $f^n$.   
% with bottom of length $ \pi/2^{2n+1}$.
This family of tongues is dense in $\Cphys$, and any central tongue is contained in one of these. 
\end{prop}

\begin{proof}
Obviously, regular lifts of central tongues are central tongues. 
Lifting the $\La_\pm$, we obtain a family of central tongues $\La(\alpha)$ attached to points $\alpha\in \PI$.
These tongues are pairwise disjoint as they are attached to different points of $\TOPphys$.     
Moreover, a central tongue $\La(\alpha)$ with $\alpha\in \PI^n$  has the bottom of length $ \pi/2^{2n+1}$. 
Since $|\PI^n|= 2^{n+1}$, we have:
$$
    \sum |\BB_{\La(\alpha)}| =  \sum_{n=0}^\infty \frac {\pi}{2^n} = 2\pi,   
$$
so the bottoms of these tongues have full measure in $\BOTTOMphys$. 
Hence this family of tongues is dense in $\Cphys$.
So, there is no room for enlarging these tongues, 
nor for fitting any extra tongue in between. 
\comm{
Let us list obvious properties of central tongues:

\ssk\nin (i) Any central tongue is contained in a maximal central tongue.

\ssk\nin (ii) A ``convex hull'' of central tongues attached to the same point is a central tongue.
  Hence there is at most one maximal central tongue attached to any point. 

\ssk These  properties imply that there is at most countable set $X\subset \TOPphys$ 
and a family of maximal central tongues $\La(x)$ attached to points $x\in X$  such each central tongue is contained in 
one of these.  
}
\end{proof}

\subsection{Orbits of typical leaves}
Given $x \in \Cphys$, for each $n\geq 0$ let $\gamma_n(x) \in \FF^c$ be the
leaf containing $\Rphys^n x$.  Invariance of $\FF^c$ gives $\Rphys \gamma_n(x)
\subset \gamma_{n+1}(x)$.  Note that we can have $\Rphys \gamma_n(x) \subsetneq
\gamma_{n+1}(x)$ if $\gamma_n(x)$ meets $\TOPphys$ at $\INDphys_\pm$.

\begin{prop}For almost every $x \in \Cphys$ have either:
\begin{itemize}

\item[(i)] there exists $N \geq 0$ so that for all $n \geq N$: $\gamma_n(x)$
meets $\TOPphys$ away from $\INDphys_\pm$, $\Rphys \gamma_n(x) =
\gamma_{n+1}(x)$, there are non-trivial intervals of points on $\gamma_n$
converging superexponentially to $\BOTTOMphys$ and to $\TOPphys$, or

\item[(ii)] $\gamma_n(x) \subset \WW^s(\BOTTOMphys)$ for all $n \geq 0$.
\end{itemize}
\end{prop}

\begin{proof}
According to Theorem \ref{attractors}, almost every $x
\in \Cphys$ is in $\WW^s(\BOTTOMphys) \cup \WW^s_0(\TOPphys)$.
It will be helpful later if we also assume that
$x \not \in \TOPphys$.

If $x \in \WW^s_0(\TOPphys)$, then for any $\bar \tau, \eta > 0$ there exits $N
\geq 0$ so that $\Rphys^n x \in \WW^s_{\bar \tau,\eta}(\TOPphys)$ for all $n
\geq N$.  (See \S \ref{SUBSEC:BASIN_T}.) For such $n$, some portion of
$\gamma_n$ agrees with a curve from the stable lamination of $\WW^s_{\bar
\tau,\eta}(\TOPphys)$ constructed in the proof of Proposition
\ref{PROP:ETA_BASINS_LAMINATED}.  In particular, $\gamma_n$ meets $\TOPphys$
away from $\INDphys_\pm$, and hence $\Rphys(\gamma_n) = \gamma_{n+1}$.  

Any leaf of $\FF^c$ intersects  $\WW^s(\BOTTOMphys)$ in a non-trivial interval.
Meanwhile, $\Rphys^n x$  has orbit superattracted to $\TOPphys$, so all points
on $\gamma_n$ above $\Rphys^n x \not \in \TOPphys$ have orbits that converge
superexponentially to $\TOPphys$. 

Now suppose $x \in \WW^s(\BOTTOMphys)$.  Let $\BOTTOMphys_0 = \cup
\Rphys^{-n}(\BOTTOMphys_{\Tongue_\pm})$, where $\BOTTOMphys_{\Tongue_\pm}$ are
the bottoms of the primary stable tongues.  By Proposition \ref{tongues are
a.e.}, $\BOTTOMphys_0$ has full Lebesgue measure.  Since $z \mapsto z^4$
preserves Lebesgue measure, the Poincar\'e Recurrence Theorem gives that the
set of points $\BOTTOMphys'_0 \subset \BOTTOMphys_0$ whose orbits return
infinitely many times to $\BOTTOMphys_0$ has full Lebesgue measure in
$\BOTTOMphys_0$.  Hence it also has full measure in $\BOTTOMphys$.

For any $x \in \BOTTOMphys'_0$, there is an infinite sequence of times $n_i$
for which $\Rphys^{n_i} x \in \BOTTOMphys_0$, implying $\gamma_{n_i}(x) \subset
\WW^s(\BOTTOMphys)$.  Then, because of the invariance $\Rphys(\gamma_n(x)) \subset
\gamma_{n+1}(x)$, we have $\gamma_n(x) \subset \WW^s(\BOTTOMphys)$ for all $n$.

Therefore, it suffices to prove that almost every point of $\WW^s(\BOTTOMphys)$
is in $\bigcup_{x \in \BOTTOMphys'_0} \gamma_0(x)$.  However, this follows since
$\FF^c$ is $C^\infty$ on $\WW^s(\BOTTOMphys)$ and that $\BOTTOMphys'_0$
has full Lebesgue measure in $\BOTTOMphys$.  \end{proof}

\subsection{Convergence}

Given a metric space $M$, the {\it Hausdorff metric} on the space of closed subsets of $M$ 
is defined as follows: $\dist_H (X,Y)$ is the infimum of $\eps>0$ such that
$X$ is contained in the $\eps$-neighborhood of $Y$, and the other way around.  

We say that a sequence of strictly vertical laminations $\FF^n$ converges to a 
lamination $\FF$ if they converge in the Hausdorff metric on subsets of the space $C^0[0,1]$.
Convergence is called exponentially with rate $\la \in (0,1)$ is there exists $C > 0$ such that
$$
    \dist_H (\FF^n,\, \FF) \leq C \la^n. 
$$

\begin{thm}\label{foliation FF}
  For any strictly vertical lamination $\FF$, 
  the pullbacks $\FF^n:=(\Rphys^n)^*\FF$ converge exponentially  
  (with the same rate) to the central  foliation $\FF^c$ of $\Rphys$.
\end{thm}

\begin{proof}
Given any $\gamma \in \FF^c$, let $\gamma_i$ be the curve from $\FF^c$
containing $\Rphys^i \gamma$ for $i=1,\ldots,n$.   Let $\eta_n$ be any leaf
from $\FF$.  We choose a sequence of iterated preimages
$\eta_{n-1},\ldots,\eta_0$, so that $\eta_{i-1}$ is obtained from $\eta_i$
under the same inverse branch of $\Rphys$ as $\gamma_{i-1}$ is obtained from
$\gamma_i$.
Let $h_0$ be the shorter of the two segments on $\BOTTOMphys$ connecting from 
$\gamma_0$ to $\eta_0$.  By construction, $l^h(h_0) \leq \pi/4^n$. 

For any $t \in [0,1]$, let $h_t$ be the genuinely horizontal curve from
$\gamma(t)$ to $\eta(t)$ chosen so that $\gamma[0,t] \cup h_t \cup
\eta[0,t]$ and $h_0$ are homotopic (relative to their endpoints) on $\Cphys$.  Then, their iterates under
$\Rphys^n$ are also homotopic.  The horizontal length of any vertical curve in
$\Cphys$ is bounded by some constant $K$.  Therefore, since $\Rphys^n \gamma[0,t]$ and
$\Rphys^n \eta[0,t]$ are vertical, we have that $l^h(\Rphys^n h_t) \leq
l^h(\Rphys^n h_0) + 2K \leq \pi+ 2K$.

Using horizontal expansion, Theorem
\ref{hor expansion thm}, we have $l^h(h_t) \leq C \lambda^{-n} \cdot l^h(\Rphys^n h_t)$
for appropriate $C > 0$ and $\lambda > 1$.  Since $l^h(\Rphys^n h_t) \leq \pi + 2K$, we 
have $l^h(h_t) \leq D \lambda^{-n}$.  Therefore, given any $\gamma \in \FF^c$ there is some
$\eta \in \FF^n$ that is $D \lambda^{-n}$ close within $C^0[0,1]$.

After switching the roles of $\gamma$ and $\eta$, the same proof shows that
given any $\eta \in \FF^n$ there is some $\gamma \in \FF^c$ that is
$D \lambda^{-n}$ close within $C^0[0,1]$.
\end{proof}

\begin{cor}\label{COR:LY_EXPONENTIAL_CONVERGENCE}
  The sequence of Lee-Yang loci $\Sphys_n$ converges exponentially to the central foliation $\FF^c$. 
\end{cor}

We have a better convergence at low temperatures.  For any $\eps >0$ let
$\FF^n_\eps$ and $\FF^c_\eps$ the truncations of $\FF^n$ and $\FF^c$ to the
cylinder $\T \times [0,t_c - \eps]$. 

\begin{prop}
For any $\eps > 0$ we have exponential convergence
\begin{eqnarray*}
\dist_H^1(\FF^n_\eps,\FF^c_\eps) \leq C(\eps) \lambda^n
\end{eqnarray*}
\noindent
where $\dist^1_H$ denotes the Hausdorff metric on subsets of $C^1 [0,1-\eps]$.
\end{prop}

\begin{proof}
Given any $\gamma \in \FF^c$, let $\eta \in \FF^n$ be the curve constructed in
the proof of Theorem \ref{foliation FF}.  Let $t \in [0,t_c-\eps]$.  By Lemma
\ref{exp shrinking of cones}, $\eta'(t)$ is exponentially close to
$\LL^c(\eta(t))$.  Since $\gamma(t)$ and $\eta(t)$ are exponentially close and
$\LL^c$ is H\"older in $\WW^s(\BOTTOMphys)$ (see Proposition
\ref{PROP:CINFTY_FOLIATION}) $\gamma'(t) = \LL^c(\gamma(t))$ and
$\LL^c(\eta(t))$ are also exponentially close.
\end{proof}

\subsection{Holonomy and the transverse measure}

\sss{Regularity of the holonomy}

%Recall from the Introduction that the {\it holonomy} transformations
%$g_t: \BOTTOMphys\ra \T_t\equiv \T\times \{t\}$,   \note{good notation!}
%$t\in [0,1)$, are defined by the property
%that $x$ and $g_t(x)$ belong to the leaf of $\FF^c$. 

\begin{prop}\label{regularity of holonomy}
   All holonomy transformations $g_t$, $0\leq t< 1$, are uniformly H\"older continuous homeomorphisms.   \note{how about $g^{-1}$?}
If $y=g_t(x)\in \WW^s(\BOTTOMphys)$ then $g_t$ is a  $C^\infty$ local diffeomorphism near $x$. 
\end{prop}

\begin{proof}
  Let us take an interval $J\subset \BOTTOMphys$, and let $J_t=g_t(J)$.
Since $\Rphys(z)=z^4$ on $\BOTTOMphys$, there is an $n\in \N$ 
such that $\Rphys^n(J)$ covers $\T$ at least once, but no more than four times.
Then the same is true for $\Rphys^n(J_t)$. Hence the horizontal length of both intervals
$\Rphys^n(J)$ and $\Rphys^n(J_t)$ is squeezed in between $2\pi$ and $8\pi$. 
It follows that $l(J)\asymp 4^{-n}$,                                   \note{make ``$l$'' standard} 
while $l(J_t)=O(\la^{-n})$ with  $\la>1$ from Theorem \ref{hor expansion thm}. 
%since $\Rphys$ is horizontally expanding.  
Hence $l(J_t)= O(l(J)^\si)$ with $\si= \log \la/ \log 4$. 

Since the $g_t: \T\ra \T_t$ 
 are continuous bijections on a compact space, they are homeomorphisms.
The last assertion follows from $C^\infty$ smoothness of the foliation 
$\FF^c|\, \WW^s(\BOTTOMphys) = \FF^{s}(\BOTTOMphys)$
(Proposition \ref{PROP:CINFTY_FOLIATION}).
\end{proof}

At the top of the cylinder, the holonomy degenerates to the {\it Devil Staircase}: 

\begin{prop}\label{devil staircase}
 As $t\to 1$,  the holonomy maps $g_t$ uniformly converge to the map $g_1$ that collapse
the bottoms $\BOTTOMphys_{\La(\alpha)}$ of the central tongues to their tips $\alpha\in\PI$.   
\end{prop}  

\begin{proof}
  Indeed, all the leaves that begin on the bottom $\BOTTOMphys_{\La(\alpha)}$ of the central tongue $\La(\alpha)$
merge at its tip $\alpha\in \TOPphys$. 
\end{proof}

\sss{Standard transverse measure}

Recall also that 
a {\it transverse invariant measure} $\mu$ for $\FF$ is a family of measures
$\mu_t$, $t\in [0,1)$,  such that $\mu_t= (g_t)_*(\mu_0)$.  Obviously, it is uniquely
determined by $\mu_0$.  It can be evaluated not only on genuinely horizontal sections but
on all local transversals to $\FF^c$, in particular, on all almost horizontal curves.

The transverse measure $\mu$ for which $\mu_0$ is the normalized Lebesgue  measure $d\phi/ 2\pi$
will be called {\it balanced}. 
For $t\in [0,1)$, we let 
$$
  O_t= \{x = (\phi,t)\in \Cphys: \, x\in \WW^s(\BOTTOMphys)\}, \quad K_t= \T\sm O_t.
$$ 
Let us also use a special notation for the
{\it horizontal expansion factor}:  
\begin{equation}\label{upper-left entry}
    \la^h(x):= \frac{\di(\phi\circ \Rphys)}{\di \phi}(x),
\end{equation}  \note{introduce earlier}
equal to  the upper-left entry of the Jacobi matrix $D\Rphys$.
Note that the transverse measure $d\mu = \rho\, d\phi/2\pi$ is transformed by the rule
\begin{equation}\label{transfer rule}
    \frac {\di (\Rphys^*\mu)} {\di \mu} =  \la^h(x)  \frac {\rho(\Rphys x)}   {\rho(x)}, \quad x\in \WW^s(\BOTTOMphys).
\end{equation}

\begin{prop}\label{hol inv meas}
  Let $\mu$ be the standard transverse measure on the central foliation $\FF^{c}$.
For any $t\in [0,1)$, the measure $\mu_t$ is absolutely continuous and $\mu_t(O_t)=1$.
Its density $\rho_t$ is  positive and $C^\infty$ on each component of $O_t$. 
Moreover,  for any almost horizontal curve $\gamma$ we have the transfer rule: 
\begin{equation}\label{balanced property}
  \mu(R(\gamma))= 4\mu(\gamma),\quad
\mbox{or equivalently:}\quad
    4\rho(x)= \la^h(x)\, \rho(\Rphys x), \quad x\in \WW^s(\BOTTOMphys).
\end{equation}
\end{prop}

\begin{proof}
By Proposition \ref{tongues are a.e.} (v),
$\mu$ is supported on the union of stable tongues $\Tongue_k(\alpha)$, hence $\mu_t(O_t)=1$.
By the second assertion of Proposition \ref{regularity of holonomy}, 
$\mu_t$  has a  positive $C^\infty$ density on $O_t$.
%real analyticity of the holonomy on $\WW^S(\BOTTOMphys)$. 

  Property (\ref{balanced property}) is obviously satisfied for the Lebesgue measure on $\BOTTOMphys$.
By holonomy invariance of $\mu$, it is satisfied for any almost horizontal curve. 
The equivalent formulation in terms of the density $\rho$ comes from the (\ref{transfer rule}).
\end{proof}

\begin{cor}\label{COR:AC_HOLONOMY}
  The holonomy maps $g_t: \BOTTOMphys\ra \T_t$ are absolutely continuous,
while the inverse maps $g_t^{-1}$ are not for $t>t_c$.
\end{cor}

\begin{proof}
 Absolute continuity of $g_t$ is equivalent to absolute continuity of the push-forward measure $\mu_t=(g_t)_*\mu_0$,
so Proposition \ref{hol inv meas} implies the first assertion.

On the other hand, by Theorem \ref{basin of T},  
the complement to the stable tongues on any section $\T_t$, $t>t_c$, has positive measure,
while on the bottom $\BOTTOMphys$, it has measure zero (by Proposition \ref{tongues are a.e.} (v)). 
This yields the second assertion. 
\end{proof}

At the top, the transverse measure becomes purely atomic:

\begin{prop}\label{m1}
  As $t \to 1$, the distributions $\mu_t$ weakly converge to the distribution $\mu_1$ supported on the pre-indeterminacy 
set $\PI$ that assigns to a point $\alpha\in \PI^n$ weight $1/4^{n+1}$.  
\end{prop}

\begin{proof}
  This follows from Proposition \ref{devil staircase}, 
taking into account that 
$\mu_0(\BOTTOMphys_{\La(\alpha)})=1/4^{n+1}.$
\end{proof}

\subsection{High-temperature hairs and critical points}
\label{SEC:HIGH_TEMP_ENDPOINTS}

As usual, we
parameterize each $\gamma \in \FF_c$ by temperature $t$.
Every $\gamma \in \FF^c \sm \{\II_{\pm \pi/2}\}$ maps by $\Rphys$
homeomorphically onto $\Rphys(\gamma)$, so there are temperatures $0 < t^-_\gamma
\leq t^+_\gamma \leq 1$ so that
\begin{eqnarray*} 
\gamma \cap \WW^s(\BOTTOMphys) &=& \gamma[0,t^-_\gamma) \,\,\, \mbox{and} \\ 
\gamma \cap \WW^s(\TOPphys) &=& \gamma(t^+_\gamma,1] \,\, \mbox{or} \,\, \gamma[t^+_\gamma,1].
\end{eqnarray*}
\noindent
We call $h_\gamma := \gamma \cap \WW^s(\TOPphys)$ the {\em high-temperature
hair} contained in $\gamma$.  Recall the notion of Cantor bouquet introduced in \S \ref{SUBSEC:FOLIATION_NEAR_T}.  

\begin{prop}$\WW^s(\TOPphys)$ is a Cantor bouquet containing all of the bouquets that were constructed in
\S \ref{SUBSEC:FOLIATION_NEAR_T}.
\end{prop}

We call $e_\gamma:= \gamma(t^+_\gamma)$ the {\em endpoint
of $h_\gamma$} and $c_\gamma:=\gamma([t^-_\gamma,t^+_\gamma])$ the {\em
critical points of $\gamma$}.  Let
\begin{eqnarray*}
\epoints:= \bigcup_\gamma e_\gamma \,\, \mbox{and} \qquad \crittemps :=
\bigcup_{\gamma} c_\gamma
\end{eqnarray*}
\noindent
Each is an $\Rphys$ invariant set.

\begin{prop}
The set of critical temperatures $\crittemps$ has zero Lebesgue measure.
\end{prop}
\noindent

\begin{proof}
By construction, the high temperature hair through any  $x \in \WW^{s,o}_0(\TOPphys)$ extends below $x$.
Therefore, $\crittemps$ lies in the complement of $\WW^s(\BOTTOMphys) \cup \WW^{s,o}_0(\TOPphys)$ so that
it has measure $0$ by Theorem \ref{attractors}.
\end{proof}

\begin{cor}\label{COR:ENDPOINTS_MEASURE_ZERO}
The set of endpoints to the high-temperature hairs $\epoints$ has zero Lebesgue measure.
\end{cor}

\section{Lee-Yang Distributions, Local Rigidity, and Critical exponents}
\label{SEC:THERM}

Now, having plowed hard in the RG dynamical cylinder, let us collect the
physics harvest:
% \footnote{Compare with  famous Douady's principle: ``You plow
%in the parameter plane, and then harvest in the dynamical plane''. }

\subsection{Lee-Yang Distributions and Critical exponents}

%\begin{cor}\label{COR:LY_EXPONENTIAL_CONVERGENCE}
%  The sequence of Lee-Yang loci $\Sphys_n$ converges exponentially to the central foliation $\FF^c$.
%\end{cor}

{\it Proof of the Main Theorem (physical version).}
Recall from \S \ref{MK RG eq-s} that the Lee-Yang locus $\SSS_n$ of level $n$
is equal to the pullback $(\Rphys^n)^*\SSS_0$ of the principal LY locus
$\SSS\equiv \SSS_0$ (\ref{S}).  By Theorem \ref{foliation FF}, these loci
converge exponentially fast to the central foliation $\FF^c$.

Moreover, on the bottom $\BOTTOMphys$ the Lee-Yang zeros are obviously asymptotically
equidistributed with respect to the Lebesgue measure $\mu_0$. It follows that on the circle $\T_t$, they are
asymptotically equidistributed with respect to the measure $(g_t)_*(\mu_0)=\mu_t$,
which is the standard transverse measure for $\FF^c$.
The rest of the physical version of the Main Theorem follows from the properties of $\mu_t$
established in \S \ref{SEC:VERTICAL_FOLIATION}.

%\bignote{Refer to the dynamical version of the Main Thm}      \note{more on rigidity?}

\subsection{Local Rigidity} 
\label{SUBSEC:LOCAL_RIGIDITY}
Recall the notion of {\em local rigidity} introduced in \S \ref{DHL intro}.

\begin{prop}\label{PROP:LOCAL_RIGIDITY}
The Lee-Yang zeros for the DHL are locally rigid at any point where the limiting density is positive $\rho(\phi,t) > 0$.
\end{prop}

\begin{proof}
 The Lee-Yang zeros $\phi_k^n(t)$ of level $n$ on $\T_t$ are the solutions to
\begin{align}\label{EQN:RIGIDITY1}
h^n_t(\phi) := \frac{1}{2\cdot4^n}  f_1 \circ \Rphys^n(\phi,t) = \frac{\pi k}{4^n} + \frac{\pi}{2\cdot 4^n}, \,\, k=0,1,\ldots 2\cdot 4^n-1,
\end{align}
where $f_1: \CC_1\ra \T$ is the first coordinate of the mapping $f$ given in (\ref{f}).
Notice that $f_1$ is a degree 2 map tangent  to $2\phi$ on $\BOTTOMphys$.

Fix a point  $x=(\phi_*,t) \in \WW^s(\BOTTOMphys)$ and a closed horizontal
interval $J_t \subset O_t$ containing $x$ in its interior. 

By  Lemma \ref{convergence of foliations}, the maps $h_t^n$ converge  in the
$C^1$ topology on $J$  to the B\"ottcher coordinate $\Phi_t(\cdot) \equiv \Phi(\cdot,t)$. 
Since $\partial \Phi / \partial \phi \neq 0$,  
the maps  $h^n_t$ are invertible on $J_t$  for sufficiently large $n$. 
 Let $g^n_t: J_0 \ra J_t$ be the inverse (where  $J_0:= h^n_t(J)\subset \BB$).  
Since  the holonomy map $g_t$ is the inverse of $\Phi_t$, we conclude that
$(g^n_t)'\ra g_t' $ uniformly on $J_0$.

For each $n$, let $\phi_l^n$ be the Lee-Yang zero that is closest to $\phi_*$.
We will show that the rescaled Lee-Yang zeros
\begin{eqnarray}\label{EQN:RIG2}
s^n_k = \frac{2\cdot 4^n}{2\pi} \ \rho_\tconv(\phi_*)\left(\phi^n_{l+k} - \phi^n_{l}\right),
\end{eqnarray}
converge locally uniformly to the integer lattice $\Z$.

After fixing $k$,
$\phi^n_{l+k}$ and $\phi^n_{l}$
will be in $J$ for all $n$ sufficiently large. 
The Mean Value Theorem gives
\begin{eqnarray}\label{EQN:RIG3}
\phi_{l+k}^n - \phi_{l}^n = (g_t^n)'(\psi_k^n) \frac{2\pi k}{2\cdot 4^n},
\end{eqnarray}
\noindent
where $\psi_k^n$ is between $\frac{\pi l}{4^n} + \frac{\pi}{2\cdot 4^n}$ and $\frac{\pi (l+k)}{4^n} + \frac{\pi}{2\cdot 4^n}$.

Equation (\ref{EQN:RIG2}) becomes
\begin{eqnarray}
s_k^n = k \cdot (g_t^{-1})'(\phi_*) \cdot ({g_t^n}) ' (\psi_k^n),
\end{eqnarray}
\noindent
since $\rho_t(\phi_*) = (g_t^{-1})'(\phi_*)$.  For fixed $k$
\begin{eqnarray*}
\lim \psi_k^n = \Phi_t(\phi_*) = (g_t)^{-1}(\phi_*).
\end{eqnarray*}
Thus, $(g_t^n)'(\psi_k^n) \ra g_t'((g_t)^{-1}\phi_*)$ giving that $s_k^n \ra k$.  
\end{proof}

\subsection{Critical Exponents}
\label{SUBSEC:CRITICAL_EXPONENTS}
Let  $x=(z,t)\in \Cphys\sm \WW^s(\BOTTOMphys)$ 
(so that the density $\rho_t$ of the transverse measure $\mu_t$ vanishes at $x$).
If
\begin{equation}\label{hor exp def}
     \mu_t(J)\asymp |J|^{\si^h+1}\ \mbox{for a horizontal interval $J$ containing $x$ on its boundary,}
\end{equation}
 then  $\si^h=\si^h(x)$
is called the {\it horizontal critical exponent} of the transverse measure at $x$
(on the left- or right-hand side of $x$, depending on $J$ -- if we do not specify the side, 
it means that the critical exponent exists on both sides).

Let additionally, $x$ lie on the boundary point of $\LL^c(x)\cap \WW^s(\BOTTOMphys)$
(in other words,  let all points on the  central leaf $\LL^c(x)$ below $x$ converge to the
bottom of the cylinder). If 
\begin{equation}\label{vert exp def}
  \rho(y)\asymp \dist (x,y)^{\si^v}\ \mbox {for $y\in \LL^c(x)$ below $x$},
\end{equation}
then $\si^v$ is called  the {\it vertical critical exponent} of the transverse measure at $x$.

Let  $x$ be a periodic point for $\Rphys$ of period $p$ with multipliers  $\la^u >  \la^c$.
Here the {\it unstable multiplier} $\la^u$ corresponds to the eigenvector of $D\Rphys^p_x$ in the horizontal cone $\KK^h(x)$,
while the {\it central multiplier} corresponds to the eigenvector tangent to the central leaf $\LL^c(x)$.
The inequality between the multipliers follows from the dominated splitting.
Also, $\la^u>1$ because of the horizontal expansion. 
%We believe that $\la^c>1$ for all periodic points as well, but we do not have a proof of this fact. 
The corresponding {\it characteristic exponents} at $x$ are defined as
$$
  \chi^u(x)= \frac 1{p} \log \la^u,\quad \chi^c(x) = \frac 1{p} \log \la^c.
$$ 

\begin{prop}\label{crit exp for periodic pts}
  Let  $x$ be a periodic point for $\Rphys$ of period $p$ with the characteristic exponents $\chi^u$ and $\chi^c$.
Then 
\begin{equation}\label{hor exp}
  \si^h(x) = \frac{\log 4}{\chi^u} - 1.
\end{equation}
Moreover, if $x$ is a boundary point of some component $J$ of the basin $O_t$, then 
\begin{equation}\label{asymp of rho}
        \rho_t(y)\asymp \dist(x,y)^{\si^h}, \quad y\in J \ \mbox{near $x$}.
\end{equation}
If $x$ is  the boundary point of $\LL^c(x)\cap \WW^s(\BOTTOMphys)$ and $\chi^c(x)>0$, then 
\begin{equation}\label{vert exp}
  \si^v(x) =\frac{\log 4 -\chi^u}{\chi^c}. 
\end{equation}
\end{prop}

\begin{proof}
  Given a horizontal interval  $J\ni x$, let us apply to it an iterate $\Rphys^n$ that stretches $J$ to
a horizontal curve that wraps around the cylinder at least once but at most four times.
Then both $l^h(\Rphys^n(J))$ and $\mu (\Rphys^n(J))$ are comparable with 1. 
 
On the other hand,  
 $l^h(\Rphys^n(J))\asymp \exp(n \chi^u) |J| $ while  $\mu (\Rphys^n(J)) = 4^n \mu_t(J)$.
Hence $  \mu_t(J) \asymp |J|^{\si+1}$ with exponent $\si=\log 4/\chi^u-1$. This proves (\ref{hor exp}). 

Let us prove (\ref{asymp of rho}).
% In terms of the density $\rho$, the transformation rule $\di (\Rphys^*\mu/) \di\mu = 4$ becomes
%$$
%        \rho(y)= q(y) \rho(\Rphys y),\quad \mbox{where}\ q = \frac{\di (\phi\circ \Rphys)}{\di\phi}.
%$$  \note{notation}
Iterating the transfer rule (\ref{balanced property}), we obtain:
\begin{equation}\label{iterated transfer rule}
     4^n \rho(y)= \la^h_n(y)\, \rho(\Rphys^n y), \quad \mbox{where}\ \la^h_n(y)= \prod_{k=0}^{n-1} \la^h(\Rphys^k y) .
\end{equation}
Let  $y\in J$ be a point near $x$. Because of the horizontal expansion, we can find an iterate $\Rphys^n$
such that  $\dist^h(\Rphys^n x, \Rphys^n y)\asymp 1$.  Then $\rho(\Rphys^n y) \asymp 1$, 
while $\la^h_n(y)\asymp \exp (n\chi^u(x)) $. Incorporating these into (\ref{iterated transfer rule}),
we obtain: 
$$ \rho(y)\asymp \exp(n(\chi^u-\log 4)). $$
On the other hand,  $\dist (x,y)\asymp \exp(-n\chi^u(x))$, 
and  (\ref{asymp of rho}) follows.

To prove (\ref{vert exp}), take a point $y\in \LL^c(x)$ below $x$.
Then $y\in \WW^s(\BOTTOMphys)$
since $x$ lies on the boundary of $\LL^c(x)\cap \WW^s(\BOTTOMphys)$. 
Hence  we can find an iterate $\Rphys^n$
such that  $\dist^c(\Rphys^n x, \Rphys^n y)\asymp 1$. 
 However, in this case $x$ repels $y$ at exponential rate with
exponent $\chi^c$, so $\dist (x,y)\asymp \exp(-n\chi^c(x))$, and (\ref{vert exp}) follows. 
\end{proof}

\begin{cor}
  The horizontal and vertical critical exponents at the fixed point $\FIXphys_c=(0,t_c)\in \II_0$ are equal to 
$\sigma^h(\FIXphys_c) = 0.06431\ldots$ and $\sigma^v(\FIXphys_c) = 0.1617\ldots$. 
\end{cor}

One can define the critical exponents at a point $x\in \WW^s(\TOPphys)$ in a {\it weak sense}:
$$
   \si^h(x)=\lim_{|J|\to 0} \frac{\log \mu(J)}{\log |J|}-1, \quad \si^v(x)= \lim_{y\to x} \frac{\log\rho(y)}{\log\dist(x,y)}
$$ 
(if the limits exist), where the meaning of $J$ and $y\in \LL^c(x)\cap \WW^s(x)$ 
are the same as in formulas (\ref{hor exp def}), (\ref{vert exp def}).
These critical exponents  can be expressed  in terms
of the unstable and central Lyapunov exponents
$$
  \chi^h(x) = \lim_{n\to \infty} \frac 1{n} \log \la_n^h(x), \quad  \chi^h(x) = \lim_{n\to \infty} \frac 1{n} \log \la_n^h(x)
$$          % \bignote{define hor and vert expansions and exponents in the dominated sections?}
by the same formulas as in the case of periodic points:

\begin{prop}\label{PROP:LYAPUNOV_WEAK_CRIT}
  Let  $x\in \Cphys\sm \WW^s(\BOTTOMphys)$ be a point with the unstable Lyapunov exponent $\chi^u$.
Then the horizontal critical exponent  $\si^h(x)$ exists in the weak sense and is given by formula
(\ref{hor exp}). 
Moreover, if $x$ is a boundary point of some component $J$ of the basin $O_t$, then 
\begin{equation}\label{asymp of rho2}
      \si^h(x) =  \lim \frac{\log \rho_t(y)}{\log \dist(x,y)}\quad \mbox{as}\  y\to x, \ y\in J.
\end{equation}
If $x$ is  the boundary point of $\LL^c(x)\cap \WW^s(\BOTTOMphys)$ and the central Lyapunov exponent
$\chi^c(x)$ exists and positive, then the vertical  critical exponent  $\si^v(x)$ 
exists in the weak sense and is given by formula (\ref{vert exp}). 
\end{prop}

The proof mimics that of Proposition \ref{crit exp for periodic pts}. 

\begin{cor} For every $x \in \Cphys \sm \WW^s(\BOTTOMphys)$ at which $\si^h(x)$ exists we have $\si^h(x) \leq 1$.  Moreover, equality holds at Lebesgue almost every such $x$.
\end{cor}

\begin{proof}
Theorem \ref{hor expansion thm} gives that $\Rphys$ is horizontally expanding
with rate $\lambda =2$, implying the upper bound.  The second
statement follows by combining Corollary \ref{COR:LYAP_LOG2} with Proposition \ref{PROP:LYAPUNOV_WEAK_CRIT}.
\end{proof}

%\note{Remove?}
%\begin{cor}
%For a.e. $y=g_t(x) \in \Cphys\sm \WW^s(\BOTTOMphys)$, the map $g_t$ behaves quadratically near $x$, i.e., 
%$\De y\asymp (\De x)^2$. Hence   $\rho_t(z, t+\De t)\asymp |\De t|$.    
%\end{cor}

\begin{rem}
 The partial derivative $\chi_T(h): =\di M(h,T)/\di h$ is called  {\it susceptibility}.  %\note{of the magnet to the field?}
Kaufman and Griffiths proved in \cite{KG2} that
the Ising model on DHL exhibits an infinite susceptibility at $h=0$ for $T\ge T_c$.
More precisely, as was shown by Bleher and \v Zalys \cite{BZ3},
the susceptibility $\chi_T(h)$ has a logarithmic singularity
at $h = 0$ for $T> T_c$.                                          % when $d=2$.
\end{rem}

\section{Periodic Leaves}\label{SEC:PERIODIC_LEAVES}
Let us distinguish between two distinct types of leaves $\gamma \subset \FF^c$
that are periodic under $\Rphys$. We say that a periodic leaf $\gamma$ is {\em
regular} if $\gamma \cap \TOPphys$ is a periodic point.  If $\gamma$ is a periodic leaf that is not regular, then $\gamma
\cap \TOPphys$ is in the preindeterminacy set $\AAA$.

Associated to any periodic point $x_1 \in \TOPphys$ is a regular periodic leaf
meeting $\TOPphys$ at $x_1$.  Meanwhile, associated to any periodic point $x_0
\in \BOTTOMphys$ there is a periodic leaf meeting $\BOTTOMphys$ at $x_0$, which
need not be regular---since almost ever point in $\BOTTOMphys$ is in the union
of the stable tongues, it is quite common to obtain a singular periodic leaf
from a periodic $x_0 \in \BOTTOMphys$.

The periodic points $x_0$ and $x_1$ at the bottom and top of a regular periodic
leaf $\gamma$ are horizontally repelling and vertically (super) attracting.
They have real-analytic (super) stable manifolds $\WW^s(x_0)$, $\WW^s(x_1)$,
which extend slightly below $\BOTTOMphys$ and above $\TOPphys$, respectively.
The high-temperature hair $h_\gamma = \WW^s(x_1) \cap \Cphys$ and the
low-temperature hair $l_\gamma := \WW^s(x_0) \cap \Cphys$ are both
real-analytic and non-trivial.  Therefore, the smoothness of a regular periodic
leaf $\gamma$ is determined within its critical points $c_\gamma = \gamma \sm
(h_\gamma \cup l_\gamma)$.

\begin{prop}\label{PROP:NON_ANALYTIC_LEAVES}
Suppose that $\gamma \in \FF_c$ is a regular periodic leaf of prime period $k
> 1$.  Then, $\gamma$ is not real-analytic.
\end{prop}

\noindent
The assumption that $k > 1$ is necessary, since the vertical interval $\II_0$ is a regular periodic
leaf of period $1$.  The assumption that $\gamma$ is a regular periodic leaf is also necessary, because there are many
periodic leaves contained entirely within $\WW^s(\BOTTOMphys)$ that are real-analytic.

\begin{proof}
We suppose that $\gamma$ is a regular periodic leaf, of prime period $k >1$, that
is real-analytic.  Let $x_0$ and $x_1$ be the periodic points at the bottom and
top of $\gamma$, respectively.  We extend $\gamma$ analytically slightly below
$x_0$ and above $x_1$ and then take a complexification $\gamma_\C$, chosen sufficiently small
so that it is an embedded complex disc.

Let $x_c = \gamma(t_\gamma^-)$ be the periodic point ``at the bottom of
$c_\gamma$''.  It has has one-dimensional central direction and one-dimensional
unstable direction, with multipliers $1 \leq \lambda_c < \lambda_u$.  Any small
piece of $\gamma$ containing $x_c$ in its interior will be a central manifold
$\WW^c_{\rm loc}(x_c)$.  Similarly, an open disc from $\gamma_\C$ containing $x_c$
will form a complex analytic central manifold $\WW^c_{\C,\rm
loc}(x_c)$.

\msk
Consider the case that $x_c$ is vertically repelling: $\lambda_c > 1$.  Let
$\rho_0: \D \ra \WW^c_{\C,\rm loc}(x_c)$ be a local linearizing coordinate,
i.e.  $\rho_0(\lambda_c  x) = \Rphys(\rho_0(x))$.  It can be globalized to form a non constant
$\rho : \C \ra \CP^2$ satisfying $\rho(\lambda_c x) = \Rphys^k(\rho (x))$ that is given by
\begin{eqnarray*}
\rho(x) := \lim_{n \ra \infty} \Rphys^{n\cdot k}(\rho_0(x/\lambda_c^n)).
\end{eqnarray*}
\noindent
Suppose $\Rphys^l(\rho_0(x_*/\lambda_c^n))$ lands on an indeterminacy point for
some $x_* \in \C$.  After taking appropriate blow-ups at the indeterminate
point, $\Rphys$ extends to some holomorphic $\tilde \Rphys$.  (See Appendix
\ref{APP:BLOW_UPS}.)  The image under $\Rphys^l(\rho_0(x/\lambda_c^n))$ of some
complex disc $D$ containing $x_*$ lifts to the blown-up space via the
proper transform, intersecting the exceptional divisor in a single point.  We
define the next iterate on this disc using $\tilde \Rphys$.  This
definition coincides with $\Rphys^{l+1}(\rho_0(x/\lambda_c^n))$ on $D \sm
\{x_*\}$ and gives a holomorphic extension through $x_*$.

A global central manifold $\WW \equiv \WW^c_{\C,\rm
glob}(x_c)$, invariant under $\Rphys^k$, is given by $\rho(\C)$.
A-priori, $\WW$ can be wild, possibly accumulating on itself.  However, we will
show that $\WW$ can be compactified to form an algebraic curve.

Given $x,y
\in \C$, let $x \sim y$ if and only if there exist neighborhoods of $N_x$ and
$N_y$ of $x$ and $y$ respectively, and a biholomorphism $h: N_x \ra N_y$ so that $\rho_{|N_x} = \rho_{|N_y} \circ h$. Then $\hat{\WW} = \C /\sim$ is a
Riemann surface whose charts are obtained by appropriate local sections of the
projection $\pi: \C \ra \hat{\WW}$.  The map $\hat \rho: \hat{\WW} \ra \WW$
that is induced by $\rho$ is holomorphic and provides a nice parameterization of $\WW$,
which is injective on all but a discrete subset of $\hat \WW$.

The action of $\Rphys^k: \WW \ra \WW$ can be lifted (in the natural way) to
$\hat \Rphys^k: \hat \WW \ra \hat \WW$.  Notice that $\pi(0) \in \hat \WW$ is a
repelling fixed point under $\hat \Rphys^k$ so that $\hat{\WW}$ is
non-hyperbolic.
%In order to form a compactification of $\WW$, we check that
%$\hat \WW$ is biholomorphic to a twice punctured Riemann sphere, and that
%$\hat \rho$ extends holomorphically across both punctures.

We construct a larger Riemann surface $V := (\hat \WW \sqcup \gamma_\C) / \hat
\rho$, where $x \in \hat \WW$ and $y \in \gamma_\C$ are identified if and only
if $\hat \rho(x) = y$ with some neighborhood of $x$ in $\hat \WW$ mapped by
$\hat \rho$ into $\gamma_\C$.  The natural inclusion $\iota: \hat \WW \ra V$ is
holomorphic.  Since $\hat \WW$ is not hyperbolic, $\iota$ can omit at most two
points of $V$.  Hence $\hat \rho$ can omit at most two points of $\gamma_\C$.
Therefore, we can (at least) choose some punctured discs $U_{0,1} \subset \gamma_\C$ in
the image of $\hat \rho$, having $x_{0,1}$ as their punctures (respectively).

Let $\hat U_0 \subset \hat \WW$  be the punctured disc mapped biholomorphically
by $\hat \rho$ onto $U_0$.  Since $\FIXphys_0$ is the only point in
$\LL_0$ having an iterated preimage under $\Rphys$ outside of
$\LL_0$, it is the only point that can possibly be in $\WW \cap
\LL_0$.  By assumption, $x_0 \neq \FIXphys_0$, so there is no
point in $\hat \WW$ mapping to $x_0$.  Thus, the puncture in $\hat U_0$
corresponds to an actual puncture in $\hat \WW$.

Let $\hat \WW_0 \equiv \hat \WW \cup \{w_0\}$ be the Riemann surface obtained
by filling in this puncture.  Both $\hat \rho$ and $\hat \Rphys^k$ extends to
$\hat \WW_0$, with $\hat \rho(w_0) = x_0$ and $\hat \Rphys^k(w_0) = w_0$.

Since $\hat \WW$ is non-hyperbolic with at least one puncture, $\hat \WW_0$ is
biholomorphic to either $\C$ or $\CP^1$.   In either case, $\hat \Rphys^k: \hat
\WW_0 \ra \hat \WW_0$ has a degree.  (If $\hat \WW_0  \cong \C$, the action of
$\hat \Rphys^k$ is polynomial, since any point has finitely many preimages under
$\hat \Rphys^k$.) Since $w_0$ is totally invariant under $\hat \Rphys^k$, with
a neighborhood mapped to itself with degree $2^k$, we see that $\Rphys^k: \hat
\WW_0 \ra \hat \WW_0$ has degree $2^k$.

Let $\hat U_1 \subset \hat \WW$  be the punctured disc mapped by $\hat \rho$
biholomorphically onto $U_1$.  If there were a point $w_1 \in \hat \WW$ filling
the puncture in $\hat U_1$, it would satisfy $\hat \rho(w_1) = x_1$, and the
local degree of $\hat \Rphys^k$ at $w_1$ would be $2^k$.
However, if $x_1 \in \WW$, there must be iterated preimages of $x_1$ under $\Rphys^k$
in $\WW$ converging to $x_c$,
violating
that the total degree of $\Rphys^k: \hat \WW \ra \hat \WW$ is $2^k$.

Therefore, the puncture in $\hat U_1$ corresponds to an actual puncture in
$\WW$.   As before, it can be filled and $\hat \rho$ extends holomorphically
to the resulting space.

Since $\hat \WW$ is non-hyperbolic and has two punctures, it is biholomorphic
to the twice punctured Riemann sphere.   Since $\hat \rho$ extends
holomorphically through both of the punctures, sending each to $x_0$ and $x_1$,
respectively, $\overline{\WW} = \WW \cup \{x_0,x_1\}$ is a compact analytic
curve.  Chow's Theorem (see, e.g., \cite{GH}) gives that $\overline{\WW}$ is therefore algebraic.

One local branch of $\overline{\WW}$ at $x_0$ is $\WW^s_{\C,\loc}(x_0)$.  Since it
intersects $\LL_0$ perpendicularly, $\overline{\WW}$ cannot have degree
$1$, for it would intersect $\TOPphys$ in a point of prime period $2k$.

Bezout's Theorem gives a second intersection of $\overline{\WW}$ with $\LL_0$.
It corresponds to some disc in $\hat \WW$, disjoint from $\hat U_0$,
whose image under $\hat \rho$ intersects $\LL_0$.  Therefore this intersection
must be at $\FIXphys_0$.  However, $\overline{\WW}$ does contain either of the
invariant separatrices $z=1$ or $t=0$, so the dynamics near $\FIXphys_0$ would
result in infinitely many branches, which is impossible.

\msk
Suppose that $x_c$ is vertically neutral.  Then, within $\WW^c_{\loc,\C}(x_c)$
is some repelling petal $\PP$ for the parabolic point $x_c$.  Then there is
some open $H \subset \C$ containing a left half-space with Fatou coordinate
$\rho_0:H \ra \PP$ a conformal isomorphism that satisfies $\rho_0(x+1) =
\Rphys^k(\rho_0(x))$.  We define $\rho: \C \ra \CP^2$ by
\begin{eqnarray*}
\rho(x) = \lim_{n \ra \infty} \Rphys^{n\cdot k}(\rho_0(x-n)).
\end{eqnarray*}
\noindent
The composition extends through indeterminate points of $\Rphys$ in the same way as in the repelling case.

Then, $\rho(\C) \subset \CP^2$ is forward invariant under $\Rphys^k$ and
contains $\PP$, which is an open subset of $\gamma_\C$.  As in the repelling case,
one can compactify $\rho(\C)$ to form a periodic
algebraic curve.  This again leads to an intersection with $\FIXphys_0$, and a contradiction.
\end{proof}

\begin{rem}
  Artur Avila has shown that for almost all points $x=(\phi,1)$  on the top $\TT$,
the leaf landing at $x$  is not real analytic.
\end{rem}

%\begin{prop}\label{PROP:REGULARITY_OF_LEAVES}
%Suppose that $\gamma$ is a periodic leaf (of period $k > 1$) with $c_\gamma$
%consisting of a single periodic point $x_c$ having multipliers $1 < \lambda^c <
%\lambda^u$.  Let $r$ be the maximal integer for which $\lambda_c^r \leq
%\lambda_u$.  Then, $\gamma$ is not of class $C^{r+1}$.
%\end{prop}

\begin{prop}\label{PROP:FINITE_SMOOTHNESS}
Suppose that $\gamma$ is a regular periodic leaf (of prime period $k > 1$) containing no vertically neutral periodic points.
Then, $\gamma$ has a finite degree of smoothness.
\end{prop}

\begin{proof}
Let $\gamma_m = \gamma([0,t_m])$ be the maximal real-analytic piece of $\gamma$
extending from $\BOTTOMphys$.  By Proposition \ref{PROP:NON_ANALYTIC_LEAVES},
$\gamma_m \subsetneq \gamma$, with $x_m := \gamma(t_m)$ a periodic point of
prime period $k$.

By hypothesis, $\lambda_c(x_m) \neq 1$.  Furthermore, $x_m$ cannot be
vertically attracting, since the stable manifold $\WW^s(x_m)$ would be a
real-analytic curve within $\gamma$ that extends above and below $x_m$.
Therefore, $x_m$ is vertically repelling.

Suppose that $x_m$ is linearizable.  Then, $\Rphys$ is
conjugate to the linear map $(u,v) \mapsto (\lambda_u u, \lambda_v v)$.
Any central invariant manifold has the form
\begin{eqnarray*}
u =
\begin{cases}
C_1 v^\alpha \,\, \mbox{if} \, \, v \geq 0, \, \mbox{and}\\
C_2 v^\alpha \,\, \mbox{if}\, \, v <  0,
\end{cases}
\end{eqnarray*}
\noindent
where $\alpha = \log(\lambda_u) / \log(\lambda_c)$.  By the choice of $x_m$, the central
manifold that is formed by $\gamma$ is not analytic at $x_m$, therefore it
does not correspond to $C_1 =  C_2 = 0$ or, if $\alpha \in \N$, to $C_1 =
C_2$.  In the remaining cases, the central manifold is not of class $C^{r}$, where $r-1
< \alpha \leq r$.

Since $\lambda_c, \lambda_u > 1$, we are in the Poincar\'e domain, with the
only obstruction to linearization being a resonance of the form $\lambda_c^r =
\lambda_u$ for some $r \in \N$.  Thus, if $x_m$ is not linearizable, the
Poincar\'e-Dulac Theorem gives that in some neighborhood $U \subset \Cphys$ of
$x_c$, $\Rphys$ is real-analytically conjugate to the normal form
\begin{eqnarray*}
(u,v) \mapsto (\lambda_u u + a v^r, \lambda_c v)
\end{eqnarray*}
\noindent
with $a \neq 0$.
(See, e.g.  \cite{IY}.)

Any central manifold is given by $u = g(v)$.  Invariance gives:
\begin{eqnarray*}
av^r = g(\lambda_c v) - \lambda_u g(v) = g(\lambda_c v) - \lambda_c^r g(v)
\end{eqnarray*}
\noindent
Differentiating $r$ and $r+1$ times, respectively, one finds
\begin{eqnarray}
a r!/\lambda_c^r &=&  g^{(r)}(\lambda_c v) - g^{(r)}(v),  \,\,  \mbox{and} \label{EQN_RTH_DERIVE}\\
0 &=& \lambda_c g^{(r+1)}(\lambda_c v) - g^{(r+1)}(v). \label{EQN_R+1_DERIVE}
\end{eqnarray}
By (\ref{EQN_R+1_DERIVE}),
either $g^{(r+1)}(v) \equiv 0$, or it is undefined at $v = 0$.  In the former
case, substitution of $g^{(r)}(v) \equiv C$ into (\ref{EQN_RTH_DERIVE}) gives a contradiction.
Thus, the central manifold is not of class $C^{r+1}$.
\end{proof}

Since the leaves of $\FF^c$ are obtained by integrating the continuous line field $\LL^c(x)$, they are all at least $C^1$ smooth.  
In fact, the regular periodic leaves have a slightly better smoothess:

\begin{prop}
Any regular periodic leaf $\gamma \in \FF^c$ is $C^{1+\de}$
for some $\de > 0$.
\end{prop}

\begin{proof}
It suffices to show that the line field $\LL^c$ is H\"older on $\gamma$ with exponent $\de$.

Replacing $\Rphys$ with an iterate of itself (keeping the same notation)
we can assume that $\gamma$ is invariant.
Below, the inverse map $\Rphys^{-1}$ will stand for $(\Rphys|\gamma)^{-1}$. 
% Since $\gamma$ is a regular periodic leaf, the orbit
% $\Rphys^i(\gamma)$ is disjoint from some neighborhood $\UU$ of $\{\INDphys_\pm\}$.  

For any $x,y \in \gamma$ and two parallel  
% \footnote{For this choice, we trivialize the tangent bundle over $\gamma$.}
tangent lines $X \in \KK^v( x)$, $ Y' \in \KK^v(y)$ we have a Lipschitz estimate:
\begin{eqnarray*}
% \dist(\Rphys x, \Rphys y) &\leq&  K  \dist(x,y), \, \mbox{and} \\
\dist_{a}( D\Rphys^{-1} X , D\Rphys^{-1} Y' ) &\leq& M \dist(x,y),
\end{eqnarray*}
\noindent
where $\dist_a$ denotes the angular distance.

On the other hand,  Lemma \ref{exp shrinking of cones} implies that  
there exists $\si \in (0,1)$ so that for any vertical lines $\tl L, L'\in \KK^v(y)$ we have
$$
    \dist_a  (D\Rphys^{-1} Y' , D\Rphys^{-1} Y) \leq \si \dist(Y', Y)
$$
Putting the last two estimates together, we obtain
$$
   \dist_{a}( D\Rphys^{-1} X , D\Rphys^{-1} Y ) \leq  \si \dist_a (X,Y) +  M \dist(x,y)
$$
for any  $X \in \KK^v(x)$,  $Y\in \KK^v(y)$.
Iterating this estimate in the backward time we obtain%
\footnote{The simplest way to see it is to notice that the iterates are bounded by the fixed point of the linear map $x\mapsto \si x+ Md$.}:
\begin{equation}\label{pullbacks of L and L'}
         \dist_{a}( D\Rphys^{-n} X , D\Rphys^{-n} Y ) \leq \frac { Md_n(x,y)} {1-\si},\quad {\mathrm{where}}
     \ d_n(x,y) =\max_{0\leq k\leq n-1} \dist (\RR^{-k}x, \RR^{-k}y), 
\end{equation}
as long as $\dist_a(X, Y)< d/(1-\si)$. (We can always start with parallel $X$ and $Y$). 

\msk
Let now $K$ be a Lipschitz constant for $\RR|\gamma$, and let $K_1 > K$.
Take two nearby  points $\alpha, \beta, \in \gamma$ and find $n$ such that
\begin{eqnarray}\label{K_1}
K_1^{-n} \leq \dist(\alpha , \beta) < K_1^{-(n-1)}.
\end{eqnarray}
\noindent
Letting $x = \RR^n \alpha$, $y = \RR^n \beta$, we obtain
\begin{equation}\label{lip estimate}
   d_n(x,y)\leq K^n d(\alpha, \beta)\leq \kappa^n, \quad \mathrm{where} \ \kappa= K/K_1< 1.
\end{equation}
By (\ref{pullbacks of L and L'}),  we have 
$$
  \dist_{a}( D\Rphys^{-n} X , D\Rphys^{-n} Y ) = O(\kappa^n). 
$$
But according to Proposition \ref{Central line field prop},
$D\Rphys^{-n} X$ is exponentially close to the tangent line $\LL^v(\alpha)$ to $\gamma$ at $\alpha$,
and likewise $D\Rphys^{-n} Y$ is exponentially close to the tangent line $\LL^v(\beta)$.
Hence
$$
       \dist_a (\LL^v(\alpha), \LL^v(\beta)) = O(\eta^n) \quad {\mathrm{for\ some}}\ \eta\in (0,1).
$$
Together with (\ref{K_1}), this implies that 
$$
   \dist_a (\LL^v(\alpha), \LL^v(\beta)) \leq C \dist(\alpha, \beta)^\de, \quad {\mathrm{with}}\ \de= \frac{\log K_1}{\log (1/\kappa)}. 
$$
\end{proof}

\begin{rem}
  The above argument applies to any vertical leaf whose forward orbit stays away from the points of indeterminacy $\alpha_\pm$. 
The  problem with other leaves is that the Lipschitz estimate for $\RR|\gamma$ may fail for leaves $\gamma$ passing through $\INDphys$
because of the big expansion near the $\alpha_\pm$.
\end{rem}

\appendix

\section{Elements of complex geometry}\label{APP:COMPLEX GEOM}

We are primarily interested in rational maps between complex projective
spaces in two dimensions.  However, in order to understand the behavior
near indeterminate points, we will need a discussion at somewhat greater
generality.  Much of the below material can be found in with greater detail in
\cite{DAN,DEMAILLY,GH,SHAF}.

\subsection{Projective varieties and rational maps}
Let $\pi: \C^{k+1} \sm \{0\} \ra \CP^k$ denote the canonical projection. 
Given $z \in \CP^k$, any $\hat z \in \pi^{-1}(z)$ is called a {\em lift} of $z$.
One calls $V \subset \CP^k$ a (projective) algebraic hypersurface if there is a
homogeneous polynomial $\hat p: \C^{k+1} \ra \C$ so that 
$$
  V= z \in \CP^k: \quad  \hat p(\hat z) = 0.
$$ 
% is the locus of points $z \in \CP^k$ so that $\hat p(\hat z) = 0$. 
More generally, a
(projective) algebraic variety is the locus satisfying a finite number projective
polynomial equations.   Any algebraic variety $V$ has the structure of a smooth
manifold away from a proper subvariety $V_{\rm sing} \subset V$ and 
the dimension of $V \sm V_{\rm sing}$ is called {\em the dimension of $V$}.

A rational map $R:\CP^k \ra \CP^l$ is given by a homogeneous polynomial map
$\hat R: \C^{k+1} \ra \C^{l+1}$ for which we will assume the components have no
common factors.  One defines $R(z):= \pi(\hat R (\hat z))$ if $\hat R(\hat z) \neq
0$, and otherwise we say that $z$ is an {\em indeterminacy point for $R$.}
Since $\hat R$ is homogeneous, the above notions are well-defined.
Because the components of $\hat R$ have no common factors, the set of
indeterminate points $I(R)$ is a projective variety of codimension greater than
or equal to two.  

Given two projective varieties, $V \subset \CP^k$ and $W \subset \CP^l$, 
a {\it rational map} $R: V\ra W$ is the restriction of a rational map
$R: \CP^k\ra \CP^l$ such that  $R(V \sm I(R)) \subset W$. 
%
% If $V \subset \CP^k$ is a variety, we say that $R: V \ra \CP^l$ is a rational
%map if it is obtained as the restriction of a rational map from $\CP^k$ to
%$\CP^l$.  If $W \subset \CP^l$ is another variety and $R(V \sm I(R)) \subset W$
%we say that $R:V \ra W$ is a rational map.  
As above, 
$I(R) \subset V$ is a projective subvariety of codimension greater than or equal to two in $V$.  
If $I(R) = \emptyset$, we say that $R$ is a (globally)  {\it holomorphic (regular)} map.

A rational mapping $R:V \ra W$ between non-singular varieties is {\em dominant}
if there is a point $z \in V \sm I(R)$ such that  ${\mathrm {rank}}\,  DR  (z) = \dim W$.  

\begin{lem}\label{LEM:DOMINANT}
Let $R:V \ra W$ be a dominant rational map between non-singular varieties.  
If $z$ is not an indeterminate point for $R$, then $R$ is locally surjective at $z$.
\end{lem}
\noindent
It is a consequence of the Weierstrass Preparation Theorem---see
Propositions 4.8 and 4.19 from \cite{DEMAILLY}.

\subsection{Blow-ups}
\label{APP:BLOW_UPS}

A non-singular variety of dimension two is called a {\em projective surface}.
Matters are simpler for maps $R:V \ra W$ 
% with $V$ and $W$ of dimension no greater than two,  
between projective surfaces since $I(R)$ is finite  in this case.
(We will refer to such maps as ``rational surface maps''.)

Given a pointed projective surface $(V,p)$,
the {\it blow-up} of $V$ at $p$ is another projective surface  $\tl V$ with a
holomorphic projection $\pi: \tl V \ra V$ such that
\begin{itemize}
\item $L_\ex(p): = \pi^{-1}(p)$ is a complex line $\CP^1$ called the {\it exceptional divisor};
\item $\pi: \tl V \sm L_\ex(p) \ra V\sm \{p\}$ is a biholomorphic map.
\end{itemize}
\noindent
See \cite{GH,SHAF}.

The construction has a local nature near $p$, so it is sufficient to provide it
for $(\C^2, 0)$.
The space of lines $l\subset \C^2$ passing through the origin is $\CP^1$  by definition.
Then $\tilde {\C^2}$ is realized as the surface $X$ in $\C^2\times
\CP^1$ given by equation $\{(u,v) \in l\}$
% (in coordinates $(u,v,l)\in \C^2\times \CP^1$)
with the natural projection $(u,v,l)\mapsto (u,v)$.
 In
this model, points of the exceptional divisor $L_\ex=\{(0,0,l): \ l \in \CP^1 \}$
get interpreted as the {\it directions} $l$ at which the origin is approached.
%One typically does calculations in the two charts $(u,m) \rightarrow (u,mu,[1:m])$ and $(v,n) \mapsto (nv,v,[n:1])$,
% where $[a:b]$ denotes the homogeneous coordinates on $\CP^1$.

Any line $l\subset \C^2$ naturally lifts to the ``line'' $\tl l= \{(u,v,l):\ (u,v)\in l\}$ in $\tilde {\C^2}$
crossing the exceptional divisor at $(0,0,l)$.%
\footnote{This turns $\tilde{\C^2}$ into a line bundle over $\CP^1$ known as the  {\it tautological}  line bundle.}
Moreover, $\tilde{\C^2}\sm \tl l $ is isomorphic to $\C^2$.
Indeed, let $\phi(u,v)=au+bv$ a linear functional that determines $l$.
It is linearly independent with one of the coordinate functionals, say with $v$ (so $a\not=0$).
Then
$$
  (u,v,l)\mapsto  (\phi, \kappa:= v/\phi)
$$
is a local chart that provides
a desired isomorphism.
In particular, two  charts corresponding to the coordinate axes in $\C^2$
provide us with  local coordinates  $(u, \kappa= v/u)$ and
 $(v , \kappa =  u/v)$ which are usually used in calculations.

The value of this construction lies in the fact that it can be used to {\em resolve} the indeterminacy of a rational map, as follows:

\begin{thm}[See \cite{SHAF}, Ch. IV, \S 3.3]\label{THM:RESOLUTION_IND}
Let  $R: V \ra W$ be a rational surface map. Then there exists a sequence of blow-ups $V_m \xrightarrow{\pi_m} \cdots
\xrightarrow{\pi_2} V_1 \xrightarrow{\pi_1} V$ so that $R$ lifts to a globally holomorphic map $\tilde R:
V_m \ra W$ making the following diagram commute\footnote{As with diagram (\ref{comm diagram}), commutativity is only at points where all maps are defined.}

%$\tilde R = R \circ \pi_1 \circ \cdots \circ \pi_m$ as
%rational maps.  

\comment{**********
\begin{diagram} 
V_m &  &\\ 
\dTo^{\pi_m} & \rdTo^{\tl R}   &\\ 
\vdots &  & \\
\dTo^{\pi_{2}} &  &\\ 
V_{1} & & \\
\dTo^{\pi_{1}} &  &\\ 
V & \rTo^{R}  & W\\
\end{diagram} 
************}

\begin{diagram}
V_m &  &\\
\dTo^{\pi} & \rdTo^{\tl R}   &\\
V & \rTo^{R} & \CP^n.\\
\end{diagram}
Here, $\pi = \pi_1 \circ \cdots \circ \pi_m$.  
\end{thm}

%\noindent 
%Theorem \ref{THM:RESOLUTION_IND} also applies
%to any rational map $R:V \ra W$, since such a map is obtained from a rational
%map from $V$ to an appropriate $\CP^n$ containing $W$.

Any analytic curve $D$ on $V$ lifts to an analytic curve $\tl D:=
\overline{\pi^{-1}(D \sm \{p\})}$ on $\tl V$, known as the {\em proper
transform} of $D$, which tends to have milder singularities than $D$:  
% (It is readily studied in the local  coordinates $(u,m)$ near the exceptional divisor.)  

\begin{thm}
Let $D$ be an irreducible curve on a non-singular projective surface $V$.
Then, there exist a sequence of blow-ups $V_m \xrightarrow{\pi_m} \cdots
\xrightarrow{\pi_2} V_1 \xrightarrow{\pi_1} V$ so that the proper transform of
$\tl D$ of $D$ is a non-singular curve on $V_m$.
\end{thm}

\noindent
See \cite[Ch. IV, \S 4.1]{SHAF}.

\subsection{Divisors}
\label{APP:DEGREE}
Divisors are a generalization of algebraic hypersurfaces that behave naturally
under dominant rational maps.  We will present an adaptation of material from
from \cite[Ch. 3]{DAN}, \cite[\S 3]{FULTON}, and \cite{SHAF} suitable for our purposes.

A {\em
divisor} $D$ on a projective surface $V$ is a collection of irreducible hypersurfaces
$C_1,\ldots,C_r$ with assigned integer multiplicities $k_1,\ldots,k_r$.  One writes $D$ as a formal sum
\begin{eqnarray}\label{EQN:DIVISOR}
D = k_1 C_1 + \cdots + k_r C_r.
\end{eqnarray} 
\noindent
Alternatively, $D$ can be described by choosing an open cover $\{U_i\}$ of $V$
and rational functions $g_i:U_i \ra \C$ with the comparability property that
$g_i/g_j$ is a non-vanishing holomorphic function on $U_i \cap U_j \neq \emptyset$.
Taking zeros and poles of the $g_i$ counted with multiplicities,
we obtain representation (\ref{EQN:DIVISOR}).
% There is a map taking any such family $\{(U_i,g_i)\}$ to a representation of
% the form (\ref{EQN:DIVISOR}) given by considering points $x \in V$ where the
% local defining functions $g_i$ have either zeros or poles, each assigned
%appropriate multiplicities.  It is a general fact that since $V$ is a
%non-singular complex variety, this map defines an equivalence between the two descriptions.  

These two equivalent descriptions of divisors allow us to pull them back and
push them forward under rational maps: If $f: V \ra W$ is a dominant
holomorphic map, and $D = \{U_i,g_i\}$ is a divisor on $W$.  The {\em pullback}
$f^* D$ is the divisor on $V$ given by $\{f^{-1}U_i,f^* g_i\} \equiv \{f^{-1}
U_i, g_i \circ f\}$.

If $f:V \ra W$ is a proper holomorphic map and $D$ is an irreducible
curve on $V$ we define its push-forward $f_* D$ to be the divisor that assigns
multiplicity\footnote{Note that it is possible for a non-generic point of $f(V)$ to have more
than $\deg_{\rm top}(f :V \ra W)$ preimages under $f|V$.
}                                                         \note{In this case, is it the image of a collapsing line? (M)}         
$\deg_{\rm top}(f|V:V \ra W)$ to the
image curve $f(V)$.  (If $f$ collapses $V$ to a point, we assign multiplicity $0$.)  This definition extends linearly to arbitrary divisors on
$V$ expressed in form (\ref{EQN:DIVISOR}).

If $R: V \ra W$ is a rational map having indeterminacy, we use Theorem
\ref{THM:RESOLUTION_IND} to define the pull-back and push forward of divisors
by 
\begin{eqnarray}\label{EQN:PUSH_PULL}
R^*D_1 := \pi_* \tl R^* D_1 \qquad R_* D_2 := \tl R_* \pi^* D_2
\end{eqnarray}
\noindent
where $D_1$ and $D_2$ are divisors on $W$ and $V$, respectively.  

Alternatively, one can pull-back $D$ under $R: V \sm I(R) \ra W$.  Since $I(R)$
is a finite collection of points, the result (in terms of local defining
functions) can be extended trivially to obtain a divisor $R^* D$ on all of $V$.
Since the trivial extension of a divisor is unique, this alternative definition
agrees with the previous one.

\msk

Any hypersurface $C$ can be triangulated as a singular cycle, thus to any
divisor $D$ is an associated {\em fundamental class} $[D] \in H_2(V)$.
Representing $D$ by local defining functions allows us to associate a
cohomology class $(D) \in H^2(V)$ called its {\em Chern class}; see \cite{FULTON}.
Furthermore, $[D]$ is the Poincar\'e dual of $(D)$.

These classes are natural satisfying $[f_* D_1] = f_* [D_1]$ and $(f^*D_2) =
f^* (D_2)$ for any holomorphic map $f: V \ra W$.  For a rational map $R: V\ra
W$, we again have
\begin{eqnarray}\label{EQN:NATURALITY}
[R_* D_1] = R_* [D_1] \,\, \mbox{and} \,\, (R^* D_2) = R^* D_2,
\end{eqnarray}
using (\ref{EQN:PUSH_PULL}) at the level of homology and cohomology and also
Poincar\'e duality.

\subsection{Composition of rational maps}
The {\em algebraic degree} of a rational map $R: \CP^k \ra \CP^l$, denoted $\deg R$, is
the degree
%\footnote{As always, we assume that the components of $\hat R$ have no common factors.} 
of the coordinates of its lift $\hat R: \C^{k+1} \ra \C^{l+1}$.

The following statement appears in \cite[Prop. 1.4.3]{S_PANORAME}:

\begin{lem}\label{LEM:DEGREE_OF_COMPOSITION}
Let  $R:\CP^k \ra \CP^l$ and $S:\CP^l \ra \CP^m$  be rational maps. Then, $\deg(S
\circ R) = \deg(S) \cdot \deg(R)$ if and only if there is no algebraic
hypersurface $V \subset\CP^k$ that is collapsed
% \footnote{In the case $k=1$, the word ``collapse'' should be replaced with ``mapped''.} 
by $R$ to an indeterminate point of $S$.
\end{lem}

\begin{proof}
Let $\hat R$ and $\hat S$ be lifts of $R$ and $S$ of degree
$\deg R$ and $\deg S$, respectively.
An irreducible algebraic hypersurface $V$ given by $\{p(\hat z)=0\}$
(where $p$ is prime)
is collapsed under $R$ to an indeterminate point for $S$ if and only if $p(\hat z)$
divides each of the components of $\hat S \circ \hat R$.
On the other hand,
$\hat S \circ \hat R$ has a common factor if and only if
$\deg(S \circ R) < \deg S \cdot \deg R$.
\end{proof}

\begin{rem}\label{geom deg deficit}
  To understand this phenomenon geometrically (for simplicity, in dimension two: k=l=m=2), 
let us consider the algebraic curve $G$
to which the indeterminacy point $\gamma$ blows up under $S$. Then any line $L$ must intersect
$G$, and hence $S^{-1} L$ passes through $\gamma$. It follows that $V\subset R^{-1}(S^{-1} L)$.
On the other hand, $V\not\subset (S\circ R)^{-1} L$ (unless $L\supset G$,
which may happen only for  a special line).
But according to Lemma \ref{LEM:PUSH_PULL} below,  
$$
   \deg S \cdot \deg R  = \deg (S^{-1} (R^{-1} L)), \quad  \deg (S\circ R)= \deg (S\circ R)^{-1} L
$$
So, components of $V$, possibly with multiplicities, account for the degree deficit.
\end{rem}

\subsection{Degree of divisors in $\CP^2$}
Associated to any homogeneous polynomial $p:
\C^{3} \ra \C$ is a divisor $D_p$ given by $\{U_i,p \circ \sigma_i\}$, where
the $\{U_i\}$ form an open covering of $\CP^2$ that admits local sections
$\sigma_i: U_i \ra \C^{3} \sm \{0\}$ of the canonical projection $\pi$.
Furthermore, every divisor can be described as a difference $D = D_p - D_q$ for
appropriate $p$ and $q$.  The following simple formula describes the pull-back:
\begin{eqnarray}\label{EQN:PULL_BACK_ON_CP2}
R^* D_p = D_{\hat R^*p} \equiv D_{p \circ \hat R}.
\end{eqnarray}

The {\em degree} of a divisor $D = D_p - D_q$ is $\deg D = \deg p - \deg q$.
Any complex projective line $L$ on which $p$ does not identically vanish
intersects the divisor $D_p$ exactly $\deg D_p$ times (counted with
multiplicity), providing an alternative geometric definition of $\deg D_p$.

More generally, {\em Bezout's Theorem} asserts that two divisors $D_1$ and
$D_2$ intersect $\deg D_1 \cdot \deg D_2$ times in $\CP^2$, counted with
appropriate {\em intersection multiplicities}.  
Suppose that $D_1$ and $D_2$ are irreducible algebraic curves assigned multiplicity one.  Then,
an intersection point $z$ is assigned multiplicity one if and only if both curves are non-singular at $z$,
meeting transversally there.
See \cite[Ch. IV]{SHAF}.

There is an alternative, homological, definition for $\deg D$. 
Namely,  any algebraic curve $D$ represents a class $[D]\in H_2(\CP^2)$. 
Moreover, $H_2(\CP^2) = \Z$ and is generated by the class $[L]$ of any line $L$. Then we have 
\begin{eqnarray}\label{EQN:HOMOLOGICAL_DEFN_DEGREE}
[D] = \deg D \cdot [L]
\end{eqnarray}

\begin{lem}\label{LEM:PUSH_PULL} Given a dominant rational map $R:\CP^2 \ra \CP^2$ 
and any divisor $D$ on $\CP^2$ we have: 
\begin{eqnarray}
\deg (R^* D) &=& \deg R \cdot \deg D, \,\, \mbox{and} \label{EQN:DEGREE_PULL_BACK} \\
\deg (R_* D) &=& \deg R \cdot \deg D. \label{DEGREE_PUSH_FORWARD}
\end{eqnarray}
In particular, $\deg (R^* L) = \deg(R_* L) = \deg R$ for any projective line $L \subset \CP^2$.
\end{lem}

\begin{proof}

Equation (\ref{EQN:DEGREE_PULL_BACK}) follows immediately from
(\ref{EQN:PULL_BACK_ON_CP2}).  To obtain (\ref{DEGREE_PUSH_FORWARD}) we make
use of the homological interpretation of degree
(\ref{EQN:HOMOLOGICAL_DEFN_DEGREE}).  By (\ref{EQN:NATURALITY}), the push-forward of divisors $R_*$
preserves homology, 
so it suffices to check (\ref{DEGREE_PUSH_FORWARD}) for any
complex projective line $L$.

We choose $L$ disjoint from $I(R)$ and we can assume that $[0:0:1] \not \in R(L)$.
Let $\iota: \CP^1 \ra \CP^2$ be the inclusion of $L$ into $\CP^2$ and ${\rm
pr}: \CP^2 \ra \CP^1$ the central projection onto the line at infinity
$L_{\infty}$ from the center $[0:0:1]$.  Note that both $i_*$ and ${\rm pr}_*$ induce
isomorphisms on the second homology.  

%Consider the composition
%\begin{eqnarray*}
%\CP^1 \xrightarrow{\iota} \CP^2 \xrightarrow{R} \CP^2 \xrightarrow{{\rm pr}} \CP^1
%\end{eqnarray*}
%\noindent

We consider the composition ${\rm pr} \circ R \circ \iota: \CP^1 \ra \CP^1$.  By Lemma \ref{LEM:DEGREE_OF_COMPOSITION},
\begin{eqnarray*}
\deg({\rm pr} \circ R \circ \iota) = \deg{\rm pr}\cdot \deg R \cdot \deg \iota = \deg R,
\end{eqnarray*}
since the image of $\iota$ avoids $I(R)$ and the image of $R \circ \iota$ avoids $I({\rm pr}) = \{[0:0:1]\}$.

Thus,
\begin{eqnarray*}
{\rm pr}_* \circ R_* [L] = {\rm pr}_* \circ R_* \circ \iota_* [\CP^1]= \deg({\rm pr} \circ R \circ \iota) [\CP^1] = \deg R \cdot [\CP^1]
\end{eqnarray*}
Since ${\rm pr}_*: H_2(\CP^2) \ra H_2(\CP^1)$ is an isomorphism we find $[R_* L] = \deg R \cdot [L_\infty]$.
\end{proof}

\begin{rem}
Formulas (\ref{EQN:DEGREE_PULL_BACK}) and (\ref{DEGREE_PUSH_FORWARD}) generalize
to other varieties using the fact that pull-back
and push-forward (under suitable maps) preserve {\em linear equivalence of
divisors,}  
% While the former is easy to find in the literature, the latter is more esoteric;
see \cite[Ch. 3 \S 5.2]{DAN} and \cite[\S 3.3]{FULTON}.  
\end{rem}

\subsection{Iteration of rational maps}
A rational mapping $R : \CP^2 \rightarrow \CP^2 $ is called {\em algebraically
stable} if there is no integer $n$ and no collapsing hypersurface $V \subset
\CP^{2}$ so that $R^n(V)$ is contained within the indeterminacy set of $R$,
\cite[p. 109]{S_PANORAME}.  
By Lemma \ref{LEM:DEGREE_OF_COMPOSITION},  $R$ is algebraically stable if and only
if $\deg R^n = ( \deg R )^n$.         
Together with Lemma \ref{LEM:PUSH_PULL}, this yields:

\begin{lem}\label{LEM:PUSH_PULL_AS}
          If $R$ is an algebraically stable map and $D$ is any divisor on $\CP^2$ we have:  
\begin{eqnarray*}\label{EQN:PULL_BACKS_AS}
\deg ((R^n)^* D) &=& (\deg R)^n \cdot \deg D \\  % \,\, \mbox{and} \\
\deg ((R^n)_* D) &=& (\deg R)^n \cdot \deg D 
\end{eqnarray*}
\end{lem}

\section {Renormalization near the indeterminacy points}
\label{APP:R_NEAR_ALPHA}

\subsection{Blow-ups}

Below we calculate blow-ups for the renormalization      % in various coordinates:
% $\INDphys_\pm$ for $\Rphys$ (and $\INDphys$ for $F$) first
in the affine coordinates $(u,v)$           % , in the physical coordinates $(z,t)$,
and in the angular coordinates $(\phi,t)$.
%\footnote{In the angular coordinates, we compute the real blow-up, which is defined analogously to the above.}

\sss{Affine coordinates}

Let us represent $R$ as $Q\circ g$ where
$$
    g: (u,w) \mapsto  \left( \frac{u^2+1}{u+w}, \frac{w^2+1}{u+w} \right)
$$
and $Q: (u,w)\mapsto (u^2, w^2)$. The indeterminacy points for $R$ and $g$ are
the same, $a_\pm= \pm (i,- i)$.
 Because of the basic symmetry $(u,w)\mapsto (w,u)$,
it is sufficient to carry the calculation at $a_+=(i,-i)$.
In  coordinates $\xi = u-i$ and $\chi =  (w+i)/(u-i)$,
we obtain the following expression for the map $\tl g: \tl\CP^2\ra \CP^2$
near  $L_\ex(\INDmig_+)$:
\begin{equation}\label{blow up near a+}
    u = \frac{\xi+2i}{1+\chi},\quad  w = \frac{\chi^2 \xi-2i \chi}{1+\chi}.
\end{equation}
So $L_\ex(\INDmig_+)= \{\xi=0\}$ is mapped by $\tl g$ biholomorphically  onto the line
\begin{equation}\label{blow-up line}
        \left\{ u= \frac{2i}{1+\chi},\ w= -\frac{2i\chi}{1+\chi} \right\} = \{u-w=2i\}.
\end{equation}
In other words, $g$ blows up the indeterminacy point $a_+$ to  line (\ref{blow-up line}).
Notice that this line connects $a_+$ to the low-temperature fixed point $b_0=(1:0:1)$
at infinity. Its slice by the Hermitian plane $C=\{w=\bar u\}$ (corresponding to the cylinder $\Cphys$)
is the horizontal line $\{ \Im u =1\}$. 

The lift $\tl \Rmig$ of $\Rmig$ to $L_\ex(\INDmig_+)$ is given by $\tl \Rmig = Q \circ \tl g$, and hence obtained
by squaring the expressions for $u$ and $w$ in (\ref{blow up near a+}).

\sss{Angular coordinates}  (compare \cite[p. 419]{BZ3}).
We will now calculate the blow-up of $\Rphys$ at the indeterminacy points $\alpha_\pm$ 
in the angular coordinates $(\phi,t)$ on $\Cphys$.
%They can be obtained from (\ref{EQN:F_BLOWUP_COMPLEX}) and
%(\ref{EQN:R_BLOWUP_COMPLEX}), or by a direct computation.  The calculations are
%nearly identical for $F$ at $\INDphys$ and for $\Rphys$ at $\INDphys_\pm$, so
%we omit the former.
As before, it suffices to consider $\INDphys_+$.  We let $\eps =
\frac{\pi}{2}-\phi$ and $\tau = 1-t$, and  $\kappa= \tau/\eps$. 
In the blow-up coordinates  $(\eps,\kappa)$ we find:

\begin{eqnarray}\label{EQN:R_BLOWUP_ANGULAR_WITH_EPS}
(z',t') = \left(\frac{-i+\kappa -  \eps -  \eps\kappa^2/2}
                      {i+\kappa -  \eps -   \eps\kappa^2/2}, \
   \frac{1-2\eps \kappa}
        {1+\kappa^2-\eps(2\kappa+\kappa^3)}\right)+O(|\eps|^2),
\end{eqnarray}
where the constant in the residual term depends on an upper bound on $\kappa$. 

\newcommand{\arctg}{\operatorname{arctg}}
\newcommand{\arcctg}{\operatorname{arcctg}}

\noindent
Thus
\begin{eqnarray}\label{EQN_PHI_PRIME_BLOWUP}
\phi' = -i \log(z') = - 2 \arcctg(\kappa- \eps + \eps\kappa^2/2) + O(|\eps|^2).
\end{eqnarray}
% $$
%  2i\arctg z =  \log\frac{z-i}{z+i}
% $$
%   $$ \arcctg z= \pi/2 -\arctg z $$

\noindent
Taking the limit as $\eps \ra 0$, we find:
\begin{eqnarray}\label{EQN:R_BLOWUP_ANGULAR}
(\phi',t') = \left(- 2\arcctg \kappa,\ \frac{1}{1+\kappa^2}\right)
\end{eqnarray}

\noindent
Letting  $\om$ be the (complexified) angle between the collapsing line $\II_{\pi/2}$ and the line with slope
$\kappa= - \ctg \om$, we come up with expression (\ref{Icurve}) for the blow-up locus $\Icurve$. 
(Notice that the blow-up loci for the maps $f$ and $\Rphys=f\circ Q$ coincide, since $Q$ is a local diffeomorphism.)

Recall that point $\INDphys = (\pi,1)$, which mapped by $\Rphys$ to the high
temperature fixed point $\beta_1$.
\begin{lem}\label{PROP:DISTANCE_TO_PI}
Let $c>0$. 
If $\zeta = (\eps,\tau)\in \C^2$ is sufficiently close to one of the indeterminacy points $\INDphys_\pm$ 
with the slope $\kappa = \tau/\eps$ sufficiently small
and $|\eps-\kappa|\geq c|\eps|$, then
\begin{equation*}
\dist^h(\Rphys(x),\INDphys) \asymp |\eps-\kappa|,
\end{equation*}
with the constant depending on $c$. 
\end{lem}

\begin{proof}
  Indeed, under our assumptions,  formula (\ref{EQN_PHI_PRIME_BLOWUP}) implies 
$
  |\phi' -  \pi  |\asymp |\kappa-\eps|.
$
\end{proof}

\subsection{The differential $D\Rphys$.}
Formula (\ref{Df}) implies the following
explicit expression for the differential $D\Rphys$ at  $x=(\phi,t)\in \Cphys$: 

{\footnotesize
% \begin{eqnarray}\label{EQN:DR}
% && D\Rphys =  \nonumber \\ 
%&& \hspace{0.7in}  
\begin{equation}\label{EQN:DR} 
  D\Rphys = \frac{4}{\zeta^2} 
                        \left(
                         \begin{array}{cc}  
                                                   \zeta & 0 \\
                                                    0 & 1- t^2
                         \end{array}   \right)
                         \left(
                          \begin{array}{cc}
                                                  1+t^2 \cos 2\phi & - \sin 2\phi \\
                                                 -t^2(1-t^2) \sin 2\phi & (1+t^2)(1+\cos 2\phi))
                          \end{array} \right)
                                      \left(  
                        \begin{array}{cc}   
                                                    1 & 0 \\
                                                    0 & t
                         \end{array}   \right)  
\end{equation}
% \end{eqnarray}
}

\noindent
where $\zeta(x) = 1+2t^2\cos 2\phi + t^4$. 

% To study projective properties  of $D\Rphys$, it is more convenient to deal with
% $A:= \zeta^2 D\Rphys/8$.
Expanding it in $\tau = 1-t$ near $\TOPphys$, we obtain: 
{\small
\begin{eqnarray}\label{EQN:A_EXPANSION}
           D\Rphys = 
\left(
       \begin{array}{cc}  2          & -2\tg \phi \\
                        -2\tau^2 \tg\phi\, (\cos\phi)^{-2}   & -\tau (\cos\phi)^{-2} 
\end{array}\right)
   (I+O(\tau))
\end{eqnarray}
}

In the $\epsilon = \pm \pi/2-\phi$ % $\tau(x) = 1-t(x)$ 
coordinate near an indeterminacy point $\alpha_\pm$  we
obtain the following asymptotic expression for the differential 
$\Rphys: (\eps, \tau)\mapsto  (\phi', \tau')$:

\begin{equation}\label{DR near alpha}
D\Rphys \sim   \frac 2 {\sigma^4}
 \left(
                         \begin{array}{cc}
                                                 (\eps^2+\tau) \sigma^2 & -\eps \sigma^2 \\
                                                 -\eps\tau^2 & \tau\eps^2
                         \end{array}   \right)
\end{equation}
where $\sigma=\sqrt{\eps^2+\tau^2}$.

\ssk

\comm{*****************

\begin{lem}\label{cones near alpha}
  Let $1<c<3/2$, and let  $x=(\eps, \tau)\in \Cphys\sm T$ be a point near $\INDphys_+$,
and let $x'=(\phi', \tau')=\Rphysx$.
If $v\in T_x\Cphys $ is a tangent vector at $x$   with slope bounded
(in absolute value) by $c\sqrt{\tau}$
then its image  $D\Rphys(x) v$ has slope bounded by $c\sqrt{\tau'}$.
\end{lem}

\begin{proof}
For $v=(1, \theta)$ we have:
Hence
$$
   D\Rphys(x)v  =
\frac  8 {|x|^4}    \left( \begin{array}{c} |x|^2(\eps^2+\tau -\theta\eps)  \\
                             -\eps\tau (\tau-\theta\eps)  \end{array}     \right)
$$
For $\theta= c\sqrt{\tau}$, the slope of $D\Rphys(x) v$ is equal
$$
   \theta' = -\frac{\eps\tau(\tau-c\sqrt{\tau} \eps)} {|x|^2(\eps^2+\tau-c\sqrt{\tau}\eps)}.
$$
  Since $\tau' =\tau^2/\eps^2$,  \note{write before}
we need to check that $|\theta'|< c \tau/|\eps|$.
Since $\eps^2<|x^2|$, it is enough to verify that
\begin{equation}\label{eps-tau}
  \frac{|\tau-c\sqrt{\tau} \eps| } {\eps^2+\tau-c\sqrt{\tau}\eps} < c.
\end{equation}
Note that denominator of this expression is a quadratic form in $(\eps, \sqrt{\tau})$
which is positive definite for $c\in (0,2)$.
Hence (\ref{eps-tau}) amounts to the following system of inequalities
$$ c\eps^2 - c(c+1)\eps\sqrt{\tau} +(1+c)\tau > 0, $$
$$ c\eps^2 - c(c-1)\eps\sqrt{\tau} +(c-1)\tau > 0.$$
These are two quadratic forms in $(\eps, \sqrt{\tau})$ that are  positive definite
in our range of $c$.

\end{proof}

*********}
 
%\subsection{Horizontal stretching near the indeterminacy points $\alpha_\pm$}
\subsection{Horizontal stretching near $\TOPphys$}
\label{APP:HOR_STRETCHING_NEAR_TOP}

Let us define the {\it horizontal expansion factor} $\hexp(x)$ at $x\in \Cphys\sm \{\alpha_\pm\}$ as
\begin{eqnarray}\label{EQN:DEF_HEXPLOW}
   \hexp(x) = \inf_{v \in \KK^h(x)} \frac{\pi (D\Rphys(v)) } {\pi(v) }.
\end{eqnarray}
Equivalently, consider an almost horizontal curve $\xi$ through $x=(\phi,t)$
 naturally  parameterized by the angular coordinate (by means of $(\pi|\xi)^{-1}$).
Let
$$\chi\equiv \chi_\xi = \pi\circ\Rphys\circ(\pi|\xi)^{-1}.$$
Then
$$
   \hexp(x) = \inf_\xi \chi_\xi' (\phi).
$$
The {\em $n$-th horizontal expansion factor} $\hexpmin_{,n}(x)$ is defined similarly, by replacing $D\Rphys$ with $D\Rphys^n$ in (\ref{EQN:DEF_HEXPLOW}).

\begin{lem}\label{hexplow}
There exists $\lambda_0 > 0$ so that for any
$x=(\eps, \tau)\in \CC\sm\{\alpha_\pm\}$ near one of the indeterminacy points $\alpha_\pm$ we have
$\hexpmin(x) \geq \lambda_0$.

Moreover, given a slope $\bar \kappa > 0$, there exits $\lambda_1 > 0$ such that 
if $x=(\eps,\tau)$ also satisfies $|\kappa| \equiv |\tau/\eps| \leq \bar \kappa$, then
\begin{eqnarray*}
   \hexpmin(x) \geq \lambda_1 \left |\frac{\kappa+\eps}{\eps} \right|.
\end{eqnarray*}
\end{lem}

\begin{proof}
  Take a horizontal vector $v=(1,s)\in \KK^h(x)$ with slope $s$. 
By definition of the horizontal cone field $\KK^h$, 
$|s| \leq \max \{\sqrt{2\tau}, |\eps|/3\}$ 
(see  \S \ref{SUBSEC:MODIFIED_ALG_CONES}, items (iii)-(iv) in the definition of  $\KK^h$).                    
Applying the asymptotical expression  (\ref{DR near alpha}), we obtain 
\begin{eqnarray*}
   \chi'(\phi) &=& \pi(D\Rphys (1,s)) = 2 \frac { \eps^2+\tau- s\eps }{|\si|^2}\\
&\geq& \frac {2} {|\si|^2}   \min \{ \eps^2+\tau-|\eps| \sqrt{2\tau}, \  (2/3)\eps^2+\tau\}
    \geq   \lambda_0 \frac{\eps^2+\tau}{\eps^2+\tau^2} \geq \lambda_0.
\end{eqnarray*}
The second to last estimate follows from positive definitiveness of the  quadratic form 
$\eps^2+\tau-\eps\sqrt{2\tau}$ in $\eps$ and $\sqrt{\tau}$.  
Finally,
$$
 \frac{\eps^2+\tau} {\eps^2+\tau^2} = \frac {|\kappa+\eps|}{|\eps|(\kappa^2+1)} \geq 
     \frac 1{\bar\kappa^2+1} \left|\frac{\kappa+\eps}{\eps}\right|,
$$
and the second estimate follows.
\end{proof}

\begin{lem}\label{LEM:BOUNDED_CONTRACT}
There exists $\lambda_2 > 0$ so that for any $x \in \Cphys \sm \{\INDphys_\pm\}$ we have $\hexpmin(x)~\geq~\lambda_2$.
\end{lem}

\begin{proof}
Away from $\{\alpha_\pm\}$ this is true since the horizontal cones are
transverse to the critical lines $\II_{\pm \pi/2}$.   Near  $\{\alpha_\pm\}$,
it follows from Lemma  \ref{hexplow}.
\end{proof}

\msk

We now estimate horizontal expansion of vectors (and hence curves) taken with respect to the algebraic cone field $\KK^\hor$.  It will be useful to consider both upper and lower bounds:
\begin{eqnarray*}
\horexpmin(x) &=& \inf_{v \in \KK^\hor(x)} \frac{\pi (D\Rphys(v)) } {\pi(v) }, \,\, \mbox{and} \,\,\, \horexpmax(x) = \sup_{v \in \KK^\hor(x)} \frac{\pi (D\Rphys(v)) } {\pi(v) }
\end{eqnarray*}
The $n$-th expansion factors $\horexpmin_{,n}(x)$ and $\horexpmax_{,n}(x)$ are defined similarly.
Note that $\horexpmin(x) \geq \hexpmin(x)$, since $\KK^\hor(x) \subset \KK^h(x)$.  In particular the estimate of Lemma \ref{hexplow} also applies to $\horexpmin(x)$.

\begin{lem}\label{LEM:STRETCHING_NEAR_TOP}
For any $\delta > 0$ there exist $\eta > 0$ and $\bar \tau > 0$
such that for any  $x \in \VV'$  we have:
\begin{eqnarray*}
(2-\delta)  < \horexpmin(x) < \horexpmax(x)  < (2 + \delta).
\end{eqnarray*}
\end{lem}

\begin{proof}
The slope a vector $v \in \KK^\hor(x)$ is bounded by $a \sqrt{\tau}$, where $a=1.45$ and $\tau = \tau(x)$.
%  was chosen previously in \S \ref{SUBSEC:WEAK_CONES}.

Near the indeterminacy points  $\INDphys_\pm$ we can use (\ref{DR near alpha})
to bound from below the horizontal stretching of the boundary  tangent vector  $v = (1,\pm a \sqrt{\tau})$:
\begin{eqnarray*}
2\, \frac {\eps^2+\tau -  a |\eps| \sqrt{\tau}} {\epsilon^2+\tau^2}
      <  \pi (D\Rphys (1,s)) < 2\, \frac {\eps^2+\tau +  a|\eps| \sqrt{\tau}} {\epsilon^2+\tau^2}.
\end{eqnarray*}
Since $x \in \VV'$, we have $\sqrt{\tau} < \sqrt{\eta}|\epsilon|$ which gives:
$$
(1-a\sqrt {\eta}) (\eps^2+\tau)  \leq \eps^2+\tau \pm  a|\eps| \sqrt{\tau} \leq (1+a\sqrt {\eta}) (\eps^2+\tau).
$$
Hence $$ 2(1-a\sqrt{\eta}) \leq   \pi (D\Rphys (1,s))  \leq 2(1+a\sqrt{\eta}),$$
which can be made arbitrary close to 2 if one's  $\eta$ is chosen  sufficiently small.

Now suppose that $v$ is based $\bar\eps$-away from $\INDphys_\pm$.
Then $\tg\phi\leq 2/\bar\eps$ and (\ref{EQN:A_EXPANSION}) implies:
\begin{align}\label{EQN:CONTROL_OF_EXPANSION_AWAY_ALPHA}
2(1 -  2a\sqrt{\tau}/\bar\eps)(1+O(\tau)) \leq     \pi (D\Rphys (1,s))  \leq 2(1+  2a\sqrt{\tau}/\bar\eps)(1+O(\tau)),
\end{align}
which can be made arbitrary close to 2 by choosing $\tau$ small enough.
\end{proof}

\section{Complex extension of the cone fields}
\label{SUBSEC:COMPLEXIFICATION_CONES}

\subsection{Terminology and notation}
Let $\CC^c :=(\C /2\pi) \times \C$ be the complexification of the full cylinder 
$(\R /2\pi) \times \R$. 
% Given complex neighborhood $U$ of $(\R /2\pi)$ and
%$V$ of $[0,1]$ we consider complex neighborhoods of $\Cphys$ of the form
%$\Cphys^c = V \times U$.

A real horizontal cone field over $\Cphys$ is given as
\begin{eqnarray}\label{EQN:REAL_CONES}
\KK(x) =  \{ v= (dt, d\phi) \in T_x \Cphys : |dt| < s(x) \, |d \phi |\}\subset T_x \Cphys
\end{eqnarray}
\noindent
for an appropriate slope function $s(x) \geq 0$.  

The cones $\KK(x)$ can be {\it complexified} to the complex cones $\KK^c(x)\subset T_x\CC^c$ 
by means of the same formula (\ref{EQN:REAL_CONES}) where $d\phi$ and $dt$ are interpreted as complex
coordinates in $T_x\CC$. Notice that $\KK^c(x)$ is obtained by rotating $\KK(x)$ by multiplications
$v\mapsto e^{i\theta} v$. Since $\RR$ commutes with this action,
invariance of a real cone field $\KK(x)$ implies invariance of its complexification.

% Any {\it complex}  horizontal tangent cone field can be defined by replacing the absolute values in (\ref{EQN:REAL_CONES}) with moduli:
% \begin{eqnarray}\label{EQN:COMPLEX_CONES}
%   \KK^c(\zeta) =  \{ v \in T_x \Cphys^c : |dt (v)| < s^c(\zeta) |d \phi(v)|\}\subset T_\zeta\Cphys^c.
%\end{eqnarray}
%\noindent
% for appropriate function $s^c(\zeta) > 0$ defined on $\Cphys^c$.

%\ssk
%The real conefield $\KK(x)$ can be converted to a complex conefield (at the moment defined only on the real slice $\Cphys$) by replacing
%the absolute values in (\ref{EQN:REAL_CONES}) with moduli.   We call this the {\em complexification} of $\KK(x)$ and denote it by $\KK^c(x)$.

%Suppose $\KK(x)$ is invariant under $\Rphys$.  Considering the projective tangent spaces as in \S \ref{SUBSEC:DOMINATED},
%for $x \in \Cphys$, $\Rphys$ acts as a real projective transformation of the
%complex projective tangent space $P T_x \Cphys^c$.  Since $\Rphys(P \KK(x))
%\Subset P \KK(\Rphys(x))$ and since $\Rphys$ sends real-symmetric circles to
%real-symmetric circles, we find $\Rphys(P \KK^c(x)) \Subset P \KK^c(\Rphys x)$
%for any $x \in \Cphys \subset \Cphys^c$.  
%Thus, the complexification $\KK^c(x)$ of the invariant real cone field $\KK(x)$ is also invariant.
%\note{Say more or less?}

%\ssk
%Suppose $s^c$ is defined on all of $\Cphys^c$ and satisfies 
%$s^c(\zeta) = s(\zeta)$ for all $\zeta \in \Cphys$.  Then, we say that $\KK^c(\zeta)$ 
%is a {\em complex extension} of $\KK^c(x)$
%from $\Cphys$ to all of $\Cphys^c$.

So, let us complexify the cone field $\KK^\hor(x)$ (for $x \in \Cphys$).     
(We will typically omit the superscripts $c$ from the complexification, to simplify the notation.)
We can further extend this cone field to a neighborhood of $\CC$ in $\CC^c$ by
extending continuously the slope function $s$.  By continuity, the extension of
$\KK^\hor$ is invariant away from the top.  However, for the application to
distortion control in \S \ref{SEC:TWO BASINS}, we will need an extension that
is invariant on a suitable ``pinched'' neighborhoods of the $\INDphys_\pm$.

\subsection{Complex extension of $\KK^\hor$}
Define an extension of $\KK^\hor \equiv \KK^{\hor,c}$ to $\CC^c$ by letting
$s^a(\zeta)  = \sqrt{|1-t^2|}= \sqrt{|\tau(2-\tau)|}$ (compare with (\ref{slope on YY})).  
For $\rho>0$, let 
\begin{align}\label{EQN:DEFN_CPHYS_C}
  \Cphys^c_\rho := \{(\phi,t)\in \Cphys^c \, : \,  |\Im \phi| < \rho, \, -\rho<  \Re t < 1+\rho ,\, \mbox{and} \, |\Im t| < \rho \}.
\end{align}
To ensure invariance of the cone field $\KK^\hor$,
we will need to appropriately ``pinch'' $\Cphys^c$ near the points of indeterminacy $\INDphys_\pm$.
Given $\theta > 0$, let
\begin{eqnarray*}
    \NN(\theta) := \{(\eps,\tau) \ : \ | \arg \eps \, (\mod \pi)| < \theta  \mbox{  and  }
                           |\arg \tau \, (\mod 2\pi) | < \theta  \} \cup \{|\eps|< \frac 12 |\tau| )\}.
\end{eqnarray*}
(Note that  in the first set, $\eps$ is allowed to  be negative, while $\tau$ is not.) 
\noindent
% Given $\hat \rho > \rho > 0,$ and $\delta > 0$ we 
Furthermore, let
\begin{eqnarray*}
{\Cphys}^c_\rho(\theta) := \{\zeta=(\eps, \tau) \in \Cphys^c_\rho \ : \  \zeta \in \NN(\theta) \ \mbox{whenever} \
|\eps | < \rho \ \mbox{and} \ |\tau | <  \rho\}.
\end{eqnarray*}

\comment{********************************
Note that horizontal width of  ${\widehat \Cphys}^c$ degenerates near $\INDphys_\pm$ with size of order
$\dist^h(x, \INDphys_{\pm})$ and ${\widehat \Cphys}^c$ has definite horizontal width at all points away from 
$\INDphys_\pm$.  {\em The way that the size of this neighborhood degenerates
near $\INDphys_\pm$ plays an essential role in \S \ref{}.}
************}

%\note{Remark on why this pinched neighborhoods are sufficient for pf of expansion?}

\begin{prop}\label{PROP:CX_EXTENSION_ALG_FIELD}
There exist $ \rho > 0$ and $\theta > 0$ sufficiently small so that if $\zeta \in \Cphys^c_\rho(\theta)$ then 
\begin{eqnarray}\label{EQN:INV_CX_ALG_CONES}
D\Rphys (\KK^\hor(\zeta)) \subset \KK^\hor(\Rphys \zeta).
\end{eqnarray}
\end{prop}

\begin{proof}
Since $\KK^\hor(x)$ is invariant on the real cylinder  $\Cphys$ 
and non-degenerate on any $K \Subset \Cphystl$, 
the extension is invariant on a complex neighborhood of $K$.
Thus, we need only find $\rho > 0$ and  $\theta >0$
sufficiently small so that (\ref{EQN:INV_CX_ALG_CONES}) holds at points in
$\Cphys^c_\rho(\de)$ with $|\tau| < \rho$.

For $c$ slightly above $\sqrt{2}$,
consider the following auxiliary complex conefield
\begin{eqnarray} 
\KSQRT (\zeta) = \{v =(dt, d\phi) \in T_\zeta \Cphys^c : | dt | <  c \sqrt{|\tau|} |d\phi | \}
\subset T_\zeta \Cphys^c. 
\end{eqnarray}
\noindent
If $\rho > 0$ is sufficiently small, then we have $\KK^\hor(\zeta) \subset
\KSQRT(\zeta)$ for all $\zeta \in \Cphys^c$.  
% To simplify notation we write $c := 1.45$.
Hence it is sufficient to verify that 
if $\zeta \in \Cphys^c_\rho(\theta)$ with $|\tau(\zeta)| < \rho$, then we have
$\Rphys(\KSQRT(\zeta)) \subset \KK^\hor(\Rphys \zeta)$.
It will be shown in Lemmas \ref{LEM:INV_NEAR_ALPHA} and
\ref{LEM:INV_AWAY_FROM_ALPHA} below.
%The real difficulty is near the indeterminate points $\INDphys_\pm$ so we begin
%our work there.  
By symmetry, it suffices to work near $\INDphys_+$.

\begin{lem}[{\bf Invariance near $\INDphys_\pm$}]\label{LEM:INV_NEAR_ALPHA}
There exist $\rho>0$ and  $\theta>0$ with the following property.
For any point 
$\zeta = \left(\frac{\pi}{2}-\epsilon,1-\tau\right)$ 
with  $|\tau| <  \rho$ and $|\epsilon| <  \rho$
lying in the pinched neighborhood $\NN(\theta)$ of $\INDphys_+$ we have:  
$D\Rphys(\KSQRT (\zeta)) \Subset \KK^\hor(\Rphys \zeta)$.
\end{lem}

\begin{proof}
% We will write the condition that $D\Rphys(\KSQRT
% (\zeta)) \subset \KK^\hor(\Rphys \zeta)$ to lowest order in $\epsilon$ and
% $\tau$, verifying that it holds if $(\eps,\tau) \in
% \NN^{1/6}$.  Then, since this invariance is an open condition, 
% we can find $\hat \rho > 0$ so that
% $|\epsilon|,|\tau| < \hat \rho$ assures that $D\Rphys(\KSQRT) \subset
% \KSQRT(\Rphys \zeta)$.
%
According to the blow-up formula (\ref{EQN:R_BLOWUP_ANGULAR}), if $\rho$ is
sufficiently small then $\zeta$ is mapped by $\Rphys$ to a point with
$\displaystyle{\tau' \approx \frac{\kappa^2}{1+\kappa^2}}$, where $\kappa =\tau /\epsilon$.

Let $v = (1,s)$ be a complex tangent vector based at $(\tau,\epsilon)$ with $v
\in \cl( \KSQRT(\zeta))$, i.e., $|s| \leq c \sqrt{|\tau|}$.  We want to show
that the slope $s'$ of the image vector $D\Rphys (v)$ satisfies
$$
  |s'| <  \sqrt{|\tau'(2-\tau')|} \approx \frac{|\kappa|}{|1+\kappa^2|} \sqrt{|2+\kappa^2|}.
$$

Equation (\ref{DR near alpha}) from the Appendix \ref{APP:R_NEAR_ALPHA} gives us the matrix
$A:= \frac{(\epsilon^2+\tau^2)^2}{2} D\Rphys$ to lowest order terms in $\epsilon$
and $\tau$: 
\begin{eqnarray*}
A v =  \left[ \begin{array}{c} (\tau + \epsilon^2- s\epsilon)(\tau^2+\epsilon^2) \\ -\epsilon \tau^2 + s\epsilon^2 \tau \end{array} \right]
%\propto \left[\begin{array}{c} 1 \\ \frac{\epsilon \tau (a \epsilon - \tau)}{(\tau^2+\epsilon^2)(\tau+\epsilon^2-a\epsilon)} \end{array} \right]
\end{eqnarray*}
Thus 
$$
 s' \approx \frac{\epsilon \tau (s \epsilon - \tau)}{(\tau^2+\epsilon^2)(\tau+\epsilon^2-s\epsilon)} = 
               \frac{|\kappa|}{|1+\kappa^2|} \frac{s\eps - \tau}{\tau+\epsilon^2-s\epsilon}
$$ 
so that
$D\Rphys (\KSQRT(\zeta)) \Subset \KK^\hor(\Rphys \zeta)$ is equivalent to:
\begin{eqnarray*}
% \left| \frac{\epsilon \tau (s \epsilon - \tau)}{(\tau^2+\epsilon^2)(\tau+\epsilon^2-s \epsilon)} \right| <  \frac{|\kappa|}{|1+\kappa^2|} \sqrt{|2+\kappa^2|}
%\Leftrightarrow
\left| \frac{s \epsilon - \tau}{\tau + \epsilon^2 -s \epsilon} \right| <  \sqrt{|2+\kappa^2|}
\quad \mbox{whenever}\ |s| \leq c \sqrt{|\tau|}.
\end{eqnarray*}

The condition $(\epsilon,\tau) \in \NN(\theta)$ with $\theta$ sufficiently small implies $\sqrt{2} \leq \sqrt{|2+\kappa^2|}$, so
it suffices to show that
\begin{eqnarray}\label{EQN:CONE_FIELD_CONDITION} 
  \left| \frac{s \epsilon - \tau}{\tau + \epsilon^2 -s \epsilon} \right| < \sqrt{2} 
\quad \mbox{whenever}\ |s| \leq  c \sqrt{|\tau|}.
\end{eqnarray}

Because  $D\Rphys$ maps cones to cones, we need only check that the boundary of $\KSQRT(\zeta)$ is mapped into $\KK^\hor(\Rphys \zeta)$.   
Thus, we substitute
$s = e^{i\theta} c \sqrt{|\tau|}$ into (\ref{EQN:CONE_FIELD_CONDITION}) 
obtaining:
\begin{eqnarray}\label{bnd estimate}
\left| e^{i\theta} c \sqrt{|\tau|} \epsilon - \tau \right| <  \sqrt{2} \left| \tau + \epsilon^2 - e^{i\theta} c \sqrt{|\tau|} \epsilon \right|
\quad \mbox{for all}\ \theta \in \mathbb{R}/2\pi \Z.
\end{eqnarray}
For real $\eps$, $\tau>0$, and $\theta=0$, 
this inequality is equivalent to postivity of a certain quadratic forms in $(\eps, \sqrt{\tau})$,
which is straightforward to check.  

For complex variables, let us square both sides:
\begin{eqnarray*} 
|\tau|^2 -  2c \sqrt{|\tau|}{\rm Re}(e^{i\theta}\overline{\tau}\epsilon)  + c^2|\tau||\epsilon|^2 + 4{\rm Re}(\tau\overline{\epsilon}^2) 
-4c \sqrt{|\tau|} |\eps|^2 {\rm Re}(  e^{i\theta} \overline{\epsilon} )
+2 |\epsilon|^4 >0. 
\end{eqnarray*}

The hypothesis $(\epsilon,\tau) \in \NN(\theta)$ with $\theta$ sufficiently small 
gives $4{\rm Re}(\tau \overline{\epsilon}^2) > \gamma |\tau||\epsilon|^2$,
where $\gamma<4$ is arbitrary close to $4$.
The other two terms with real parts we can replace with absolute values
% by $-2c\sqrt{|\tau|}|\tau||\epsilon|$ and $-4c \sqrt{|\tau|} |\epsilon|^3$, 
obtaining:
\begin{eqnarray*}
 |\tau|^2 -2c|\tau|^{3/2}|\epsilon| +(c^2+\gamma)|\tau||\epsilon|^2  -4c |\tau|^{1/2}|\eps|^3 +2|\epsilon|^4 >0 
\end{eqnarray*}
In the variables $u = |\tau|^{1/2}$ and $v = |\epsilon|$, this reduces to positivity of the real quartic form
\begin{eqnarray}\label{quartic form}
  u^4 -2c u^3v+ (c^2+\gamma)u^2v^2 -4c uv^3 + 2 v^4 > 0 .
\end{eqnarray}
Notice that this inequality with $\gamma=4$ is equivalent to (\ref{bnd estimate}) for real $\eps$, $\tau>0$, and $\theta=0$.  
Hence (\ref{quartic form}) is valid for $\gamma=4$,
and by continuity,  it is also valid for $\gamma$ close to $4$. 
%that we wish to show is positive definite for $c=1.45$.  An explicit check using Sturm
%Sequences, for example, gives the result. 
%It is sufficient to check that
%$$
%    z^4-4cz^3+ (c^2+\gamma) z^2 -4z +1 > 0,  
%$$
%which  reduces to 
%$$
%  \zeta^2 - ... 0, \quad \mbox{where} \ \zeta= z+ \frac 1z
%$$
\end{proof}

% The following lemma will complete the proof of Proposition \ref{PROP:CX_EXTENSION_ALG_FIELD}:

\begin{lem}[{\bf Invariance away from $\INDphys_\pm$}]\label{LEM:INV_AWAY_FROM_ALPHA}
There exist $\rho>0$ and $\bar\tau>0$ such that for any 
$\zeta = \left(\phi,1-\tau\right)$ 
with $|\tau| < \bar\tau$ and  $|\phi \pm \frac{\pi}{2} | \geq  \rho$
we have $D\Rphys(\KSQRT (\zeta)) \subset \KK^\hor(\Rphys \zeta)$.
\end{lem}

\begin{proof}
We select $\rho>0$ as in Lemma \ref{LEM:INV_NEAR_ALPHA}.
%Note that $|\phi \pm \frac{\pi}{2}| \geq \hat \rho$ implies $\tg \phi$ and $1 / \cos
%\phi$ are bounded.  We will then choose $\rho$ can be chosen small (if
%necessary) relative to to these bounds in order to guarantee the desired
%invariance.
%
Then for  $|\phi \pm \frac{\pi}{2} | \geq  \rho$ formula (\ref{EQN:A_EXPANSION})  implies 
\begin{eqnarray*}
   DR = \left( \begin{array}{cc}  2          & O(1) \\ 
                                     O(\tau^2)    & O(\tau) \end{array}
               \right)
                                       (I+O(\tau)), 
\end{eqnarray*}
with the coefficients depending on $\rho$.
Applied to a tangent vector $v = (1,s)\in T_\zeta\CC^c$ with $|s| = c \sqrt{|\tau|}$ we find
that $D\Rphys (v)$ has slope $|s'| = O(|\tau|^{3/2}) = o(|\tau|) = o(\sqrt{|\tau'|})$,
which is less than  $\sqrt{|\tau'(2-\tau')|}$ for $\tau$ (and hence $\tau'$) sufficiently small.
\end{proof}

%\msk
Lemmas \ref{LEM:INV_NEAR_ALPHA} and \ref{LEM:INV_AWAY_FROM_ALPHA} complete to proof of Proposition \ref{PROP:CX_EXTENSION_ALG_FIELD}.
\end{proof}

\subsection{Stretching of horizontal holomorphic curves}

We now consider how holomorphic curves $\xi$ that are horizontal with respect
to $\KK^\hor$ are stretched under $\Rphys$.  
%Our results will follow from estimates on the derivatives at points $x =(z,t) \in \Cphys$ that were
%obtained in \S \ref{APP:HOR_STRETCHING_NEAR_TOP}.

Let $\pi(z,t) = z$.  If $\xi$ is a horizontal holomorphic curve with $\pi(\xi)$  a
round disc $\D(z, r)$ centered at $z$, we will say that $\rh(\xi,x) = r$.
More generally, we let $\rhmin(\xi, x)$ stand for the supremum of the  radii of
the disks $\D(z, r)$ centered at $z$ that can be inscribed
into $\pi(\xi)\subset \C$ and $\rhmax(\xi,x)$ the infimum of the discs $\D(z,
r)$ that can be circumscribed about $\pi(\xi)$.  They measure the ``horizontal
size'' of $\xi$ at $x$.

\begin{prop}\label{PROP:ESCAPE_FROM_ALPHA}
Fix any $\bar \kappa > 0$.  There is a $C > 0$ so that for any sufficiently
small $a > 0$ and any $x \in \Cphys~\sm~\{\INDphys_\pm\}$ sufficiently close to
$\INDphys_\pm$ and satisfying $|\kappa(x)| \leq \bar \kappa$, there is an
iterate $N \equiv N(x)$ so that, if $\xi$ is any horizontal holomorphic curve based at
$x$ with
\begin{eqnarray*}
\rh(\xi,x) =  a|\eps(x)|,
\end{eqnarray*}
then the subdisc $\xis \subset \xi$ of radius $\rh(\xis,x) = a|\eps(x)|/2$ will
have $\Rphys^i \xis$ in the domain of definition of $\KK^\hor$ for
$i=1,\cdots,N$ and $\rhmin(\Rphys^N \xis,\Rphys^N x)~\geq~C \cdot a$.
\end{prop}

The proof of Proposition \ref{PROP:ESCAPE_FROM_ALPHA} will be the consequence of several lemmas.

Throughout the following lemmas we will suppose that $a > 0$ is sufficiently
small so that a holomorphic disc $\xi$ centered at $x$ with $\rh(\xi,x) = a
|\eps(x)|$ is entirely within the domain of definition for $\KK^\hor$.   This
is possible by Proposition \ref{PROP:CX_EXTENSION_ALG_FIELD}.

\begin{lem}\label{LEM:KOEBE}
If $\xi$ is any horizontal holomorphic curve centered at $x \in \Cphys \sm \{\INDphys_\pm\}$
of radius $\rh(\xi,x) = a|\eps(x)|$, then the restriction of $\Rphys$ to the subdisc $\xis \subset \xi$ of radius $\rh(\xis,x) = a|\eps(x)|/2$ has 
bounded
horizontal distortion.
\end{lem}

\begin{proof}
    Let $r=\rh(\xis,x)$.
The function $\chi_\xis = \pi\circ\Rphys\circ(\pi|\xis)^{-1}$ extends to a univalent function in the disk $\D(\phi,2r)$, since $\xis$ is the restriction of $\xi$ to half of its radius.
The conclusion then follows from the Koebe Distortion Theorem.
\end{proof}

In particular, we will have $\rhmin(\Rphys \xis,\Rphys x) \asymp \rhmax(\Rphys
\xis,\Rphys x)$ and also there is a uniform constant $d \geq 1/4$ so that
$\rhmin(\Rphys \xis,\Rphys x) \geq d \ \hexpmin(x) \rh(\xis,x)$.

\ssk

Recall the point $\INDphys = (\pi,1)$, which is mapped by $\Rphys$ to the high
temperature fixed point $\beta_1 =(0,1)$.

\begin{lem}\label{One iterate near indeterminacy points}

   Given $\underline \kappa>0$ there exists $C_0 > 0$ so that for any  $a > 0$
we have the following property.  If $\xi$ and $\xis \subset \xi$ are as in
Lemma \ref{LEM:KOEBE} and are based at a point $x\in \Cphys\sm \{\alpha_\pm\}$
in an appropriately small neighborhood of $\INDphys_\pm$ with a bounded slope:
$|\kappa(x)|\leq \underline{\kappa}$, then
$$   
  \rhmin(\Rphys \xis, \Rphys x)  \geq C_0 \cdot a \, {\mathrm {dist}}^h (\Rphys x, \INDphys).
$$
\end{lem}

\begin{proof}
 Let $x=(\eps, \tau)$.

Let us select $c=c(a)$ in such a way that any horizontal curve $\xi$  of size $\geq a\eps$
centered  in one of the parabolic sectors $\{  |\kappa-\eps| \leq c |\eps|\} $ near $\alpha_\pm$ crosses
 the corresponding curve $\SSS_\pm$.
The image  $\Rphys \xi$ of such a curve crosses $\II_\pi$, so
that $\rhmax(\Rphys \xi,\Rphys x) \geq \dist^h(\Rphys x, \INDphys)$.
This is sufficient, since $\rhmin(\Rphys \xis,\Rphys x) \asymp \rhmax(\Rphys \xis,\Rphys x)$.

On the other hand,
if $|\kappa-\eps| \geq c |\eps|$ then ${\mathrm {dist}}^h (\Rphys x, \INDphys) \asymp |\kappa - \eps|$
by Lemma \ref{PROP:DISTANCE_TO_PI}.  Then, Lemmas \ref{hexplow} and \ref{LEM:KOEBE} imply that
\begin{equation}\label{ad}
 \rhmin(\Rphys\xis, \Rphys x)\geq d \hexpmin(x) \rh(\xis,x) \geq \tl C_0 \cdot a \, |\kappa-\eps| \asymp {\mathrm {dist}}^h (\Rphys x, \INDphys).
\end{equation}
\end{proof}

We now set up a complex neighborhood of $\beta_1$ designed so that suitable holomorphic curves $\xi$ near $\beta_1$ can regrow to definite size:
Let
\begin{eqnarray*}
\VV^c \equiv \VV_{\bar \tau}^c &:=& \{(\phi,t) \in \Cphys^c_\rho \,:\, |\tau| \leq \bar \tau,  \, |\Im \phi| < \bar \tau\} \,\, \mbox{and} \\
\UU^c \equiv \UU^c_{\bar \eps} &:=& \{(\phi,t) \in \VV^c \, : \, |\eps(\phi)| < \bar \eps \}.
\end{eqnarray*}
They are complex versions of the regions $\VV$ and $\UU$ from \S
\ref{SUBSEC:MODIFIED_ALG_CONES}.  We will take $\bar \tau$
sufficiently small relative to $\bar \eps >0$ so that $\VV^c \sm \UU^c$ lies in
the domain of definition of the complex extension of $\KK^\hor$.

Choosing $\bar \tau$ sufficiently 
small compared to $\bar \eps$, we can ensure that $\Rphys$ is uniformly horizontally
expanding on any horizontal holomorphic curve $\xi \subset \VV^c \sm \UU^c$.  (If follows by continuity from
(\ref{EQN:CONTROL_OF_EXPANSION_AWAY_ALPHA}).)

\comment{%%%%%%%%%%%
\begin{lem}\label{LEM:ESCAPE_FROM_BETA1}
For any $x \in \VV \sm \UU \subset \Cphys$, $0 < r_{\min} < r_{\max}$, and $c > 0$, there exists
an iterate $M$ with the following property.  If
$\eta \subset \VV^c \sm \UU^c$ is a horizontal holomorphic curve centered at $x$ with
\begin{eqnarray*}
r_{\min} \leq \rhmin(\eta,x) < \rhmax(\eta,x) < r_{\max} \qquad \mbox{and} \qquad
\rhmin(\eta,x) \geq b \dist^h(x,\FIXphys_1)
\end{eqnarray*}
then
$\Rphys^i \eta \subset \VV^c \sm \UU^c$ for $0 \leq i \leq M$ and
$\rhmin(\Rphys^M \eta,\Rphys^M x)$ of a definite size (depending on $c$ and on $r_{\max}/r_{\min}$.)
\end{lem}

\begin{proof}
This follows from the standard distortion estimates near a hyperbolic fixed point.
\end{proof}
%%%%%%%%%%%%%%%%%%
}

\begin{lem}\label{LEM:ESCAPE_FROM_BETA1}
Given $0 < r_{\min} < r_{\max}$ there exists $C_1 > 0$ so that for any $b > 0$
and any $x \in \VV \sm \UU \subset \Cphys$ there is an iterate $n(x)$ with the
following property. If $\eta \subset \VV^c \sm \UU^c$ is a horizontal
holomorphic curve centered at $x$ with
\begin{eqnarray*}
r_{\min} \leq \rhmin(\eta,x) < \rhmax(\eta,x) < r_{\max} \qquad \mbox{and} \qquad
\rhmin(\eta,x) \geq b \dist^h(x,\FIXphys_1)
\end{eqnarray*}
then
$\Rphys^i \eta \subset \VV^c \sm \UU^c$ for $0 \leq i \leq M$ and
$\rhmin(\Rphys^M \eta,\Rphys^M x) \geq C_1 \cdot b$.  Moreover, $C_1$ depends only on $r_{\max}/r_{\min}$.
\end{lem}

\begin{proof}
The standard distortion estimates near the hyperbolic fixed point $\FIXphys_1$ show that the
discs $\Rphys^{i} \eta$ grow at the same rate as the horizontal distances between
any point of $\Rphys^{i} \eta$ and $\FIXphys_1$.
\end{proof}

\begin{proof}[Proof of Proposition \ref{PROP:ESCAPE_FROM_ALPHA}:]
Equation (\ref{EQN:DR}) gives that $D\Rphys$ expands horizontal vectors based
at such points by $O(1/|\eps(\zeta)|)$.  So, we can assume that $a$ is
sufficiently small so that $\Rphys \xi$ is in the domain of definition of
$\KK^\hor$. 
Moreover, using (\ref{EQN:R_BLOWUP_ANGULAR_WITH_EPS}) we can choose $\underline
\kappa \leq \overline{\kappa}$ so that if $\xi$ is centered at $x$ with
$|\kappa(x)| \leq \underline \kappa$ then $\Rphys^i \xi \subset \VV^c \sm
\UU^c$ for $i=1,2$.

Lemma \ref{LEM:KOEBE} gives that $\Rphys$ has bounded horizontal distortion on
the subdisc $\xis \subset \xi$ of radius $a|\eps(x)|/2$.

If $|\kappa(x)| \geq \underline \kappa$ then Lemma \ref{hexplow} implies that
$\hexpmin(x) \geq K/|\eps(x)|$.  Together with bounded distortion, this is
sufficient to give that $\rhmin(\Rphys \xis,\Rphys x) \geq C \cdot a$.

If $|\kappa(x)| \leq \underline \kappa$ then Lemma \ref{One iterate near
indeterminacy points} gives that $\rhmin(\Rphys \xis,\Rphys x) \geq C_0 \cdot a \,
\dist^h(\Rphys x, \INDphys)$.  There is some $\tl C_0 > 0$ so that after one further iterate we have  
$$\rhmin(\Rphys^2 \xis, \Rphys^2 x) \geq \tl C_0 \cdot a \, \dist^h(\Rphys^2
x,\FIXphys_1).$$  Since the horizontal distortion of $\Rphys^2$ is also bounded
on $\xi_0$, there is a uniform bound on $\rhmax(\Rphys^2
\xis,\Rphys^2 x)/ \rhmin(\Rphys^2 \xis,\Rphys^2 x)$.  Lemma
\ref{LEM:ESCAPE_FROM_BETA1} then gives $M$ further iterates so that
$\Rphys^{M+2} \xi \subset \VV^c \sm \UU^c$ and \\
$$\rhmin(\Rphys^{M+2} \xis, \Rphys^{M+2} x) \geq C \cdot a.$$
\end{proof}

\section {Critical locus and Whitney folds}\label{App: crit locus}
%\label{APP:R_NEAR_ALPHA}

\subsection{Critical locus}
\sss{Five lines and a conic}\label{five lines and conic}
The Jacobian of $\hat R: \C^3\ra \C^3$ (\ref{hat R}) is equal to
$$
   \det D\hat R = 32\, V\, (UW-V^2)\, (U+W)^2\, (U^2+V^2)\, (W^2+V^2),  
$$ 
so its critical locus comprises 6 complex planes
(where we count $\{U+W=0\}$ only once, without multiplicity) 
and the cone $\{ UW=V^2\} $. 
They descend to 6 complex lines and one conic on $\CP^2$:

\begin{eqnarray*}
\Lzero &:=& \{V=0\}      = \mbox{line at infinity},\\
\Lone &:=& \{UW= V^2\} = \mbox{conic}\ \{ uw=1\}, \\
\Ltwo &:=& \{U = -W\}  =  \{ u=-w\},      
                \\
\Lthree^\pm &:=& \{U = \pm i V \} = \{ u=\pm i\}, \\
\Lfour^\pm &:=& \{W = \pm i V\} =  \{ w=\pm i\}.
\end{eqnarray*}
(Here the  curves are written in the homogeneous coordinates $(U:V:W)$
and in the affine ones, $(u=U/V, w=W/V)$.)  
The configuration  of these curves is shown in Figure \ref{FIG:CRITICAL_CURVES}.

The following general lemma shows that this coincides with the critical locus 
of $R$ on $\CP^2$:

\begin{lem}
   Let $\hat R: \C^{m+1}\ra \C^{m+1}$ be a homogeneous polynomial,
and let $R: \CP^m\ra \CP^m$ be the corresponding rational map of the projective space.
Let $\hat z\in \C^{m+1}$ be such that $\hat R(\hat z)\not=0$
(so that $z:=\pi(\hat z)$ is not a point of indeterminacy of $R$).
Then $\hat z$   is a critical for $\hat R$  iff $z$ is critical for $R$. 
\end{lem}

\begin{proof}
   Let  $z$ be  critical for $R$; then there exists a non-vanishing tangent vector $v\in T_z\CP^n$
such that $DR(z)\cdot v =0$. Let $\hat v$ be a lift of $v$ to $T_{\hat z} \C$.
Then $  D\hat R (\hat z)\cdot \hat v  = \la\hat R (\hat z)$ 
for some $\la \in \C$. On the other hand,
$  D\hat R (\hat z)\cdot \hat z   = \mu\, \hat R (\hat z)$ for some $\mu\in \C$.
This shows that $D\hat R(\hat z)$ is not injective.    

Vice versa, by Euler's formula, $ D\hat R (z)\cdot \hat z  = d\, \hat R(\hat z)$ where $d=\deg \hat R$.
Hence, if $z$ is not a point of indeterminacy then $\hat R$  is non-degenerate along
the complex line spanned by $\hat z$. So, if $  D\hat R(\hat z)\cdot \hat v  =0$
then $\hat v$ is linearly independent of $\hat z$, and hence it projects to
a non-vanishing vector $v\in \Ker DR(z)$. 
\end{proof}

%\begin{lem}
%   Let $\hat R: \C^{m+1}\ra \C^{m+1}$ be a homogeneous polynomial map,
%and let $L$ be a complex line through the origin.
%Then $\hat \R$ is non-degenerate along $L$ unless $L$ gets collapsed to the origin
%(i.e., $\hat R(L)=\{0\}$).    
%\end{lem}

\sss{Critical points on the exceptional divisor}

To complete the picture, we need to take an account of the critical points 
hidden inside the points of indeterminacy, $a_\pm$. To make them visible, 
we blow up $\CP^2$ at $a_\pm$ and lift $R$ to a holomorphic map 
$\tl R = Q\circ \tl g : \tl{\CP}^2\ra \CP^2$.

By symmetry, it is enough to analyze the blow-up of $a_+$. 
Jacobian of $\tl g$  (\ref{blow up near a+}) on the exceptional divisor $\{\xi=0\}$
is equal to $2i(\chi-1)/(\chi+1)^2$, so $\tl g$ has two critical points on it,
$\chi=-1$ and $\chi=1$, which are the intersections of the exceptional divisor with 
 the collapsing line $\tl L_2$ and the separatrix $\tl L_1$ respectively
(recall that  $\tl L$ stands for the lift of $L$  to $\tl {\CP}^2$). 
 The first critical point is mapped by $\tl R$  
to the low temperature fixed point $b_0$,
while the second one is mapped to $(-1,-1)$, which is a preimage of the high temperature fixed point $b_1$. 

Also, points $\chi=\infty$ and $\chi=0$ are mapped by $\tl g$ to the
coordinate lines $\{u=0\}$ and $\{w=0\}$ respectively which are 
critical for the squaring map $Q$. It creates two more critical points on 
the exceptional divisor, its intersections with the critical lines $\tl L_3^+$ and $\tl L_4^-$.

\subsection{Complex Whitney folds}

To simplify calculations near the critical points, it is convenient to bring $R$ to a normal form.
A complex Whitney fold is a generic and the simplest one. 
Let $R: (\C^2,0)\ra (\C^2,0)$ be a germ of holomorphic map with a critical point at $0$.
The map $R$ (and the corresponding critical set) is called a {\it complex Whitney fold} if

\begin{itemize}
\item [(W1)] The critical set $L$ is a non-singular curve near $0$;
\item [(W2)] $DR(0)$ has rank 1 and $\Ker DR(0)$ is transverse to $L$;
\item [(W3)] The second differential $D^2 R(0)$ is not vanishing in the direction of  $\Ker DR(0)$. 

\end{itemize}

\begin{lem}\label{Whitney normal form}
    A Whitney fold can be locally brought to a normal form $(u,w)\mapsto (u, w^2)$ in holomorphic coordinates. 
\end{lem}

\begin{proof}
   Since $L$ is non-singular and $\Ker DR(0)$ is transverse to $L$,
 the image $L'=R(L)$ is a non-singular holomorphic curve near $0$,
so we can select local coordinates in such a way that both $L$ and $L'$ locally coincide with the axis $\{w=0\}$.
Then $u'= \phi(u,w)$ with $\di_u \phi(0)\not=0$, so $(u',w)$ can be selected  as local coordinates in the domain of $R$.
This brings $R$ to a fibered map $u'=u,\ w'=\psi(u,w)$. 

Since $\{w=0\}$  is in the critical locus of $R$, 
$$
   \psi(w)= p_2(u) w^2+p_3(u)w^3+\dots.
$$
Moreover, by definition of the fold, $\psi_2(0)\not=0$, so we can select a local branch of the square root
$w \sqrt{p_2(u)+p_3(u)w+\dots}$ as a local coordinate replacing $w$. 
This brings $R$ to the desired form. 
\end{proof}

\begin{lem}\label{description of folds}
  All critical points of $\tl R$ except the fixed points $e,e'$, 
the collapsing line $\tl \Ltwo$, and two points $\{\pm (i,i)\} = \tl \Lthree\cap \tl \Lfour$,
are Whitney folds.  
\end{lem}

\begin{proof}
  Let us treat the components of the critical locus one by one.

\ssk {\it Separatrix $\Lzero$.} In affine coordinates $\xi= U/W$, $\eta= V/W$, the map $R$ near 
$\Lzero=\{\eta=0\}$ looks like this:
$$
  \xi'= \left(\frac{\xi^2+\eta^2}{1+\eta^2}\right)^2=\xi^4(1+O(\eta^2)),
$$
$$
   \eta'= \eta^2\left(\frac{1+\xi}{1+\eta^2}\right)^2=\eta^2(1+\xi)^2 (1+O(\eta^2)),
$$
which shows that $\Lzero$ is a fold outside the fixed point $b_0$ ($\xi=0$) 
and the collapsing line $\Ltwo$ ($\xi=-1$). 
(The other fixed point $b_1$ lies at infinity, $\xi=\infty$.) 

\ssk {\it Separatrix $\tl \Lone$. }
 In local coordinates $(u, \tau= uw-1)$, the map $R$ near 
$\Lone=\{\tau =0\}$ looks like this:
$$
   u'= \left( \frac{u^2+1}{u+(\tau+1)/u}\right)^2 = u^2\left( 1+O\left(\frac \tau{u^2+1}\right) \right), 
$$
\begin{equation}\label{tau squared}
   \tau' =  \frac{\tau^2}{(u+w)^2} \left(-2+ \frac{\tau^2}{(u+w)^2} \right),
\end{equation}
which shows that $\Lone$ is a fold outside the indeterminacy points $\{u=\pm i\}= \Lone\cap \{u+w=0\}$
(and outside the fixed points $b_0$ and $b_1$ at infinity). 

Let us now analyze the intersection $\tl a_+ =(u=i, \chi=1)$ of $\tl\Lone$ with the exceptional divisor $L^+_\ex$
(the intersection with $L^-_\ex$ is symmetric). 
Let us use local coordinates $(\xi=u-i, \la=\chi -1)$ near $\tl a^+$. 
Representation (\ref{blow up near a+}) of $\tl g$ near $L^+_\ex$ gives:
$$
   u= i+\frac 1{2} (\xi-i \la) + \frac i{4} \la^2- \frac 1{4} \xi \la+\dots,
$$
$$
    w= -i + \frac 1{2} (\xi-i \la)+\frac i{4} \la^2 + \frac 3{4} \xi \la+\dots,
$$
so the vanishing direction for $D\tl g(\tl a_+)$ is $d\xi= i d\la$. 
On the other hand, in these coordinates the separatrix $\{uw=1\}$ assumes the form 
$\xi+i\la+\xi^2\la=0$, so it is a non-singular curve tangent to $\{d\xi=- i\, d \la\}$ at $\tl a_+$.
This yields conditions (W1) and (W2). Moreover, at the kernel direction $d\xi=i\, d\la$, 
the second differential assumes the form $(0, i\, d\la^2)$, so it is non-vanishing.   

Thus, $\tl a^+$ is a Whitney fold for $\tl g$. Since the squaring map $Q$ is non-singular at $(i,-i)$,
it is a a Whitney fold for $\tl R = \tl g \circ Q$ as well. 

\ssk {\it  Lines $\tl \Lthree^\pm $ and $\tl \Lfour^\pm$.}
  These lines intersect the line at infinity at the fixed points $b_0$ and $b_1$,
so we need to analyze only their affine parts. 
The map $\tl g$ is non-singular on these lines (including their intersections with the exceptional divisors),
and it maps them isomorphically onto the coordinate axes $\{u=0\}$ and $\{w=0\}$.
Outside the origin $\B0$, these axes are Whitney folds for the squaring map $Q$. 
Hence the lines  $\tl \Lthree^\pm $ and $\tl \Lfour^\pm$
are folds for $\tl R$ outside points $\{\pm (i,i)\} = \tl g^{-1}(\B0)=\tl \Lthree\cap \tl \Lfour$. 
\end{proof}

\sss{Double points}

A {\it double point} of a holomorphic curve $X$ is a point $a\in X$ such that
the germ of $\gamma$ at $a$ consists of two regular branches, $X_1$ and $X_2$,
meeting at $a$. A double point is called {\it transverse} 
if the branches $X_i$ intersect transversely at $a$.
Otherwise, it is called {\it tangential}. 

We say that a regular curve $L$ intersects a curve $X$ {\it transversely} at the double point $a\in X$
if it intersects transversely both branches $X_i$. Such intersection has multiplicity 2. 
%In this case, the intersection multiplicity of $X$ and $L$ at $a$ is equal to 2. 

\begin{lem}\label{pullback of parabola}
 Let $R$ be  a Whitney fold at $a$, and let $X$ be a germ of regular holomorphic curve with the first order tangency to
the critical value locus $R(L)$. Then the pullback $R^*X$ has a transverse double point at $a$
intersecting $L$ transversally. 
\end{lem}

\begin{proof}
In the normal coordinates, the pullback  under $R$  of a regular curve $w=cu^2(1+O(u))$, $c\not=0$, 
tangent to $R(L)=\{w=0\}$ is a pair of regular curves $w=\pm \sqrt{c} u (1+O(u))$.
\end{proof}

\begin{lem}\label{pencil}
Let $\pi: (\tl M,\EE_\ex) \ra (M, a)$ be the blow-up of $M$ at $a$, and let $\tl p\in \EE_\ex$.
Let  $\tl X\subset \tl M$ be a holomorphic curve with a transverse double point at $\tl p$
 that intersect $\EE_\ex$ transversely. Then $X=\pi(\tl X)$ is a holomorphic curve in $M$
with a (first order) tangential double point at $a$.  
\end{lem}

\begin{proof}
Let us use the local coordinates $(u,v, m=v/u)$ from the definition of the the blow-up.  
In this coordinates, 
The pencil of lines $m-m_0=\la u$, $\la\in \CP^1$
centered at  $(0,0,m_0)\in \EE_\ex$ projects to the pencil of
parabolas $v=(\la^{-1} u+m_0)u$ in $M$ tangent to the line $v=m_0 u$. 
\end{proof}

\section{Computational proof of horizontal expansion}
\label{SEC:PARTIAL_HYPERBOLICITY}

\comment{********************
In Subsection \ref{SUBSEC:PARTIAL_HYP} we will show that $\Rphys: \Cphystl \ra \Cphystl$ is {\em partially hyperbolic}.  In Subsection \ref{SUBSEC:DISTORTION} we prove estimates of the distortion for $\Rphys^n$ along certain horizontal curves.

\subsection{Partial Hyperbolicity}
\label{SUBSEC:PARTIAL_HYP}
Recall that
$\Rphys: \Cphystl \ra \Cphystl$ is {\em partially hyperbolic} if
\begin{itemize}
\item there is an invariant horizontal tangent conefield $\KK^h(x) \subset T_x \Cphystl$, and
\item tangent vectors $v \in \KK^h(x)$ get exponentially stretched under iterates of $\Rphys$.
\end{itemize}
\noindent
\noindent
In \S \ref{sec: alg cone field} we have constructed the invariant cone-field $\KK^h(x)$.  If $\Rphys$ satisfies the second condition, we say that it is {\em horizontally expanding}.

\note{References for partial hyperbolicity?  For endos with cone-field?}
 
\begin{prop}\label{PROP:EXPANDING}$\Rphys: \Cphystl \ra \Cphystl$ is horizontally expanding.
\end{prop}
***************}

We will give an alternative, computational, proof that $\Rphys: \Cphys \ra
\Cphys$ is horizontally expanding, which, unlike Theorem \ref{hor expansion thm}, does not
result in a lower bound on the rate of expansion.

\begin{prop}\label{PROP:COMPUTATIONAL_EXPANSION}
The map $\Rphys: \Cphys \ra \Cphys$ is horizontally expansing on $\Cphys$ with respect to the horizontal
cone field $\KK^h$.
\end{prop}
\noindent

\begin{lem}\label{LEM:NONUNIFORM_EXPAND}Let $x \in \Cphystl$ then 
\begin{eqnarray}\label{EQN:NON_CONTRACT}
d(\phi \circ \Rphys)(v) > d\phi(v) \,\,\,\, \mbox{for any} \,\, v \in \KK^h(x).
\end{eqnarray}
\end{lem}

\begin{proof}
Recall the region $\VV'$ used in \S \ref{SUBSEC:MODIFIED_ALG_CONES} when constructing $\KK^h$.
If $x \in \Cphystl \sm \VV'$, then $\KK^h(x) = \KK^\hor(x)$.  Linearity of
(\ref{EQN:NON_CONTRACT}) allows us to assume that $d\phi(v) = 1$ in order to
check that $d(\phi \circ \Rphys)(v) \geq 1$.  Furthermore,  the minimum of
$d(\phi \circ \Rphys)(v) \geq 1$  over all such $v \in \KK^\hor(x)$ occurs at one
of the vectors $v_\pm = (1,\pm \sqrt{1-t^2})$ on the boundary of $\KK^\hor(x)$.

Using (\ref{EQN:DR}) we find
\begin{eqnarray}\label{EQN:EXPANSION_ONE_ITERATE}
d(\phi \circ \Rphys)(v_\pm)  = 4\,{\frac {1+{t}^{2}\cos \left( 2\,\phi \right) \mp \sin \left( 2\,\phi
 \right) t\sqrt {1-{t}^{2}}}{1+2\,{t}^{2}\cos \left( 2\,\phi \right) +
{t}^{4}}}
\end{eqnarray}
\noindent
For $x \in \Cphys \sm \{\INDphys_\pm\}$ let $\lambda(x)$ be given by the minimum of the two terms in (\ref{EQN:EXPANSION_ONE_ITERATE}).
It is a continuous function.  

Equivalent to $\lambda(x) > 1$ is
\begin{eqnarray}\label{EQN:DESIRED_PHI_T}
3 - t^4 +2 t^2\cos 2\phi \mp 4 \sin 2\phi \, t\sqrt {1-t^2} > 0,
\end{eqnarray}
\noindent
which holds when $t = 0$ or when $t = 1$ and $\phi \neq \pm \pi/2$.  

If we let $s = |\sin 2\phi|$ so that $|\cos 2\phi| = \sqrt{1-s^2}$, then (\ref{EQN:DESIRED_PHI_T}) is bounded below by
\begin{eqnarray*}
g(s,t):= 3 - t^4 - 2 t^2 \sqrt{1-s^2} - 4 s t \sqrt{1-t^2}.
\end{eqnarray*}
\noindent
\noindent
which we check is positive for 
$(s,t) \in [0,1]\times (0,1)$.
Because $g(s,t)$ is continuous on $[0,1]^2$ with $g(1/2,1/2) = 1.63 > 0$, it
suffices to check that $g(s,t) \neq 0$ for $(s,t) \in [0,1] \times (0,1)$.

Replace $g(s,t) = 0$ with $3 - t^4 - 4 s t \sqrt{1-t^2} = 2 t^2 \sqrt{1-s^2}$.  After squaring both sides, we find
\begin{eqnarray*}
(9-10{t}^{4}+{t}^{8}) - \left(8t(3-t^4)
\sqrt {1-{t}^{2}} \right) s+ \left( 16\,{t}^{2}-12\,{t}^{4} \right) {s
}^{2}
\end{eqnarray*}
\noindent
Considered as a quadratic polynomial in $s$ the discriminant is:
\begin{eqnarray*}
-16t^4 (1-t^2)^2(t^4+2t^2+9).
\end{eqnarray*}
\noindent
For each $t \in (0,1)$ this is negative, so that there are no real roots for $s$.  

\msk
Now consider  $x \in \VV'$.  If $x \in \VV' \sm \UU'$,
(\ref{EQN:NON_CONTRACT}) follows from (\ref{EQN:A_EST}).

If $x \in \UU'$ the minimum of $d(\phi \circ \Rphys)(v)$
occurs at $v_\pm = (1,\pm w(x))$, where $w(x) \sim |\eps(x)|/3.$
Equation (\ref{DR near alpha}) gives
\begin{eqnarray*}
d(\phi \circ \Rphys)(v_\pm) = \frac{2}{\eps^2 + \tau^2} \left(\eps^2 + \tau \pm \frac{\eps^2}{3} \right)  \geq \frac{4\eps^2/3 + 2 \tau}{\eps^2 + \tau^2}  > \frac{4}{3}.
\end{eqnarray*}

\msk
Thus, $\lambda(x) > 1$ for all $x \in \Cphys \sm \{\INDphys_\pm\}$.
\end{proof}

\begin{rem}\label{REM:EXPANSION_BOUNDED_BELOW_BY_1}
Lemma \ref{LEM:NONUNIFORM_EXPAND} give non-uniform expansion for horizontal
tangent vectors $v \in \KK^h(x)$ under the first iterate of $\Rphys$.  
It is not possible for this expansion to be uniform because
(\ref{EQN:EXPANSION_ONE_ITERATE}) limits to $1$ as a point $x$
approaches the indeterminate points $\INDphys_\pm$ along the parabolas given by
$\tau = \eps^2/2$.  
\end{rem}

\begin{proof}[Proof of Proposition \ref{PROP:COMPUTATIONAL_EXPANSION}:]
Recall the neighborhood $\UU$ of $\INDphys_\pm$ that was constructed in \S
\ref{SUBSEC:MODIFIED_ALG_CONES}.  The estimates from the proof of Lemma
\ref{LEM:NONUNIFORM_EXPAND} give uniform expansion at points $x \in \Cphystl
\sm \UU$. 
According to Lemma \ref{inclusions}, we have $\Rphys(\UU) \cap \UU =
\emptyset$, so that vectors $x \in \KK^h(x)$ experience a definite amount of
expansion under at least every other iterate.
\end{proof}

\section{Extra bits of stat mechanics}\label{APP:STATMECH}

\subsection {The Lee-Yang Theorem}\label{APP:LY_THM}

To prove the Lee-Yang Theorem,
we need to consider a more general, anisotropic, Ising model.
It is parameterized by  a symmetric matrix $\BJ=(J_{vw})_{v,w\in \EE}$ of couplings
between the atoms and by a vector $\Bh=(h_v)_{v\in \VV}$ of interaction strengths of the external field with the atoms.
Then the energy of a spin configuration $\si:\VV\ra \{\pm 1\}$ assumes the form 
\begin{eqnarray}\label{EQN:H_GENERAL}
- \Hconv(\sigma) = <\BJ \si, \si> + <\Bh, \si>.  
\end{eqnarray}
It is convenient to assume that $\Gamma$ is a complete graph without loops (connecting a vertex to itself),
but to to allow some of the coupling constants vanish. 
In the ferromagnetic model, $J_{vw} \geq 0$. 

Let us consider  the ``support'' of a configuration $\si$,
$$
 \VV(\si)= \{c\in \VV: \ \si(v)=  -1\},
$$
and let 
$$
 \EE(\si)= \{(v,w):\ v\in \VV(\si), \ w\in \VV\sm \VV(\si)\}.
$$ 
Let $l(\BJ)$ and $l(\Bh)$ be the  the sums of all components of $\BJ$ and $\Bh$ respectively.        
We will work with a modified Hamiltonian 
$$
  - \check{\Hconv}(\sigma) = - \Hconv(\si) -   l(\BJ)  - l(\Bh)
     = - 2\sum_{(vw)\in \EE(\si)} J_{vw} - 2\sum_{v\in \VV(\si)} h_v  .
$$
Let us introduce the temperature-like and field-like variables:
$$
 t_{vw}= e^{-2J_{vw}/T}\quad \text{ and}\quad  \zeta_v= e^{-2h_v/T}.
$$
Given subsets $X\subset \VV$ and  $Y\subset \EE$, we will use notation 
$$ 
   \zeta^X = \prod_{v\in X} \zeta_v, \quad t^Y = \prod_{(vw)\in Y} t_{vw} 
$$    
In this notation, we obtain the following expression for the modified Gibbs weights:
\begin{equation*}
\check\Weight(\si) =   \exp(-\check{\Hconv}(\si)/T) =  \Weight(\si)\, t^{\EE}\, \zeta^{\VV} = t^{\EE(\si)} \,  \zeta^{\VV(\si)}  
%   \prod_{(vw)\in \EE} t_{vw}  \prod_{v\in \VV} z_v      
%             = \prod_{(vw)\in \EE(\si)} t_{vw} \, \prod_{v\in \V(\si)} z_v^2,
\end{equation*}
and for the modified partition function:
\begin{equation}\label{modified Z}
  \check Z =  t^{\EE}\, \zeta^{\VV}\, Z =  \sum_\si  t^{\EE(\si)} \,  \zeta^{\VV(\si)}. 
% Z  \prod_{(vw)\in \EE} t_{vw}  \prod_{v\in \VV} z_v^1  = \sum_\si \prod_{(vw)\in \EE(\si)} t_{vw} \, \prod_{v\in \VV(\si)} z_v^2.
\end{equation}
Obviously, the modification does not affect the roots of the partition function (modulo clearing up the denominator), 
so we can work with $\check Z$ instead of $Z$. 

\begin{lem}\label{LY lemma}
  Fix arbitrary $t_{vw}\in  [-1,1]$.
If $\hat Z(\zeta_1,\dots, \zeta_n)=0$ and $|\zeta_i|\leq 1 $ for $i=1,\dots, n-1$, then $|\zeta_n|\geq 1$,
where the inequality is strict unless all $|\zeta_i|=1$, $i=1,\dots, n-1$.    
\end{lem}

\begin{proof}
To simplify notation and to emphasize dependence on $n=|\VV|$,  we let $P_n\equiv \hat Z$. 
We will carry induction in $n$.
Without loss of generality we can  assume that $t_{vw}\not= 0,\pm 1$.
% If the proposition is proved under this restriction, the general case follows by continuity.

For $n = 2$, we have
$$
   P_2(z _1,z_2) = 1 + t_{12} \zeta_1 + t_{21} \zeta_2 + \zeta_1 \zeta_2,
$$
which implies
$$
\zeta_2 = - \frac{1 + t_{12}\zeta_1}{t_{21} + \zeta_1}.
$$
Since $t_{12} = t_{21} \in (-1,1)$, this is a M\"obius map sending $\mathbb{D}$ to $\C \sm \overline{\mathbb{D}}$.

\msk
To pass from $n$ to $n+1$, observe that
$$
  P_{n+1} (\zeta_1,\dots, \zeta_{n+1}) = P_n (u_1, \dots, u_n) + \zeta_1\dots \zeta_{n+1} \, P_n (v_1, \dots, v_n), 
$$
where $u_i =  t_{i,n+1} \zeta_i$, $v_i= t_{i, n+1}/ \zeta_i$.

\begin{rem}
This formula is directly related to the Basic Symmetry of the Ising model which is ultimately responsible for the
Lee-Yang Theorem.   
\end{rem}

If $P_n=0$ then 
\begin{equation}\label{fraction}
   \zeta_{n+1} = - \frac 1{\zeta_1\dots \zeta_n}\,  \frac {P_n(u_1,\dots, u_n) }{P_n(v_1,\dots, v_n)}. 
\end{equation}
If $\zeta_i\in \hat\C\sm \bar \D$ for $i=1,\dots, n$, then $|v_i|< 1$ and by the Induction Assumption,
$P_n(v_1,\dots, v_n)\not=0$.  
Hence the right-hand side of (\ref{fraction}) is a  well-defined holomorphic function in the polydisk
$ \De^n := (\hat\C\sm \bar \D)^n$. On its Shilov boundary 
$$
   \di^s \De^n=\T^n \equiv  \{ |\zeta_i|=1,\quad i=1\dots, n\}\}
$$  
we have $v_i=\bar w_i$. Since $P_n$ has real coefficients, we conclude that $|\zeta_{n+1}|=1$ on $\T^n$. 
By the Maximum Principle, $|\zeta_{n+1}|\leq 1$ in $\De^n$, with equality only  on $\T^n$,
and we are done.
\end{proof}

% The general Lee-Yang Theorem immediately follows:

\begin{LY Theorem2}[\cite{YL,LY}] \label{THM:YANG_LEE2}
Fix $J_{vw} \in [0,+\infty]$, and assume $h_v/h_w \in \R\cup\{\infty\}$. 
Then, all zeros of the partition function $Z(\zeta_1, \dots, \zeta_n)$ lie on the unit torus $\T^n$.
\end{LY Theorem2}

\begin{proof}
  Under these circumstances, all $\zeta_v= e^{-h_v/T}$ have the same modulus,
which must be equal to 1 by the previous lemma.
 (The lemma  is applicable since $t_{vw} = e^{-2J_{vw}/T}\in [0,1]$.) 
\end{proof}

To obtain the classical Lee-Yang Theorem, corresponding to Hamiltonian with
magnetic moment $\Mconv$ given by (\ref{M conventional}), one sets $h_v \equiv
h$.  To get the result for Hamiltonian with magnetic moment $M$ given by
(\ref{M}), obtained by summing over magnetic moments of edges, one sets $h_v = h \cdot
|v|/2$, where $|v|$ is the valence of the vertex $v$.

\msk

Many extensions and new proofs of this theorem have appeared since the 1950s: see
Asano~\cite{Asa}, Suzuki and Fisher \cite{SF}, Heilmann and Lieb \cite{HL},
Ruelle~\cite{Rue,Ruelle_ANN}, Newman~\cite {New}, Lieb and Sokal~\cite {LS}, Borcea and Br{\"a}nd{\'e}n~\cite{BB} and further
references therein.  The proof given above is based upon the original idea of
Lee and Yang \cite{LY}, compare \cite[Thm 5.1.2]{Ruelle_book}.

\subsection {The Lee-Yang Theorem with Boundary conditions}
Given any coupling $J_{vw} \in [0,+\infty]$, let $\Gamma' \subset \Gamma$ be the subgraph containing
only the edges with $J_{vw} > 0$.
The Lee-Yang Theorem with Boundary Conditions, stated in \S \ref{background},
is a consequence of the following more general statement:
\begin{LY Theorem2BC}
Fix $J_{vw} \in [0,+\infty]$ and assume that $\Gamma'\subset \Gamma$ is connected
and $h_v/h_w \in \R\cup\{\infty\}$.
Let $\sigma_\UU \equiv +1$ for $\UU \subset V$ is given by
$\{m+1,\ldots,n\}$ with $1 < m < n$.

Then all of the  zeros of the conditional partition function
$Z^+(\zeta_1,\ldots,\zeta_m)$ lie in the open polydisc $\De^m = (\hat \C \sm
\overline{\D})^m$.
\end{LY Theorem2BC}

\begin{proof}
Since $\Gamma'$ is connected we can assume that
$J_{1,n} > 0$ giving $|t_{1n}| < 1$.  Let $\eta_i = \zeta_i \prod_{j=m+1}^n
t_{ij}$ for $i=1\ldots,m$ so that $|\eta_i| \leq |\zeta_i|$ and $|\eta_1| <
|\zeta_1|$.

The (modified) conditional partition function 
\begin{eqnarray*}
\check Z_{\Gamma | \, \si_\UU}(\zeta_1,\ldots,\zeta_n) = t^\EE z^{\VV \sm \UU} \ Z_{\Gamma | \, \si_\UU}(\zeta_1,\ldots,\zeta_n)
\end{eqnarray*}
satisfies
\begin{eqnarray*}
\check Z_{\Gamma | \, \si_\UU}(\zeta_1,\ldots,\zeta_m) = \check Z(\eta_1,\ldots,\eta_m),
\end{eqnarray*}
\noindent
where $\check Z$ is the modified partition function corresponding to $\Gamma \sm \UU$.

Since $h_v/h_w \in \R\cup\{\infty\}$, all of the $\zeta_i$ have the same modulus.
If this common modulus is less than or equal to $1$, then $|\eta_i| \leq 1$ and $|\eta_1| < 1$. 
Lemma \ref{LY lemma} gives 
\begin{eqnarray*}
\check Z_{\Gamma|\, \si_\UU}(\zeta_1,\ldots,\zeta_m) = \check Z(\eta_1,\ldots,\eta_m) \neq 0.
\end{eqnarray*}
\end{proof}

\begin{rem}
The same statement holds if we assign $\sigma_\UU \equiv \sigma_0 > 0$ and the classical Lee-Yang Theorem can be obtained from it by taking a limit $\sigma_0 \ra 0$.
\end{rem}

\subsection{The Lee-Yang zeros in the 1D Ising model}

By means of  the well-known ``Transfer Matrix technique'' (see \cite{Baxter}),
one can find explicitly the Lee-Yang zeros of the 1D Ising model.  
Let $\Gamma_n$ be the linear chain with $n+1$ vertices $\{0,1,\dots, n\}$.  
For simplicity, we will consider periodic boundary conditions: $\sigma(n) = \sigma(0)$
(so that our graph is the circle $\Z/n\Z$).  This assumption does not affect the thermodynamical limit.

The Hamiltonian of this lattice is:
\begin{eqnarray*}
\Hconv_n (\sigma) = - \sum_{i=0}^{n-1} \{ J \sigma(i)\sigma(i+1) + \frac{h}{2} (\sigma(i)+\sigma(i+1) \}.
\end{eqnarray*}

Two neighboring spins $\{\si(i), \si(i+1)\}$ contribute the following factor to the Gibbs weight (compare (\ref{UVW-zt})):
$$
  W(++)= e^{(J+h)/T}= \frac 1{\tconv z}, \quad W(+-)=W(-+)=e^{-J/T} = \tconv, \quad W(--) =  e^{(J- h)/T} = \frac z{\tconv},
$$
which can be organized into the {\it Transfer Matrix}
\begin{equation*} % \label{1d2}
W =
\begin{pmatrix}
  (\tconv z)^{-1}  & \tconv  \\
         \tconv  &  \tconv^{-1} z
\end{pmatrix}.
\end{equation*}
\noindent
We see that the partition function $Z_n = \sum_\sigma \exp(-\Hconv_n(\sigma)/T)$ can be expressed as
\begin{eqnarray*}
Z_n = \sum_{\sigma} W(\sigma_1,\sigma_2) \cdot W(\sigma_2,\sigma_3) \cdots W(\sigma_n,\sigma_1) = \tr\,  W^n = \tau_1^n+\tau_2^n,
\end{eqnarray*}
\noindent
where $\tau_{1,2}$ are the eigenvalues of $W$.%  
\footnote{Different boundary conditions would result in $Z_n=a\tau_1^n+b\tau_2^n$ for appropriate $a$ and $b$.}

Thus, the zeros of the partition functions are solutions of
the equation 
$$
\tau_1^n+\tau_2^n=0,
$$
or
\begin{equation}\label{1d4}
\frac{\tau_2}{\tau_1}=e^{i\al_k},\quad \al_k=\frac{\pi}{n}+\frac{2\pi k}{n};\qquad k=0,1,\ldots,n-1.
\end{equation}
The eigenvalues $\tau_{1,2}$ are the roots of the quadratic equation
$$
\tau^2-p\tau +q=0,\quad {\mathrm{where}}\ p= \frac 1{\tconv}(z+\frac 1 {z}) = \frac 2 {\tconv} \cos\phi  ,\quad q= \frac 1{\tconv^2} -\tconv^2.
$$
This gives 
$$
2\cos\al_k=\frac{\tau_2}{\tau_1}+\frac{\tau_1}{\tau_2}=\frac{p^2}{q}-2=\frac{4\cos^2\phi_k}{1-\tconv^4}-2,
$$
so
\begin{equation}\label{1d8}
(1-\tconv^4)\cos^2\frac{\al_k}{2}=\cos^2\phi_k, \quad {\mathrm{where}}\ z_k=e^{i\phi_k}. % \ t=\tconv^2.
\end{equation}
Together with (\ref{1d4}), this gives expression (\ref{in11a}) for the zeros of the partition function.
%\begin{equation*}
%z_k^\pm=e^{i\phi_k^\pm},\quad \phi_k^\pm=
%\pm \arccos \left[\sqrt{1-\tconv^4}\,\cos\left(\frac{\pi}{2n}+\frac{\pi k}{n}\right)\right]\,;\qquad 
%k=0,1,\ldots,n-1.
%\end{equation*}

Since the angles $\alpha_k$ are equidistributed with respect to the Lebesgue measure $d\alpha/2\pi$ on the circle,
the distribution $\rho_t d\phi$ of the Lee-Yang zeros is obtained by pushing this measure to the interval $[-1,1]$
by $\cos$,%
\footnote{which gives the ``Chebyshev measure'' $dx/\sqrt{1-x^2}$ on $[-1,1]$} 
scaling it by $\sqrt{1-\tconv^4}$,
and then pulling it back to the circle by $\cos$. 
The calculation gives  expression (\ref{in11}):
% for the density of limiting distribution of Lee-Yang zeros:
\begin{eqnarray*}
\rho_t(\phi) = \frac 1{2\pi} \left| \frac {d\alpha} {d\phi}\right| = 
\frac{1}{2\pi} \left| \frac{d}{d\phi}  \arccos \left( \frac{\cos \phi}{\sqrt{1-\tconv^4}}  \right)\right|
= \frac{|\sin\phi|}{2\pi\sqrt{1-\tconv^4-\cos^2 \phi}}.
\end{eqnarray*}

\subsection{ Diamond model as anisotropic regular 2D lattice model}\label{reg lattice}
The diamond model can be viewed as the Ising model on the regular 2D lattice
with a special anisotropic choice of the interaction parameters.
Namely, let us consider a $2^n\times (2^n+1)$ rectangle
\[
\De_n=\{(i,j)\in\Z^2:\; 0\le i\le 2^n-1,\; 0\le j\le 2^n\}
\]
in $\Z^2$,
% For a spin configuration $\si$ on $V_n$,
% \[
% \sigma=\{\sigma(x)=\pm 1,\;x\in V_n\},
% \]
with the Hamiltonian
\begin{equation}\label{H on 2D lattice}
H_n(\sigma)=-\sum_{|x-y|=1} J(x,y)\sigma(x)\sigma(y)-\sum_{x\in \De_n} h(x)\sigma(x),
\end{equation}
where
\[
h(i,j)=
\left\{
\begin{aligned}
h\;\;&\textrm{if}\;\; 0<j<2^n,\\
\frac{h}{2}\;\;&\textrm{if}\;\; j=0,\,2^n.
\end{aligned}
\right.
\]
The interaction parameters $J(x,y)$ are defined as follows.

%%%%%%%%%%%% Fig.  %%%%%%%%%%%%%%
\begin{center}
 \begin{figure}
\begin{center}
   %\scalebox{0.52}{\includegraphics{figures/4x5.eps}}
   \input{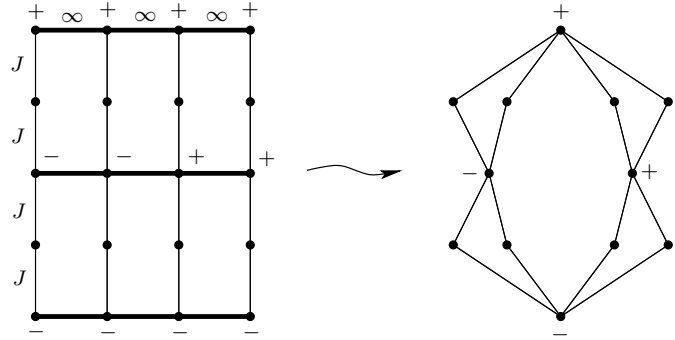}
\end{center}
        \caption{\label{fig 4x5}The Migdal interaction.}
    \end{figure}
\end{center}
%%%%%%%%%%%%%%%%%%%%%%%%%%%%%%%%%%

Representing any $j=0,1,\ldots, 2^n$ in the dyadic arithmetics,
\[
j= \sum_{k=0}^n j_k 2^k,\quad j_k\in \{0,1\},
\]
let
\[
o(j)=\min\{k:\; j_k\not= 0\}\;\;\textrm{if}\;\; j>0;\quad o(0)=n.    %\qquad \|j\| = 2^{-o(j)}.
\]
%(Note that for $j\not=0$,  $\|j\|$ is the dyadic norm of $j$.)
\comm{
Similarly, for
\[
i= \sum_{k=0}^{n-1} i_k 2^k \in [0,2^n-1], % \quad i_k\in \{0,1\},
\]
we can consider the ``dual'' dyadic norm and the corresponding ultra-metric:
\[
o'(i)=\min\{k:\; i_{n-1-k}\not= 0\}, \quad \|i\|'=\|-i\|' = 2^{-o'(i)}, \quad
              d'(i,\tl i) = \|i-\tl i\|'. % \; (i\geq \tl i).
\] }
Let us partition each horizontal level $\De_n(j)= \{(i,j)\in \De_n\}$
into $2^{n-o(j)}$ intervals $[s 2^{o(j)}, (s+1)2^{o(j)}) $ of length $2^{o(j)}$, $s=0,1,\dots, 2^{n-o(j)}-1$.
Let $\underset{j}\sim$ denote the corresponding equivalence relation.

Now we let
$$
J(x,y) = \left\{
\begin{aligned}
    J\quad & \textrm{if}\;\;    y-x=\pm (0,1);\\   % x=(i,j),\; y=(i,j+1)\;\;\textrm{for some}\;\;i,j,\\
    \infty\quad & \textrm{if}\;\; y-x=\pm (1,0)\quad \textrm{and}\quad x\underset{j}\sim y.\\
                       %  x=(i,j),\;\;  %  y=(i+1,j), \;\;\textrm{and}\;\;
                       %  \frac k{\|j|} \le i <\frac{k+1}{\|j\|}-1\;  \textrm{for some}\;\; k,\;\; y-x=(1,0);\\
                0 \quad & \textrm{otherwise}.
\end{aligned} \right.
$$
In other words,  within a horizontal level $\De_n(j)$,
non-equivalent sites do not interact, while
the  equivalent neighbors   interact with infinite strength.
Infinite interaction $J(x,y)$ is interpreted as  condition  $\sigma(x)=\sigma(y)$,
so that the equivalent sites $x$ and $y$ can be identified.
This leads to the diamond graph $\Gamma_n$.
(Figure \ref{fig 4x5}                       
illustrates the $4\times 5$ rectangle,
with the vertical solid lines representing interaction $J$ and horizontal dash lines representing the infinite interaction.)
Moreover, the Hamiltonian (\ref{H on 2D lattice}) takes the form of (\ref{Hamiltonian}):
$$
H_n(\sigma)=- \frac J{2} \sum_{(x,y)\in \EE_n^v} \sigma(x)\sigma(y) - \frac{h}{2} \sum_{(x,y)\in \EE^v_n}  (\sigma(x)+\si(y)),
$$
where $\EE_n^v$ is the set of vertical edges.

Thus, we have obtained the diamond model.

\subsection{Gibbs states}
The thermodynamic limit of the Gibbs distributions for $DHL$  was
studied by Griffiths and Kaufman \cite {GK} and by Bleher and  Zalys \cite{BZ2}.
 As noticed in \cite {GK}, there are uncountably many non-isomorphic injective limits of the
 hierarchical lattice $\Gamma_n$ as $n\to\infty$, which give rise to uncountably many non-isomorphic
 infinite hierarchical lattices.                                 % \note{elaborate?}
In \cite{BZ2}, limit Gibbs states on the infinite hierarchical lattices are constructed.
It is proven that for any non-degenerate infinite hierarchical lattice,
if $T<T_c$ and $h=0$ then there exist exactly two pure infinite Gibbs states,
 in the sense of Dobrushin-Lanford-Ruelle, while if $T\ge T_c$ or $h\not=0$
then the infinite Gibbs state is unique.

\section{Open Problems}\label{APP:PROBLEMS}

\begin{problem}[\bf Critical exponents for the low-temperature intervals]\label{PROB:STABLE_INTERVALS}

A consequence of Theorem \ref{attractors} and Corollary
\ref{COR:LYAP_LOG2} is that unstable Lyapunov exponents $\chi^u$
exist at almost every point of $\Cphystl$.  However, the union of all endpoints
of the intervals from $O_t =\WW^s(\BOTTOMphys) \cap \T_t$, taken over all $t
\in [t_c,1)$, has measure zero. Thus, we do not know that ``most endpoints''
have Lyapunov exponents.
Do Lyapunov exponents and hence, by Proposition \ref{PROP:LYAPUNOV_WEAK_CRIT}, weak critical exponents exist at the endpoints of the intervals from $O_t$?

\end{problem}

\begin{problem}[\bf Principal stable tongues]\label{PROB:STABLE_TONGUES}

Consider the principal stable tongues $\Tongue_\pm$.  Are $\Tongue_\pm$ bounded by high-temperature hairs of some positive length?

If this is the case, the discussion from Problem \ref{PROB:STABLE_INTERVALS} gives that for high enough
values of $t$, the critical exponents $\sigma^h = 1$  at the endpoints of the intervals formed by $\Tongue_\pm \cap \TT_t$.

The question can be asked for any of the stable tongues.
\end{problem}

\begin{problem}[\bf Endpoints of hairs]\label{PROB:HIGH_TEMP_ENDPOINTS}  Recall
the set $\epoints$ of endpoints to the high-temperature hairs that constructed in \S
\ref{SEC:HIGH_TEMP_ENDPOINTS}.  According to Corollary
\ref{COR:ENDPOINTS_MEASURE_ZERO}, $\epoints$ has Lebesgue measure zero.

\begin{itemize}
\item[(a)] What is the Hausdorff dimension of $\epoints$? 
\item[(b)] Do any of the high-temperature hairs contain their endpoints, i.e. is there any endpoint within $\WW^s(\TOPphys)$?
\item[(c)] Is $\WW^s(\TOPphys)$ a ``straight hairy brush'' in the sense of \cite{AARTS_OVERSTEEGEN}?  One consequence would be that the endpoints of the high-temperature hairs must accumulate from both sides to every point on every high-temperature hair.  In particular, this would give a negative solution to Problems \ref{PROB:STABLE_TONGUES} and \ref{PROB:STABLE_INTERVALS}(b).
\end{itemize}
  
These questions are partly motivated by the structure of the {\it Devaney hairs} for the exponential maps,
see \cite{DT,McM,Karpinska}. 
\end{problem}

\begin{problem}[\bf Control of expansion]\label{PROB:CONTROL_EXPANSION} Do we have 
\begin{eqnarray}\label{EQN:BOUNDS_ON_EXPANSION}
  \limsup \frac{\log d(\phi \circ \Rphys^n)(v)}{n} \leq 4
\end{eqnarray}
\noindent
for any $v \in \KK(x)$ based at any $x \in \Cphystl$?
% By Proposition \ref{PROP:LYAPUNOV_WEAK_CRIT}, the lower bound gives
%$\sigma^h(x) \in [0,1]$ for all horizontal critical exponents (that
%exist).\note{Why do we want critical exponents in this range, if some of them
%for the Ising model in 2D are not predicted to fall there.} 
This bound
would give continuity in $\phi$ of the density $\rho_t(\phi)$ by an estimate similar
to the proof of Propositions \ref{crit exp for periodic pts} and \ref{PROP:LYAPUNOV_WEAK_CRIT}.
Notice, however  that it does not hold for the first iterate: 
%We saw in \S \ref{SUBSEC:PARTIAL_HYP} as $x$ approaches $\INDphys_\pm$ along the certain parabolas 
% %$1-t = \frac{1}{2}(\phi\pm \frac{\pi}{2})^2$ 
%the horizontal expansion of certain vectors in $\partial \KK^h(x)$ limits to $1$.
One can see from  (\ref{DR near alpha}) that if $x$ approaches $\INDphys_\pm$ at a definite slope $\tau/\eps = \bar \kappa$, 
then the horizontal expansion of vectors in $\KK^h(x)$ blows up like $1/\tau$.  
\end{problem}

\begin{problem}[\bf Critical temperatures and regularity] \label{PROB:CRITICAL_TEMPURATURES_AND_REGULARITY}
Given $\gamma \in \FF^c$, there are $0 <
t^-_c(\gamma)  \leq t^+_c(\gamma) \leq 1$ so that points on $\gamma$
below $t=t^-_c(\gamma)$ are in $\WW^s(\BOTTOMphys)$ and points above
$t=t^+_c(\gamma)$ are in $\WW^s(\TOPphys)$.  We call the points on $\gamma$
having $t^-_c(\gamma)  \leq t \leq t^+_c(\gamma)$ the {\em $\gamma$-critical
temperatures}.

\begin{itemize}
\item[(a)]
Is there a unique $\gamma$-critical temperature $t_c(\gamma) := t^-_c(\gamma) = t^+_c(\gamma)$ on each $\gamma \in \FF^c$?

\item[(b)] The union of $\gamma$-critical temperatures over all $\gamma \in \FF^c$ is invariant under $\Rphys$.  Is there a ``natural'' invariant measure $\nu_{\rm crit}$ supported on this set?  What is the entropy of this measure?
% \note{$\mu$ as a notation for measures collides with $\mu_c$ for the LYF current.}

\item[(c)]
It is a consequence of Propositions \ref{PROP:STABLE_MANIFOLD} and
\ref{PROP:ETA_BASINS_LAMINATED} that each leaf $\gamma \in \FF^c$ is real
analytic below $t^-_c(\gamma)$ and $C^1$ above $t^+_c(\gamma)$.  Does $\gamma$
have only finite smoothness within the range of $\gamma$-critical temperatures?
% \note{Does every non-endpoint in $\WW^s(\TOPphys)$ eventually land in one of
% the $\FF^s_\eta(\TOPphys)$?}

\item[(d)] Proposition \ref{PROP:NON_ANALYTIC_LEAVES} gives a partial answer to
the previous question for periodic leaves.  A natural open question here is
whether a periodic leaf can contain a neutral periodic point?

\end{itemize}

Cylinder maps having  property (a) on almost every leaf are constructed in
\cite[\S 3]{BM}.  To ask questions (a) and (c) for almost
every leaf in our situation, one must first choose a transverse invariant
measure on $\FF^c$.  With respect to $\mu_t$, almost every leaf is in the union
of stable tongues and the result is trivial.
The question is more interesting with respect to the
transverse measure induced on $\FF^c$ by Lebesgue measure on $\TOPphys$.
\end{problem}

%\begin{problem}[\bf Periodic leaves of $\FF^c$]\label{PROB:PERIODIC_LEAVES}
%Problems \ref{PROB:CRITICAL_TEMPURATURES_AND_REGULARITY}(a,c) are interesting for periodic leaves $\gamma$ since $t^+_c(\gamma) < 1$.
%Furthermore, periodic leaves may be more easily studied by means of complexification.
%\end{problem}

\section{Table of notation}\label{APP:NOTATION}

In the course of this paper
various objects appear in parallel in two coordinate systems:  
the ``physical coordinates '' $(z,t)$  and the affine coordinates
$(u,v) \mapsto [u:1:v]$.\footnote{If not to count homogeneous coordinates $(U:V:W)$ and angular
coordinates $(\phi,t)$ as systems in their own right.}  They are
related by the semi-conjugacy $\correspond$ from \S \ref{SEC:STRUCTURE}.
% with their corresponding renormalization operators $\Rmig$ and $\Rphys$ which are related
% by a semiconjugacy (see \S \ref{SUBSEC:SEMICONJUGACY}),   
We have attempted (not fully consistently) to use similar  notation for
corresponding objects, roughly using calligraphic and Greek symbols in the
physical coordinates and the corresponding non-calligraphic
and Latin symbols in the  affine coordinates. 
For reader's convenience, some of the notation is collected in the following table:

\begin{center}
\begin{tabular}{|l|l|l|}
\hline
Object & Physical coordinates  & Affine coordinates  \\
       &  $(z,t)=$		   & $(u,w)$  \\
\hline
Renormalization map & $\Rphys$ & $\Rmig$ \\
Invariant cylinder & $\Cphys=\T\times [0,1]$ & $\Cmig=\{w=\bar u,\ |u|\geq 1\}$ \\
Topless cylinder  & $\Cphys^0 =\T\times [0,1)$ & $\Cmig^0 =\{w=\bar u,\ |u| > 1\}$   \\
Horizontal/vertical algebraic cone field & $\KK^{\hor/\ver}$ & $K^{\hor/\ver}$ \\   
Modified horizontal/vertical cone fields & $\KK^{h/v}$ & $K^{h/v}$ \\
Strong separatrix & $\LLzero=\{t=0\}$ & $\Lzero =\mbox{line at infinity}$\\
Weak separatrix & $\LLone=\{t=1\}$ & $\Lone =\{uw=1\}$\\
Bottom of the cylinder & $\BOTTOMphys=\T\times \{0\}$ & $\BOTTOMmig\subset \Lzero, \ |w/u|=1$ \\
Top of the cylinder & $\TOPphys=\T\times \{1\}$ & $\TOPmig=\T$ \\
Main indeterminacy pts  & $\INDphys_\pm=(\pm i, 1)$ &  $\INDmig_\pm= \pm (i,- i)$ \\
Accidental indeterminacy pts & $\gamma, \B0$ & none \\
% Fixed points on the real $t$ line & $\FIXphys_0, \FIXphys_c, \FIXphys_1$ & $\FIXmig_0, \FIXmig_c, \FIXmig_1$ \\
Low temp fixed point  & $\FIXphys_0 = (1,0)$   & $\FIXmig_0 = [1:0:1] \in \Lzero$ \\
Critical temp fixed point  & $\FIXphys_c \approx (1,0.2956) $   & $\FIXmig_c \approx (3.3830,3.3830)$ \\
High temp fixed point  & $\FIXphys_1 = (1,1)$   & $\FIXmig_1= (1,1)$ \\
Attracting fixed points in $\CP^2$ & $\CFIXphys =(0,1),\CFIXphys'=(\infty,0)$ & $\CFIXmig=(\infty,0),\CFIXmig'=(0,\infty)$ \\
Principal LY locus & $\Sphys=\{z^2+2tz+1=0\}$ & $\Smig=\{u+w=-2\}$ \\
Blow-up locus & $\Icurve = \{z^2+4zt-2z+1=0\}$ & $\Imig = \{u-w = 2i\}$\\
 & \,\,\,\,  $= \{t=\sin^2 \phi/2 \}$ & \\
\hline

\end{tabular}
\end{center}

\comm{*****
\msk 

Because of the physical motivation, in some cases we prefer to work in the physical coordinates $(z,t)$ and particularly
$(\phi,t) \in \Cphys$.   We use three different cone-fields on, or in a complex neighborhood of, $\Cphys$: \note{Inconsistent notation for weak cone field: doesn't have exponent $^h$.}
 
\msk

\begin{center}
\begin{tabular}{|l|l|l|}
\hline
Name of conefield & Notation  & Initially defined in \\
\hline
Horizontal and vertical algebraic conefield & $\KK^{h,v}$ & \S \ref{sec: alg cone field} \\
Non-degenerate conefield near $\TOPphys \sm \{\INDphys_\pm\}$ & $\tl\KK^{h,v}$ & \S \ref{SUBSEC:FOLIATION_NEAR_T} \\
Weak conefield (defined in a complex neighborhood of $\TOPphys \sm \{\INDphys_\pm\}$) & $\KSQRT^c$ & \S \ref{SUBSEC:COMPLEXIFICATION_CONES} \\
\hline
\end{tabular}
\end{center}
*****************}

\msk
The following is further notation specific to $\Cphys$:

\begin{center}
\begin{tabular}{|l|l|l|}
\hline
Object & Physical coordinates  & Initially defined in \\
\hline
Topless cylinder  & $\Cphystl$ & \S \ref{SUBSEC:MAP_ON_CYL} \\
Bottomless cylinder & $\Cphysbl$ & \S \ref{SUBSEC:MAP_ON_CYL} \\
Low temperature cylinder & $\Cphyslow$ & \S \ref{SUBSEC:BASIN_BOTTOM} \\
Primary stable tongues       &  $\Tongue(\INDphys_\pm)$               &  \S \ref{stable tongues sec}  \\
Secondary stable tongues   & $\Tongue^{n}_k(\alpha)$   & \S \ref{stable tongues sec} \\
Basins of attraction for $\TOPphys$ with prescribed control & $\WW^s_\eta(\TOPphys)$, $\WW^s_0(\TOPphys)$ & \S \ref{SUBSEC:BASIN_T},  
                      \S \ref{SEC:TWO BASINS} \\
Central foliation & $\FF^c$ & \S \ref{SEC:VERTICAL_FOLIATION} \\
Horizontal critical exponent  & $\si^h$ & \S \ref{SUBSEC:CRITICAL_EXPONENTS} \\
Vertical critical exponent & $\si^v$ & \S \ref{SUBSEC:CRITICAL_EXPONENTS} \\

\hline

\end{tabular}
\end{center}

\bsk

%\begin{center} DYNAMICS and COMPLEX GEOMETRY \end{center}
\renewcommand\refname{References from  dynamics and complex geometry}

\msk

%\begin{center} MATHEMATICAL  PHYSICS \end{center}
\renewcommand\refname{References from  mathematical physics}

\end{document}